\documentclass[DIV=12]{scrartcl}

\usepackage{amsmath}
 \allowdisplaybreaks[1]
 \numberwithin{equation}{section}
\usepackage{amssymb}
\usepackage{amsthm}
 \newtheorem{thm}{Theorem}[section]
 
 \newtheorem{lem}[thm]{Lemma}
 \newtheorem{prop}[thm]{Proposition}
\theoremstyle{definition}
 \newtheorem{defn}[thm]{Definition}
\theoremstyle{remark}
 \newtheorem{nota}[thm]{Notation}
 \newtheorem{exam}[thm]{Example}
 \newtheorem{rem}[thm]{Remark}
\usepackage[english]{babel}
\usepackage[noisbn,noissn,fixlanguage]{babelbib}
 \setbtxfallbacklanguage{english}
\usepackage{cite}%
\usepackage[T1]{fontenc}
\usepackage{hyperref}
\usepackage{lmodern}
\usepackage{mathscinet}
\usepackage{mathtools}
 \mathtoolsset{mathic=true}
\usepackage{microtype}
\usepackage{pmat}
\usepackage{url}

\usepackage{170arxiv}

\title{A Potapov-type approach to a truncated matricial Stieltjes-type power moment problem}
\author{B.~Fritzsche \and B.~Kirstein \and C.~M\"adler \and T.~Makarevich}

\begin{document}
\maketitle

\begin{abstract}
 The paper gives a parametrization of the solution set of a matricial Stieltjes-type truncated power moment problem in the non-degenerate and degenerate cases.
 The key role plays the solution of the corresponding system of Potapov's fundamental matrix inequalities.
\end{abstract}

\begin{description}
 \item[Keywords:] Stieltjes moment problem, Potapov's fundamental matrix inequalities, Herglotz--Nevanlinna functions, Stieltjes functions
\end{description}

\section{Introduction and preliminaries}
 The starting point of studying power moment problems on semi-infinite intervals was the famous two part memoir of  T.~J.~Stieltjes~\zitas{MR1508747,MR1508159}.
 A complete theory of the treatment of power moment problems on semi-infinite intervals in the scalar case was developed by M.~G.~Krein in collaboration  with A.~A.~Nudelman (see~\zitaa{MR0044591}{\cSect{10}},~\zita{MR0233157},~\zitaa{MR0458081}{\cchap{V}}).
 What concerns an operator-theoretic treatment of the power moment problems named after Hamburger and Stieltjes and its interrelations, we refer the reader to Simon~\zita{MR1627806}.
 
 In the 1970's, V.~P.~Potapov developed a special approach to discuss matrix versions of classical interpolation and moment problems.
 The main idea of his method is based on transforming such problems into equivalent matrix inequalities with respect to the L\"owner semi-ordering.
 Using this strategy, several matricial interpolation and moment problems could successfully be handled (see, \teg{}~\zitas{MR1362524,CR01,MR2222521,MR1152328,Dyu01,MR2570113,MR645305,MR752057,Mak14,MR1395706,Thi06,Dub,MR686076,G1,MR722914,MR645308,MR701996,MR734686,MR738449,MR777324,MR1473259,MR703593,MR1009144}).
 L.~A.~Sakhnovich enriched Potapov's method by unifying the particular instances of Potapov's procedure under the framework of one type of operator identities~\zitas{MR1310360,MR1631843,MR1722780}.

 Matrix versions of the classical Stieltjes moment problem were studied by Adamyan/Tkachenko~\zitas{MR2155645,MR2215856}, And\^o~\zita{MR0290157}, Bolotnikov~\zitas{MR975671,MR1362524,MR1433234}, Bolotnikov/Sakhnovich~\zita{MR1722780}, Chen/Hu~\zita{MR1807884}, Chen/Li~\zita{MR1670527}, Dyukarev~\zitas{Dyu81,MR686076}, Dyukarev/Katsnelson~\zitas{MR645305,MR752057}, and Hu/Chen~\zita{MR2038751}.
 The considerations of this paper deal with the more general case of an arbitrary semi-infinite interval \(\rhl \), where \(\ug\) is an arbitrarily given real number.

 In order to formulate the moment problem,  we are going to study, we first review some notation.
 Throughout this paper, let \(p\)\index{p@\(p\)} and \(q\)\index{q@\(q\)} be positive integers.
 Let \(\C\)\index{c@\(\C\)}, \(\R\)\index{r@\(\R\)}, \(\NO\)\index{n@\(\NO\)}, and \(\N\)\index{n@\(\N\)} be the set of all complex numbers, the set of all real numbers, the set of all \tnn{} integers, and the set of all positive integers, respectively.
 For every choice of \(\upsilon,\omega\in\R\cup\set{-\infty,\infp}\), let \(\mn{\upsilon}{\omega}\)\index{z@\(\mn{\alpha}{\beta}\)} be the set of all integers \(k\) for which \(\upsilon\leq k\leq\omega\) holds.
 If \(\cX\) is a \tne{} set, then \(\cX^\pxq\)\index{\(\cX^\pxq\)} stands for the set of all \tpqa{matrices} each entry of which belongs to \(\cX\), and \(\cX^p\)\index{\(\cX^p\)} is short for \(\cX^\xx{p}{1}\).
 If \((\Omega,\gA)\) is a measurable space, then each countably additive mapping whose domain is \(\gA\) and whose values belong to the set \(\Cggq\)\index{c@\(\Cggq\)} of all \tnnH{} complex \tqqa{matrices} is called a \tnnH{} \tqqa{measure} on \((\Omega, \gA)\).
 By  \(\Mggqa{\Omega,\gA}\) we denote the set of all \tnnH{} \tqqa{measures} on \((\Omega,\gA)\). 
 For the integration theory for \tnnH{} measures, we refer to~\zitas{MR0080280,MR0163346}.
 If \(\mu =\mat{\mu_{jk}}_{j,k=0}^q\) is a \tnnH{} \tqqa{measure} on a measurable space \((\Omega, \gA)\) and if \(\mathbb{K}\in\set{\R,\C}\), then we use \(\LoaaaK{1}{\Omega}{\gA}{\mu}\) to denote the set of all Borel-measurable functions  \(f\colon\Omega \to\mathbb{K}\) for which the integral exists, \tie{}, that \(\int_{\Omega} \abs{f}\dif\tilde{\mu}_{jk} < \infp\) for every choice of \(j\) and \(k\) in \(\mn{1}{q}\), where \(\tilde{\mu}_{jk}\) is the variation of the complex measure  \(\mu_{jk}\).
 If \(f\in\LoaaaK{1}{\Omega}{\gA}{\mu}\), then let \(\int_A f\dif\mu \defeq\mat{\int_\Omega 1_A f\dif\mu_{jk}}_{j,k=1}^q\) for all \(A\in\gA\) and we  will also write \(\int_A f(\omega) \mu (\dif\omega)\) for this integral.

 Let \(\BorR\)\index{b@\(\BorR\)} (\tresp{}\ \(\BorC \)) be the \(\sigma\)\nobreakdash-algebra of all Borel subsets of \(\R\) (\tresp{}\ \(\C \)).
 For all \(\Omega\in\BorR\setminus\set{\emptyset}\), let \(\BorO\)\index{b@\(\Bori{\Omega}\)} be the \(\sigma\)\nobreakdash-algebra of all Borel subsets of \(\Omega\), let \(\Mggqa{\Omega} \defeq  \Mggqa{\Omega,\BorO}\)\index{m@\(\Mggqa{\Omega}\)}  and, for all \(\kappa\in\NOinf \), let \(\Mgguqa{\kappa}{\Omega}\)\index{m@\(\Mgguqa{\kappa}{\Omega}\)} be the set of all \(\sigma\in\Mggqa{\Omega}\) such that for all \(j\in\mn{0}{\kappa}\) the function \(f_j\colon\Omega\to\C\) defined by \(f_j (t) \defeq  t^j\) belongs to \(\LoaaaC{1}{\Omega}{\BorO}{\sigma}\).
 If \(\kappa\in\NOinf \) and if  \(\sigma\in\Mgguqa{\kappa}{\Omega}\), then we set
\begin{align}\label{FR1148}
 \suo{j}{\sigma}
 &\defeq\int_\Omega t^j\sigma(\dif t)&\text{for each }j&\in\mn{0}{\kappa}.
\end{align}
 The following matricial power moment problem lies in the background of our considerations:

\begin{Problem}[\mprob{\Omega}{m}{\lleq}]
 Let \(\Omega\in\BorR\setminus\set{\emptyset}\), let \(m\in\NO\), and let \(\seqs{m}\) be a sequence of complex \tqqa{matrices}.
 Describe the set \(\Mggqaakg{\Omega}{\seqs{m}}\)\index{m@\(\Mggqaakg{\Omega}{\seqs{m}}\)} of all \(\sigma\in\Mgguqa{m}{\Omega}\) for which the matrix \(s_{m}-\suo{m}{\sigma}\) is \tnnH{} and for which, in the case \(m>0\), moreover \(\suo{j}{\sigma}=s_{j}\) is fulfilled for all \(j\in\mn{0}{m-1}\).\index{m@\mprob{\Omega}{m}{\leq}}
\end{Problem}

 Note that we also sometimes turn our attention to the following power moment problem:

\begin{Problem}[\mprob{\Omega}{\kappa}{=}]
 Let \(\Omega\in\BorR\setminus\set{\emptyset}\), let \(\kappa\in\NOinf \), and let \(\seqska \) be a sequence of complex \tqqa{matrices}.
 Describe the set \(\Mggqaag{\Omega}{\seqska }\)\index{m@\(\Mggqaag{\Omega}{\seqska }\)} of all \(\sigma\in\Mgguqa{\kappa}{\Omega}\) for which \(\suo{j}{\sigma}=s_{j}\) is fulfilled for all \(j\in\mn{0}{\kappa}\).\index{m@\mprob{\Omega}{\kappa}{=}}
\end{Problem}

 The considerations of this paper are mostly concentrated on the case that the set \(\Omega\) is a one-sided bounded and closed infinite interval of the real axis.
 Such moment problems are called to be of Stieltjes type.
 
 The key role for solving the moment problem \mprob{\rhl}{m}{\lleq}, where \(\ug\) is an arbitrarily given real number, is \rthm{MT4315} below.
 It will turn out that the solution set of the moment problem (obtained via Stieltjes transformation) coincides with the solution set of a certain system of Potapov's fundamental matrix inequalities.
 The considerations in this paper are aimed to solve these inequalities.
 In \rsec{S1313}, we give a parametrization of the solution set \(\MggqKskg{2n+1}\) of Problem \mprob{\rhl}{2n+1}{\lleq}, where \(\ug\) is an arbitrarily given real number and where \(n\) is an arbitrarily given \tnn{} integer.
 Note that Problem~\mprob{\rhl}{2n}{\lleq} can be discussed by similar methods.
 This will be done somewhere else.
 
 In \rsec{sec2}, we recall necessary and sufficient conditions of solvability of the moment problems in question.
 In \rsec{sec3}, we reformulate these problems in the language of certain matrix-valued functions.
 \rsec{S1415} is aimed at showing that every solution of the moment problem fulfills necessarily the corresponding system of Potapov's fundamental matrix inequalities. Some integral estimates for the scalar case are given in \rsec{S1639}.
 In \rsec{S0846}, we will prove that each solution of the system of Potapov's fundamental matrix is a solution of the moment problem as well.
 \rsec{S1709} is aimed to give some identities for block Hankel matrices.
 In \rsec{S1710}, we study special subspaces of \(\Cq \), so-called Dubovoj  subspaces.
 Particular matrix polynomials in connection with a signature matrix are considered in \rsec{S1711}.
 In \rsecss{S1223}{S1321}, we study distinguished classes of meromorphic functions, which occur as parameters in our description of the solution set, which is stated in \rsec{S1313}.
 
 At the end of this section, let us introduce some further notations, which are useful for our considerations.
 We will write \(\Iq\) for the identity matrix in \(\Cqq \), whereas \(\Opq \) is the null matrix belonging to \(\Cpq \).
 If the size of the identity matrix or the null matrix is obvious, then we will also omit the indexes.
 The notations \(\CHq \) and \(\Cggq \) stand for the set of all \tH{} complex \tqqa{matrices} and the set of all \tnnH{} complex matrices, respectively.
 If \(A\) and \(B\) are complex \tqqa{matrices}, then we will write \(A\lleq B\) or \(B\lgeq A\) to indicate that \(A\) and \(B\) are  \tH{} matrices such that the matrix \(B-A\) is \tnnH{}.
 For each \(A\in\Cpq \), let \(\nul{A}\) be the null space of \(A\), let \(\ran{A}\) be the column space of \(A\),  and let \(\rank  A\) be the rank of \(A\).
 For each \(A\in\Cqq \), we will use \(\Re  A\) and \(\Im A\) to denote the real part of \(A\) and the imaginary part of \(A\), respectively:
 \(\Re  A \defeq  \frac{1}{2} (A + A^\ad)\) and \(\Im A \defeq  \frac{1}{2\iu} (A-A^\ad)\).
 Furthermore, for each \(A\in\Cpq \), let \(\normF{A}\) be the Frobenius norm of \(A\) and let \(\normS{A}\) be the operator norm of \(A\).
 For each \(x\in\Cq\), we write \(\normE{x}\) for the Euclidean norm of \(x\).
 If \(A\in\Cqq \), then \(\det A\) stands for the determinant of \(A\). 
 
 For each complex \tpqa{matrix} \(A\), let \(A\{1\}\defeq  \setaca{X \in \Cqp}{AXA=A}\).
 Obviously, for each \(A\in\Cpq \), the Moore--Penrose inverse \(A^\mpi \) of \(A\) belongs to \(A\{1\}\).
 
 If \(n\in\N \), if \((p_j)_{j=1}^n\) is a sequence of positive integers, and if \(x_j\in\Coo{p_j}{q}\) for each \(j\in\mn{1}{n}\), then let \(\col (x_j)_{j=1}^n  \defeq \smat{x_1\\ x_2 \\ \vdots \\ x_n}\).
 If \(n \in \N\), if \((q_k)_{k=1}^n\) is a sequence of positive integers, and if \(y_k\in\Coo{p}{q_k}\) for each \(k\in\mn{1}{n}\), then let \(\row (y_k)_{k=1}^n \defeq\mat{y_1, y_2,\dotsc, y_n}\).
 If \(n\in\N\), if \((p_j)_{j=1}^n\) and \((q_j)_{j=1}^n\) are sequences of positive integers, and if \(A_j\in\Coo{p_j}{q_j}\) for every choice of \(j\) in \(\mn{1}{n}\), then let \(\diag (A_1, A_2,\dotsc, A_n)\defeq\mat{\Kronu{jk} A_j}_{j,k=1}^n\), where \(\Kronu{jk}\) is the Kronecker delta:
 \(\Kronu{jk} \defeq 1\) in the case \(j=k\) and \(\Kronu{jk} \defeq 0\) if \(j\ne k\).
 We also use the notation \(\diag (A_j)_{j=1}^n\) instead of \(\diag (A_1,A_2,\dotsc, A_n)\).
 For each \(n\in\N\) and each \(A\in\Cpq \), we will also write \(\Iu{n}\otimes A\) for \(\diag (A)_{j=1}^n\).
 
 If \(\mathcal{M}\) is a \tne{} subset of \(\Cq \), then let \(\mathcal{M}^\orth\) be the set of all vectors in \(\Cq \) which are orthogonal to \(\mathcal{M}\) (with respect to the Euclidean inner product).
 If \(\mathcal{X}\), \(\mathcal{Y}\), and \(\mathcal{Z}\) are \tne{} sets with \(\mathcal{Z}\subseteq\mathcal{X}\) and if \(f\colon\mathcal{X}\to\mathcal{Y}\) is a mapping, then \(\Rstr_\mathcal{Z}f\) stands for the restriction of \(f\) onto \(\mathcal{Z}\).
 
 Furthermore, let \(\uhp  \defeq  \setaca{z\in\C}{\Im z\in (0,\infp)}\) and let \(\lhp  \defeq  \setaca{z\in\C}{\Im z\in (-\infty, 0)}\).

\section{On the solvability of matricial power moment problems}\label{sec2}
 In this section, we recall necessary and sufficient conditions for the solvability of the Stieltjes moment problems \mprob{\rhl }{m}{\lleq} and \mprob{\rhl }{m}{=}, where \(\ug\) is an arbitrarily given real number and where \(m\) is an arbitrarily given \tnn{} integer.
 First we introduce certain sets of sequences of complex \tqqa{matrices}, which are determined by the properties of particular block Hankel matrices built of them.
 For each \(n\in\NO \), let \(\Hggq{2n}\) be the set of all sequences \(\seqs{2n}\) of complex \tqqa{matrices} such that the block Hankel matrix \(  \Hu{n} \defeq\mat{s_{j+k}}_{j,k=0}^n\)   \index{h@\(\Hu{n}\)}is \tnnH{}.
 Furthermore, let \(\Hggqinf\)\index{h@\(\Hggqinf\)} be the set of all sequences \(\seqsinf \) of complex \tqqa{matrices} such that, for all \(n\in\NO\), the sequence \(\seqs{2n}\) belongs to \(\Hggq{2n}\).
 The elements of the set \(\Hggq{2\kappa}\),  where \(\kappa\in\NOinf \), are called \emph{\tHnnd} sequences.
 For all \(n\in\NO\), let \(\Hggeq{2n}\)\index{h@\(\Hggeq{2n}\)} be the set of all sequences \(\seqs{2n}\) of complex \tqqa{matrices} for which there are matrices \(s_{2n+1}\in\Cqq\) and \(s_{2n+2}\in\Cqq\) such that \(\seqs{2(n+1)}\) belongs to \(\Hggq{2(n+1)}\).
 Furthermore, for all \(n\in\NO\), we will use \(\Hggeq{2n+1}\)\index{h@\(\Hggeq{2n+1}\)} to denote the set of sequences \(\seqs{2n+1}\) of complex \tqqa{matrices} for which there is some \(s_{2n+2}\in\Cqq\) such that \(\seqs{2(n+1)}\) belongs to \(\Hggq{2(n+1)}\).
 For all \(m\in\NO\), the elements of the set \(\Hggeq{m}\) are called \emph{\tHnnde} sequences.
 For technical reasons, we set \(\Hggeqinf\defeq\Hggqinf\)\index{h@\(\Hggeqinf\)}.
 Observe that the solvability of the matricial Hamburger moment problems can be characterized  by the introduced classes of sequences of complex \tqqa{matrices}:

\begin{thm}[see, \teg{}~\zitaa{MR1624548}{\cthm{3.2}} or~\zitaa{MR2570113}{\cthm{4.16}}]\label{T1.1}
 Let \(n\in\NO\) and let \(\seqs{2n}\) be a sequence of complex \tqqa{matrices}.
 Then \(\MggqRskg{2n}\neq\emptyset\) if and only if \(\seqs{2n}\in\Hggq{2n}\).
\end{thm}

\begin{thm}[see~\zitaa{MR2570113}{\cthm{4.17}},~\zitaa{MR2805417}{\cthm{6.6}}]\label{T1.2}
 Let \(\kappa\in\NOinf \) and let \(\seqska \) be a sequence of complex \tqqa{matrices}.
 Then \(\MggqRsg{\kappa}\neq\emptyset\) if and only if \(\seqska \in\Hggeq{\kappa}\).
\end{thm}
 
 Let \(\ug\in\C\), let \(\kappa\in\Ninf \), and let \(\seqska \) be a sequence of complex \tpqa{matrices}.
 Then let the sequence \(\seqsa{\kappa-1}\) be defined by
\begin{align}\label{FR1733}
 \sau{j}&\defeq-\ug s_{j}+s_{j+1}&\text{for all }j&\in\mn{0}{\kappa-1}.
\end{align}
 The sequence  \(\seqsa{\kappa-1}\) is called the \emph{sequence generated from \(\seqska \) by right-sided \(\ug\)\nobreakdash-shifting}.
 (An analogous left-sided version is discussed in~\zitaa{MR3014201}{\cdefn{2.1}}.)
 The sequence \(\seqsa{\kappa-1}\) is used to define further sets of sequences of complex matrices, which are useful to discuss the Stieltjes moment problems we consider.
 Let \(\Kggq{0}\defeq\Hggq{0}\)\index{k@\(\Kggq{0}\)}.
 For every choice of \(n\in\N\), let \( \Kggq{2n}   \defeq\setaca{\seqs{2n}\in\Hggq{2n}}{\seqsa{2(n-1)}\in\Hggq{2(n-1)}}\).\index{k@\(\Kggq{2n}\)}
 For all \(m\in\NO\), by \(\seqset{m}{\Cqq}\) we denote the set of all sequences \(\seqs{m}\) of complex \tqqa{matrices}.
 Then we set  \(\Kggq{2n+1} \defeq  \setaca{\seqs{2n+1} \in  \seqset{2n+1}{\Cqq}}{\set{\seqs{2n}, \seqsa{2n}} \subseteq \Hggq{2n}}\).  \index{k@\(\Kggq{2n+1}\)}
 For all \(m\in\NO\), let \(\Kggeq{m}\)\index{k@\(\Kggeq{m}\)} be the set of all sequences \(\seqs{m}\) of complex \tqqa{matrices} for which there exists a complex \tqqa{matrix} \(s_{m+1}\) such that \(\seqs{m+1}\) belongs to \(\Kggq{m+1}\).
 Obviously, we  have  \( \Kggeq{2n}   =\setaca{\seqs{2n}\in\Hggq{2n}}{\seqsa{2n-1}\in\Hggeq{2n-1}}\) for all \(n\in\N\) and \(\Kggeq{2n+1} =\setaca{\seqs{2n+1}\in\Hggeq{2n+1}}{\seqsa{2n}\in\Hggq{2n}}\) for all \(n\in\NO\).
 
\breml{R1440}
 Let \(\ug\in\R\) and let \(m\in\NO\).
 Then \(\Kggeq{m} \subseteq \Kggq{m}\).
 Furthermore, if  \(\seqs{m}\in\Kggq{m}\) (\tresp{}\ \(\Kggeq{m}\)), then we easily see that \(\seqs{\ell}\in\Kggq{\ell}\) (\tresp{} \(\seqs{\ell}\in\Kggeq{\ell} \)) holds true for all \(\ell\in\mn{0}{m}\).
\erem

 In view of \rrem{R1440}, for all \(\ug\in\R\), let \(\Kggqinf \)\index{k@\(\Kggqinf\)} be the set of all sequences \(\seqsinf \) of complex \tqqa{matrices} such that \(\seqs{m}\) belongs to \(\Kggq{m}\) for all \(m\in\NO\), and let \(\Kggeqinf \defeq\Kggqinf \)\index{k@\(\Kggeqinf\)}.
 For all \(\kappa\in\NOinf \), we call a sequence \(\seqs{\kappa}\) \emph{\taSrsnnd{\ug}} (\tresp{}\ \emph{\taSrsnnde{\ug}}) if it belongs to \(\Kggq{\kappa}\) (\tresp{}\ to \(\Kggeq{\kappa}\)).
 Note that left versions of these notions are used in~\zitaa{MR3014201}{\cdefn{1.3}}.

 Using the introduced sets of sequences of complex \tqqa{matrices}, we are able to recall solvability criterions of the problems~\mprob{\rhl }{m}{\lleq} and \mprob{\rhl }{m}{=}:

\begin{thm}[\zitaa{MR2735313}{\cthm{1.4}}]\label{T1122}
 Let \(\ug\in\R\), let \(m\in\NO\), and let \(\seqs{m}\) be a sequence of complex \tqqa{matrices}.
 Then \(\MggqKskg{m}\neq\emptyset\) if and only if \(\seqs{m}\in\Kggq{m}\).
\end{thm}

\begin{thm}\label{T1121}
 Let \(\ug\in\R\), let \(\kappa\in\NOinf \), and let \(\seqska \) be a sequence of complex \tqqa{matrices}.
 Then \(\MggqKsg{\kappa}\neq\emptyset\) if and only if \(\seqska \in\Kggeq{\kappa}\).
\end{thm}
 
 In the case \(\kappa\in\NO\), a proof of \rthm{T1121} is given in~\zitaa{MR2735313}{\cthm{1.3}}.
 If \(\kappa=\infi\), then the asserted equivalence can be proved using the equation \(\MggqKsg{\infi}=\bigcap_{m=0}^\infi\MggqKsg{m}\) and a matricial version of the Helly--Prohorov theorem (see~\zitaa{MR975253}{\cSatz{9}}).
 We omit the details of the proof, the essential idea of which is originated in~\zitaa{MR0184042}{proof of \cthm{2.1.1}}.
 
 We note that we do not need to apply Theorems~\ref{T1.2},~\ref{T1.1},~\ref{T1122}, and~\ref{T1121} for our further considerations in this paper.
 In the case of an odd integer \(m\), we will obtain a new proof of \rthm{T1122}.
 
 For the description of the solution set \(\MggqKskg{m}\) of Problem~\mprob{\rhl}{m}{\lleq}, it is essential that one can suppose expendable data without loss of generality:

\begin{thm}[\zitaa{MR2735313}{\cthm{5.2}}]\label{NT1}
 Let \(\ug\in\R\), let \(m\in\NO\), and let \(\seqs{m}\in\Kggq{m}\).
 Then there is a unique sequence \((\tilde{s}_j)_{j=0}^m\in\Kggeq{m}\) such that \(\MggqKakg{(\tilde{s}_j)_{j=0}^m} = \MggqKskg{m} \).
\end{thm}

\section{Some classes of holomorphic matrix-valued functions}\label{sec3}
 The class \(\RFq \) of all \tqqa{Herglotz--Nevanlinna} functions in the upper half-plane \(\uhp  \) consists of all matrix-valued functions \(F\colon\uhp  \to\Cqq \) which are holomorphic in \(\uhp \) and which satisfy \(\Im  [F(\uhp )] \subseteq \Cggq \).
 Detailed considerations of matrix-valued Herglotz--Nevanlinna functions can be found in~\cite{MR2988005,MR1784638}.
 In particular, the functions belonging to \(\RFq \) admit a well-known integral representation:

\begin{thm}\label{T31DN}
 \benui
  \item For each \(F\in\RFq \), there exist unique matrices \(A\in\CHq \) and \(B\in\Cggq \) and a unique \tnnH{} measure \(\nu\in\MggqR \) such that
\begin{align}\label{N31DN}
 F(z)&= A + zB + \int_\R  \frac{1 + tz}{t-z} \nu (\dif t)&\text{for each }z&\in\uhp .
\end{align}
  \item If \(A\in\CHq \), if \(B\in\Cggq \), and if \(\nu\in\MggqR \), then \(F\colon\uhp \to\Cqq \) defined by \eqref{N31DN} belongs to \(\RFq \).
 \eenui
\end{thm}

 For each \(F\in \RFq \), the unique triple \((A,B,\nu)\in\CHq  \times \Cggq  \times \MggqR \) for which the representation \eqref{N31DN} holds true is called the \emph{Nevanlinna parametrization of \(F\)} and we will also write \((A_F, B_F, \nu_F)\) for \((A,B,\nu)\).
 In particular, \(\nu_F\) is said to be the \emph{Nevanlinna measure  of \(F\)}.
 If \(F\) belongs to \(\RF{1}\), then \(\mu_F\colon \BorR  \to [0, \infp]\) defined by 
\begin{align}\label{N1052D}
 \mu_F (B)&\defeq  \int_B (1 + t^2) \nu_F (\dif t)&\text{for all }B&\in\BorR  
\end{align}
 is a measure, which is called the \emph{spectral measure of \(F\)}.
 By \(\RFsq \) we denote the set of all \(F\in\RFq \) for which \(g\colon\R\to\R\) defined by \(g(t) \defeq  1+t^2\) belongs to \(\LoaaaR{1}{\R}{\BorR}{\nu_F}\).
 Obviously, \(\RFsq  =\setaca{F\in\RFq }{\nu_F \in\Mgguqa{2}{\R}}\).
 If \(F\) belongs to \(\RFsq \), then \(\mu_F\colon \BorR  \to\Cggq \) given by \eqref{N1052D} is a well-defined \tnnH{} \tqqa{measure} belonging to \(\MggqR \), which is said to be the \emph{matricial spectral measure of \(F\)}.
 Obviously, for functions which belong to \(\RFs{1}\), the notions spectral measure and matricial spectral measure coincide.

 For our considerations, the class \(\RFOq \) of all \(F\in\RFq \) for which
\beql{N1302D}
 \sup_{y\in[1,\infp)} y\normS*{F (\iu y)}
 < \infp
\eeq 
 holds true plays an essential role.
 The class \(\RFOq \) is a subclass of \(\RFsq \) (see, \teg{}~\cite[\clem{6.1}]{MR2988005}).
 Furthermore, the functions belonging to \(\RFOq \) admit a particular integral representation:

\begin{thm}\label{N87DN}
 \benui
  \item For each \(F\in\RFOq \), there is a unique  \(\mu\in\MggqR \) such that
\begin{align}\label{N1140D}
 F(z)&=  \int_\R  \frac{1}{t-z} \mu (\dif t)&\text{for each }z&\in\uhp ,
\end{align}
 namely the matricial spectral measure of \(F\), and
\[
 \mu (\R )
 = \lim_{y\to\infp}\rk*{y \Im\ek*{F(\iu y)}}
 = - \iu \lim_{y\to\infp}\ek*{y F(\iu y)}
 =  \iu \lim_{y\to\infp}\ek*{y F^\ad (\iu y)}.
\]
 \item If \(F\colon\uhp  \to \Cqq \) is a matrix-valued function for which there exists a \tnnH{} measure \(\mu\in \MggqR \) such that \eqref{N1140D} holds true, then \(F\) belongs to  \(\RFOq \).
 \eenui
\end{thm}

 A proof of \rthm{N87DN} is given, \teg{}, in~\cite[\cthm{8.7}]{MR2222521}.
 If \(F\in\RFOq \), then the unique \(\mu\in\MggqR \) for which \eqref{N1140D} holds true is also called the \emph{Stieltjes measure of \(F\)}.
 If a \tnnH{} \tqqa{measure} \(\mu\in\MggqR \) is given, then \(F\colon\uhp  \to \Cqq \) defined by \eqref{N1140D} is said to be the \emph{Stieltjes transform of \(\mu\)}.

\begin{lem}\label{N1219DN}
 Let \(M\in\Cqq \) and let \(F\colon\uhp  \to\Cqq \) be a matrix-valued function which is holomorphic in \(\uhp \) and which satisfies the inequality
\[
\begin{bmatrix}
 M & F(z)\\
 F^\ad (z) & \frac{F(z) - F^\ad (z)}{z-\ko{z}}
\end{bmatrix}
\lgeq 0
\]
 for each \(z\in\uhp \).
 Then \(F\) belongs to \(\RFOq \) and the inequality  \(\sup_{y\in (0,\infp)} y\normS{F (\iu y)}\lleq \normS{M}\) holds true. 
 Furthermore, the Stieltjes measure \(\mu\) of \(F\) fulfills \(\mu (\R )\lleq M\).
\end{lem}

 A proof of \rlem{N1219DN} is given, \teg{}, in~\cite[\clem{8.9}]{MR2222521}.

 In view of the Stieltjes moment problem, a further class of matrix-valued functions plays a key role:
 For each \(\ug\in\R \), let \(\SFq \) be the set of all matrix-valued functions \(S\colon\Cs  \to \Cqq \) which are holomorphic in \(\Cs  \) and which satisfy \(\Im  [S (\uhp )]\subseteq \Cggq \) as well as  \(S (\crhl ) \subseteq \Cggq \).

 In~\cite[Theorems~3.1 and~3.6, Proposition~2.16]{MR3644521}, integral representations of functions  belonging to \(\SFq \) are proved.
 Furthermore, several characterizations of the class \(\SFq \) are given in~\cite[Section~4]{MR3644521}.

 For each \(\ug\in\R \), let \(\SFOq \) be the class of all \(F\in \SFq \) which satisfy \eqref{N1302D}.
 The functions belonging to \(\SFOq \) admit a particular  integral representation.
 Before we state this, let us note the following: 

\begin{rem}\label{L1.53}
 For every choice of \(\ug \in \R \) and \(z\in\Cs  \), the function \(b_{\ug, z}\colon\rhl  \to\C \) given by \(b_{\ug, z} (t) \defeq1/(t-z)\) is a bounded and continuous function which, in particular, belongs to \(\LoaaaC{1}{\rhl}{\BorK}{\sigma}\) for all \(\sigma\in\MggqK\).
\end{rem}

\begin{thm}[\zitaa{MR3644521}{\cthm{5.1}}]\label{T1316D}  
 Let \(\ug\in\R\).
\benui
\item If \(S\in \SFOq \), then there is a unique \(\sigma\in\MggqK \) such that 
\begin{align}\label{N1319D}
 S(z)&=\int_{\rhl } \frac{1}{t-z} \sigma (\dif t)&\text{for each }z&\in\Cs.
\end{align}
\item If  \(\sigma\in\MggqK \) is such that \(S\colon\Cs   \to \Cqq \) can be represented via \eqref{N1319D}, then \(S\) belongs to \(\SFOq \).
\eenui
\end{thm}

 If \(F\in \SFOq \) is given, then the unique \(\sigma\in\MggqK \) which fulfills the representation \eqref{N1319D} of \(F\) is called the \emph{{F}}.
 If  \(\sigma\in\MggqK \) is given, then \(F\colon\Cs   \to \Cqq \) defined by \eqref{N1319D} is said to be the \emph{\taSto{\sigma}}.
 In view of \rthm{T1316D}, the moment problems \mprob{\rhl }{m}{\lleq} and \mprob{\rhl}{\kappa}{=} admit reformulations in the  language of \taSt{s}:
 
\begin{Problem}[\iproblem{\rhl}{m}{\lleq}]
 Let \(\ug\in\R \), let \(m\in\NO \), and let \(\seqs{m}\) be a sequence of complex \tqqa{matrices}.
 Describe the set \(\SFOqskg{m}\) of all  \(F\in \SFOq \) the \taSm{} of which belongs to \(\MggqKskg{m}\).
\end{Problem}

\begin{Problem}[\iproblem{\rhl}{\kappa}{=}]
 Let \(\ug\in\R \), let \(\kappa\in\NOinf \), and let \(\seqska \) be a sequence of complex \tqqa{matrices}.
 Describe the set \(\SFOqsg{\kappa}\) of all  \(F\in \SFOq \) the \taSm{} of which belongs to \(\MggqKsg{\kappa}\).
\end{Problem}

\begin{rem}\label{R1457D}
 Let \(\ug\in\R \) and let \(F\in \SFOq \).
 Then \(F_\square \defeq  \Rstr_{\uhp } F\) belongs to \(\RFOq \),  the matricial spectral measure \(\mu_\square\) of \(F_\square\) fulfills \(\mu_\square (\crhl ) =\NM\), and \(\sigma \defeq  \Rstr_{\BorK } \mu_\square\) is exactly the \taSmo{F} (see~\cite[Proposition~2.16]{MR3644521}).
\end{rem}

 At the end of this section, we state two results, which are essential to discuss the so-called degenerate case.

\begin{prop}[cf.~\zitaa{MR3644521}{Proposition~5.3}]\label{P1525}
 Let \(\ug\in\R \) and let \(F\in \SFOq \).
 Then the \taSm{} \(\sigma\) of \(F\) fulfills \(\sigma (\rhl ) = -\iu \lim_{y\to\infp} y F (\iu y)\) and, for each \(z\in \Cs  \), furthermore \(\ran{F(z)}=\ran{\sigma (\rhl)}\)  and \(\nul{F(z)} = \nul{\sigma (\rhl)}\).
\end{prop}

 If \(\cG\) is a region of \(\C\), \tie{}, a \tne{} open connected subset of \(\C\), then a matrix-valued function \(S\colon\cG\to\Cpq \) is called \emph{\tpqa{Schur} function in \(\cG\)} if \(S\) is both holomorphic and contractive in \(\cG\).
 The set of all \tpqa{Schur} functions in a region \(\cG\) of \(\C\) will be denoted by \(\SchF{p}{q}{\cG}\). 
 
\begin{rem} \label{TS415}
 Let \(\cG\) be a region of \(\C\) and let \(S\in\SchF{q}{q}{\cG}\).
 Then \(\cG^\lor \defeq \setaca{ z\in\C}{\ko{z} \in \cG }\) is a region of \(\C\) and \(S^\lor \colon\cG^\lor\to \Cqp \) given by \(S^\lor (z) \defeq \ek{S(\ko{z})}^\ad\) belongs to \(\SchF{q}{p}{\cG^\lor}\).
\end{rem}

\begin{lem} \label{TS417}
 Let \(\cG\) be a region of \(\C\) and let \(S\in\SchF{q}{q}{\cG}\).
 Then:
\begin{enui}
 \item\label{TS417.a} If \(U\in\Cpq \) fulfills \(U^\ad U =\Iq \), then \(\nul{U + S(w)} = \nul{U+S(z)}\) for every choice of \(w\) and \(z\) in \(\cG\).
 \item\label{TS417.b} If \(V\in\Cpq \) fulfills \(VV^\ad =\Ip \), then \(\ran{V + S(w)}= \ran{V+S(z)}\) for every choice of \(w\) and \(z\) in \(\cG\). 
\end{enui}
\end{lem}
\begin{proof}
 \eqref{TS417.a} Let \(U\in\Cpq \) be such that \(U^\ad U =\Iq \), let \(w\in\cG\), and let \(z\in\cG\).
 We consider an arbitrary \(v\in\nul{U + S(w)}\setminus \set{\Ouu{q}{1}}\).
 Then \(S(w) v = -Uv\).
 Consequently, 
\(
 \normEs{S (w) v}
 =\normEs{Uv}
 = v^\ad U^\ad Uv
 = v^\ad v
 =\normEs{v}
\).
 According to~\cite[\clem{2.1.2}, \cpage{61}]{MR1152328}, this implies \(S(w) v = S(z) v\).
 Hence, \( \ek{U + S(z)} v  = Uv + S(w) v = \Ouu{p}{1}\).
 Thus, \(v\in\nul{U + S(z)}\).
 Therefore, \(\nul{U + S(w)}\subseteq\nul{U + S(z)}\) is proved.
 For reasons of symmetry, we also have \(\nul{U + S(z)}\subseteq \nul{U + S(w)}\).
 \rPart{TS417.a} is checked.
 
 \eqref{TS417.b} In order to prove \rpart{TS417.b}, we first apply \rrem{TS415}.
 Thus, we see that \(S^\lor\) belongs to \(\SchF{q}{p}{\cG^\lor}\).
 Consequently, \rpart{TS417.a} yields
\[
 \Nul{\ek*{V + S(w)}^\ad}
 = \Nul{V^\ad + S^\lor (\ko{w})}
 =\Nul{V^\ad + S^\lor (\ko{z})}
 =\Nul{\ek*{V + S(z)}^\ad }
\]
 for every choice of \(w\) and \(z\) in \(\cG\).
 In view of \rrem{A.1.}, the proof is complete.
\end{proof}

\section{From the Stieltjes moment problem to the system of Potapov's fundamental inequalities}\label{S1415}
 In this section, we introduce the system of Potapov's fundamental matrices corresponding to the matricial Stieltjes moment problem \mprob{\rhl }{m}{\lleq}.
 We will see that each solution of this moment problem fulfills necessarily the system of Potapov's fundamental matrix inequalities.
 First we are going to introduce further notations and, in particular, several block Hankel matrices which will play a key role in our considerations.
 For technical reason, let \(s_{-1} \defeq  \Opq \).

 Let \(\kappa\in\NOinf  \) and let \(\seqska \) be a sequence of complex \tpqa{matrices}.
 For each \(n\in\NO \) with \(2n\le \kappa\), let \(H_n \defeq\mat{s_{j+k}}_{j,k=0}^n\), for each \(n\in\NO \) with \(2n + 1 \le \kappa\), let \(K_n \defeq\mat{s_{j+k+1}}_{j,k=0}^n\), and, for each \(n\in\NO \) with \(2n +2\le \kappa\), let \(G_n \defeq\mat{s_{j+k+2}}_{j,k=0}^n\).
 If \(m\) and \(n\) are integers such that \(-1 \le m \le n\le \kappa\), then we set
\begin{align}\label{YZ}
 y_{m,n}&\defeq  \col (s_j)_{j=m}^n&
&\text{and}&
 z_{m,n} \defeq  & \row (s_k)_{k=m}^n.
\end{align}
 Let \(u_0 \defeq  \Opq\), \(\mathfrak{u}_0 \defeq  \Opq\), \(w_0 \defeq  \Opq \), and \(\mathfrak{w}_0 \defeq  \Opq \).
 For all \(n\in\N \) with \(n\le\kappa +1\), let \(u_n \defeq  -y_{-1,n-1}\), and \(w_{n} \defeq  z_{-1,n-1}\).
 Further, for each \(n\in\NO \) with \(2n\le \kappa\), let  \(\mathfrak{u}_n \defeq\tmat{-y_{n+1,2n}\\ \Opq}\) and \(\mathfrak{w}_{n} \defeq\mat{z_{n+1,2n}, \Opq}\).

 If a real number \(\ug\) is additionally given, then we continue to use the notation given by \eqref{FR1733}, and we set \(H_{\ug \triangleright n} \defeq\mat{s_{\ug\triangleright j + k}}_{j,k=0}^n\) for each \(n\in\NO\) with \(2n+1\le \kappa\).

\begin{rem} \label{ML415}
 Let \(\kappa\in\NOinf \) and let \(\seqska \) be a sequence of complex \tqqa{matrices}.
 If \(n\in\NO \) is such that \(2n\leq\kappa\), then \(H_n\in\CHo{(n+1)q}\) if and only if \(\setaca{ s_j}{j\in\mn{0}{2n}} \subseteq\CHq \).
 Furthermore, if \(\ug\in\R\), if \(\kappa\geq 1\), and if \(n\in\NO \) is such that \(2n+1\leq\kappa\), then \(\set{ H_n, H_{\at{n}} } \subseteq \CHo{(n+1)q}\) if and only if \(\setaca{ s_j}{ j\in\mn{0}{2n+1}} \subseteq\CHq \).
\end{rem}

\begin{rem}\label{R1502}
 Let \(n\in\N\) and let \(\seqs{2n}\) be a sequence of complex \tpqa{matrices}.
 Then the block Hankel matrix \(\Hu{n}\) admits the block representations
\begin{align} 
 H_n&=\bMat H_{n-1} & y_{n, 2n-1} \\ z_{n, 2n-1} & s_{2n}\eMat,& && H_n&=\bMat s_0 & z_{1, n} \\ y_{1,n} & G_{n-1}\eMat,\label{Nr.B6}\\
 H_n&=\bMat y_{0,n-1} & K_{n-1} \\ s_{n} & z_{n+1,2n}\eMat,& &\text{and}& H_n&=\bMat z_{0, n-1} & s_n \\ K_{n-1} & y_{n+1,2n}\eMat\notag.
\end{align}
\end{rem}

 For each \(n\in\NO \), we set
\begin{align*}
T_{q,n}&\defeq\mat{\Kronu{j,k+1}\Iq}_{j,k=0}^n,&
v_{q,n}&\defeq  \col (\Kronu{j,0}\Iq)_{j=0}^n,&
&\text{and}&
\mathfrak{v}_{q,n}&\defeq \col (\Kronu{n-j,0}\Iq)_{j=0}^n,
\end{align*}
 where \(\Kronu{j,k}\) is again the Kronecker delta.
 Obviously, \(\Tqn^\ad =\mat{\Kronu{j+1,k}\Iq}^{n}_{j,k=0}\) for each \(n\in\NO \).

 It seems to be useful to recall well-known Ljapunov identities for block Hankel matrices. 
 (These equations can be also easily proved by straightforward calculation.)

\begin{rem} \label{lemC21-1}
 Let \(\kappa \in \NOinf \) and let \(\seqska \) be a sequence of complex \tpqa{matrices}.
 \benui
  \item\label{lemC21-1.a} For each \(n \in \NO \) with \(2n \leq \kappa\), then \(H_n \Tqn^\ad -T_{p,n} H_n = u_n v^\ad _{q,n} - v_{p,n} w_n\) and \(H_n \Tqn-T_{p,n}^\ad  H_n = \su_n \sv^\ad _{q,n} - \sv_{p,n} \sw_n\).
 In particular, if \(p=q\) and if \(s_j^\ad = s_j\) for each \(j\in\mn{0}{\kappa}\), then \(H_n \Tqn^\ad -\Tqn H_n = u_n v^\ad _{q,n} - v_{q,n} u_n^\ad \) and \(H_n \Tqn-\Tqn^\ad  H_n = \su_n \sv^\ad _{q,n} - \sv_{q,n} \su_n^\ad \) for  each \(n \in \NO \) with \(2n \leq \kappa\).
  \item\label{lemC21-1.b} For each \(n \in \NO \) with \(2n+1\leq \kappa\), we have \(H_{\at{n}}= -\alpha H_n +K_n\), \(v_{p,n} v_{p,n}^\ad H_n = \ek*{R_{\Tpn}(\alpha)}^\inv H_n -\Tpn H_{\at {n}}\), and, in the case that \(p=q\) and \(s_{j}^\ad=s_{j}\) for each \(j\in\mn{0}{\kappa}\) hold true, moreover \(H_{\at{n}} T_{q,n}^\ad - T_{q,n} H_{\at{n}} = (-\alpha u_n - y_{0,n}) v_{q,n}^\ad - v_{q,n} (-\alpha u_n - y_{0,n})^\ad\) for each \(n \in \NO \) with \(2n+1\leq \kappa\).
  \item\label{lemC21-1.bb} For each \(n \in \NO \) with \(2n+2\leq \kappa\), we have \(H_{\at{n}}T_{q,n}-T_{p,n}^\ad H_{\at{n}}=(-\alpha \su_n - y_{n+2, 2n+2})\sv^\ad _{q,n} - \sv_{p,n}( -\alpha \sw_{n} - z_{n+2,2n+2})\) and, in particular, if \(p=q\) and if \(s_j^\ad = s_j\) for each \(j\in\mn{0}{\kappa}\), then \(H_{\at{n}} T_{q,n} - T_{q,n}^\ad H_{\at{n}} = (-\alpha\su_n - y_{n+2, 2n+2}) \sv_{q,n}^\ad - \sv_{q,n}  ( -\alpha \su_{n} - y_{n+2,2n+2})^\ad\) for each \(n \in \NO \) with \(2n+2\leq \kappa\).
  \item\label{lemC21-1.c} The equations   \(H_n v_{q,n} = y_{0,n}\) and \(- T_{p,n}H_n v_{q,n} = u_n\) hold true for each  \(n \in \NO \) with \(2n\leq \kappa\).
 \eenui
\end{rem}

\begin{rem}   \label{21112N}
 For each \(n \in \NO \), the matrix-valued functions \(R_{\Tqn}\colon\C \to \Coo{(n+1)q}{(n+1)q}\) and \(R_{\Tqn^\ad }\colon\C \to \Coo{(n+1)q}{(n+1)q}\)  given by 
\begin{align*}
 R_{\Tqn}(z)&\defeq (\Iu{(n+1)q}-z\Tqn)^\inv &
&\text{and}&
 R_{T^\ad _{q,n}}(z)&\defeq (\Iu{(n+1)q}-zT^\ad _{q,n})^\inv  
\end{align*}
 are well-defined matrix polynomials of degree \(n\), which can be represented, for each \(z\in\C\), via \(R_{\Tqn}(z)=\sum_{j=0}^n z^j \Tqn^j\) and   \(R_{\Tqn^\ad} (z)=\sum_{j=0}^n z^j (T_{q,n}^\ad)^j\), respectively.
 In particular,   \(R_{\Tqn^\ad} (z)=[R_{T_{q,n}} (\ko{z})]^\ad\) for all \(z\in\C \).
\end{rem}

 For each \(n \in \NO \), let \(E_{q,n}\colon\C \to \C^{(n+1)q \times q}\) and \(F_{q,n}\colon\C \to \C^{(n+1)q \times q}\) be defined by 
\begin{align} \label{3.64.-1}
 E_{q,n}(z)&\defeq  \col ( z^j \Iq)_{j=0}^{n}&
&\text{and}&
 F_{q,n} (z)&\defeq zE_{q,n}(z),
\end{align}
 respectively.
 Obviously, for each \(n\in \NO \) and each \(z\in\C\), we have \(R_{\Tqn}(z)v_{q,n}=E_{q,n}(z)\).

\begin{nota}\label{Note44}
 Let \(\alpha \in \R\), let \( \kappa \in \NOinf \), and let \(\seqska \) be a sequence of complex \tqqa{matrices}.
 Further, let \(\cG\) be a subset of  \(\C\) with \(\cG \setminus \R\neq \emptyset \) and let  \(f\colon\cG \to \Cqq \) be a matrix-valued function.
 Then, for each \(n \in \NO \) with \(2n \leq \kappa\), let \(P^{[f]}_{2n}\colon\cG \setminus \R \to \Coo{(n+2)q}{(n+2)q}\) be defined by
\beql{def-P2n}
 P^{[f]}_{2n} (z)\defeq 
 \begin{bmatrix}
 H_n & R_{\Tqn} (z) [v_{q,n} f(z)-u_n]\\
 (R_{\Tqn} (z) [v_{q,n} f(z)-u_n] )^\ad  &\frac{f(z)-f^\ad (z)}{z -\ko{z}}
 \end{bmatrix}.
\eeq  
 If \(\kappa \geq 1\), then, for each \(n \in \NO \) with \(2n+1 \leq \kappa\), let \(P^{[f]}_{2n+1}\colon\cG \setminus \R \to \Coo{(n+2)q}{(n+2)q}\) be given by 
\begin{multline}\label{Pf2n+1}
 P^{[f]}_{2n+1} (z)\\
 \defeq
 \begin{pmat}[{|}]
  H_{\at n}  &
  \begin{gathered}
   R_{\Tqn} (z) (v_{q,n}\ek{(z- \alpha) f(z)}\\
   \qquad-(- \alpha u_n-y_{0,n}))
  \end{gathered}\cr\-
 \ek{R_{\Tqn} (z)\rk{v_{q,n}\ek{(z- \alpha) f(z)}-(- \alpha u_n-y_{0,n})}}^\ad  & \frac{(z-\alpha)f(z)-[(z-\alpha)f(z)]^\ad }{z -\ko{z}}\cr
 \end{pmat}.
\end{multline}
 Furthermore, let \(P_{-1}^{[f]}\colon\cG\setminus\R \to \Cqq\) be defined by
\[
 P_{-1}^{[f]}(z)
 \defeq  \frac{(z-\alpha)f(z)-[(z-\alpha)f(z)]^\ad }{z -\ko{z}}.
\]
\end{nota}
 
 With respect to the Stieltjes moment problem \mprob{\rhl}{m}{\lleq} if \(\cG = \C\), then the functions \eqref{def-P2n} and \eqref{Pf2n+1} are called the Potapov fundamental matrix-valued functions connected to the Stieltjes moment problem (generated by \(f\)).
 If these matrices  are both \tnnH{}, then one says that the Potapov's fundamental matrix inequalities for the function \(f\) are fulfilled.

\begin{rem}\label{remark4.5}
 Let  \( \kappa \in \NOinf \) and let \(\seqska \) be a sequence  from \(\Cqq \).
 Furthermore, let \(\cG\)  be a subset of \(\C\) with \(\cG \setminus \R \neq \emptyset\), let \(f\colon\cG \to \Cqq\) be a matrix-valued function, and let \(z \in \cG \setminus \R\).
 Then \rrem{DFK1}  shows that the following statements hold true:
\benui
 \item Let \(n\in \NO \) with  \(2n \leq \kappa\).
 Then the matrix  \(P_{2n}^{[f]}(z)\) is \tnnH{} if and only if the following three conditions are fulfilled:
 \baeqii{0}
  \item \(\seqs{2n}\in \Hggq{2n}\).
  \item  \(\ran{R_{\Tqn} (z) [v_{q,n}f(z)-u_n]}\subseteq \ran{\Hu{n}}\).
  \item    The matrix
   \begin{multline}\label{N54}
    \Sigma_{2n}^{[f]}(z)
    \defeq\frac{f(z)-f^\ad (z)}{z-\ko{z}}\\
    -\rk*{R_{\Tqn} (z) [v_{q,n}f(z)-u_n]}^\ad  H_n^\mpi \rk*{R_{\Tqn} (z) [v_{q,n}f(z)-u_n]}
   \end{multline}
 is \tnnH{}.
 \eaeqii
 If \(P_{2n}^{[f]}(z)\in \Cggo{(n+2)q}\), then, for each \(H_n^{(1)} \in H_n\{1\}\), we have
 \begin{multline*}
    \Sigma_{2n}^{[f]}(z)=
    \frac{f(z)-f^\ad (z)}{z-\ko{z}}\\
    -\rk*{R_{\Tqn} (z) [v_{q,n}f(z)-u_n]}^\ad  H_n^{(1)} \rk*{R_{\Tqn} (z) [v_{q,n}f(z)-u_n]}.
   \end{multline*} 
 \item
 Let \(\alpha \in \R\), let \(\kappa\ge 1\), and let  \(n\in \NO \) with \(2n+1 \leq \kappa\).
 Then  the matrix \(P_{2n+1}^{[f]}(z)\) is \tnnH{} if and only if the following three conditions are valid:
 \baeqii{3}
  \item \((\seqsa{2n} \in \Hggq{2n}\).
  \item \(\ran{R_{\Tqn} (z) [v_{q,n}[(z-\alpha)f(z)] - (-\alpha u_n -y_{0,n})]} \subseteq \ran{H_{\at n}}\).
  \item The matrix
   \begin{multline}\label{N55}
    \Sigma_{2n+1}^{[f]}(z)
    \defeq\frac{(z-\alpha)f(z)-\ek{(z-\alpha)f(z)}^\ad }{z-\ko{z}}\\
    -\ek*{R_{\Tqn}(z)\rk*{v_{q,n}\ek*{(z-\alpha)f(z)} - (-\alpha u_n -y_{0,n})}}^\ad\\
    \times H_{\at n}^\mpi \ek*{R_{\Tqn} (z)\rk*{v_{q,n} \ek*{(z-\alpha)f(z)} - (-\alpha u_n -y_{0,n})}}
   \end{multline}
 is \tnnH{}.
 \eaeqii
 If \(P_{2n+1}^{[f]}(z)\in \Cggo{(n+2)q}\), then for each \(H_{\at n}^{(1)} \in H_{\at n}\{1\}\), we have
   \begin{multline*}
    \Sigma_{2n+1}^{[f]}(z)
    =\frac{(z-\alpha)f(z)-[(z-\alpha)f(z)]^\ad }{z-\ko{z}}\\
    -\ek*{R_{\Tqn}(z)\rk*{v_{q,n}\ek*{(z-\alpha)f(z)} - (-\alpha u_n -y_{0,n})}}^\ad\\
    \times H_{\at n}^{(1)} \ek*{R_{\Tqn} (z)\rk*{v_{q,n} \ek*{(z-\alpha)f(z)} - (-\alpha u_n -y_{0,n})}}.
   \end{multline*}
 \eenui
\end{rem}

\begin{rem} \label{lemM423}
 Let  \(\kappa \in \NOinf \), let \(\seqska \) be a sequence of complex \tqqa{matrices}, let \(\cG\) be a subset of  \(\C\) with \(\cG \setminus \R \neq  \emptyset\), and let  \(S\colon\cG \to \Cqq \) be a matrix-valued function.
 Straightforward calculations show then that the following statements hold true:
  \begin{enui} 
  \item For every choice of \(n \in \NO \) with \(2n \leq \kappa\) and \(z\in\cG\setminus\R\), we have 
   \beql{N68-1}
    \begin{bmatrix}
     s_0 & S (z)\\ 
     S^\ad (z) & \frac{S (z) - S^\ad  (z)}{z - \ko{z}}
    \end{bmatrix} 
    =\mat{v_{q,n+1},  \sv_{q,n+1}}^\ad       P^{[S]}_{2n}(z)\mat{v_{q,n+1},  \sv_{q,n+1}}.
  \eeq     
  \item If \(\kappa \geq 1\), for each \(n \in \NO \) with \(2n +1 \leq \kappa\) and each \(z\in\cG\setminus\R\), then
    \begin{multline} \label{N68-2}
     \begin{bmatrix}
       -\alpha s_0 + s_1   &  (z - \alpha) S(z)+s_0\\
       [(z-\alpha)S (z) +s_0 ] ^\ad  & \frac{(z-\alpha) S(z)- [(z-\alpha) S (z)]^\ad }{z- \ko{z}}
      \end{bmatrix}\\
      =\mat{v_{q,n+1},  \sv_{q,n+1}}^\ad       P^{[S]}_{2n+1}(z)\mat{v_{q,n+1},  \sv_{q,n+1}}.
    \end{multline}  
  \end{enui} 
\end{rem}

 In the following, for each \(k\in\NO\), let 
\begin{align}\label{N1435D}
 m_{2k}&\defeq  k&
&\text{and}&
 m_{2k+1}&\defeq  k.
\end{align}
 
\begin{lem} \label{MP36N}
 Let  \(\kappa \in \NOinf \) and let \(\seqska \) be a sequence of \tH{} complex \tqqa{matrices}.
 Let \(\cG\) be a subset of  \(\C\) with \(\cG \setminus \R \neq  \emptyset\).
 Further, let  \(f \colon\cG \to \Cqq \) be a matrix-valued function, let \(\cG^\vee \defeq  \setaca{z\in\C}{\ko{z}\in \cG}\), and let \(f^\vee \colon\cG^\vee \to \Cqq \) be defined by \(f^\vee (z) \defeq  f^\ad (\ko{z})\).
 For each \(k \in\mn{-1}{\kappa}\) and each \(z\in\cG^\vee  \setminus \R \), then there is a complex \taaa{(m_k +2)q}{(m_k + 2)q}{matrix} \(X_k (z)\) such that \(P_k^{[f^\vee]} (z) = X_k (z)P_k^{[f]} (\ko{z}) X_k^\ad (z)\).
\end{lem}
\begin{proof}
 We give the proof which is stated with more details in~\cite[Lemma~3.6]{MP11}.
 We consider an arbitrary  \( z \in \cG^\vee \setminus \R\).
 From \rrem{lemC21-1}, for each \(n\in\NO \), we see that
\beql{Nr.FN1}
 \begin{split}
  &\ek*{R_{T_{q,n}}(\ko{z})}^\inv  H_n \ek*{R_{T_{q,n}}(\ko{z})}^\invad  + (\ko{z}-z)(v_{q,n}u^\ad _{n}-u_n v_{q,n}^\ad )\\
  &=  H_n-zH_nT_{q,n}^\ad -\ko{z}T_{q,n}H_n+\abs{z}^2T_{q,n}H_nT_{q,n}^\ad +(\ko{z}-z)(T_{q,n}H_n-H_nT^\ad _{q,n})\\
  &=  H_n+\abs{z}^2T_{q,n}H_nT_{q,n}^\ad -\ko{z}H_nT^\ad _{q,n}-zT_{q,n}H_n\\
  &=  \rk{\Iu{(n+1)q} - z T_{q,n}}  H_n   \rk{\Iu{(n+1)q} -z T_{q,n} }^\ad     
  =  \ek*{R_{T_{q,n}}(z)}^\inv    H_n  \ek*{R_{T_{q,n}}(z)}^\invad.
 \end{split}
\eeq 
 Obviously, for each \(n \in \NO \), we also get
\beql{Nr.54}
   (z-\ko{z})v_{q,n}[f^\ad (\ko{z})v^\ad _{q,n}-u^\ad _n] - (z-\ko{z})[v_{q,n} f^\ad (\ko{z})-u_n]v^\ad _{q,n}
   = (z-\ko{z}) (u_n v^\ad _{q,n} - v_{q,n} u^\ad _n)
\eeq 
 and
\beql{Nr.54-1}
   f^\ad (\ko{z}) v^\ad _{q,n} - u_n^\ad - [f^\ad (\ko{z})-f(\ko{z})] v^\ad _{q,n} 
   =     [v_{q,n} f^\ad (\ko{z}) - u_n]^\ad.
\eeq 
 For each \(n \in \NO \), let
\begin{align*}
  A_{2n}(z)&\defeq  \diag \rk*{ \ek*{ R_{T_{q,n}} (\ko{z}) }^\inv ,\Iq },&
  B_{2n}(z)&\defeq  \begin{bmatrix}                                        
                   \Iu{(n+1)q} & (z-\ko{z})v_{q,n}    \\
                    \Ouu{q}{(n+1)q} & \Iq
              \end{bmatrix},       
       \end{align*}   
 \(C_{2n}(z) \defeq   \diag \rk{ R_{T_{q,n}}(z),\Iq }\), \(A_{2n+1}(z) \defeq   A_{2n}(z)\), \( B_{2n+1}(z) \defeq   B_{2n}(z)\), and \(C_{2n+1}(z) \defeq   C_{2n}(z)\).
 For each \(m \in \NO \), let \(X_{m}(z) \defeq   C_{m}(z) B_m(z) A_m(z)\).\\ 
 First we observe now that \(P_{-1}^{[f^\vee]} (z) = X_{-1} (z) P_{-1}^{[f]} (\ko{z}) X_{-1}^\ad (z)\) is true with \(X_{-1} (z) \defeq  \Iq\).

 Now we consider an arbitrary \(n \in \NO \) with \(2n \leq \kappa\).
 Then we have
\[\begin{split}
  &P_{2n}^{[f]}(\ko{z})A^\ad _{2n}(z)\\
  &= \begin{bmatrix}
                        H_n               &                             R_{T_{q,n}} (\ko{z}) [v_{q,n} f(\ko{z})-u_n]           \\
                        \rk{R_{T_{q,n}} (\ko{z}) [v_{q,n} f(\ko{z})-u_n] }^\ad  &         \frac{f(\ko{z})-f^\ad (\ko{z})}{\ko{z} - z}
                    \end{bmatrix}     \cdot      \diag     \rk*{ \ek*{R_{T_{q,n}}(\ko{z}) }^\invad , \Iq}\\
  &= \begin{bmatrix}
                        H_n \ek{ R_{T_{q,n}}(\ko{z}) }^\invad   &  R_{T_{q,n}} (\ko{z}) [v_{q,n} f(\ko{z})-u_n]           \\
                        f^\ad (\ko{z}) v^\ad _{q,n} -u_n^\ad   & \frac{f(\ko{z})-f^\ad (\ko{z})}{\ko{z} - z}
                    \end{bmatrix}.
\end{split}\]
 Consequently, 
\begin{equation*}           
 A_{2n}(z) P_{2n}^{[f]}(\ko{z}) A_{2n}^\ad  (z)    
 = \begin{bmatrix}
                        \ek{ R_{T_{q,n}} (\ko{z}) }^\inv  H_n \ek{ R_{T_{q,n}}(\ko{z}) }^\invad  & v_{q,n} f(\ko{z})-u_n   \\
                        f^\ad (\ko{z}) v^\ad _{q,n} -u_n^\ad  & \frac{f^\ad (\ko{z})-f(\ko{z})}{z - \ko{z}} 
           \end{bmatrix}.                              
\end{equation*}
 Thus, we conclude
\[\begin{split}
 &B_{2n}(z) A_{2n}(z) P_{2n}^{[f]}(\ko{z}) A_{2n}^\ad  (z)\\                          
 &= \begin{bmatrix}                                        
                   \Iu{(n+1)q} & (z-\ko{z})v_{q,n}    \\
                    \Ouu{q}{(n+1)q} & \Iq
                 \end{bmatrix}
                 \begin{bmatrix}
                        \ek{ R_{T_{q,n}} (\ko{z}) }^\inv  H_n \ek{ R_{T_{q,n}}(\ko{z}) }^\invad  & v_{q,n} f(\ko{z})-u_n   \\
                        f^\ad (\ko{z}) v^\ad _{q,n} -u_n^\ad  & \frac{f^\ad (\ko{z})-f(\ko{z})}{z - \ko{z}} 
           \end{bmatrix}\\
           &= \begin{pmat}[{|}]
                        \ek{ R_{T_{q,n}} (\ko{z}) }^\inv  H_n \ek*{ R_{T_{q,n}}(\ko{z}) }^\invad+(z-\ko{z})v_{q,n} \ek*{ f^\ad (\ko{z}) v^\ad _{q,n} -u_n^\ad } & v_{q,n} f^\ad (\ko{z})-u_n  \cr\-
                        f^\ad (\ko{z}) v^\ad _{q,n} -u_n^\ad  & \frac{f^\ad (\ko{z})-f(\ko{z})}{z - \ko{z}}\cr
           \end{pmat}.
\end{split}\]
 Using \eqref{Nr.54}, \eqref{Nr.54-1}, and  \eqref{Nr.FN1}, we get then
\begin{equation*}
 \begin{split}
  &B_{2n}(z) A_{2n}(z) P_{2n}^{[f]}(\ko{z}) A_{2n}^\ad  (z) B^\ad _{2n}(z)\\
  &=
  \begin{pmat}[{|}]
   \begin{gathered}
    \ek{ R_{T_{q,n}} (\ko{z}) }^\inv  H_n \ek{ R_{T_{q,n}}(\ko{z}) }^\invad\\
    +(z-\ko{z})v_{q,n} \ek{ f^\ad (\ko{z}) v^\ad _{q,n} -u_n^\ad }\\
    -(z- \ko{z})\ek{ v_{q,n} f^\ad (\ko{z})-u_n } v_{q,n}^\ad
   \end{gathered}& v_{q,n} f^\ad (\ko{z})-u_n  \cr\-
   f^\ad (\ko{z}) v_{q,n}^\ad  -u_n^\ad- \ek*{ f^\ad (\ko{z})-f(\ko{z}) } v_{q,n}^\ad&\frac{f^\ad (\ko{z})-f(\ko{z})}{z - \ko{z}}   \cr
  \end{pmat}\\
  &=
  \begin{pmat}[{|}]   
   \ek{ R_{T_{q,n}} (\ko{z}) }^\inv  H_n \ek{ R_{T_{q,n}}(\ko{z}) }^\invad+(z-\ko{z}) \ek*{u_n v_{q,n}^\ad  - v_{q,n}u_n^\ad }&v_{q,n} f^\ad (\ko{z})-u_n\cr\-
   \ek{ v_{q,n} f^\ad (\ko{z})-u_n }^\ad& \frac{f^\ad (\ko{z})-f(\ko{z})}{z - \ko{z}}\cr
  \end{pmat}\\
  &=
  \begin{pmat}[{|}] 
   \ek{ R_{T_{q,n}} (z) }^\inv  H_n \ek{ R_{T_{q,n}}(z) }^\invad& v_{q,n} f^\ad (\ko{z})-u_n\cr\-
   \ek{ v_{q,n} f^\ad (\ko{z})-u_n }^\ad&\frac{f^\ad (\ko{z})-f(\ko{z})}{z - \ko{z}}\cr
  \end{pmat}.
 \end{split}
\end{equation*}
 Hence, we obtain
\[\begin{split}
 X_{2n}(z) P_{2n}^{[f]}(\ko{z}) & X_{2n}^\ad (z)
 = C_{2n}(z) B_{2n}(z) A_{2n}(z) P_{2n}^{[f]}(\ko{z}) A_{2n}^\ad  (z) B^\ad _{2n}(z) C^\ad _{2n}(z)                \\
 &=
 \begin{bmatrix}       
  H_n & R_{T_{q,n}}(z) \ek{ v_{q,n} f^\ad (\ko{z})-u_n }          \\
  \ek{ v_{q,n} f^\ad (\ko{z})-u_n }^\ad  R^\ad _{T_{q,n}}(z) & \frac{f^\ad (\ko{z})-f(\ko{z})}{z - \ko{z}}                
 \end{bmatrix}\\
 &=
 \begin{bmatrix}
  H_n & R_{T_{q,n}}(z) \ek{ v_{q,n} f^\vee (z)-u_n }          \\
  \rk{R_{T_{q,n}}(z) \ek*{ v_{q,n} f^\vee (z)-u_n } }^\ad  & \frac{f^\vee (z) - \ek*{ f^\vee (z) }^\ad  }{z - \ko{z}}     
 \end{bmatrix}                       =    P^{[f^\vee]}_{2n}(z).
\end{split}\]
 Now we consider the case that \(n \in \NO \) is such that  \(2n+1 \leq \kappa\).
 Then \rrem{lemC21-1} yields
\beql{Nr.FN2}
 \begin{split}
  &\ek*{R_{T_{q,n}}(\ko{z})}^\inv  H_{\at{n}}\ek*{R_{T_{q,n}}(\ko{z})}^\invad+ (\ko{z}-z) \ek*{ v_{q,n}(-\alpha u_{n} -y_{0,n})^\ad  - (-\alpha u_n - y_{0,n}) v_{q,n}^\ad }\\
  &= \rk{\Iu{(n+1)q}-\ko{z}T_{q,n}} H_{\at{n}}\rk{\Iu{(n+1)q}- z T_{q,n}^\ad }\\
  &\qquad+(\ko{z}-z) \ek*{ v_{q,n}(-\alpha u_{n} -y_{0,n})^\ad  - (-\alpha u_n - y_{0,n}) v_{q,n}^\ad }\\
  &= \rk{\Iu{(n+1)q}-\ko{z}T_{q,n}}  H_{\at{n}} \rk{\Iu{(n+1)q}- z T_{q,n}^\ad } + (\ko{z}-z) \rk{ T_{q,n} H_{\at{n}} - H_{\at{n}} T_{q,n}^\ad  }\\
  &=H_{\at{n}}+\abs{z}^2 T_{q,n} H_{\at{n}} T_{q,n}^\ad -\ko{z} H_{\at{n}} T_{q,n}^\ad  - z T_{q,n} H_{\at{n}}\\
  &= \rk{\Iu{(n+1)q}- z T_{q,n}} H_{\at{n}} \rk{\Iu{(n+1)q}-\ko{z} T_{q,n}^\ad }
  = \ek*{R_{T_{q,n}}(z)}^\inv  H_{\at{n}} \ek*{R_{T_{q,n}}(z)}^\invad.
 \end{split}
\eeq 
 Obviously,
\begin{align}
 &v_{q,n} (\ko{z}-\alpha) f(\ko{z}) - (-\alpha u_n-y_{0,n}) +(z-\ko{z}) v_{q,n} \frac{(z-\alpha)f^\ad (\ko{z})-(\ko{z}-\alpha)f(\ko{z})}{z-\ko{z}}\notag\\
 &= v_{q,n}(z-\alpha) f^\ad (\ko{z}) - (-\alpha u_n - y_{0,n}),\label{Nr.VQF}\\
&(z-\ko{z})v_{q,n} \ek*{ (z-\alpha) f^\ad (\ko{z})v^\ad _{q,n}- (-\alpha u^\ad _n-y^\ad _{0,n}) } \notag\\
&\qquad- (z-\ko{z}) \ek*{ v_{q,n}(z-\alpha)f^\ad (\ko{z}) - (-\alpha u_n-y_{0,n}) } v^\ad _{q,n}\notag\\
  &= (\ko{z}-z) ( v_{q,n} (-\alpha u^\ad _n - y^\ad _{0,n}) - (- \alpha u_n-y_{0,n}) v^\ad _{q,n} ),\label{Nr.Z-ZQ} 
\end{align}
 and
\begin{multline*}
 (z-\alpha) f^\ad (\ko{z}) v^\ad _{q,n} - (-\alpha u^\ad _n-y^\ad _{0,n}) - \ek*{ (z-\alpha)f^\ad (\ko{z}) - (\ko{z}-\alpha) f(\ko{z}) } v^\ad _{q,n}\\
 = \ek*{ (z-\alpha) v_{q,n} f^\ad (\ko{z}) - (-\alpha u_n-y_{0,n}) }^\ad.
\end{multline*}
 Taking into account
\begin{multline*}
 P_{2n+1}^{[f]}(\ko{z})A^\ad _{2n+1}(z)\\
 =
 \begin{pmat}[{|}]
  H_{\at{n}} \ek{ R_{T_{q,n}}(\ko{z}) }^\invad& R_{T_{q,n}} (\ko{z}) [v_{q,n} (\ko{z}- \alpha) f(\ko{z})-(- \alpha u_n-y_{0,n})]\cr\-
  (z - \alpha)f^\ad (\ko{z}) v^\ad _{q,n}-(- \alpha u_n^\ad  - y^\ad _{0,n})&\frac{(\ko{z}-\alpha)f(\ko{z})-(z-\alpha)f^\ad (\ko{z})}{\ko{z} -z}\cr
 \end{pmat}  ,                                                                         
\end{multline*}
 we conclude
\begin{multline*}           
 A_{2n+1}(z) P_{2n+1}^{[f]}(\ko{z}) A_{2n+1}^\ad  (z)\\
 =
 \begin{pmat}[{|}]
  \ek{ R_{T_{q,n}}(\ko{z})}^\inv H_{\at{n}} \ek{ R_{T_{q,n}}(\ko{z}) }^\invad& v_{q,n} (\ko{z}- \alpha) f(\ko{z})-(- \alpha u_n-y_{0,n}) \cr\-
  (z - \alpha)f^\ad (\ko{z}) v^\ad _{q,n}-(- \alpha u_n^\ad  - y^\ad _{0,n})&\frac{(z-\alpha)f^\ad (\ko{z})-(\ko{z}-\alpha)f(\ko{z})}{z -\ko{z}}\cr
 \end{pmat}.
\end{multline*}
 In  view of \eqref{Nr.VQF}, \eqref{Nr.Z-ZQ}, and \eqref{Nr.FN2}, it follows
\[\begin{split}
  &B_{2n+1}(z) A_{2n+1}(z) P_{2n+1}^{[f]}(\ko{z}) A_{2n+1}^\ad  (z) B^\ad _{2n+1}(z)\\
  &=
  \begin{pmat}[{|}]
   \begin{gathered}
    \ek{R_{T_{q,n}}(\ko{z})}^\inv H_{\at{n}}\ek{ R_{T_{q,n}}(\ko{z}) }^\invad\\
    +(z-\ko{z})v_{q,n}\ek{(z-\alpha)f^\ad (\ko{z})v^\ad _{q,n}-(- \alpha u_n^\ad  - y^\ad _{0,n})}\\
    -(z-\ko{z})\ek{v_{q,n}(z - \alpha)f^\ad (\ko{z})-(- \alpha u_n - y_{0,n})}v^\ad _{q,n}   
   \end{gathered}&
   \begin{gathered}
    v_{q,n}(z-\alpha)f^\ad (\ko{z})\\
    -(- \alpha u_n-y_{0,n})
   \end{gathered}\cr\-
   \begin{gathered}
    (z - \alpha)f^\ad (\ko{z}) v^\ad _{q,n}-(- \alpha u_n^\ad  - y^\ad _{0,n})\\
    -\ek{ (z- \alpha)f^\ad (\ko{z})-(\ko{z}-\alpha)f(\ko{z})} v^\ad _{q,n}
   \end{gathered}&\frac{(z-\alpha)f^\ad (\ko{z})-(\ko{z}-\alpha)f(\ko{z})}{z -\ko{z}}\cr
 \end{pmat}\\
 &=
 \begin{pmat}[{|}]
  \begin{gathered}
   \ek{R_{T_{q,n}}(\ko{z})}^\inv H_{\at{n}}\ek{ R_{T_{q,n}}(\ko{z}) }^\invad\\
   +(\ko{z}-z) \ek{ v_{q,n} (- \alpha u_n - y_{0,n})^\ad  - (- \alpha u_n - y_{0,n}) v^\ad _{q,n} }
  \end{gathered}&
  \begin{gathered}
   v_{q,n}(z-\alpha)f^\ad (\ko{z})\\
   -(- \alpha u_n-y_{0,n})
  \end{gathered}\cr\-
  \ek{ v_{q,n}(z-\alpha)f^\ad (\ko{z})-(- \alpha u_n-y_{0,n})}^\ad&\frac{(z-\alpha)f^\ad (\ko{z})-\ek{(z-\alpha)f^\ad (\ko{z})}^\ad }{z -\ko{z}}\cr
 \end{pmat}\\
 &=
 \begin{pmat}[{|}]
  \ek{R_{T_{q,n}}(z)}^\inv H_{\at{n}}\ek{ R_{T_{q,n}}(z) }^\invad&  v_{q,n}(z-\alpha)f^\ad (\ko{z})-(- \alpha u_n-y_{0,n})\cr\-
  \ek{ v_{q,n}(z-\alpha)f^\ad (\ko{z})-(- \alpha u_n-y_{0,n})}^\ad&\frac{(z-\alpha)f^\ad (\ko{z})-\ek{(z-\alpha)f^\ad (\ko{z})}^\ad }{z -\ko{z}}\cr
 \end{pmat}.
\end{split}\]
 Consequently,
\[\begin{split}
 &X_{2n+1}(z) P_{2n+1}^{[f]}(\ko{z}) X_{2n+1}^\ad (z)\\
 &=
 \begin{pmat}[{|}]
  H_{\at{n}}&
  \begin{gathered}
   R_{T_{q,n}}(z) [ v_{q,n}(z-\alpha)f^\ad (\ko{z})\\
   \qquad-(- \alpha u_n-y_{0,n}) ]
  \end{gathered}\cr\-
  \rk{ R_{T_{q,n}}(z) \ek{v_{q,n}(z-\alpha)f^\ad (\ko{z})-(- \alpha u_n-y_{0,n})} }^\ad&\frac{(z-\alpha)f^\ad (\ko{z})-\ek*{(z-\alpha)f^\ad (\ko{z})}^\ad }{z -\ko{z}}\cr
 \end{pmat}\\
 &= P^{[f^\vee]}_{2n+1}(z).\qedhere
\end{split}\]
\end{proof}

 In the following, we will use \(\Borpq \) to denote the \(\sigma\)\nobreakdash-algebra of all Borel subsets of \(\Cpq \).
 Let \((\Omega, \gA)\) be a measurable space and let \(\mu \in \Mggqa{\Omega,\gA}\).
 Then \(\mu\) is absolutely continuous with respect to its trace measure \(\tau \defeq\tr\mu\).
 Let \(\mu'_\tau\) be a version of the Radon--Nikodym derivative of \(\mu\) with respect to \(\tau\).
 A pair \([\Phi,\Psi]\) of an \(\gA\)\nobreakdash-\(\Borpq \)\nobreakdash-measurable mapping \(\Phi\colon\Omega\to\Cpq \) and an \(\gA\)\nobreakdash-\(\Boruu{r}{q}\)\nobreakdash-measurable mapping \(\Psi\colon\Omega\to\Coo{r}{q}\) is called left-integrable with respect to \(\mu\) if \(\Phi\mu'_\tau\Psi^\ad  \) belongs to  \([\LoaaaC{1}{\Omega}{\gA}{\tau}]^\xx{p}{r}\).
 In this case, the corresponding integral is defined by \(\int_\Omega \Phi\dif\mu  \Psi^\ad \defeq \int_\Omega\Phi\mu'_\tau \Psi^\ad  \dif\tau\) and we also use the notation \(\int_\Omega\Phi(\omega)\mu(\dif\omega)\Psi^\ad (\omega)\) for it.
 In the following, when we write such an integral \(\int_\Omega \Phi\dif\mu  \Psi^\ad \), then we also mean that the pair \([\Phi, \Psi]\) is left-integrable with respect to \(\mu\).
 By \(\pqLsaaaC{\Omega}{\gA}{\mu}\) we denote the set of all \(\gA\)\nobreakdash-\(\Borpq \)\nobreakdash-measurable mappings for which the pair \([\Phi, \Phi]\) is left-integrable which respect to \(\mu\).
 Furthermore, for each subset \(A\) of \(\Omega\), we will use \(1_A\) to denote the indicator function of the set \(A\) (defined on \(\Omega\)).

\begin{rem}\label{D1229}
 Let \(\Omega\in\BorR  \setminus\set{\emptyset}\), let \(m\in\NO \), and let \(\sigma\in\MggqO \).
 In view of \rlem{B26}, it is readily checked that \(\sigma\) belongs to \(\Mgguqa{2m}{\Omega}\) if and only if \(\Rstr_\Omega E_{q,m}\) belongs to \(\aaLsaaaC{(m+1) q}{q}{\Omega}{\BorO}{\sigma}\), where \(E_{q,m}\) is given by \eqref{3.64.-1}.
 If \(\sigma\in\Mgguqa{2m}{\Omega}\), then \rlem{B26} also shows that, for each \(n\in\NO \) with \(n\le m\), the block Hankel matrix \(H_n^{[\sigma]} \defeq\mat{s_{j+k}^{[\sigma]}}_{j,k=0}^n\) admits the integral representation
    \beql{Mi,L3.9-28}
      H_n^{[\sigma]} 
      = \int_\Omega E_{q,n}(t)\sigma(\dif t) E_{q,n}^\ad (t).
    \eeq  
\end{rem}

 If \(\alpha\in\R \), if \(\kappa\in\Ninf \), and if \(\sigma\in\Mgguqa{\kappa}{\rhl}\), then let \(H_{\at{n}}^{[\sigma]} \defeq\mat{s_{\at{j+k}}^{[\sigma]}}_{j,k=0}^n\) for each \(n\in\NO\) with \(2n+1 \le \kappa\).

\begin{rem} \label{Mi,L3.6}
 Let \(\alpha\in\R\) and let \(\sigma \in \Mgguqa{1}{\rhl}\).
 Using \rprop{M24} and \rrem{M22}, it is readily checked that the following statements hold true (for details see~\cite[Lemma~5.7]{Sch11} and~\zitaa{MP11}{\clem{3.12}}):
\begin{enui}
 \item\label{Mi,L3.6.a} The function \(\phi\colon\rhl\to\Cqq\) defined by \(\phi(t) \defeq \sqrt{t-\alpha}\Iq \) belongs to \(\aaLsaaaC{q}{q}{\rhl}{\BorK}{\sigma}\) and  \(\sigma^\#\colon\BorK \to \Cqq\) given by
  \beql{S1421N}
   \sigma^\#(B)
   \defeq \int_{B}(\sqrt{t-\alpha}\Iq)\sigma(\dif t)(\sqrt{t-\alpha}\Iq)^\ad 
  \eeq 
 belongs to \(\MggqK \).
 \item\label{Mi,L3.6.b} If \(n\in\NO \) and if \(\sigma\in\Mgguqa{2n+1}{\rhl}\), then 
   \beql{D2012}
      H_{\at{n}}^{[\sigma]}
      = \int_{\rhl}\ek*{\sqrt{t-\alpha}E_{q,n}(t)}\sigma(\dif t)\ek*{\sqrt{t-\alpha}E_{q,n}(t)}^\ad.
       \eeq 
 \item\label{Mi,L3.6.c} If \(n \in \NO \) and if \(\sigma^\#\in \Mgguqa{2n}{\rhl}\), then \(\sigma\) belongs to \(\Mgguqa{2n+1}{\rhl} \) and furthermore \(s_j^{[\sigma^\#]} =s_{j+1}^{[\sigma]}-\alpha s_j^{[\sigma]}\) for all  \(j\in\mn{0}{2n}\) and \(H_n^{[\sigma^\#]}= H_{\at{n}}^{[\sigma]}\).
\end{enui}
\end{rem}

 The next proposition shows that each solution of problem~\mprob{\rhl}{m}{\lleq}  fulfills necessarily the system of the corresponding Potapov's fundamental matrix inequalities.

\begin{prop}   \label{lemM4213-2}
 Let  \(\alpha \in \R\), let  \(m \in \NO \), and let \(\seqs{m}\) be a sequence of complex \tqqa{matrices} such that \(\MggqKskg{m} \neq \emptyset\).
 Let  \(\sigma \in \MggqKskg{m}\)  and let \(S\) be the  \taSto{\sigma}.
 For each \(j \in \mn{0}{m}\), let \(\suo{j}{\sigma}\) be given by \eqref{FR1148}.
Then
  \[
    P^{[S]}_{2n}(z)
    =\int_{\rhl}
     \begin{bmatrix}         E_{q,n}(t) \\       \frac{1}{\ko {t-z}} \Iq     \end{bmatrix}     \sigma(\dif t) 
     \begin{bmatrix}         E_{q,n}(t) \\       \frac{1}{\ko {t-z}} \Iq     \end{bmatrix}^\ad+
     \begin{bmatrix}         \sv_{q,n} \\       \Oqq     \end{bmatrix}     ( s_{2n}-s_{2n}^{[\sigma]} ) 
     \begin{bmatrix}         \sv_{q,n} \\       \Oqq     \end{bmatrix}^\ad 
  \]
 for each  \(n\in\NO \) with \(2n\le m\) and all \(z\in\C\setminus\R\), where \(E_{q,n}\) is given by \eqref{3.64.-1}, and 
  \begin{multline*}
    P^{[S]}_{2n+1}(z)
    = \int_{\rhl}      \rk*{ \sqrt{t-\alpha}     \begin{bmatrix}         E_{q,n}(t) \\       \frac{1}{\ko {t-z}} \Iq     \end{bmatrix}}     \sigma(\dif t)
    \rk*{ \sqrt{t-\alpha}        \begin{bmatrix}         E_{q,n}(t) \\       \frac{1}{\ko {t-z}} \Iq     \end{bmatrix}         }^\ad\\ +\begin{bmatrix}         \sv_{q,n} \\       \Oqq     \end{bmatrix}     ( s_{2n+1}-s_{2n+1}^{[\sigma]} ) 
     \begin{bmatrix}         \sv_{q,n} \\       \Oqq     \end{bmatrix}^\ad      
  \end{multline*}
 for each \(n\in\NO \) with \(2n + 1\le m\) and all \(z\in\C\setminus\R\).
 In particular, for every choice of  \(k\in \mn{0}{m}\) and \(z \in\C\setminus\R\), the matrix \(P^{[S]}_{k}(z)\) is \tnnH{}.
 
\end{prop}

 \rprop{lemM4213-2} can be proved using standard arguments of integration theory of \tnnH{} measures (\rlem{B26} and \rrem{M22}).
 We omit the details.

\section{Some integral estimates for the scalar case}\label{S1639}
 In this section, we state some integral representations and integral estimates in the scalar case \(q=1\).
 
\begin{lem} \label{MB2210}
 Let \(\alpha\in\R\) and let \(F\in\RF{1}\) with Nevanlinna parametrization \((A,B,\nu)\) and spectral measure \(\mu\).
 Then:
\begin{enui}
 \item\label{MB2210.a} For each \(w\in\uhp\), the integral \(\int_\R\abs{t-w}^{-2}\mu (\dif t)\)  is finite and
   \beql{N1943D}
   \Im F(w) 
   = (\Im w)\ek*{ B + \int_\R \frac{1}{\abs{t-w}^2} \mu (\dif t)}.
    \eeq 
 \item\label{MB2210.b} For each \(w\in\uhp\), the integral    \(\int_\R\abs{ t\ek{\abs{t-w}^{-2}-\rk{1+t^2}^\inv}-\alpha\abs{t-w}^{-2}}\mu(\dif t)\)  is finite and \(F^\#\colon\uhp\to\C\) defined by
   \beql{Nr.FLG}
    F^\#(w)
    \defeq (w-\alpha)F(w)
   \eeq 
  satisfies, for each \(w\in \uhp\), the equation
   \begin{multline}\label{N2032D}
    \Im F^\#(w) \\
    =(\Im w)\rk*{ A+B(2\Re w-\alpha) +\int_\R\ek*{ t\rk*{ \frac{1}{\abs{t-w}^2} - \frac{1}{1+t^2}}-\frac{\alpha}{\abs{t-w}^2}}\mu (\dif t)}.
   \end{multline}
 \end{enui}   
\end{lem}
\begin{proof}  
 In view of 
\[
 \int_\R \frac{1}{1+t^2} \mu (\dif t)
 =\int_\R \frac{1}{1+t^2}(1+t^2)\nu(\dif t)
 = \nu(\R) 
 < \infp,
\]
 we see that, for each \(w\in\uhp\), the function \(\psi_w\colon\R\to\C\) given by  the equation \(\psi_w(t)\defeq\rk{t-w}^\inv-t\rk{1+t^2}^\inv\) belongs to \(\LoaaaC{1}{\R}{\BorR}{\mu}\).
 By virtue of a result due to R.~Nevanlinna (see, \teg{}~\zitaa{MR0458081}{\cthm{A.2}}), for each  \(w\in\uhp\), we have
\beql{Nr.222}
       F(w)           =A+Bw+\int_\R\rk*{\frac{1}{t-w}-\frac{t}{1+t^2}}\mu (\dif t).
      \eeq 
 \eqref{MB2210.a} Let \(w\in\uhp\).
 For each \(t\in\R\), then \(\Im \psi_w(t)= (\Im w)\abs{t-w}^{-2}\).
 Thus, 
\begin{equation*}
 \begin{split}
  \int_\R \abs*{\frac{1}{\abs{t-w}^2}}\mu (\dif t)
  =  \frac{1}{\Im w}  \int_\R \Im  \psi_w(t)  \mu(\dif t) 
 &=  \frac{1}{\Im w}  \Im \ek*{\int_\R   \psi_w(t)  \mu (\dif t)}\\
 & \leq \frac{1}{\Im w}\int_\R\abs*{ \psi_w(t) } \mu (\dif t) 
 < \infp
 \end{split}
 \end{equation*}
 and 
\beql{Nr.45-1}
 \Im\ek*{\int_\R \psi_w(t) \mu (\dif t)}
 =\int_\R \Im \psi_w(t) \mu(\dif t)    
 = (\Im w)\int_\R \frac{1}{\abs{t-w}^2} \mu (\dif t).
\eeq 
 Because of \(A\in\R \) and \(B\in[0,\infp)\), we have \(\Im A=0\) and \(\Im(wB)=(\Im w) B\).
 Consequently, from  \eqref{Nr.222}, and \eqref{Nr.45-1} we get then \eqref{N1943D}.
 
 \eqref{MB2210.b} Let \(w\in\uhp\).
 In view of \eqref{Nr.FLG} and \eqref{Nr.222}, we obtain
\beql{M,S60-2}
 F^\#(w)
 =A(w-\alpha)+Bw(w-\alpha)+\int_\R\ek*{\frac{w-\alpha}{t-w}-\frac{t(w-\alpha)}{1+t^2}}\mu (\dif t).
\eeq 
 For each  \(t\in\R\), we see that \((w-\alpha)\psi_w(t)=\rk{w-\alpha}/\rk{t-w}-t(w-\alpha)/\rk{1+t^2}\) holds true.
 Hence, \((w-\alpha)\psi_w \in \LoaaaC{1}{\R}{\BorR}{\mu}\) and, for each \(t\in\R\), we have furthermore
\[
 2\iu  \Im\ek*{(w-\alpha)\psi_w(t)}
 =2\iu (\Im w)\ek*{ t\rk*{\frac{1}{\abs{t-w}^2} - \frac{1}{1+t^2}}-\frac{\alpha}{\abs{t-w}^2}}.
\]
 This implies
\begin{equation*}
 \begin{split}
 \int_\R \abs*{ t\rk*{\frac{1}{\abs{t-w}^2}-\frac{1}{1+t^2}}-\frac{\alpha}{\abs{t-w}^2}}\mu (\dif t)
 &= \frac{1}{\Im w} \int_\R  \abs*{\Im\ek*{(w-\alpha)\psi_w(t)}} \mu(\dif t)\\
 &\leq \frac{1}{\Im w}\int_\R\abs*{ (w-\alpha)\psi_w(t) }\mu (\dif t)  
 < \infp
 \end{split}
 \end{equation*}
 and
\beql{M,S61-2}
 \begin{split}
 \Im\ek*{\int_\R (w-\alpha)\psi_w(t) \mu (\dif t)}
 &=\int_\R \Im\ek*{(w-\alpha)\psi_w(t)}\mu (\dif t)\\
 &=(\Im w)\int_\R\ek*{ t\rk*{ \frac{1}{\abs{t-w}^2} - \frac{1}{1+t^2}}-\frac{\alpha}{\abs{t-w}^2}}\mu (\dif t).
 \end{split}
\eeq 
 Obviously,  \(\Im(w^2)= 2(\Re w)(\Im w)\).
 Consequently,  \(\Im[w(w-\alpha)] =\Im(w^2)-\Im(w\alpha)=(\Im w)(2\Re w-\alpha)\).
 Thus,  \(\Im[Bw(w-\alpha)]=B (\Im w)(2\Re w-\alpha)\).
 Then, by virtue of  \eqref{M,S60-2}, and \eqref{M,S61-2}, we get
\begin{equation*}
 \begin{split}
 &\Im F^\# (w)   
 =\Im\rk*{ A(w-\alpha)+Bw(w-\alpha)+\int_\R\ek*{\frac{w-\alpha}{t-w}-\frac{t(w-\alpha)}{1+t^2}}\mu (\dif t)} \\
 &= \Im\ek*{A (w-\alpha)} + \Im\ek*{Bw(w-\alpha)} + \Im\rk*{\int_\R\ek*{\frac{w-\alpha}{t-w}-\frac{t(w-\alpha)}{1+t^2}}\mu (\dif t)}\\
 &=A \Im w+B(\Im w)(2\Re w-\alpha)+(\Im w)\int_\R\ek*{ t\rk*{\frac{1}{\abs{t-w}^2} - \frac{1}{1+t^2}}-\frac{\alpha}{\abs{t-w}^2}}\mu (\dif t).
\end{split}
\end{equation*}
 Thus, \eqref{N2032D} follows.
\end{proof}

\begin{rem} \label{MB2212} 
 Let \(\alpha \in \R\) and let \(F \in \RF{1}\) with spectral measure \(\mu\).
 Further, let  \(\ell_1, \ell_2\in\R \) be such that  \(\ell_1<\ell_2 <\alpha\).
 Then it is readily checked that for every choice of \(a\in (-\infty, \ell_1)\) and \(b\in (\ell_2, \infp)\), there exists a \(K_{a,b}\in\R \) such that, for each \(x\in [\ell_1,\ell_2]\), the inequality  \(\int_{\R \setminus (a,b)}\rk{t-x}^{-2}\mu (\dif t)<K_{a,b}\) holds true.
\end{rem}

\begin{rem}  \label{Mi2213}
 Let \(r,s\in\R\).
 Then it is readily checked that the following statements hold true (for details, see also~\cite[Lemma~3.7]{Sch11}):
\begin{enui}
 \item If \(r < s\) and \(s\neq0\), then there exists a number  \(a\in(-\infty,r)\cap(-\infty,0)\) such that
 \beql{N948D}
  \abs*{t\ek*{\frac{1}{(t-x)^2+y^2}-\frac{1}{1+t^2}}} 
  <\rk*{2+\abs*{\frac{r}{s}}}\cdot\abs*{t\ek*{\frac{1}{(t-s)^2+1}-\frac{1}{1+t^2}}}
 \eeq 
 is valid for every choice of  \(x\in[r,s]\) and \(y\in(0,1)\) and \(t\in(-\infty,a]\).
 \item If \(s <  r\) and \(r\neq0\), then there exists a  \(b\in(r,\infp)\cap(0,\infp)\) such that, for every choice of  \(x\in[s,r]\) and \(y\in(0,1)\) and \(t\in[b,\infp)\), inequality \eqref{N948D} holds true.
\end{enui}
\end{rem}

\begin{lem} \label{MB2214}
 Let \(\alpha\in\R\) and let \(F\in\RF{1}\) with spectral measure \(\mu\).
 Further, let \(\ell_1\) and \(\ell_2\) be real numbers with \(\ell_1 < \ell_2 < \alpha\).
 Then there are real numbers \(a\), \(b\), and \(C\) with \(a<\ell_1\) and \(\ell_2 < b<\alpha\) such that
\[
 \int_{\R \setminus (a,b)}\abs*{ t\ek*{\frac{1}{(t-x)^2+y^2}-\frac{1}{1+t^2}}-\frac{\alpha}{(t-x)^2+y^2}}\mu (\dif t) 
 < C
\]
 holds true for every choice of  \(x\in[\ell_1,\ell_2]\) and  \(y\in(0,1)\).
\end{lem}

 Using \rlem{MB2210} and \rremss{MB2212}{Mi2213}, \rlem{MB2214} can be proved analogous to the well-known special case \(\alpha =0\).
 However, in the general case of on arbitrary real number \(\alpha\), these straightforward calculations are very lengthy (see~\cite[Lemma~3.8]{Sch11} for details).

\begin{lem} \label{MB2215}
 Let \(\alpha\in\R\) and let \(F\in\RF{1}\) be such that \(F^\#\colon\uhp\to\C\) defined by \eqref{Nr.FLG} belongs to \(\RF{1}\).
 Further, let \(\mu\) be  the spectral measure of \(F\) and let \(\ell_1\) and \(\ell_2\) be real numbers with \(\ell_1< \ell_2 < \alpha\).
 Then there are real numbers \(a\), \(b\), and \(C\) with \(a < \ell_1\) and \(\ell_2 < b<\alpha\) such that
\begin{align*} %
 \int_{(a,b)}\abs*{\frac{t-\alpha}{(t-x)^2+y^2}} \sigma(\dif t)&< C&
&\text{and}&
 \int_{(a,b)}\abs*{\frac{1}{(t-x)^2}} \sigma(\dif t)&< C
\end{align*}
 hold true for every choice of  \(x\in[\ell_1,\ell_2]\) and  \(y\in [0,\infp)\).
\end{lem}

 \rlem{MB2215} can be proved, using \rlemss{MB2210}{MB2214} and Beppo Levi's Theorem of monotone convergence (for details, see~\cite[Lemma~3.12]{Sch11}).
 
\begin{rem} \label{ML2216}
 Let \(\alpha\in\R\) and let \(F\in\RF{1}\) be such that \(F^\#\colon\uhp\to\C\)  defined by \eqref{Nr.FLG} belongs to \(\RF{1}\).
 Let \(\mu\) be the spectral measure of \(F\) and let \(\ell_1\) and \(\ell_2\) be real numbers with \(\ell_1 < \ell_2 < \alpha\).
 Then one can easily see from \rrem{MB2212} and \rlem{MB2215} that there is a real number \(C\) such that \(\int_\R(t-x)^{-2}\mu(\dif t) < C\) for all  \(x\in[\ell_1,\ell_2]\).
\end{rem}

\begin{lem} \label{ML2217}
 Let  \(\alpha\in\R\) and let \(F\in\RF{1}\) be such that  \(F^\#\colon\uhp\to\C\) defined by  \eqref{Nr.FLG} belongs to  \(\RF{1}\).
 Then the Nevanlinna   measure \(\nu\) of \(F\) and the spectral measure \(\mu\) of \(F\) fulfill \(\nu(\crhl )=0\) and \(\mu (\crhl )=0\).
\end{lem}
\begin{proof}
 We give the proof stated in~\cite[Lemmata~3.13 and~3.14]{Sch11}.
 
 (I) In the first step of the proof, we consider arbitrary real numbers \(\ell_1\) and \(\ell_2\) with \(\ell_1<\ell_2<\alpha\).
 Let \((A,B,\nu)\) be the  Nevanlinna parametrization of  \(F\).
 Because of \rrem{ML2216}, there is a \(C\in\R\) such that  \(\int_\R(t-x)^{-2}\mu(\dif t) < C\) is true for all \(x\in[\ell_1,\ell_2]\).
 Since \(F\) belongs to  \(\RF{1}\), for each  \(x\in[\ell_1,\ell_2]\) and each \(\epsilon\in(0,\infp)\), from  \rlem{MB2210} we get then
\[
 0 
 \leq \Im F(x+\iu\epsilon)
 = \epsilon\ek*{ B+\int_\R\frac{1}{(t-x)^2+\epsilon^2}  \mu (\dif t)}
 <\epsilon(B+C)
\]
 and, consequently,
\beql{NM0947}
 0
 \le  \int_{[\ell_1,\ell_2]}\Im F(x+\iu\epsilon) \lambda^{(1)}(\dif x) 
 \leq  \epsilon(B+C)(\ell_2-\ell_1),
\eeq 
 where \(\lambda^{(1)}\) is the Lebesgue measure defined on \(\BorR \).
 In view of  \(F\in\RF{1}\), the inversion formula of Stieltjes--Perron (see, \teg{}~\zitaa{MR0458081}{Appendix, p.~390}) yields
\beql{M,S82-3}
 \frac{1}{2}\ek*{\sigma\rk*{\set{\ell_1}}+\sigma\rk*{\set{\ell_2}}}+\sigma\rk*{(\ell_1,\ell_2)}
 =\frac{1}{\pi} \lim_{\epsilon\to 0+0}\int_{[\ell_1,\ell_2]}\Im F(x+\iu\epsilon) \lambda^{(1)}(\dif x).
\eeq 
 Combining \eqref{M,S82-3} and \eqref{NM0947}, we obtain \(\sigma((\ell_1,\ell_2))=0\), from
\begin{equation*}
\begin{split}
 0     &\leq \sigma\rk*{(\ell_1,\ell_2)}
 \leq\frac{1}{2}\ek*{\sigma\rk*{\set{\ell_1}}+\sigma\rk*{\set{\ell_2}}}+\sigma\rk*{(\ell_1,\ell_2)}\\
 &= \frac{1}{\pi} \lim_{\epsilon\to 0+0} \int_{[\ell_1,\ell_2]}\Im F(x+\iu\epsilon) \lambda^{(1)}(\dif x)
 \leq \frac{1}{\pi}\lim_{\epsilon\to 0+0}\ek*{\epsilon(B+C)(\ell_2-\ell_1)}
 = 0.
\end{split}
\end{equation*}

 (II) For each \(n\in\N\), the real numbers  \(a_n\defeq \alpha-(1+n)\) and \(b_n\defeq \alpha-\frac{1}{n}\) fulfill \(a_n < b_n < \alpha\).
 Consequently, part~(I) of the proof provides us \(\mu ((a_n,b_n))=0\).
 Obviously,  \((a_n,b_n) \subseteq (a_{n+1},b_{n+1})\) for each \(n\in\N \) and \(\bigcup_{n=1}^\infi (a_n,b_n) = \crhl \).
 Hence, \(\mu (\crhl )=\mu(\bigcup_{n=1}^\infi (a_n,b_n))  = \lim_{n\to\infi} \mu ((a_n,b_n)) =  0\).
 Thus, \(\nu(\crhl )=0\) follows from
\[
 0 
 \leq \nu\rk*{\crhl }
 = \int_{\crhl }1\nu(\dif t)
 \leq \int_{\crhl }(1+t^2) \nu(\dif t)
 = \mu\rk*{\crhl }
 =0.\qedhere
\]
\end{proof}

\section{From the system of Potapov's fundamental matrix inequalities to the moment problem}\label{S0846}
 \rprop{lemM4213-2} showed that the Stieltjes transform of an arbitrary solution of problem~\mprob{\rhl}{m}{\lleq} fulfills necessarily the system of corresponding Potapov's fundamental matrix inequalities.
 In this section, we are going to prove that the validity of the system of Potapov's fundamental matrix inequalities for a holomorphic \tqqa{matrix-valued} function defined on \(\Cs \)  is also sufficient to be the Stieltjes transform of some solution of this matricial Stieltjes-type moment problem.
 For the convenience of the reader, first we state two well-known facts.

\begin{rem} \label{MS2221}
 Let \(\cD\) be a discrete subset of \(\uhp\) and let \(F\colon\uhp\setminus\cD\to\Cqq\) be a matrix-valued function which is holomorphic in \(\uhp\setminus\cD\)  and which fulfills \(\Im F(z) \in\Cggq\) for all  \(z\in\uhp\setminus\cD\).
 Then one can easily see from~\cite[Lemma~2.1.9]{MR1152328} that  there is a function \(F^\triangle\in\RFq\) such that \(\Rstr_{\uhp\setminus\cD}F^\triangle=F\).
\end{rem}

\begin{rem} \label{Th3}
 Let \(A,B\in\Cqq\), let \(M\) be an open subset of  \(\R\), and let \(\nu\in\Mggqa{\R\setminus M}\).
 In view of a well-known result on integrals which depend on a complex parameter (see, \teg{}~\cite[Satz~5.8]{MR2257838}), it is readily checked that  \(\phi\colon\uhp\cup M\cup\lhp \to\Cqq\) given by
\[
 \phi(z)
 \defeq  A +Bz +\int_{\R\setminus M}\frac{1+tz}{t-z}\nu(\dif t)
\]
 is holomorphic in \(\uhp\cup M\cup\lhp \) (see also, \teg{}~\cite[Lemma~3.17]{Sch11}).
\end{rem}

 In the following, for all \(\alpha\in\R\), let \(\lehpa \defeq\setaca{z\in\C}{\Re z\in\crhl }\)\index{c@\(\lehpa \)}.

\begin{lem} \label{ML2223}
 Let \(\alpha\in\R\) and let \(F\in\RFq\) be such that \(F^\#\colon\uhp\to\Cqq\) defined by \(F^\#(w)\defeq (w-\alpha)F(w)\) belongs to \(\RFq\).
 Further, let \(\nu\) be the Nevanlinna measure of  \(F\).
 Then \(\nu(\crhl )=\NM\) and the following two statements hold true:
\begin{enui}
 \item\label{ML2223.a} There is a function \(F_\alpha \colon\Cs \to \Cqq \) such that \(\Rstr_{\uhp } F_\alpha = F\) and \(F_\alpha(\crhl )\subseteq\CHq\) are fulfilled.
 \item\label{ML2223.b} There exists a unique function \(S\in\SFq\)  with \(\Rstr_{\uhp} S=F\).
\end{enui}
\end{lem}
\begin{proof} 
 Since \(F\) and \(F^\#\) belong to \(\RFq\), for all \(u\in\Cq\), we see that \(\set{ u^\ad Fu, u^\ad F^\#u } \subseteq \RF{1}\) and that  \(u^\ad \nu u\) is the Nevanlinna measure of \(u^\ad Fu\).
 Because of \rlem{ML2217}, for  all \(u\in\Cq\), we have \(u^\ad  \nu(\crhl ) u =(u^\ad \nu u)(\crhl ) = 0  =u^\ad \Oqq u\).
 Consequently, \(\nu(\crhl )=\Oqq\).
 
 \eqref{ML2223.a} Obviously,  \(\tilde{\nu}\defeq \Rstr_{\BorK}\nu\) belongs to \(\MggqK \).
 By virtue of \(F\in\RFq\) and \rthm{T31DN}, there are matrices \(A\in\CHq\) and \(B\in\Cggq\) such that \eqref{N31DN} holds for each \(z\in\uhp\).
 \rrem{Th3} shows that \(F_{\alpha}\colon\Cs\to\Cqq\) given by
\beql{Mi87-1}
 F_\alpha (z)
 \defeq  A+Bz+\int_{\rhl} \frac{1 + tz}{t-z} \tilde{\nu} (\dif t)
\eeq 
 is holomorphic in \(\Cs\).
 Comparing \eqref{N31DN} and  \eqref{Mi87-1}, we get \(F_\alpha(z) = F(z)\) for each \(z\in\uhp \).
 For every choice of \(x\in\R \), we have
\[
 \ek*{  \int_{\rhl } \frac{1+tx}{t-x} \tilde{\nu} (\dif t)}^\ad 
 = \int_{\rhl } \ko{\rk*{\frac{1+tx}{t-x}}} \tilde{\nu} (\dif t)  
 = \int_{\rhl } \frac{1+tx}{t-x} \tilde{\nu} (\dif t).
\]
 In view of \eqref{Mi87-1}, \(A\in\CHq\),  and \(B\in\Cggq\), then \([F_\alpha (x)]^\ad = F_\alpha (x)\) follows for each  \(x\in\crhl \).

 \eqref{ML2223.b} Because of \rpart{ML2223.a}, there is a holomorphic function \(S\colon\Cs\to\Cqq\) such that 
\begin{align} \label{M,S89-1}
 \Rstr_{\uhp}S&= F&
&\text{and}&
 S(\crhl )&\subseteq \CHq
\end{align}
 hold true.
 According to \(\set{ F, F^\#} \subseteq \RFq\) and \eqref{M,S89-1}, for all \(z\in\uhp\), then
\begin{align} \label{M,S89-3}
 \Im S(z)&= \Im F(z) \in\Cggq&
&\text{and}&
 \Im\ek*{(z-\alpha)S(z)}& =  \Im F^\#(z) \in\Cggq.
\end{align} 
 For all \(z\in\lehpa \cap\uhp\), we have \(\Im[(z-\alpha)S(z)]= [\Re (z-\alpha)]\Im S(z) + (\Im z)\Re S(z)\) and, by virtue of \eqref{M,S89-3}, consequently, 
\beql{M,S89-13}
 \Re S(z)
 = \frac{\Im[(z-\alpha)S(z)]}{\Im z}+\ek*{-\Re (z-\alpha)}\frac{\Im S(z)}{\Im z}
 \in \Cggq.
\eeq 
 Now we consider an arbitrary monotonically nondecreasing sequence \((y_n)_{n=1}^\infi\) of positive real numbers with \(\lim_{n\to\infi}y_n = 0\).
 Since the function \(S\) is holomorphic in \(\Cs  \), the functions \( \Re S\) and \(\Im S\) are continuous in \(\Cs  \).
 Thus, for each \(x\in \crhl \), we have \(x + \iu y_n \in\C _{\alpha, -}\cap \uhp \) for all \(n\in\N \) and, hence, because of \eqref{M,S89-13}, and \eqref{M,S89-3}, then
\begin{align} \label{M,S90-3-1}
 \Re S(x)&=\lim_{n\to\infi}\Re S(x+\iu y_n) \in\Cggq&
&\text{and}&
 \Im S(x)&= \lim_{n\to\infi}\Im S(x+\iu y_n) \in\Cggq.
\end{align}
 Combining  \eqref{M,S89-1} and \eqref{M,S90-3-1}, for each \(x\in\crhl \), we get  \(\Re S(x)+\iu\Im S(x) =S(x)=[S(x)]^\ad = \Re S(x)-\iu\Im S(x) \) and, hence, \(\Im S(x)=0\).
 From \eqref{M,S90-3-1} then  \(S(x)\in \Cggq\) follows for each \(x\in \crhl \).
 Consequently,  \(S\in\SFq\).
 
 Now we consider an arbitrary \(S^\square\in\SFq\)  such that \(\Rstr_{\uhp} S^\square=F\).
 From  \eqref{M,S89-1} we get then \(S^\square(z)=F(z)=S(z)\) for each \(z\in\uhp \).
 Thus, the identity theorem for holomorphic functions provides us \(S^\square=S\).
\end{proof}

\begin{prop} \label{MT2224}
 Let \(\alpha\in\R\) and  let \(\cD\) be a discrete subset of \(\uhp\).
 Let \(F\colon\uhp\setminus\cD\to\Cqq\) be a holomorphic matrix-valued function and let \(F^\#\colon\uhp\to\Cqq\) be defined by \(F^\#(w)\defeq (w-\alpha)F(w)\).
 Suppose  \(\set{\Im F(w), \Im F^\#(w) } \subseteq \Cggq\) for all \(w\in\uhp\setminus\cD\).
 Then there is a unique \(S\in\SFq\) such that \(\Rstr_{\uhp\setminus\cD}S=F\).
\end{prop}

 \rprop{MT2224} can be easily proved using \rrem{MS2221}, \rlem{ML2223}, and the identity theorem for holomorphic functions (see also~\cite[Theorem~3.19]{Sch11}).
 We omit the details.

\begin{thm}\label{T1747F}
 Let \(\alpha\in\R \), let \(\kappa\in\NOinf \), let \(\seqska \) be a sequence of complex \tqqa{matrices}, and let \(m\in\mn{0}{\kappa}\).
 Further, let \(\cD\) be a discrete subset of \(\uhp \) and let \(F\colon\uhp  \setminus\cD\to \Cqq \) be a holomorphic matrix-valued function such that 
\begin{align}\label{N1838S}
 P_m^{[F]} (z)&\lgeq\NM&
&\text{and}&
 P_{m-1}^{[F]} (z)&\lgeq\NM&\text{for each }z&\in\uhp \setminus\cD.
\end{align}
 Then there exists a unique \(S\in\SFOq \) such that \(\Rstr_{\uhp \setminus\cD} S = F\).
 Moreover,  the inequality \(P_k^{[S]} (z)\lgeq\NM\) holds true for each \(k\in\mn{-1}{m}\) and each \(z\in\C\setminus\R\).
\end{thm}
\begin{proof}
 We give a version of the proof stated with more details in~\cite[Theorem~4.10]{Sch11}.
 From \eqref{N1838S}, \rnota{Note44}, and \eqref{Nr.B6}, we see that \(H_n\ge 0\) for each \(n\in\NO \) with \(2n\le m\), that \(H_{\at{n}} \lgeq\NM\) for each \(n\in\N \) with \(2n+1\le m\), that \(s_j^\ad = s_j\) for each \(j\in\mn{0}{m}\), and that \(\Im F(z) =(\Im z)\frac{F(z) - F^\ad (z)}{z - \ko{z}} \lgeq\NM\) and \(\Im [(z-\alpha) F(z)]  = (\Im z)\frac{(z-\alpha) F(z) - [(z-\alpha)F(z)]^\ad}{z - \ko{z}}\lgeq\NM\) hold true for each  \(z\in\uhp  \setminus\cD\).
 Thus, because of \rprop{MT2224}, there exists a unique \(S\in\SFq \) such that \(\Rstr_{\uhp \setminus \cD} S = F\).
 By continuity arguments, from \eqref{N1838S} we get then \(\set{ P_m^{[S]} (z), P_{m-1}^{[S]} (z)} \subseteq \Cggq \) for each \(z\in\uhp \) and, consequently,
\begin{align}\label{N2149S}
 P_k^{[S]} (z)&\ge 0&\text{for each }k&\in\mn{-1}{m}\text{ and each }z\in\uhp.
\end{align}
 In particular, \(\tilde{S} \defeq  \Rstr_{\uhp } S\) fulfills
\begin{align*}
 \begin{bmatrix}
  s_0 & \tilde{S} (z)\\ 
  \tilde{S}^\ad (z)  &  \frac{\tilde{S} (z) - \tilde{S}^\ad (z)}{z - \ko{z}}
 \end{bmatrix}
 &= P_0^{[S]} (z)\lgeq\NM&\text{for each }z&\in\uhp .
\end{align*}
 Consequently, \rlem{N1219DN} provides us \(\tilde{S} \in\RFOq \) and  \(\sup_{y\in [1,\infp)} y \normS{S (\iu y)} <  \infp\).
 Hence, \(S\) belongs to \(\SFOq \).
 Then \rthm{T1316D} shows that there is a \(\sigma\in\MggqK \) such that \eqref{N1319D} holds true.
 Let \(\tilde{S}^\vee \colon\lhp \to\Cqq \) be defined by \(\tilde{S}^\vee (z) \defeq  S^\ad (\ko{z})\).
 Thus, from \eqref{N1319D}, we get  
\[
 \tilde{S}^\vee (z)
 = \ek*{ \int_\rhl \frac{1}{t- \ko{z}} \sigma (\dif t)}^\ad
 = \int_{\rhl } \frac{1}{t-z} \sigma (\dif t)
 = S(z)
\] 
 for each \(z\in\lhp \).
 Taking into account \eqref{N2149S} and \rlem{MP36N}, we see then that, for each \(k\in\mn{-1}{m}\) and each \(z\in\lhp \), there exists a matrix \(X_k (z)\) such that \(P_k^{[S]} (z)=  P_{k}^{[\tilde{S}^\vee]} (z)= X_k (z) P_k^{[S]} (\ko{z}) X_k^\ad (z)\) is fulfilled for every choice of \(k\in\mn{-1}{m}\) and \(z\in\lhp \).
 In view of \eqref{N2149S}, this implies \(P_k^{[S]} (z)\lgeq\NM\) for each \(k\in\mn{-1}{m}\) and each \(z\in \lhp \).
 Because of \(\C \setminus\R  =\uhp \cup\lhp \), the proof is complete.
\end{proof}

\begin{rem}\label{D1407}
 For each \(n\in\NO \) and every choice of \(w\) and \(z\) in \(\C \), it is readily checked that
\[
 (z - \ko{w})\ek*{R_{T^\ad _{q,n}} (w)}^\ad  T_{q,n} R_{T_{q,n}}  (z)
 =  R_{T_{q,n}} (z)   - \ek*{R_{T^\ad _{q,n}} (w)}^\ad.
\]
\end{rem}

\begin{lem}  \label{D1451}
 Let \(\kappa \in \NOinf \) and let \(\seqska \) be a sequence of \tH{} complex \tqqa{matrices}.
 Then
\begin{multline} \label{MT24M}
 H_n T^\ad _{q,n} R_{T^\ad _{q,n}} (z) - \ek*{R_{T^\ad _{q,n}} (w)}^\ad  T_{q,n} H_n  +   \ek*{R_{T^\ad _{q,n}} (w)}^\ad  (v_{q,n} u^\ad _n - u_n v^\ad _{q,n}) R_{T^\ad _{q,n}} (z)\\
 = (z - \ko{w}) \ek*{R_{T^\ad _{q,n}} (w)}^\ad  T_{q,n} H_n T^\ad _{q,n} R_{T^\ad _{q,n}} (z)
\end{multline} 
 for  all \(n \in \NO \) with \(2n \le \kappa\) and every choice of \(w\) and \(z\) in \(\C \).
 Furthermore,
\begin{multline} \label{MT24N}
 H_{\at{n}} T^\ad _{q,n} R_{T^\ad _{q,n}} (z) - \ek*{R_{T^\ad _{q,n}} (w)}^\ad  T_{q,n} H_{\at{n}}\\
 + \ek*{R_{T^\ad _{q,n}} (w)}^\ad  \ek*{v_{q,n} (- \alpha u_n -y_{0,n})^\ad- (- \alpha u_n -y_{0,n}) v^\ad _{q,n}} R_{T^\ad _{q,n}} (z)\\
 = (z - \ko{w}) \ek*{R_{T^\ad _{q,n}} (w)}^\ad  T_{q,n} H_{\at{n}} T^\ad _{q,n} R_{T^\ad _{q,n}} (z)
\end{multline}   
 for all \(\alpha \in \R \), all \(n \in \NO \) with \(2n+1 \le \kappa\), and every choice of \(w\) and \(z\) in \(\C  \).
\end{lem}
\begin{proof} 
 By virtue of \rremp{lemC21-1}{lemC21-1.a}, we have 
\begin{equation*}
\begin{split}
 &H_n T^\ad _{q,n} R_{T^\ad _{q,n}} (z) - \ek*{R_{T^\ad _{q,n}} (w)}^\ad  T_{q,n} H_n + \ek*{R_{T^\ad _{q,n}} (w)}^\ad  (v_{q,n} u^\ad _n - u_n v^\ad _{q,n}) R_{T^\ad _{q,n}} (z)\\
 &= \ek*{ R_{T^\ad _{q,n}}(w) }^\ad  \rk*{ \ek*{ R_{T^\ad _{q,n}}(w) }^{-\ast} H_n T^\ad_{q,n} - T_{q,n} H_n R^\inv _{T^\ad _{q,n}}(z)    + (v_{q,n} u^\ad _n - u_n v^\ad _{q,n}) } R_{T^\ad _{q,n}}(z)  \\
 &= \ek*{ R_{T^\ad _{q,n}}(w) }^\ad\\
 &\times\ek*{(\Iu{(n+1)q} - \ko{w} T_{q,n}) H_n T^\ad _{q,n} - T_{q,n} H_n (\Iu{(n+1)q} - z T^\ad _{q,n})                  + (v_{q,n} u^\ad _n - u_n v^\ad _{q,n}) } R_{T^\ad _{q,n}}(z)  \\
 &= \ek*{ R_{T^\ad _{q,n}}(w) }^\ad  \Big[ (z - \ko{w}) T_{q,n} H_n T^\ad _{q,n} -(T_{q,n} H_n - H_n T^\ad _{q,n})       + (v_{q,n} u^\ad _n - u_n v^\ad _{q,n}) \Big] R_{T^\ad _{q,n}}(z)\\ 
 &= \ek*{ R_{T^\ad _{q,n}}(w) }^\ad  \ek*{(z - \ko{w}) T_{q,n} H_n T^\ad _{q,n}}R_{T^\ad _{q,n}}(z) 
  = (z - \ko{w}) \ek*{ R_{T^\ad _{q,n}}(w) }^\ad  T_{q,n} H_n T^\ad _{q,n} R_{T^\ad _{q,n}}(z).
\end{split}
\end{equation*}
 Using \rremp{lemC21-1}{lemC21-1.b}, equation~\eqref{MT24N} can be proved analogous to \eqref{MT24M}.
\end{proof}

\begin{nota} \label{Mi,B3.1}
 Let \(\alpha\in\R\), let \(\kappa\in\NOinf \), and let \(\seqska \) be a sequence from \(\Cqq \).
 Let \(\cG\) be a subset of \(\C\) with \(\cG\setminus\R\neq\emptyset\) and let \(f \colon\cG\to\Cqq \) be a matrix-valued function.
 For each  \(n\in\NO \) with \(2n\leq\kappa\), let  \(F_{2n}\colon\cG\to\Coo{(n+1)q}{(n+1)q}\) be given by  
\beql{Nr.F2N}
 F_{2n}(z)
 \defeq  H_n \Tqn ^\ad  R_{\Tqn ^\ad }(z) + R_{\Tqn }(z)\ek*{v_{q,n} f(z) - u_n} v_{q,n}^\ad  R_{\Tqn ^\ad }(z)
\eeq                
 and let \(Q_{2n}^{[f]}\colon\cG\setminus\R\to\Coo{(2n+2)q}{(2n+2)q}\) be defined by
\beql{Nr.Q2N}
 Q_{2n}^{[f]}(z)
 \defeq
 \begin{bmatrix}
  \Hu{n} & F_{2n}(z) \\
  F_{2n}^\ad (z) & \frac{F_{2n}(z) -F_{2n}^\ad (z)}{z-\ko z}
 \end{bmatrix}.
\eeq 
 If \(\kappa\geq1\), then, for all \(n\in\NO \) with \(2n+1\leq\kappa\), let \(F_{2n+1}\colon\cG\to\Coo{(n+1)q}{(n+1)q}\) be given by
\begin{multline}  \label{Nr.F2N+1}
 F_{2n+1}(z)
 \defeq H_{\at n}\Tqn ^\ad  R_{\Tqn ^\ad }(z)\\
  + R_{\Tqn }(z) \ek*{v_{q,n}(z-\alpha) f(z) - (-\alpha u_n - y_{0,n})} v_{q,n}^\ad  R_{\Tqn ^\ad }(z)
\end{multline}              
 and let  \(Q_{2n+1}^{[f]}\colon\cG\setminus\R\to\Coo{(2n+2)q}{(2n+2)q}\) be defined by
\beql{Nr.Q2N+1}
 Q_{2n+1}^{[f]}(z)
 \defeq
 \begin{bmatrix}
  H_{\at n} & F_{2n+1}(z) \\
  F_{2n+1}^\ad (z) & \frac{F_{2n+1}(z) -F_{2n+1}^\ad (z)}{z-\ko z}
 \end{bmatrix}.
 \eeq                
\end{nota}

\begin{prop} \label{Mi,L3.3}
 Let \(\alpha\in\R\), let \(\kappa\in\NOinf \), and let \(\seqska \) be a sequence of \tH{} complex \tqqa{matrices}.
 Let  \(f\colon\Cs\to\Cqq \) be a matrix-valued function.
 Further, for each \(k\in\NO \), let \(m_k\) be given by \eqref{N1435D} and let \(F_k\colon\Cs\to\Coo{(m_k +1)q}{(m_k +1)q}\)  be defined by  \rnota{Mi,B3.1}.
 For all \(k\in\mn{0}{\kappa}\), then there are functions \(\Gamma_k\colon\C \setminus \R \to\Coo{(m_k+2)q}{(2m_k +2)q}\) and \(\Delta_k\colon\C \setminus \R \to\Coo{(2m_k +2)q}{(m_k+2)q}\) such that \(P^{[f]}_k (z)=\Gamma_k(z) Q^{[f]}_k(z) \Gamma^\ad _k(z)\) and \(Q^{[f]}_k (z)=\Delta_k(z) P^{[f]}_k(z) \Delta^\ad _k(z)\) hold true for each \(z\in \C \setminus \R\).
\end{prop}
\begin{proof}
 (I) In the trivial case  \(k=0\), choose  \(\Gamma_0\colon\C \setminus \R \to\Coo{2q}{2q}\) and \(\Delta_0\colon\C \setminus \R\to\Coo{2q}{ 2q}\) given by \(\Gamma_0 (z)\defeq \Iu{2q}\) and \(\Delta_0 (z)\defeq \Iu{2q}\).
 
 (II) Now we consider the case that \(\kappa\geq 1\) and that \(n\in\NO \) is such that \(2n+1\leq\kappa\).
 Let  \(\Delta_{2n+1} \colon\C \setminus\R \to \Coo{(2n+2)q}{(n+2)q}\) be defined by
\beql{Nr.D2N+1}
 \Delta_{2n+1}(z)
 \defeq
 \begin{bmatrix}
    \Iu{(n+1)q} &\Ouu{(n+1)q}{q} \\
    [R_{T_{q,n}^\ad }(z)]^\ad T_{q,n} & [R_{T_{q,n}^\ad }(z)]^\ad v_{q,n} 
    \end{bmatrix}
\eeq 
 and let \(\Gamma_{2n+1} \colon\C \setminus\R \to\Coo{(n+2)q}{(2n+2)q}\) be given by
\beql{Nr.G2N+1}
\Gamma_{2n+1}(z)
\defeq \begin{bmatrix}
    \Iu{(n+1)q} & \Ouu{(n+1)q}{(n+1)q} \\      
    -v_{q,n}^\ad [R_{T_{q,n}^\ad }(z)]^\ad T_{q,n} & v_{q,n}^\ad  
    \end{bmatrix}.
\eeq 
 Since \(s_j^\ad =s_j\) holds true for each \(j\in\mn{0}{\kappa}\), we have 
\beql{Nr.3.3.1-1}
H_ {\at{n}}^\ad  
= H_{\at{n}}.
\eeq 
 We consider an arbitrary \(z\in\C \setminus\R \).
 Let
\beql{Nr.3.3.1-2}
 B_{2n+1}(z)
 \defeq R_{T_{q,n}}(z) \ek*{v_{q,n}(z-\alpha) f(z) - (-\alpha u_n - y_{0,n})},
\eeq 
 let
\beql{N7.12N}
 C_{2n+1}(z)
 \defeq \frac{(z-\alpha)f(z) -[(z-\alpha)f(z)]^\ad }{z-\ko{z}},
\eeq 
 and let 
\beql{Nr.3.DPD}
 \Delta_{2n+1}(z) P^{[f]}_{2n+1}(z) \Delta_{2n+1}^\ad (z)
 =
 \begin{bmatrix}
  X_{2n+1}(z) & Y_{2n+1}(z) \\
  Z_{2n+1}(z) & W_{2n+1}(z) 
 \end{bmatrix}
\eeq 
 be the \taaa{(n+1)q}{(n+1)q}{block} representation of \(\Delta_{2n+1}(z) P^{[f]}_{2n+1}(z) \Delta_{2n+1}^\ad (z)\).
 Then 
\beql{Nr.3.3.2}
 P^{[f]}_{2n+1} (z)
 =
 \begin{bmatrix}
    H_{\at{n}} & B_{2n+1}(z) \\
    B_{2n+1}^\ad (z) & C_{2n+1}(z) 
 \end{bmatrix}.
\eeq 
 Using \eqref{Nr.3.DPD}, \eqref{Nr.D2N+1}, and \eqref{Nr.3.3.2}, straightforward calculations show that 
\begin{align} 
X_{2n+1}(z)&= H_{\at{n}},\label{Nr.3.3.3-1}\\
Y_{2n+1}(z)&= H_{\at{n}} T_{q,n}^\ad  R_{T_{q,n}^\ad }(z)+B_{2n+1}(z)v_{q,n}^\ad R_{T_{q,n}^\ad }(z),\label{Nr.3.3.3-2}\\
Z_{2n+1}(z)&=\ek*{R_{T_{q,n}^\ad }(z)}^\ad  T_{q,n} H_{\at{n}} +\ek*{R_{T_{q,n}^\ad }(z)}^\ad v_{q,n} B^\ad _{2n+1}(z),\label{Nr.3.3.3-3}
\end{align}
and
\begin{multline} \label{Nr.3.3.3-4}
 W_{2n+1}(z)  =\ek*{R_{T_{q,n}^\ad }(z)}^\ad T_{q,n} H_{\at{n}} T_{q,n}^\ad  R_{T_{q,n}^\ad }(z)  +\ek*{R_{T_{q,n}^\ad }(z)}^\ad T_{q,n} B_{2n+1}(z)v_{q,n}^\ad  R_{T_{q,n}^\ad }(z)\\
 +\ek*{R_{T_{q,n}^\ad }(z)}^\ad  v_{q,n} B^\ad _{2n+1}(z) T_{q,n}^\ad  R_{T_{q,n}^\ad }(z) +\ek*{R_{T_{q,n}^\ad }(z)}^\ad  v_{q,n} C_{2n+1}(z) v_{q,n}^\ad R_{T_{q,n}^\ad }(z) 
\end{multline}
 hold true.
 Because of \eqref{Nr.3.3.3-3}, \eqref{Nr.3.3.1-2}, and \eqref{Nr.F2N+1}, we see that
\begin{multline} \label{Nr.3.3.4}
 Y_{2n+1}(z)
 = H_{\at{n}} T_{q,n}^\ad  R_{T_{q,n}^\ad }(z)+R_{T_{q,n}}(z)\ek*{v_{q,n}(z-\alpha)f(z)-(-\alpha u_n -y_{0,n})}v_{q,n}^\ad R_{T_{q,n}^\ad }(z)\\
 =F_{2n+1}(z)
\end{multline}
 is valid.
 From \eqref{Nr.3.3.3-3}, \eqref{Nr.3.3.1-1}, \eqref{Nr.3.3.3-2}, and \eqref{Nr.3.3.4} we obtain then
\beql{Nr.3.3.5}
Z_{2n+1}(z)
=Y_{2n+1}^\ad (z)
=F_{2n+1}^\ad (z).
\eeq 
 Using \eqref{Nr.3.3.3-4}, it follows
\begin{multline} \label{Nr.3.3.6}
 W_{2n+1}(z)  =\ek*{R_{T_{q,n}^\ad }(z)}^\ad T_{q,n} H_{\at{n}} T_{q,n}^\ad  R_{T_{q,n}^\ad }(z)    
+\ek*{R_{T_{q,n}^\ad }(z)}^\ad T_{q,n}B_{2n+1}(z)v_{q,n}^\ad  R_{T_{q,n}^\ad }(z)\\
 +\rk*{\ek*{R_{T_{q,n}^\ad }(z)}^\ad T_{q,n}B_{2n+1}(z)v_{q,n}^\ad  R_{T_{q,n}^\ad }(z)}^\ad  
 +\ek*{R_{T_{q,n}^\ad }(z)}^\ad v_{q,n}C_{2n+1}(z) v_{q,n}^\ad R_{T_{q,n}^\ad }(z).
\end{multline}
 In view of \rlem{D1451}, we have
\beql{Nr.3.3.7}
 \begin{split}
 &(z-\ko z)\ek*{R_{T_{q,n}^\ad }(z)}^\ad T_{q,n} H_{\at{n}} T_{q,n}^\ad  R_{T_{q,n}^\ad }(z) \\
 &= H_{\at{n}} T_{q,n}^\ad  R_{T_{q,n}^\ad }(z)-\ek*{R_{T_{q,n}^\ad }(z)}^\ad T_{q,n} H_{\at{n}}\\
 &\qquad+\ek*{R_{T_{q,n}^\ad }(z)}^\ad \ek*{v_{q,n}(-\alpha u_n -y_{0,n})^\ad -(-\alpha u_n -y_{0,n})v_{q,n}^\ad }R_{T_{q,n}^\ad }(z).
 \end{split}
\eeq 
 By virtue of \eqref{Nr.3.3.1-2}, \rrem{D1407}, and \eqref{Nr.F2N+1}, we conclude
\beql{Nr.3.3.8}
 \begin{split}
  &(z-\ko z)\ek*{R_{T_{q,n}^\ad }(z)}^\ad T_{q,n} B_{2n+1}(z) v_{q,n}^\ad  R_{T_{q,n}^\ad }(z)\\
  &=(z-\ko z)\ek*{R_{T_{q,n}^\ad }(z)}^\ad T_{q,n}R_{T_{q,n}}(z)\ek*{v_{q,n}(z-\alpha)f(z)-(-\alpha u_n -y_{0,n})}v_{q,n}^\ad R_{T_{q,n}^\ad }(z)\\
  &=\rk*{R_{T_{q,n}}(z)-\ek*{R_{T_{q,n}^\ad }(z)}^\ad }\ek*{v_{q,n}(z-\alpha)f(z)-(-\alpha u_n -y_{0,n})}v_{q,n}^\ad R_{T_{q,n}^\ad }(z)\\
  &=R_{T_{q,n}}(z)\ek*{v_{q,n}(z-\alpha)f(z)-(-\alpha u_n -y_{0,n})}v_{q,n}^\ad R_{T_{q,n}^\ad }(z)\\
  &\qquad-\ek*{R_{T_{q,n}^\ad }(z)}^\ad \ek*{v_{q,n}(z-\alpha)f(z)-(-\alpha u_n -y_{0,n})}v_{q,n}^\ad R_{T_{q,n}^\ad }(z)\\
  &=F_{2n+1}(z)- H_{\at{n}} T_{q,n}^\ad  R_{T_{q,n}^\ad }(z)\\
  &\qquad-\ek*{R_{T_{q,n}^\ad }(z)}^\ad \ek*{v_{q,n}(z-\alpha)f(z)-(-\alpha u_n -y_{0,n})}v_{q,n}^\ad R_{T_{q,n}^\ad }(z)\\
  &=F_{2n+1}(z)- H_{\at{n}} T_{q,n}^\ad  R_{T_{q,n}^\ad }(z)-\ek*{R_{T_{q,n}^\ad }(z)}^\ad v_{q,n}(z-\alpha)f(z)v_{q,n}^\ad R_{T_{q,n}^\ad }(z)\\
  &\qquad+\ek*{R_{T_{q,n}^\ad }(z)}^\ad (-\alpha u_n -y_{0,n})v_{q,n}^\ad R_{T_{q,n}^\ad }(z),
 \end{split}
\eeq 
 which implies
\begin{multline} \label{Nr.3.3.9}
 (z-\ko z)\rk*{\ek*{R_{T_{q,n}^\ad }(z)}^\ad T_{q,n} B_{2n+1}(z) v_{q,n}^\ad  R_{T_{q,n}^\ad }(z)}^\ad\\
 =-F_{2n+1}^\ad (z)+\ek*{R_{T_{q,n}^\ad }(z)}^\ad  T_{q,n} H_{\at{n}}^\ad  + \ek*{R_{T_{q,n}^\ad }(z)}^\ad v_{q,n}\ek*{(z-\alpha)f(z)}^\ad v_{q,n}^\ad R_{T_{q,n}^\ad }(z)\\
 -\ek*{R_{T_{q,n}^\ad }(z)}^\ad v_{q,n}(-\alpha u_n -y_{0,n})^\ad R_{T_{q,n}^\ad }(z).
\end{multline}
 Taking into account \eqref{N7.12N} we get 
\begin{multline} \label{Nr.3.3.10}
 (z-\ko z)\ek*{R_{T_{q,n}^\ad }(z)}^\ad v_{q,n} C_{2n+1}(z) v_{q,n}^\ad  R_{T_{q,n}^\ad }(z)\\
 =\ek*{R_{T_{q,n}^\ad }(z)}^\ad v_{q,n}(z-\alpha)f(z)v_{q,n}^\ad R_{T_{q,n}^\ad }(z)\\
 -\ek*{R_{T_{q,n}^\ad }(z)}^\ad v_{q,n}\ek*{(z-\alpha)f(z)}^\ad v_{q,n}^\ad R_{T_{q,n}^\ad }(z),
\end{multline}
 and, taking \eqref{Nr.3.3.6}, \eqref{Nr.3.3.7}, \eqref{Nr.3.3.8}, \eqref{Nr.3.3.9}, \eqref{Nr.3.3.10}, and  \eqref{Nr.3.3.1-1} into account, furthermore
\beql{Nr.3.3.11}
 \begin{split}
  &(z-\ko z)W_{2n+1}(z) \\
  &=(z-\ko z)\ek*{R_{T_{q,n}^\ad }(z)}^\ad T_{q,n}  H_{\at{n}} T_{q,n}^\ad  R_{T_{q,n}^\ad }(z)\\
  &\qquad+(z-\ko z)\ek*{R_{T_{q,n}^\ad }(z)}^\ad T_{q,n}B_{2n+1}(z)v_{q,n}^\ad  R_{T_{q,n}^\ad }(z)\\
  &\qquad+(z-\ko z)\rk*{\ek*{R_{T_{q,n}^\ad }(z)}^\ad T_{q,n}B_{2n+1}(z)v_{q,n}^\ad  R_{T_{q,n}^\ad }(z)}^\ad\\
  &\qquad+(z-\ko z)\ek*{R_{T_{q,n}^\ad }(z)}^\ad v_{q,n}C_{2n+1}(z) v_{q,n}^\ad R_{T_{q,n}^\ad }(z)\\
  &= H_{\at{n}} T_{q,n}^\ad  R_{T_{q,n}^\ad }(z)-\ek*{R_{T_{q,n}^\ad }(z)}^\ad T_{q,n} H_{\at{n}}\\
  &\qquad+\ek*{R_{T_{q,n}^\ad }(z)}^\ad \ek*{v_{q,n}(-\alpha u_n -y_{0,n})^\ad -(-\alpha u_n -y_{0,n})v_{q,n}^\ad }R_{T_{q,n}^\ad }(z)+F_{2n+1}(z)\\
  &\qquad-  H_{\at{n}} T_{q,n}^\ad  R_{T_{q,n}^\ad }(z)-\ek*{R_{T_{q,n}^\ad }(z)}^\ad v_{q,n}(z-\alpha)f(z)v_{q,n}^\ad R_{T_{q,n}^\ad }(z)\\
  &\qquad+\ek*{R_{T_{q,n}^\ad }(z)}^\ad (-\alpha u_n -y_{0,n})v_{q,n}^\ad R_{T_{q,n}^\ad }(z)-F_{2n+1}^\ad (z)\\
  &\qquad+\ek*{R_{T_{q,n}^\ad }(z)}^\ad  T_{q,n}  H_{\at{n}}^\ad+\ek*{R_{T_{q,n}^\ad }(z)}^\ad v_{q,n}\ek*{(z-\alpha)f(z)}^\ad v_{q,n}^\ad R_{T_{q,n}^\ad }(z)\\
  &\qquad-\ek*{R_{T_{q,n}^\ad }(z)}^\ad v_{q,n}(-\alpha u_n -y_{0,n})^\ad R_{T_{q,n}^\ad }(z)  %
 +\ek*{R_{T_{q,n}^\ad }(z)}^\ad v_{q,n}(z-\alpha)f(z)v_{q,n}^\ad R_{T_{q,n}^\ad }(z)\\
  &\qquad-\ek*{R_{T_{q,n}^\ad }(z)}^\ad v_{q,n}\ek*{(z-\alpha)f(z)}^\ad v_{q,n}^\ad R_{T_{q,n}^\ad }(z)\\
  &=F_{2n+1}(z)-F_{2n+1}^\ad (z).
 \end{split}
\eeq 
 From \eqref{Nr.3.DPD}, \eqref{Nr.3.3.3-1}, \eqref{Nr.3.3.4}, \eqref{Nr.3.3.5}, \eqref{Nr.3.3.11},  and \eqref{Nr.Q2N+1} we infer
\beql{Nr.3.3.12}
 \Delta_{2n+1}(z) P^{[f]}_{2n+1}(z) \Delta_{2n+1}^\ad (z)
 = Q^{[f]}_{2n+1}(z).
\eeq 
 In view of \(v_{q,n}^\ad [R_{T_{q,n}^\ad }^\ad (z)]v_{q,n}=  \Iq\), we easily see that the matrices \(\Gamma_{2n+1} (z)\) and \(\Delta_{2n+1}(z)\) given by \eqref{Nr.G2N+1} and  \eqref{Nr.D2N+1} obviously fulfill
\beql{Nr.3.GD}
 \Gamma_{2n+1}(z)\Delta_{2n+1}(z)
 = \Iu{(n+2)q}.
\eeq 
 Thus, because of  \eqref{Nr.3.3.12}, we obtain
\begin{equation*}
 \begin{split}
  P^{[f]}_{2n+1}(z)
  &= \Iu{(n+2)q} P^{[f]}_{2n+1}(z) \Iu{(n+2)q}^\ad\\
  &=\Gamma_{2n+1}(z)\Delta_{2n+1}(z)P^{[f]}_{2n+1}(z)\Gamma_{2n+1}^\ad (z)\Delta_{2n+1}^\ad (z)
  =\Gamma_{2n+1}(z)Q^{[f]}_{2n+1}(z)\Gamma_{2n+1}^\ad (z). 
 \end{split}
\end{equation*}
 In this case \(k=2n+1\) with some \(n\in\NO \), the proof is complete.
 
 (III) Now  we consider the case that \(\kappa \ge 2\) and that there is an \(n\in\N \) such that \(k=2n\).
 Let \(\Gamma_{2n} \defeq  \Gamma_{2n+1}\) and let \(\Delta_{2n} \defeq  \Delta_{2n+1}\).
 We consider again an arbitrary \(z\in\C \setminus\R \).
 Let
\beql{Nr.3.3.15}
\Delta_{2n}(z) P^{[f]}_{2n}(z) \Delta_{2n}^\ad (z)
=
\begin{bmatrix}
    X_{2n}(z) & Y_{2n}(z) \\ 
    Z_{2n}(z) & W_{2n}(z) 
    \end{bmatrix}
\eeq  
 be the \taaa{(n+1)q}{(n+1)q}{block} representation of \(\Delta_{2n}(z) P^{[f]}_{2n}(z) \Delta_{2n}^\ad (z)\).
 Setting
\begin{align} \label{Nr.3.3.16}
 B_{2n}(z)&\defeq R_{T_{q,n}}(z)\ek*{v_{q,n} f(z)- u_n}&
&\text{and}&
 C_{2n}(z)&\defeq \frac{f(z)-f^\ad (z)}{z- \ko z},
\end{align}
 we have 
\beql{Nr.3.3.18}
P^{[f]}_{2n}(z)
=
\begin{bmatrix}
    \Hu{n} & B_{2n}(z) \\
    B_{2n}^\ad (z) & C_{2n}(z) 
    \end{bmatrix}.
\eeq 
 From \eqref{Nr.3.3.15} and  \eqref{Nr.3.3.18} we easily see then that
\begin{align} 
 X_{2n}(z)&=H_n,\qquad Y_{2n}(z)=H_n T_{q,n}^\ad  R_{T_{q,n}^\ad }(z)+B_{2n}(z)v_{q,n}^\ad  R_{T_{q,n}^\ad }(z) \label{Nr.3.3.19-1},\\
Z_{2n}(z)&=\ek*{R_{T_{q,n}^\ad }(z)}^\ad T_{q,n}H_n+\ek*{R_{T_{q,n}^\ad }(z)}^\ad v_{q,n}B_{2n}^\ad (z),\label{Nr.3.3.19-3} 
\end{align}
and
\begin{multline}
W_{2n}(z)
=\ek*{R_{T_{q,n}^\ad }(z)}^\ad T_{q,n}H_n T_{q,n}^\ad R_{T_{q,n}^\ad }(z)+\ek*{R_{T_{q,n}^\ad }(z)}^\ad v_{q,n}B_{2n}^\ad (z)T_{q,n}^\ad R_{T_{q,n}^\ad }(z)\\
+\ek*{R_{T_{q,n}^\ad }(z)}^\ad T_{q,n}B_{2n}(z)v^\ad _{q,n}R_{T_{q,n}^\ad }(z)+\ek*{R_{T_{q,n}^\ad }(z)}^\ad v_{q,n}C_{2n}(z)v_{q,n}^\ad R_{T_{q,n}^\ad }(z)  \label{Nr.3.3.19-4}
\end{multline}
 hold true.
 Because of \eqref{Nr.3.3.19-1}, \eqref{Nr.3.3.16}, and \eqref{Nr.F2N}, we obtain
\beql{Nr.3.3.20}
 Y_{2n}(z)
 =H_n T_{q,n}^\ad  R_{T_{q,n}^\ad }(z)+R_{T_{q,n}}(z)\ek*{v_{q,n} f(z)-u_n}v_{q,n}^\ad R_{T_{q,n}^\ad }(z)
 = F_{2n}(z).
\eeq 
 Since \(s_j^\ad =s_j\) is supposed for each \(j\in\mn{0}{\kappa}\), we get \(H_n^\ad =H_n\).
 Consequently, in view of  \eqref{Nr.3.3.19-3},  \eqref{Nr.3.3.19-1}, and \eqref{Nr.3.3.20},  then
\beql{Nr.3.3.22}
 Z_{2n}(z)
 = Y_{2n}^\ad (z)
 =F_{2n}^\ad (z)
\eeq 
 follows.
 By virtue of  \eqref{Nr.3.3.19-4} and \eqref{Nr.3.3.16}, we see that
\begin{multline} \label{Nr.3.3.24}
 W_{2n}(z)  
 =\ek*{R_{T_{q,n}^\ad }(z)}^\ad T_{q,n}H_n T_{q,n}^\ad R_{T_{q,n}^\ad }(z)\\
 +\ek*{R_{T_{q,n}^\ad }(z)}^\ad v_{q,n}\ek*{f^\ad (z) v^\ad _{q,n} - u_n^\ad} \ek*{R_{T_{q,n}}(z)}^\ad T_{q,n}^\ad R_{T_{q,n}^\ad }(z)\\
 +\ek*{R_{T_{q,n}^\ad }(z)}^\ad T_{q,n} R_{T_{q,n}}(z)\ek*{v_{q,n} f(z)- u_n} v^\ad _{q,n}R_{T_{q,n}^\ad }(z)\\
 +\ek*{R_{T_{q,n}^\ad }(z)}^\ad v_{q,n}\ek*{\frac{f(z)-f^\ad (z)}{z- \ko z}}v_{q,n}^\ad R_{T_{q,n}^\ad }(z) 
\end{multline}
 holds true.
 Taking into account  \eqref{Nr.3.3.24} and \rrem{D1407}, we conclude
\begin{multline} \label{Nr.3.3.24-3}
 W_{2n}(z)
 =\ek*{R_{T_{q,n}^\ad }(z)}^\ad T_{q,n}H_n T_{q,n}^\ad R_{T_{q,n}^\ad }(z)\\
 +\ek*{R_{T_{q,n}^\ad }(z)}^\ad v_{q,n}\ek*{f^\ad (z) v^\ad _{q,n} - u_n^\ad} \ek*{ \frac{1}{z- \ko z} \rk*{R_{T_{q,n}^\ad }(z)-\ek*{R_{T_{q,n}}(z)}^\ad }}\\
 +\ek*{ \frac{1}{z- \ko z} \rk*{R_{T_{q,n}}(z)-\ek*{R_{T_{q,n}^\ad }(z)}^\ad }} \ek*{v_{q,n} f(z)- u_n} v^\ad _{q,n}R_{T_{q,n}^\ad }(z)\\
 +\ek*{R_{T_{q,n}^\ad }(z)}^\ad v_{q,n}\ek*{\frac{f(z)-f^\ad (z)}{z- \ko z}}v_{q,n}^\ad R_{T_{q,n}^\ad }(z). 
\end{multline}
 Using \rlem{D1451},  the equation \(H_n^\ad = H_n\), \eqref{Nr.F2N}, and \eqref{Nr.3.3.24-3}, we infer
\beql{Nr.3.3.28}
 \begin{split}
 W_{2n}(z)
 &=\frac{1}{z- \ko z}\biggl\{H_n T_{q,n}^\ad R_{T_{q,n}^\ad }(z) - \ek*{R_{T_{q,n}^\ad }(z)}^\ad T_{q,n}H_n\\
 &\qquad+\ek*{R_{T_{q,n}^\ad }(z)}^\ad  (v_{q,n} u_n^\ad  - u_n v_{q,n}^\ad )R_{T_{q,n}^\ad }(z)\\
 &\qquad+\ek*{R_{T_{q,n}^\ad }(z)}^\ad v_{q,n}\ek*{f^\ad (z) v^\ad _{q,n} - u_n^\ad}\rk*{R_{T_{q,n}^\ad }(z)-\ek*{R_{T_{q,n}}(z)}^\ad }\\
 &\qquad+\rk*{R_{T_{q,n}}(z)-\ek*{R_{T_{q,n}^\ad }(z)}^\ad}\ek*{v_{q,n} f(z)- u_n} v^\ad _{q,n}R_{T_{q,n}^\ad }(z)\\
 &\qquad+\ek*{R_{T_{q,n}^\ad }(z)}^\ad v_{q,n}[f(z)-f^\ad (z)]v_{q,n}^\ad R_{T_{q,n}^\ad }(z)\biggr\}\\
 &=\frac{1}{z- \ko z}\biggl\{H_n T_{q,n}^\ad R_{T_{q,n}^\ad }(z) + R_{T_{q,n}}(z) \ek*{v_{q,n} f(z)- u_n } v_{q,n}^\ad  R_{T_{q,n}^\ad }(z)\\
 &\qquad-\rk*{H_n T_{q,n}^\ad R_{T_{q,n}^\ad }(z) + R_{T_{q,n}}(z) \ek*{v_{q,n} f(z) - u_n}v_{q,n}^\ad  R_{T_{q,n}^\ad }(z)}^\ad\biggr\}\\
 &=\frac{1}{z- \ko z} \ek*{F_{2n}(z)-F_{2n}^\ad (z) }. 
 \end{split}
\eeq 
 Thus, \eqref{Nr.3.3.15}, the first equation in  \eqref{Nr.3.3.19-1}, \eqref{Nr.3.3.20}, \eqref{Nr.3.3.22},  \eqref{Nr.3.3.28}, and \eqref{Nr.Q2N} show that
\beql{Nr.3.3.29}
 \Delta_{2n}(z) P^{[f]}_{2n}(z) \Delta_{2n}^\ad (z)
   =
   \begin{bmatrix}
    \Hu{n} & F_{2n}(z) \\
    F_{2n}^\ad (z) & \frac{F_{2n}(z) -F_{2n}^\ad (z)}{z-\ko{z}}
    \end{bmatrix}
    =Q_{2n}^{[f]}(z)
\eeq 
 is valid.
 Because of \(\Gamma_{2n} =\Gamma_{2n+1}\) and \(\Delta_{2n} =\Delta_{2n+1}\), equation~\eqref{Nr.3.GD} implies \(\Gamma_{2n}(z)\Delta_{2n}(z)= \Iu{(n+2)q}\).
 Consequently, from \eqref{Nr.3.3.29} we get 
\begin{equation*}
 \begin{split}
 P^{[f]}_{2n}(z)
 &=\Iu{(n+2)q} P^{[f]}_{2n}(z) \Iu{(n+2)q}
 =\Gamma_{2n}(z)\Delta_{2n}(z) P^{[f]}_{2n}(z) \ek*{\Gamma_{2n}(z)\Delta_{2n}(z)}^\ad\\
 &=\Gamma_{2n}(z)\Delta_{2n}(z) P^{[f]}_{2n}(z) \Delta_{2n}^\ad (z)\Gamma_{2n}^\ad (z)
 =\Gamma_{2n}(z)Q^{[f]}_{2n}(z) \Gamma_{2n}^\ad (z).\qedhere
 \end{split}
\end{equation*}
\end{proof}

\begin{lem} \label{Mi,L3.4}
 Let \(\alpha\in\R\), let \(\kappa\in\NOinf \), and let \(\seqska \) be a sequence from \(\Cqq \).
 Let  \(f\colon\Cs\to\Cqq \) be a holomorphic matrix-valued function.
 Then: 
\begin{enui}
 \item\label{Mi,L3.4.a} Let \(n \in \NO \) be such that  \(2n \leq \kappa\).
 If  \(P_{2n}^{[f]}(z) \in\Cggo{(n+2)q}\) holds true for each \(z\in\C \setminus\R \), then \(F_{2n}\colon\uhp \to\Coo{(n+1)q}{(n+1)q}\) given by \eqref{Nr.F2N} belongs to \(\RFO{(n+1)q}\) and the matricial spectral measure \(\mu_{2n}\) of \(F_{2n}\) fulfills  \(\mu_{2n}(\R) \lleq H_n\).
 \item\label{Mi,L3.4.b} Let \(n \in \NO \) be such that \(2n+1 \leq \kappa\).
 If \(P_{2n+1}^{[f]}(z)\in\Cggo{(n+2)q}\) for each \(z\in\C\setminus\R\), then \(F_{2n+1}\colon\uhp\to\Coo{(n+1)q}{(n+1)q}\) defined by \eqref{Nr.F2N+1} belongs to \(\RFO{(n+1)q}\) and the matricial spectral measure \(\mu_{2n+1}\) of \(F_{2n+1}\) fulfills \(\mu_{2n+1}(\R) \lleq H_{\at n}\).
\end{enui}
\end{lem}
\begin{proof}
 \eqref{Mi,L3.4.a} Let \(z\in\uhp \).
 Suppose that the matrix \(P_{2n}^{[f]} (z)\) is \tnnH{}.
 From \eqref{def-P2n} we see then that \(H_n\) is \tnnH{} as well.
 In particular, \(H_n^\ad = H_n\).
 This implies \(s_j^\ad = s_j\) for each \(j\in\mn{0}{2n}\).
 According to \rprop{Mi,L3.3}, there is a function \(\Delta_{2n} \colon\C \setminus\R \to\Coo{(2n+2)q}{(n+2)q}\) such that \(\Delta_{2n}(z) P^{[f]}_{2n}(z) \Delta_{2n}^\ad  (z) = Q^{[f]}_{2n}(z) \).
 Thus, since \( P^{[f]}_{2n}(z) \) is \tnnH{}, the matrix \( Q^{[f]}_{2n}(z) \) is \tnnH{} as well.
 Taking \rnota{Mi,B3.1} into account, this means that the block matrix on the right-hand side of \eqref{Nr.Q2N} is \tnnH{}.
 Since \(R_{T_{q,n}}\) and \(R_{T_{q,n}^\ad}\) are matrix polynomials, we see  from \eqref{Nr.F2N} that \(F_{2n}\) is holomorphic in \(\uhp \), and, thus, the application of \rlem{N1219DN} yields \(F_{2n} \in\RFO{(n+1)q}\) and \(\mu_{2n} (\R )\lleq H_n\).
 
 \eqref{Mi,L3.4.b}  \rPart{Mi,L3.4.b} can be proved analogously.
 We omit the details.
\end{proof}

\begin{lem} \label{Mi,L3.5}  
 Let  \(\alpha \in \R\), let \(f\colon\Cs\to\Cqq\) be a matrix-valued function, let  \( \kappa \in \NOinf \), and let  \(\seqska \) be a sequence of \tH{} complex \tqqa{matrices}.
 Then:
\begin{enui}
  \item\label{Mi,L3.5.a} Let \(n \in \NO  \) be such that \(2n \leq \kappa\), let \(F_{2n}\colon\uhp \to \Coo{(n+1)q}{(n+1)q}\) be defined by \eqref{Nr.F2N}, and let  \(\Psi_{2n}\colon\C \to\Coo{(n+1)q}{(n+1)q}\) be given by 
 \beql{Nr.PH2NN}
  \Psi_{2n}(z)\defeq R_{\Tqn}(z)(H_n\Tqn^\ad -u_n v_{q,n}^\ad -z\Tqn H_n\Tqn^\ad  )R_{\Tqn^\ad }(z).
 \eeq 
 Then  \(\Psi_{2n}\) is a continuous matrix-valued function such that  \(\Psi_{2n}(\R)\subseteq \CHo{(n+1)q}\).
 In view of \eqref{3.64.-1}, furthermore,
 \begin{align}\label{N1439D}
  F_{2n}(z)&=\Psi_{2n}(z) + E_{q,n}(z)f(z)E_{q,n}^\ad (\ko z)&\text{for each }z&\in \uhp.
 \end{align}
 \item\label{Mi,L3.5.b} Let \(n \in \NO  \) be such that \(2n+1 \leq \kappa\), let \(F_{2n+1}\colon\uhp \to \Coo{(n+1)q}{(n+1)q}\) be defined by  \eqref{Nr.F2N+1}, and let \(\Psi_{2n+1}\colon\C \to\Coo{(n+1)q}{(n+1)q}\) be given by
\beql{Nr.PH2N+1}
 \Psi_{2n+1}(z)
 \defeq  R_{\Tqn}(z)\ek*{H_{\at n}\Tqn^\ad -(-\alpha u_n-y_{0,n})v_{q,n}^\ad  -z\Tqn H_{\at n}\Tqn^\ad }R_{\Tqn^\ad }(z).
\eeq 
 Then \(\Psi_{2n+1}\) is a continuous matrix-valued function such that  \(\Psi_{2n+1}(\R)\subseteq \CHo{(n+1)q}\).
  Moreover,
 \begin{align*} %
  F_{2n+1}(z)&=\Psi_{2n+1}(z) + E_{q,n}(z)\ek*{(z-\alpha)f(z)}E_{q,n}^\ad (\ko z)&\text{for each }z&\in \uhp.
 \end{align*}
\end{enui}
\end{lem}
\begin{proof}  
 \eqref{Mi,L3.5.a} The case \(n=0\)  is trivial.
 Suppose now \(0< 2n\leq\kappa\).
 \rrem{21112N} shows that \(\Psi_{2n}\) is a matrix polynomial.
 In particular, \(\Psi_{2n}\) is continuous.
 By assumption, we have \(s_j^\ad = s_j\) for each \(j\in\mn{0}{2n}\).
 Thus, \(H_n^\ad = H_n\).
 For each  \(x\in\R \), we have \( R_{\Tqn^\ad }(x)=[R_{\Tqn}(\ko {x})]^\ad =[R_{\Tqn}(x)]^\ad \) and, consequently,
\[
 \ek*{\Psi_{2n}(x)}^\ad
 =R_{\Tqn}(x)(\Tqn H_n-v_{q,n} u_n^\ad -x\Tqn H_n^\ad \Tqn^\ad )R_{\Tqn^\ad }(x),
\]
 which, in view of  \(H_n^\ad =H_n\), implies that
\[\begin{split}
 \ek*{\Psi_{2n}(x)}^\ad 
 &=R_{\Tqn}(x)\rk*{-[H_n\Tqn^\ad  - \Tqn H_n] +H_n\Tqn^\ad  - v_{q,n}u_n^\ad -x\Tqn H_n\Tqn^\ad }R_{\Tqn^\ad }(x) \\
 &=R_{\Tqn}(x)\rk*{-[u_n  v_{q,n}^\ad  - \vqn u_n^\ad ] +H_n\Tqn^\ad  - v_{q,n}u_n^\ad -x\Tqn H_n\Tqn^\ad }R_{\Tqn^\ad }(x) \\
 &=R_{\Tqn}(x)(H_n\Tqn^\ad  - u_n  v_{q,n}^\ad  - x\Tqn H_n\Tqn^\ad )R_{\Tqn^\ad }(x)
 =\Psi_{2n}(x)
\end{split}\]
 holds true for all  \(x\in\R\).
 Hence, \(\Psi_{2n}(\R) \subseteq \CHo{(n+1)q}\).
 Taking into account  \eqref{Nr.F2N}, \rrem{21112N}, and \eqref{Nr.PH2NN}, for all \(z\in\uhp\), we conclude
\[\begin{split}
 &F_{2n}(z)\\
 &=R_{\Tqn}(z) \ek*{R_{\Tqn}(z)}^\inv  H_n\Tqn^\ad R_{\Tqn^\ad }(z)\\
 &\qquad+R_{\Tqn}(z)v_{q,n}f(z)v_{q,n}^\ad R_{\Tqn^\ad }(z)-R_{\Tqn}(z) u_n v_{q,n}^\ad R_{\Tqn^\ad }(z)\\
 &=R_{\Tqn}(z)\ek*{\rk{\Iu{(n+1)q}-z\Tqn}H_n\Tqn^\ad -u_nv_{q,n}^\ad }R_{\Tqn^\ad }(z)+R_{\Tqn}(z)v_{q,n}f(z)v_{q,n}^\ad \ek*{R_{\Tqn}(\ko z)}^\ad  \\
 &=R_{\Tqn}(z)(H_n\Tqn^\ad -z\Tqn H_n\Tqn^\ad  -u_nv_{q,n}^\ad )R_{\Tqn^\ad }(z)+ R_{\Tqn}(z)v_{q,n}f(z)[R_{\Tqn}(\ko z)v_{q,n}]^\ad  \\
 &=\Psi_{2n}(z) + E_{q,n}(z)f(z)E_{q,n}^\ad (\ko z).
\end{split}
\]

 \eqref{Mi,L3.5.b} \rPart{Mi,L3.5.b} can be proved analogously.
\end{proof}

\begin{lem} \label{Mi,L3.7}
 Let \(\alpha\in\R\),  let \(\kappa\in\Ninf \),  let \(\seqska\) be a sequence from \(\Cqq\), and let \(n \in \NO \) be such that  \(2n+1 \leq \kappa\).
 Further, let \(S\in\SFOq  \) be such that 
\begin{align} \label{Nr.PS2}
 P_{2n}^{[S]}(z)&\in \Cggo{(n+2)q}&
 &\text{and}&
 P_{2n+1}^{[S]}(z)&\in\Cggo{(n+2)q} & \text{for all }z&\in\uhp.
\end{align}
 Then the \taSm{} \(\sigma\) of  \(S\) belongs to \(\Mgguqa{1}{\rhl}\).
\end{lem}
\begin{proof} 
 (I) For all  \(z\in\uhp\),  from \rrem{lemM423} we see that \eqref{N68-1} holds true and,  in view of  \eqref{Nr.PS2}, hence, that the block matrix on the left-hand side of \eqref{N68-1} is \tnnH{}.
 Consequently, since \(S\) is holomorphic in \(\Cs  \), \rlem{N1219DN} yields that  \(F\defeq \Rstr_{\uhp} S\) belongs to \(\RFOq \) and that the matricial spectral measure \(\mu\) of \(F\) fulfills \(\mu(\R)\lleq s_0\).
 Thus, \rrem{R1457D} provides us 
\beql{N1043F}
 \sigma\rk*{\rhl}
 =\Rstr_{\BorK }\mu\rk*{\rhl}
 =\mu\rk*{\rhl}
 \lleq \mu(\R) 
 \lleq s_0.
\eeq 
 Because of \eqref{Nr.PS2} and \eqref{def-P2n}, the matrix \(H_n\) is \tnnH{}. 
 In particular, \(s_0\in\Cggq\). 
 Hence,
\begin{align} \label{Nr.Mi,L3.7-4}
 s_0^\ad&= s_0&
 &\text{and}&
 \set*{ u^\ad \sigma\rk*{\rhl}u, u^\ad s_0u}&\subseteq[0,\infp)&\text{for all }u&\in \Cq.
\end{align}

 (II) In the second part of the proof, we consider an arbitrary \(n\in\N \) and an arbitrary \(u\in\Cq\).
 From \rrem{L1.53} we see then that
\beql{SR1}
 \int_{\rhl}\abs*{\frac{\iu n}{t-(\iu n+\alpha)}}(u^\ad \sigma u)(\dif t)
 = n\int_{\rhl}\abs*{\frac{1}{t-(\iu n+\alpha)}}(u^\ad \sigma u)(\dif t)
 < \infp.
\eeq 
In view of 
\beql{D01038}
 \frac{\iu n}{t-(\iu n+\alpha)}
 = -\frac{n^2}{\abs{t-\alpha-\iu n}^2} + \iu\frac{(t-\alpha)n}{\abs{t-\alpha-\iu n}^2}
\eeq 
 and \eqref{SR1}, we obtain
\[%beql{SR2}
 \int_{\rhl}\abs*{-\frac{n^2}{\abs{t-\alpha-\iu n}^2}} (u^\ad \sigma u)(\dif t)
 = \int_{\rhl}\abs*{\Re\ek*{\frac{\iu n}{t-(\iu n+\alpha)}}}(u^\ad \sigma u)(\dif t)
 < \infp
\]
and
\beql{N7.49}
 \int_{\rhl}\abs*{\frac{(t-\alpha)n}{\abs{t-\alpha-\iu n}^2}} (u^\ad \sigma u)(\dif t)
 =  \int_{\rhl}\abs*{\Im\ek*{\frac{\iu n}{t-(\iu n+\alpha)}}}(u^\ad \sigma u)(\dif t)
 < \infp.
\eeq 
 For each \(t\in\rhl\), we have
\beql{N2054F}
 n\ek*{\frac{\iu n}{t-(\iu n+\alpha)}+1}
 = \frac{(t-\alpha)n}{t-\alpha-\iu n} 
 = \frac{(t-\alpha)^2 n}{(t-\alpha)^2+n^2} + \iu\frac{(t-\alpha)n^2}{(t-\alpha)^2+n^2}.
\eeq 
 Consequently, we get  that
\[
 \abs*{\Re\rk*{ n\ek*{\frac{\iu n}{t-(\iu n+\alpha)}+1}}}
 = n\cdot\frac{(t-\alpha)^2}{(t-\alpha)^2+n^2}
 \leq n
 = n\cdot 1_{\rhl}(t)
\]
 and 
\[
 \abs*{\Im\rk*{ n\ek*{\frac{\iu n}{t-(\iu n+\alpha)}+1}}}
 = \frac{(t-\alpha)n^2}{(t-\alpha)^2+n^2}
 \leq n\cdot\frac{2\abs{t-\alpha}n}{(t-\alpha)^2+n^2}
 \leq n
 = n\cdot 1_{\rhl}(t)
\]
 hold true for each \(t\in \rhl \).
 This implies
\[\begin{split}
 \int_{\rhl}\abs*{\Re\rk*{ n\ek*{\frac{\iu n}{t-(\iu n+\alpha)}+1}}}(u^\ad \sigma u)(\dif t)
 &\leq \int_{\rhl}n\cdot 1_{\rhl}\dif(u^\ad \sigma u)\\
 &=  n u^\ad \sigma\rk*{\rhl}u 
 < \infp
\end{split}\] 
 and, analogously,
\[
 \int_{\rhl}\abs*{\Im\rk*{ n\ek*{\frac{\iu n}{t-(\iu n+\alpha)}+1}}}(u^\ad \sigma u)(\dif t)
 \leq  n  u^\ad \sigma\rk*{\rhl}u
 < \infp.
\]
 Thus, the function \(g_n\colon\rhl\to\C\) given by \(g_n(t)\defeq n\rk{\iu n/\ek{t-(\iu n+\alpha)}+1}\) fulfills 
\beql{Nr.UFF}
 g_n
 \in\LoaaaC{1}{\rhl}{\BorK}{u^\ad\sigma u}.
\eeq 
 Using \rthm{T1316D}, \rrem{B8},   \eqref{D01038}, and  \eqref{N7.49}, we conclude
\[\begin{split} 
 &u^\ad\ek*{\iu n\cdot S(\iu n+\alpha)}u
 = u^\ad \ek*{\iu n \int_{\rhl}\frac{1}{t-(\iu n+\alpha)}\sigma(\dif t)} u \\
 &= \int_{\rhl}\frac{\iu n}{t-(\iu n+\alpha)}(u^\ad \sigma u)(\dif t) \\
 &= \int_{\rhl}\ek*{ -\frac{n^2}{\abs{t-\alpha-\iu n}^2} + \iu\frac{(t-\alpha)n}{\abs{t-\alpha-\iu n}^2}}(u^\ad \sigma u)(\dif t) \\
 &= -n^2 \int_{\rhl} \frac{1}{\abs{t-\alpha-\iu n}^2}(u^\ad \sigma u)(\dif t) + \iu n\int_{\rhl}\frac{t-\alpha}{\abs{t-\alpha-\iu n}^2}(u^\ad \sigma u)(\dif t)
\end{split}\]
 and, in particular,
\beql{CN7-1}
 \Re\rk*{ u^\ad \ek*{\iu n\cdot S(\iu n+\alpha)}u}
 = -n^2 \int_{\rhl} \frac{1}{\abs{t-\alpha-\iu n}^2}(u^\ad \sigma u)(\dif t).
\eeq 
 Taking into account \eqref{N1043F}, \eqref{CN7-1}, and that  \(1 -n^2/\abs{t-\alpha-\iu n}   =(t-\alpha)^2/\ek{(t-\alpha)^2+n^2}\) holds true, for each \(t\in \rhl \), we get
\[\begin{split}
 &\Re \rk*{u^\ad \ek*{\iu n\cdot S(\iu n+\alpha)}u} + u^\ad s_0 u
 \geq \Re \rk*{u^\ad \ek*{\iu n\cdot S(\iu n+\alpha)}u} + u^\ad \sigma\rk*{\rhl} u \\
 &= -n^2\int_{\rhl}\frac{1}{\abs{t-\alpha-\iu n}^2}(u^\ad \sigma u)(\dif t)     + \int_{\rhl} 1_{\rhl}\dif(u^\ad \sigma u)\\
 &= \int_{\rhl}\rk*{ 1-\frac{n^2}{\abs{t-\alpha-\iu n}^2}}(u^\ad \sigma u)(\dif t)
 = \int_{\rhl}\frac{(t-\alpha)^2}{(t-\alpha)^2 +n^2}(u^\ad \sigma u)(\dif t)
 \geq 0
\end{split}\]
 and, consequently,
\begin{multline} \label{SRT2}
 \ek*{\Re \rk*{u^\ad \ek*{\iu n\cdot S(\iu n+\alpha)}u} + u^\ad s_0 u}^2  \geq \ek*{\Re \rk*{u^\ad \ek*{\iu n\cdot S(\iu n+\alpha)}u} + u^\ad \sigma\rk*{\rhl} u}^2.
\end{multline}
 Because of \eqref{Nr.Mi,L3.7-4}, \eqref{SRT2}, and again \eqref{Nr.Mi,L3.7-4}, it follows 
\[\begin{split}
 &\abs*{nu^\ad\ek*{\iu n\cdot S(\iu n+\alpha)+s_0} u}^2
 = n^2 \abs*{u^\ad \ek*{\iu n\cdot S(\iu n+\alpha)}u + u^\ad s_0 u}^2 \\
 &= n^2 \rk*{\ek*{\Re \rk*{u^\ad \ek*{\iu n\cdot S(\iu n+\alpha)}u} + u^\ad s_0 u}^2       + \ek*{\Im \rk*{u^\ad \ek*{\iu n\cdot S(\iu n+\alpha)}u} }^2 } \\
 &\geq n^2 \rk*{\ek*{\Re \rk*{u^\ad \ek*{\iu n\cdot S(\iu n+\alpha)}u} + u^\ad \sigma\rk*{\rhl} u}^2     + \ek*{\Im \rk*{u^\ad \ek*{\iu n\cdot S(\iu n+\alpha)}u} }^2} \\
 &= n^2 \abs*{u^\ad  \ek*{\iu n\cdot S(\iu n+\alpha)}u + u^\ad \sigma\rk*{\rhl} u }^2
 = \abs*{n u^\ad    \ek*{\iu n\cdot S(\iu n+\alpha) + \sigma\rk*{\rhl}} u}^2
\end{split}\]
 and, therefore,
\beql{BMR}
 \abs*{ n u^\ad\ek*{\iu n\cdot S(\iu n+\alpha)+s_0}u}
 \geq \abs*{  n u^\ad\ek*{\iu n\cdot S(\iu n+\alpha) + \sigma\rk*{\rhl}} u}.
\eeq 
 Since \(S\) belongs to \(\SFOq  \), the function \(G\colon\uhp\to\Cqq\) given by \(G(w)\defeq wS(w+\alpha)+s_0\) is holomorphic in \(\uhp\).
 From \rrem{lemM423} we know that, for all \(z\in\Cs  \), equation \eqref{N68-2} is true.
 Hence, from \eqref{Nr.PS2} we see that the block  matrix on the left-hand side of \eqref{N68-2} is \tnnH{}.
 Consequently, we conclude
\begin{multline} \label{588-2}
 \begin{bmatrix}  
  -\alpha s_0 + s_1  &  G(w) \\
  G^\ad (w)  &  \frac{G(w)-G^\ad (w)}{w-\ko w}
 \end{bmatrix} 
 =
 \begin{bmatrix}  
  -\alpha s_0 + s_1  &  wS(w+\alpha)+s_0 \\
  [wS(w+\alpha)+s_0]^\ad &  \frac{[wS(w+\alpha)+s_0]-[wS(w+\alpha)+s_0]^\ad }{w-\ko w}
 \end{bmatrix} \\
 =
 \begin{bmatrix}  
  -\alpha s_0 + s_1  &  [(w+\alpha)-\alpha]S(w+\alpha)+s_0 \\
  \rk*{[(w+\alpha)-\alpha]S(w+\alpha)+s_0}^\ad  &   \frac{[(w+\alpha)-\alpha]S(w+\alpha) -\rk*{[(w+\alpha)-\alpha]S(w+\alpha)}^\ad }{w-\ko w}  
 \end{bmatrix} 
 \in\Cggo{2q}.
\end{multline}
 Since \(G\) is holomorphic, from  \eqref{588-2} and \rlem{N1219DN} we infer  \(\sup_{y\in(0,\infp)}\rk*{ y\normS{G(\iu y)}} \leq \normS{-\alpha s_0+s_1}\) and, hence,  \(\sup_{n\in\N}\rk*{ n\normS{\iu n\cdot S(\iu n+\alpha)+s_0}} \leq \normS{-\alpha s_0+s_1}\).
 Thus, the Bunjakowski--Cauchy--Schwarz inequality provides us
\begin{multline} \label{CN13-1}
  \abs*{u^\ad \rk*{n\ek*{\iu n\cdot S(\iu n+\alpha)+s_0}}u}
  \leq \normE*{n\ek*{\iu n\cdot S(\iu n+\alpha)+s_0}u} \cdot \normE{u}\\
  \leq n \normS{\iu n\cdot S(\iu n+\alpha)+s_0} \cdot \normE{u}^2
  \leq \normS{-\alpha s_0+s_1}\cdot\normE{u}^2.
\end{multline}
 For each \(t\in\rhl\), we have \( \abs{t-\alpha}  = \liminf_{n\to\infi}(t-\alpha)n^2/\ek{(t-\alpha)^2+n^2}\).
 Fatou's lemma yields then 
\beql{CN3}
 \begin{split}
  \int_{\rhl}\abs{t-\alpha}(u^\ad \sigma u)(\dif t) 
  &= \int_{\rhl} \liminf_{n\to\infi}\frac{(t-\alpha)n^2}{(t-\alpha)^2+n^2}(u^\ad \sigma u)(\dif t)\\
  &\leq \liminf_{n\to\infi}\int_{\rhl}\frac{(t-\alpha)n^2}{(t-\alpha)^2+n^2}(u^\ad \sigma u)(\dif t).
 \end{split}
\eeq 
 Obviously, from \eqref{Nr.UFF} and \eqref{N2054F} it follows
\beql{CN15}
 \int_{\rhl}\Im\rk*{ n\ek*{\frac{\iu n}{t-(\iu n+\alpha)}+1}}(u^\ad \sigma u)(\dif t)
 =  \int_{\rhl}\frac{(t-\alpha)n^2}{(t-\alpha)^2+n^2}(u^\ad \sigma u)(\dif t) 
\eeq 
 and
\begin{multline} \label{CN16}
 \int_{\rhl}\Im \rk*{ n\ek*{\frac{\iu n}{t-(\iu n+\alpha)}+1}}(u^\ad \sigma u)(\dif t)\\
 =\Im\rk*{\int_{\rhl} n\ek*{\frac{\iu n}{t-(\iu n+\alpha)}+1}(u^\ad \sigma u)(\dif t)} \\
 \leq \abs*{\int_{\rhl} n\ek*{\frac{\iu n}{t-(\iu n+\alpha)}+1}(u^\ad \sigma u)(\dif t)}.
\end{multline}

 (III) Since \eqref{Nr.UFF} holds true for every choice of \(u\in\Cq\) and \(n\in\N \), \rrem{B8} yields 
\beql{Nr.GNL}
 g_n
 \in\LoaaaC{1}{\rhl}{\BorK}{\sigma}.
\eeq 
 Hence, \rrem{B8} shows that
\beql{CN17}
 \int_{\rhl} n\ek*{\frac{\iu n}{t-(\iu n+\alpha)}+1}(u^\ad \sigma u)(\dif t)
 = u^\ad \rk*{ \int_{\rhl} n\ek*{\frac{\iu n}{t-(\iu n+\alpha)}+1} \sigma(\dif t) } u
\eeq 
 is valid for each \(u\in\Cq\) and each \(n\in\N\).
 Combining \eqref{CN15},   \eqref{CN16}, and \eqref{CN17},  we have
\beql{CN18}
 0
 \le  \int_{\rhl}\frac{(t-\alpha)n^2}{(t-\alpha)^2+n^2}(u^\ad \sigma u)(\dif t)
 \leq \abs*{u^\ad \rk*{\int_{\rhl} n\ek*{\frac{\iu n}{t-(\iu n+\alpha)}+1}\sigma(\dif t)} u}
\eeq 
 for each \(u\in\Cq\) and each \(n\in\N \).
 For all \(n\in\N\) and all \(t\in\rhl\), we see that \(g_n(t)-n \cdot 1_{\rhl}(t)  = \tilde{g}_n(t)\)   holds true, where  \(\tilde{g}_n\colon\rhl\to\C\) is given by \(\tilde{g}_n(t)\defeq\iu n^2/\ek{t-(\iu n+\alpha)} \).
 Thus, for each \(n\in\N\), we get \(\tilde{g}_n = g_n -n \cdot 1_{\rhl}\), and, in view of \eqref{Nr.GNL}, then \( \tilde{g}_n \in \LoaaaC{1}{\rhl}{\BorK}{\sigma}\) and
\[
 \int_{\rhl}\tilde{g}_n\dif\sigma
 = \int_{\rhl}g_n \dif\sigma - n\int_{\rhl} 1_{\rhl}\dif\sigma
 =\int_{\rhl}g_n \dif\sigma - n\sigma\rk*{\rhl}.
\]
 Consequently, for each \(n\in\N \), we conclude
\[\begin{split}
 \int_{\rhl} n&\ek*{\frac{\iu n}{t-(\iu n+\alpha)}+1}\sigma(\dif t)
 = \int_{\rhl}\frac{\iu n^2}{t-(\iu n+\alpha)}\sigma(\dif t) + n\sigma\rk*{\rhl}\\
 &= \iu n^2\int_{\rhl}\frac{1}{t-(\iu n+\alpha)}\sigma(\dif t) + n\sigma\rk*{\rhl}\\
 &= \iu n^2\cdot S(\iu n+\alpha) + n\sigma\rk*{\rhl}= n\ek*{\iu n\cdot S(\iu n+\alpha) + \sigma\rk*{\rhl}}.
\end{split}\]
 Thus, because of \eqref{BMR}, for each \(u\in\Cq\) and each \(n\in\N\), we obtain
\beql{BM1}
 \abs*{u^\ad \rk*{\int_{\rhl}n\ek*{\frac{\iu n}{t-(\iu n+\alpha)}+1}\sigma(\dif t) } u}
 \leq \abs*{  nu^\ad\ek*{\iu n\cdot S(\iu n+\alpha) + s_0} u}.
\eeq 
 Taking into account \eqref{CN3}, \eqref{CN18},  \eqref{BM1}, and \eqref{CN13-1}, for each \(u\in\Cq\), we get
\beql{FAE}
 \begin{split}
  \int_{\rhl} & \abs{t-\alpha}(u^\ad \sigma u)(\dif t)
  \leq\liminf_{n\to\infi}\int_{\rhl}\frac{(t-\alpha)n^2}{(t-\alpha)^2+n^2}(u^\ad \sigma u)(\dif t)  \\
  &\leq \liminf_{n\to\infi} \abs*{u^\ad \rk*{\int_{\rhl} n\ek*{\frac{\iu n}{t-(\iu n+\alpha)}+1}\sigma(\dif t)} u} \\
  &\leq \liminf_{n\to\infi} \abs*{  nu^\ad\ek*{\iu n\cdot S(\iu n+\alpha) + s_0} u }\\
  &\leq \liminf _{n\to\infi} \normS{-\alpha s_0+s_1}\cdot\normE{u}^2 =\normS{-\alpha s_0+s_1}\cdot\normE{u}^2 
  < \infp.
 \end{split}
\eeq 
 Therefore, from \eqref{FAE} we obtain that 
\[\begin{split}
 \int_{\rhl}\abs{t}(u^\ad \sigma u)(\dif t)
 & \leq \int_{\rhl}\rk*{\abs{t-\alpha}+\abs{\alpha}}(u^\ad \sigma u)(\dif t) \\
 & = \int_{\rhl}\abs{t-\alpha}(u^\ad \sigma u)(\dif t) + \int_{\rhl}\abs{\alpha}(u^\ad \sigma u)(\dif t) \\
 &\leq \normS{-\alpha s_0+s_1}\cdot\normE{u}^2 + \abs{\alpha}(u^\ad \sigma u)\rk*{\rhl}
 < \infp
\end{split}\]
 is true for all \(u\in\Cq\).
 Thus, \rrem{B8} provides us \(\sigma\in\Mgguqa{1}{\rhl}\).
\end{proof}

\begin{lem} \label{Mi,L3.8}
 Let \(\alpha\in\R\), let \(\kappa\in\Ninf \), let \(\seqska\) be a sequence from \(\Cqq\), and let  \(n\in\NO \) be such that  \(2n+1 \leq \kappa\).
 Further, let \(S\in\SFOq  \) be such that  \(P_{2n}^{[S]}(z) \in\Cggo{(n+2)q}\) and  \( P_{2n+1}^{[S]}(z) \in\Cggo{(n+2)q}\) hold true for all  \(z\in\uhp\).
 Then: 
\begin{enui}
 \item\label{Mi,L3.8.a} The \taSm{} \(\sigma\) of \(S\) belongs to \(\Mgguqa{1}{\rhl}\).
 \item\label{Mi,L3.8.b} The function \(\phi\colon\rhl\to\Cqq\) given by \(\phi (t) \defeq  \sqrt{t-\alpha} \Iq\) belongs to \(\aaLsaaaC{q}{q}{\rhl}{\BorK}{\sigma}\) and  \(\sigma^\#\colon\BorK \to\Cqq\) defined by \eqref{S1421N} belongs to \(\MggqK \).
 \item\label{Mi,L3.8.c} The function \(\tilde{S} \colon\Cs\to\Cqq\) given by
\beql{Nr.SRD}
 \tilde{S} (z)
 \defeq (z-\alpha)S(z)
\eeq     
 and the \taSt{} \(S^{[\sigma^\#]}\) of \(\sigma^\#\) fulfill 
\begin{align} \label{Nr.SRG}
 \tilde{S}(z)&=S^{[\sigma^\#]}(z)-\sigma\rk*{\rhl}&\text{for each }z&\in\Cs.
\end{align}         
 \item\label{Mi,L3.8.d} The function \((\tilde{S})_\square\defeq \Rstr_{\uhp} \tilde{S}\) belongs to \(\RFsq\) and \((\tilde{\sigma})_\square\colon\BorR\to\Cqq\) given by \( (\tilde{\sigma})_\square(B)\defeq \sigma^\#\rk{ B\cap\rhl}\) is exactly the matricial spectral measure of \((\tilde{S})_\square\).
\end{enui}
\end{lem}
\begin{proof} 
 \eqref{Mi,L3.8.a} \rPart{Mi,L3.8.a} is proved in \rlem{Mi,L3.7}.
 
 \eqref{Mi,L3.8.b} In view of~\eqref{Mi,L3.8.a}, \rpart{Mi,L3.8.b} follows immediately from  \rrem{Mi,L3.6}.

 \eqref{Mi,L3.8.c} Let \(z\in\Cs\).
 According to \rrem{L1.53} and \rthm{T1316D}, the function \(g_{\alpha,z}\colon\rhl\to\C\) given by \(g_{\alpha, z} (t) \defeq  \rk{z-\alpha}/\rk{t-z}\) belongs to  \(\LoaaaC{1}{\rhl}{\BorK}{\sigma}\) and
\[
 (z-\alpha)S(z)
 =\int_{\rhl}\frac{z-\alpha}{t-z}\sigma(\dif t)
\]
 is true.
 Consequently, in view of \rlem{B26}, the pair \([g_{\alpha,z} \Iq,1_{\rhl}\Iq]\)  is left-integrable with respect to \(\sigma\) and 
\[
 \tilde{S} (z)
 =  (z-\alpha)S(z)
 =\int_{\rhl}\rk*{\frac{z-\alpha}{t-z}\Iq} \sigma(\dif t) \Iq^\ad 
\]
 is valid.
 Thus, \rrem{M22} shows that the pair  \([g_{\alpha,z}\Iq+ 1_{\rhl}\Iq,\Iq]\) is left-integrable with respect to \(\sigma\) and that 
\[
 \int_{\rhl}\ek*{\rk*{\frac{z-\alpha}{t-z}+1} \Iq}\sigma(\dif t) \Iq^\ad
 = \int_{\rhl}\rk*{\frac{z-\alpha}{t-z}\Iq}\sigma(\dif t) \Iq^\ad      + \int_{\rhl}\Iq \sigma(\dif t) \Iq^\ad 
\]
 is fulfilled.
 Taking into account 
\[
 \sigma \rk*{\rhl}
 =\int_{\rhl} 1_{\rhl} \dif\sigma
 =\int_{\rhl} ( 1_{\rhl}\Iq) \dif\sigma ( 1_{\rhl}\Iq)^\ad
 =\int_{\rhl} \Iq\sigma(\dif t)\Iq^\ad 
\]
 and that  \(\rk{z-\alpha}/\rk{t-z}+1 = \rk{t-\alpha}/\rk{t-z}\) holds true for each \(t\in\rhl\),  we get then
\beql{Nr.FSR}
 \tilde{S} (z)
 = \int_{\rhl}\rk*{\frac{t-\alpha}{t-z}\Iq}\sigma(\dif t) \Iq^\ad  - \sigma\rk*{\rhl}.
\eeq 
 Because of   \rlem{B26}, \rprop{M24}, and  \eqref{S1421N}, we have
\[\begin{split}
 \int_{\rhl}\rk*{\frac{t-\alpha}{t-z}\Iq}\sigma(\dif t) \Iq^\ad
 & = \int_{\rhl}\ek*{\rk*{\frac{1}{t-z}\Iq} (\sqrt{t-\alpha}\Iq)}\sigma(\dif t) \ek*{ \Iq\rk{\sqrt{t-\alpha}\Iq}}^\ad \\
 & =\int_{\rhl}\rk*{\frac{1}{t-z}\Iq} \sigma^\#(\dif t)\Iq^\ad  =\int_{\rhl} \frac{1}{t-z} \sigma^\# (\dif t).
\end{split}\]
 Thus, from \eqref{Nr.FSR} it follows
\[
 \tilde{S}(z)
 =\int_{\rhl}\frac{1}{t-z} \sigma^\#(\dif t) - \sigma\rk*{\rhl}
 = S^{[\sigma^\#]} (z) - \sigma (\rhl ).
\]

 \eqref{Mi,L3.8.d} In view of \rthm{T1316D},  the function \(S^{[\sigma^\#]}\)  belongs to \(\SFOq  \).
 Thus, \rrem{R1457D} shows that \( \Rstr_{\uhp} S^{[\sigma^\#]} \in\RFOq \subseteq\RFsq\), that the matricial spectral measure \(\mu^\#\) of \(\Rstr_{\uhp}S^{[\sigma^\#]}\) fulfills \( \sigma^\#=\Rstr_{\BorK }\mu^\#\), and that \(\mu^\#(\R\setminus\rhl )=\mu^\#(\crhl )=\Oqq\).
 Consequently,  \((\tilde{\sigma})_\square\) is the matricial spectral measure of \(\Rstr_{\uhp} S^{[\sigma^\#]}\).
 From \rthm{T31DN} one can easily see that the function \(F\colon\uhp\to\Cqq\) given by  \(F(z)\defeq -\sigma(\rhl)\) belongs to \(\RFsq\) and that the matricial spectral measure \(\theta\) of \(F\) fulfills \(\theta(B)=\Oqq \) for all \(B\in\BorR\) (see also~\cite[Beispiel~1.2.1]{CR01}).
 Because of \eqref{Nr.SRG}, we have \((\tilde{S})_\square=\Rstr_{\uhp} S^{[\sigma^\#]} + F\).
 Since \(\Rstr_{\uhp} S^{[\sigma^\#]}\) and \(F\) both belong to \(\RFsq\), from~\cite[Remark~4.4]{MR2988005} we see that  \((\tilde{S})_\square \in\RFsq\) and that \((\tilde{\sigma})_\square + \theta\)  is the matricial spectral measure of \((\tilde{S})_\square\).
 In view of \((\tilde{\sigma})_\square +\theta = (\tilde{\sigma})_\square\), the proof is complete.
\end{proof}

\begin{lem} \label{Mi,L3.9}
 Let \(\alpha\in\R\), let \(\kappa\in\NOinf \), and let \(\seqska\) be a sequence of complex \tqqa{matrices}.
 Then:
\begin{enui} 
 \item\label{Mi,L3.9.a} Let \( n \in \NO  \) be such that \(2n\leq\kappa\)  and let  \(S\in\SFOq  \) be such that 
\begin{align} \label{Mi,L3.9-1}
 P_{2n}^{[S]}(z)&\in \Cggo{(n+2)q}&\text{for all }z&\in\C\setminus\R.
\end{align}
 Then the \taSm{} \(\sigma\) of \(S\) belongs to \(\Mgguqa{2n}{\rhl}\) and the inequality  \(H_n^{[\sigma]} \lleq H_n\) holds true.
 \item\label{Mi,L3.9.b} Let \(n \in \NO \) be such that \(2n+1\leq\kappa\) and let \(S\in\SFOq  \) be such that 
\begin{align} \label{Mi,L3.9-30}
 \set*{ P_{2n}^{[S]}(z), P_{2n+1}^{[S]}(z)}&\subseteq \Cggo{(n+2)q}&\text{for all }z&\in\C\setminus\R.
\end{align} 
 Then the \taSm{} \(\sigma\) of \(S\) belongs to \( \Mgguqa{2n+1}{\rhl} \) and the inequality \( H_{\at{n}}^{[\sigma]}  \lleq H_{\at n}\) holds true.
 \end{enui}
\end{lem}
\begin{proof}
 \eqref{Mi,L3.9.a} Because of \eqref{Mi,L3.9-1}, we get \(H_n\in\Cggo{(n+1)q}\subseteq\CHo{(n+1)q}\) and, in particular, \(s_j^\ad = s_j\) for each \( j\in \mn{0}{2n}\).
 In view of  \(S \in\SFOq  \),  we see that the function \(S\) is holomorphic in \(\Cs \) and, using additionally~\zitaa{MR2988005}{\cpropss{8.9}{8.8}}, we also obtain \(\Rstr_{\uhp}S\in\RFOq \subseteq\RFsq\).
 Let \(f\defeq  S\) and let  \(F_{2n}\colon\uhp \to \Coo{(n+1)q}{(n+1)q}\) be given by \eqref{Nr.F2N}.
 Using  \rlem{Mi,L3.4} and~\zitaa{MR2988005}{\cpropss{8.9}{8.8}}, we conclude that \( F_{2n} \in \RFO{(n+1)q} \subseteq \RFs{(n+1)q}\) and that the matricial spectral measure \(\mu_{2n}\) of \(F_{2n}\) fulfills \(\mu_{2n}(\R)\lleq H_n\).
 Let \(\Psi_{2n}\colon\C\to\Coo{(n+1)q}{(n+1)q}\) be given  \eqref{Nr.PH2NN}.
 Since \(s_j^\ad = s_j\) holds true for each \(j\in\mn{0}{2n}\), from  \rlem{Mi,L3.5} we see that  \(\Psi_{2n}\) is a continuous matrix-valued function with \(\Psi_{2n}(\R)\subseteq \CHo{(n+1)q}\).
 Furthermore,  \rlem{Mi,L3.5} yields \eqref{N1439D}.
 According to \rrem{R1457D}, the matricial spectral measure \(\sigma_{\square}\) of  \(\Rstr_{\uhp}S\) fulfills
\begin{align} \label{Mi,L3.9-8}
 \sigma&=\Rstr_{\BorK }\sigma_{\square}&
&\text{and}&
 \sigma_{\square}\rk*{\R\setminus\rhl}&=\NM.
\end{align}
 Standard arguments of measure theory show that we can choose sequences \((a_k)_{k=1}^\infi \) and \((b_k)_{k=1}^\infi \) of real numbers such that
\begin{gather} 
 \sigma_{\square}\rk*{\set{a_k}}=\NM,\quad\sigma_{\square}\rk*{\set{b_k}}=\NM,\quad\mu_{2n}\rk*{\set{a_k}}=\NM,\quad\mu_{2n}\rk*{\set{b_k}}=\NM,\label{7.81-11}\\
   a_k< b_k,\quad\text{and}\quad(a_k,b_k)\subseteq (a_{k+1},b_{k+1})\label{Mi,L3.9-13}
\end{gather}
 hold true for each \(k\in\N\) and that \( \bigcup_{k=1}^\infi (a_k,b_k) =\R\).
 In view of \(F_{2n}\in\RFs{(n+1)q}\), a matricial version of Stieltjes' inversion formula (see~\cite[Theorem~8.6]{MR2222521}),  and  \eqref{7.81-11} provide us
\beql{Mi,L3.9-16}
 \begin{split}
  \mu_{2n}\rk*{(a_k,b_k)}
  &=\frac{1}{2}\ek*{\mu_{2n}\rk*{\set{ a_k}}+\mu_{2n}\rk*{\set{ b_k}}}+\mu_{2n}\rk*{(a_k,b_k)}\\
  &=\frac{1}{\pi}\lim_{\epsilon\to 0+0}\int_{[a_k,b_k]}\Im F_{2n}(x+\iu\epsilon) \lambda^{(1)}(\dif x)
 \end{split}
\eeq 
 for all \(k\in\N\), where \(\lambda^{(1)}\) is the Lebesgue measure defined on \(\BorR \).
 The function \(E_{q,n}\colon\C\to\Coo{(n+1)q}{q}\) given by \eqref{3.64.-1} is holomorphic in \(\C\).
 Moreover, \(\Psi_{2n}\) is continuous with \(\Psi_{2n} (\R ) \subseteq\CHo{(n+1)q}\).
 Thus, for all \(k\in\N \), we get from \eqref{N1439D},  a matricial version of Stieltjes' inversion formula (see~\cite[Theorem~8.6]{MR2222521}) and  \eqref{7.81-11} that
\beql{Mi,L3.9-17}
 \begin{split}
  &\frac{1}{\pi}\lim_{\epsilon\to 0+0}  \int_{[a_k,b_k]}\Im F_{2n}(x+\iu\epsilon) \lambda ^{(1)}(\dif x)\\
  &= \frac{1}{2}\rk*{  E_{q,n}(a_k)\sigma_{\square}\rk*{\set{ a_k}}  \ek*{E_{q,n}(a_k)}^\ad+ E_{q,n}(b_k)\sigma_{\square}\rk*{\set{b_k}}\ek*{E_{q,n}(b_k)} ^\ad  }\\
  &\qquad  + \int_{(a_k,b_k)}E_{q,n}(t)\sigma_{\square}(\dif t) E_{q,n}^\ad (t)\\
  &=  \int_{(a_k,b_k)}E_{q,n}(t)\sigma_{\square}(\dif t) E_{q,n} ^\ad (t).
 \end{split}
\eeq 
 Combining \eqref{Mi,L3.9-17} and \eqref{Mi,L3.9-16}, we obtain 
\begin{align} \label{Mi,L3.9-18}
 \int_{(a_k,b_k)}E_{q,n}(t)\sigma_{\square}(\dif t) E_{q,n}^\ad (t)&=\mu_{2n}\rk*{(a_k,b_k)}\lleq\mu_{2n}(\R)&\text{for all }k&\in\N 
\end{align}
 and, consequently, 
\begin{align}  \label{Mi,L3.9-20}
 \tr\ek*{\int_{(a_k,b_k)} E_{q,n}(t) \sigma_\square(\dif t) E_{q,n}^\ad (t)}
 &\leq \tr\ek*{\mu_{2n}(\R)}
 < \infp&\text{for all }k&\in\N.
\end{align}
 The trace measure \(\tau \defeq  \tr \sigma_\square\) of \(\sigma_\square\) is a finite measure and \(\sigma_\square\) is absolutely continuous with respect to \(\tau\).
 We can choose a version \((\sigma_\square)'_\tau\) of the matricial Radon--Nikodym derivative of \(\sigma_\square\) with respect to \(\tau\) such that  \((\sigma_\square)_{\tau}'(t)\in\Cggq\) for all \(t\in\R\).
 For all \(k\in\N \), then  \(g_k \defeq\normF{1_{(a_k,b_k)}(\Rstr_\R E_{q,n}) \sqrt{(\sigma_\square)_{\tau}'}}^2\in\LoaaaC{1}{\R}{\BorR}{\tau}\) and  \(\tr  \ek{ \int_\R  \rk{  1_{(a_k,b_k)}\Rstr_\R E_{q,n}} \dif\sigma_\square \rk{  1_{(a_k,b_k)} \Rstr_\R E_{q,n} }^\ad } = \int_\R g_k \dif\tau\).
 Thus, by virtue of  \eqref{Mi,L3.9-20}, we get  
\beql{Mi,L3.9-22}
 \int_\R g_k \dif\tau
 \leq \tr\ek*{\mu_{2n}(\R)}
 <\infp
\eeq 
 for all \(k\in\N\).
 Obviously, \(g\colon\R\to\C\) defined by \(g(t) \defeq\normF{E_{q,n}(t)\sqrt{(\sigma_\square})'_\tau}^2\) is an \(\BorR\)\nobreakdash-\(\BorC\)\nobreakdash-measurable function  with \(g(\R) \subseteq [0,\infp)\).
 For  all \(t\in\R\), we see that 
\beql{Mi,L3.9-24}
 g(t)
 =\normF*{\ek*{\lim_{k\to\infi} 1_{(a_k,b_k)}(t)}\cdot\ek*{(\Rstr_\R E_{q,n})\sqrt{(\sigma_\square)_{\tau}'}}(t)}^2
 =\lim_{k\to\infi} g_k(t)
 =\liminf_{k\to\infi} g_k(t).
\eeq 
 In view of \eqref{Mi,L3.9-24} and \eqref{Mi,L3.9-22}, Fatou's lemma yields then
\[
 \int_\R  \abs*{g(t)} \tau (\dif t)
 =\int_\R\liminf_{k\to\infi} g_k(t) \tau(\dif t)
 \leq\liminf_{k\to\infi}\int_\R g_k(t)\tau (\dif t)
 \leq \tr \ek*{\mu_{2n}(\R)}
 <\infp,
\]
 and, consequently,  \(g\in\LoaaaC{1}{\R}{\BorR}{\tau}\).
 Because of \rlem{B26}, it follows
\beql{Mi,L3.9-25}
 \Rstr_\R E_{q,n}
 \in \aaLsaaaC{(n+1)q}{q}{\R}{\BorR}{\sigma_\square}.
\eeq 
 Hence, from   \eqref{Mi,L3.9-8} we obtain that \(\Rstr_{\rhl} E_{q,n}\) belongs to \( \aaLsaaaC{(n+1)q}{q}{\rhl}{\BorK}{\sigma}\) and  that
\beql{Mi,L3.9-27}
 \int_\R E_{q,n}(t)\sigma_\square(\dif t)E_{q,n}^\ad(t)
 =\int_{\rhl} E_{q,n}(t) \sigma(\dif t) E_{q,n}^\ad (t).
\eeq 
 Furthermore, applying  \rrem{D1229}, we get  \(\sigma \in  \Mgguqa{2n}{\rhl} \) and \eqref{Mi,L3.9-28}.
 Because of  \eqref{Mi,L3.9-25},  we see that \(\Theta_n\colon\BorR\to\Coo{(n+1)q}{(n+1)q}\) defined by
\beql{L3.9-28-1}
 \Theta_n(B)
 \defeq  \int_B  E_{q,n} (t) \sigma_\square (\dif t) E_{q,n}^\ad  (t)
\eeq 
 is a well-defined \tnnH{}  \taaa{(n+1)q}{(n+1)q}{measure} on  \((\R,\BorR)\).
 Using  \eqref{L3.9-28-1}, \( \bigcup_{k=1}^\infi (a_k, b_k) =\R\),  \eqref{Mi,L3.9-13}, \(\Theta_n \in \Mggoa{(n+1)q}{\R,\BorR}\),  \eqref{Mi,L3.9-18}, and \(\mu_{2n} \in \Mggoa{(n+1)q}{\R,\BorR}\), we conclude  
\beql{Mi,L3.9-29}
 \begin{split}
  \int_\R E_{q,n} &(t)\sigma_\square(\dif t)E_{q,n}^\ad(t)
  =\Theta_n(\R) 
  =\Theta_n \rk*{ \bigcup_{k=1}^\infi(a_k,b_k) } 
  =\lim_{k\to\infi}\Theta_n\rk*{(a_k,b_k)}\\
  &=\lim_{k\to\infi}\int_{(a_k,b_k)} E_{q,n}(t) \sigma_\square(\dif t) E_{q,n}^\ad (t)
  = \lim_{k\to\infi}\mu_{2n}\rk*{(a_k,b_k)}\\
  &     =\mu_{2n}\rk*{ \bigcup_{k=1}^\infi(a_k,b_k)} 
  =\mu_{2n}(\R).
 \end{split}
\eeq 
 The combination of \eqref{Mi,L3.9-28}, \eqref{Mi,L3.9-27}, \eqref{Mi,L3.9-29}, and \(\mu_{2n} (\R) \lleq H_n\) provides us then 
\[
 H_n^{[\sigma]}
 =\int_{\rhl} E_{q,n}(t)\sigma(\dif t) E_{q,n}^\ad (t)
 =\int_\R \Rstr_\R E_{q,n} \dif\sigma_\square\rk{\Rstr_\R E_{q,n}}^\ad
 =\mu_{2n}(\R)
 \lleq H_n.
\]

 \eqref{Mi,L3.9.b} Because of  \eqref{Mi,L3.9-30}, we have \(\set{ H_n,  H_{\at n}} \subseteq\Cggo{(n+1)q}\subseteq\CHo{(n+1)q}\) and, consequently, \(s_j^\ad = s_j\) for each \(j\in\mn{0}{2n+1}\).
 Since \(S\) belongs to \(\SFOq \), from \eqref{Mi,L3.9-30} and   \rlem{Mi,L3.8} we infer that \(\sigma\) belongs to \(\Mgguqa{1}{\rhl}\), that  \(\sigma^\#\colon\BorK \to\Cqq\) defined by \eqref{S1421N} belongs to \(\MggqK \), that \(\tilde{S} \colon\Cs \to\Cqq \) given by \eqref{Nr.SRD} is a function with \(\Rstr_{\uhp} \tilde{S} \in\RFsq\), and that   \((\sigma^\#)_\square\colon\BorR\to\Cqq\) given by \((\tilde{\sigma})_\square (B)\defeq  \sigma^\# (B \cap \rhl )\) is the matricial spectral measure of \((\tilde{S})_\square \defeq  \Rstr_{\uhp} \tilde{S}\).
 Observe that \rremp{Mi,L3.6}{Mi,L3.6.b} shows that \eqref{D2012} holds true.
 Now \rpart{Mi,L3.9.b} can be proved analogous to \rpart{Mi,L3.9.a}, where \(F_{2n+1} \colon\uhp\to\Coo{(n+1)q}{(n+1)q}\) given by \eqref{Nr.F2N+1} and \(\Psi_{2n+1} \colon\C \to \Coo{(n+1)q}{(n+1)q}\) defined by \eqref{Nr.PH2N+1} play the roles of \(F_{2n}\) and \(\Psi_{2n}\), respectively (for details, see~\cite[Lemma~7.9]{Mak14}).
\end{proof}

\begin{rem} \label{Mi,L3.13}
 Let \(\alpha\in\R\), let \(\kappa\in\NOinf \), and let\(\seqska\) be a sequence from \(\Cqq\).
 Using  \rrem{NR41N} and the definition of the class \(\SFOq \), it is readily checked that the following statements hold true:
\begin{enui} 
 \item\label{Mi,L3.13.a} If \(n \in \NO \) is such that  \(2n\leq\kappa\) and if  \(S\in\SFOq  \) is such that  \(P_{2n}^{[S]}(\iu y)\in \Cggo{(n+2)q}\) for all \(y\in (0,\infp)\), then 
\beql{M1836M}
 \lim_{y\to\infp}R_{\Tqn}(\iu y)\ek*{\vqn S(\iu y)-u_n}
 =\NM.
\eeq 
 \item\label{Mi,L3.13.b} If \(\kappa\geq 1\) and \( n \in \NO  \) are such that \(2n+1\leq\kappa\) and if \(S\in\SFOq \) is such that  \(P_{2n+1}^{[S]}(\iu y)\in\Cggo{(n+2)q}\) holds true for each \(y\in(0,\infp)\), then
\[
 \lim_{y\to\infp}R_{\Tqn}(\iu y)\ek*{\vqn (\iu y-\alpha)S(\iu y)-(-\alpha u_n-y_{0,n})} 
 =\NM. 
\]
\end{enui}
\end{rem}

\begin{rem} \label{Mi,L3.15}
 Let \( n \in \NO  \) and let \(y\in\R\).
 If \(u\in\Coo{(n+1)q}{p}\) is such that \(\lim_{y\to\infp}[u^\ad R_{\Tqn}(\iu y)u]=\NM\), then from \rrem{21112N} one can easily see that  \(u=\NM\).
\end{rem}

\begin{lem} \label{Mi,L3.16}
 Let \(\alpha\in\R\), let \(\kappa\in\NOinf \), and let \(\seqska\) be a sequence of complex \tqqa{matrices}.
 Then:
\begin{enui} 
 \item\label{Mi,L3.16.a} Let \(n\in\NO \) be such that \(2n\leq\kappa\) and let \(S\in\SFOq \) be such that  \(P_{2n}^{[S]}(z)\in \Cggo{(n+2)q}\) holds true for all  \(z\in\C\setminus\R\).
 Then the \taSm{} \(\sigma\) of \(S\) belongs to \(\Mgguqa{2n}{\rhl}\) and  \(S\) belongs to \(\SFOqskg{2n}\).
 \item\label{Mi,L3.16.b} Let \(n\in \NO \) be such that  \(2n+1\leq\kappa\) and let \(S\in\SFOq  \) be such that \(\set{ P_{2n}^{[S]}(z), P_{2n+1}^{[S]}(z)} \subseteq \Cggo{(n+2)q}\) holds true for each  \(z\in\C\setminus\R\).
 Then the \taSm{} \(\sigma\) of \(S\) belongs to \(\Mgguqa{2n+1}{\rhl} \) and \(S\) belongs to \(\SFOqskg{2n+1}\). 
\end{enui}
\end{lem}
\begin{proof}
 We give a shortened version of the detailed proof stated in~\cite[Lemma~5.15]{Sch11}.
 
 \eqref{Mi,L3.16.a} \rlem{Mi,L3.9} yields \(\sigma \in \Mgguqa{2n}{\rhl}\) and \(H_n^{[\sigma]} \lleq H_n\).
 If \(n=0\), then \(\sigma\in\MggqKskg{2n}\) follows.
 Suppose now \(n\ge 1\).
 \rrem{Mi,L3.13} shows that \eqref{M1836M} is valid.
 Obviously, \(\sigma\in\MggqKakg{\seq{\suo{j}{\sigma}}{j}{0}{2n}}\), where \(\seq{\suo{j}{\sigma}}{j}{0}{2n}\) is defined by \eqref{FR1148}.
 Thus, \rprop{lemM4213-2} and \rrem{Mi,L3.13} provide us
\beql{M1907M}
 \lim_{y\to\infp} R_{T_{q,n}} (\iu y) \ek*{ v_{q,n} S(\iu y) - u_n^{[\sigma]}}
 =\NM
\eeq 
 where \(s_{-1}^{[\sigma]} \defeq  \Oqq \) and where \(u_n^{[\sigma]} \defeq-\col (s_{j-1}^{[\sigma]})_{j=0}^n\).
 Combining \eqref{M1836M}  and \eqref{M1907M}, we get
\[
 \lim_{y\to\infp} \rk{u_n^{[\sigma]}-u_n}^\ad R_{T_{q,n}} (\iu y) \rk{ u_n^{[\sigma]} - u_n}
 =\NM.
\]
 Consequently, \rrem{Mi,L3.15} yields \(u_n^{[\sigma]} = u_n\).
 Let \(d_j \defeq  s_j - \suo{j}{\sigma}\) for each \(j\in\mn{0}{2n}\).
 Then \(u_n^{[\sigma]} = u_n\) and \(n\ge 1\) imply \(d_0 =\Oqq \).
 Furthermore, the inequality \(H_n^{[\sigma]}\lleq H_n\) shows that the block Hankel matrix \(\mat{d_{j+k}}_{j,k=0}^n\) is \tnnH{}.
 Thus, \(d_{2n} \in\Cggq \) and \rrem{CDFK06} yield \(d_j =\Oqq\) for each \(j\in\mn{0}{2n-1}\).
 Hence, \(\sigma\) belongs to \(\MggqKskg{2n}\).
 
 \eqref{Mi,L3.16.b} \rPart{Mi,L3.16.b} can be proved analogously.
\end{proof}

 Now we are able to prove that the solution set of the (reformulated) truncated Stieltjes-type moment problem and the solution set of the corresponding system of  the fundamental  Potapov's matrix inequalities coincide.

\begin{thm} \label{MT4315}
 Let \(\alpha\in\R\), let \(\kappa\in\NOinf \), and let \(\seqska \) be a sequence of complex \tqqa{matrices}.
 Let \(\cD\)  be a discrete subset of \(\uhp\) and let \(S\colon\Cs\to\Cqq\) be a holomorphic matrix-valued function.
 Then:
\begin{enui}
 \item\label{MT4315.a} Let \(n\in\NO \) be such that \(2n\leq \kappa\).
 Then the following statements are equivalent:
\begin{aeqii}{0}
 \item\label{MT4315.i} \(S\in\SFOqskg{2n}\).
 \item\label{MT4315.ii} \(P_{2n-1}^{[S]}(z) \in \Cggo{(n+1)q}\) and \(P_{2n}^{[S]}(z)\in \Cggo{(n+2)q}\)  for all \(z\in\uhp\setminus\cD\).
\end{aeqii}
 \item\label{MT4315.b} Let \(n\in\NO \) be such that  \(2n+1\leq \kappa\).
 Then the following statements are equivalent:
\begin{aeqii}{2}
 \item\label{MT4315.iii} \(S\in\SFOqskg{2n+1}\).
 \item\label{MT4315.iv} \(\set{ P_{2n}^{[S]}(z), P_{2n+1}^{[S]}(z)}\subseteq\Cggo{(n+2)q}\) for all \(z\in\uhp\setminus\cD\).
\end{aeqii}    
\end{enui}
\end{thm}
\begin{proof}
 \impl{MT4315.i}{MT4315.ii}, \impl{MT4315.iii}{MT4315.iv}: Use \rprop{lemM4213-2}.

\bimp{MT4315.ii}{MT4315.i}
 Let \(m\defeq  2n\).
 Observe that the function \( F  \defeq  \Rstr_{\uhp \setminus \cD} S\) is holomorphic.
 Because of~\ref{MT4315.ii}, the inequalities \(P_{m-1}^{[ F ]} (z) \lgeq\NM\) and \(P_{m}^{[ F ]} (z) \lgeq\NM\) hold true for each \(z\in\uhp \setminus \cD\).
 From \rthm{T1747F} we get then that there is a unique function \(\hat S\in\SFOq \) such that \(\Rstr_{\uhp \setminus \cD}\hat S =  F \), namely \(\hat S=S\), and that \(P_{k}^{[\hat S]} (z) \lgeq\NM\)  are valid for all \(k\in\mn{-1}{m}\) and all  \(z\in\C\setminus\R\).
 Applying \rlem{Mi,L3.16}, we get then~\ref{MT4315.i}.
\eimp

\bimp{MT4315.iv}{MT4315.iii}
 Let \(m\defeq 2n+1\) and use the same argumentation as in the proof of the implication ``\impl{MT4315.ii}{MT4315.i}''.
\eimp
\end{proof}

\section{Some considerations on block Hankel matrices}\label{S1709}   
 First we introduce some further block Hankel matrices.
 In \rsec{S1415}, in particular in \rremss{ML415}{lemC21-1}, we already discussed some aspects on such matrices.
 
 Let \(\kappa \in \NOinf \) and let \(\seqska \) be a sequence of complex \tpqa{matrices}.
 Let \(H_n \defeq\mat{s_{j+k}}_{j,k=0}^n\) for each \(n \in \NO \) with \(2n\leq \kappa\), let \(K_n \defeq\mat{s_{j+k+1}}_{j,k=0}^n\) for each \(n \in \NO \) with \(2n+1 \leq \kappa\), and let \(G_n \defeq\mat{s_{j+k+2}}_{j,k=0}^n\) for each \(n\in\NO \) with \(2n + 2\le \kappa\).
 For every choice of integers \(m\) and \(n\) with \(0\le m\le n\le \kappa\), let \(y_{m,n}\) and \(z_{m,n}\) be given by \eqref{YZ}, let \(\hat{y}_{m,n} \defeq \col (s_{n-j})_{j=0}^{n-m}\), and let \(\hat{z}_{m,n} \defeq \row (s_{n-k})_{k=0}^{n-m}\).

 We will see that certain Schur complements play an essential role for our considerations.
 Let \(L_0 \defeq s_0\) and, for each \(n\in\N\) with \(2n \le \kappa\), furthermore 
\[
 L_n
 \defeq s_{2n} - z_{n,2n-1} H_{n-1}^\mpi  y_{n,2n-1}.
\]
 For every choice of integers \(m\) and \(n\) with \(0\le m\le n\le \kappa -1\), let \(y_{\alpha \triangleright m,n} \defeq \col (s_{\alpha \triangleright m+j})_{j=0}^{m-n}\) and \(z_{\alpha \triangleright m,n} \defeq \row (s_{\alpha \triangleright m+k})_{j=0}^{m-n}\).
 Let \(L_{\alpha\triangleright 0} \defeq s_{\alpha\triangleright 0}\) and, for each \(n\in\NO \) with \(2n + 1\le\kappa\), moreover \(L_{\at{n}} \defeq s_{\alpha\triangleright 2n} - z_{\alpha \triangleright n,2n-1} H_{\alpha\triangleright n-1}^\mpi  y_{\alpha \triangleright n, 2n-1}\).

\begin{rem} \label{lemM431-2}
 Let \(\kappa \in \NOinf \) and let \(\seqska \) be a sequence of \tH{} complex \tqqa{matrices}.
 In view of \rrem{lemC21-1}, it is readily checked that 
\begin{multline*} 
 \ek*{R_{\Tqn^\ad }(w)}^\invad H_n \Tqn^\ad -\Tqn H_n \ek*{R_{\Tqn^\ad}(z)}^\inv+(\ko{w}-z) \Tqn H_n \Tqn^\ad\\
 = v_{q,n}v_{q,n}^\ad H_n \Tqn^\ad -\Tqn H_n v_{q,n}v_{q,n}^\ad\label{SKP3-12-II}
\end{multline*} 
 and 
\beql{B204}
 \ek*{R_{\Tqn}(z)}^\inv H_n\Tqn^\ad -\Tqn H_n \ek*{R_{\Tqn}(\ko z)}^\invad
 = v_{q,n} v_{q,n}^\ad H_n \Tqn^\ad -\Tqn H_n v_{q,n} v_{q,n}^\ad 
\eeq
 are fulfilled for every choice of \(n\in\NO \) with \(2n\le \kappa\) and \(w,z\in\C\).
\end{rem}

\begin{rem} \label{ZF1}
 Let \(\alpha \in \R\), let \(\kappa \in \Ninf \), and let \(\seqska \) be a sequence of \tH{} complex \tqqa{matrices}.
 In view of \rremss{ML415}{lemC21-1}, then
\begin{multline} \label{FIDW3}
 \ek*{R_{\Tqn}(z)}^\inv H_{\at{n}}\Tqn^\ad- \Tqn H_{\at{n}}\ek*{R_{\Tqn}(\ko z)}^\invad\\
 = v_{q,n} v_{q,n}^\ad H_n\ek*{R_{\Tqn}(\alpha)}^\invad - \ek*{R_{\Tqn}(\alpha)}^\inv H_n v_{q,n} v_{q,n}^\ad 
\end{multline} 
 is valid for each \(n \in \NO \) with \(2n+1 \leq \kappa\) and each \(z\in\C\).
 \rrem{lemC21-1} yields 
\begin{align} 
 \ek*{R_{\Tqn}(\alpha)}^\inv H_n v_{q,n}&= \ek*{R_{\Tqn}(\alpha)}^\inv y_{0,n},&
 \ek*{R_{\Tqn}(\alpha)}^\inv H_n v_{q,n}& = \alpha u_n+ y_{0,n}\label{RHV-1}
\end{align}
 and
\begin{align*}
 z_{0,n} \ek*{R_{\Tqn}(\alpha)}^\invad&= v_{q,n}^\ad H_n \ek*{R_{\Tqn}(\alpha)}^\invad,&
 \alpha w_n+ z_{0,n}&= v_{q,n}^\ad H_n \ek*{R_{\Tqn}(\alpha)}^\invad  \label{RHV-2}
\end{align*} 
 for each \(n\in\NO \) with \(2n\le \kappa\).
\end{rem}

\begin{rem}[{\cite[\crem{2.1}]{MR2570113}}] \label{DFKMT2.1} 
 Let \(n \in \N\) and let \(\seqs{2n}\) be a sequence of complex \tqqa{matrices}.
 In view of \eqref{Nr.B6} and \rlem{DFK1}, one can easily see that \(\seqs{2n}\) belongs to \(\Hggq{2n}\) if and only if the four conditions \(\seqs{2(n-1)} \in \Hggq{2(n-1)}\), \(\ran{y_{n,2n-1}}\subseteq \ran{H_{n-1}}\), \(s_{2n-1}^\ad =s_{2n-1}\), and \(L_n \in \Cggq\) hold true.
 If \(\seqs{2n}\) belongs to \(\Hggq{2n}\), then \(\rank H_n = \sum_{j=1}^n \rank L_j\).
\end{rem}

\begin{rem} \label{1582012-2}
 Let \(\alpha \in \R\), let \(\kappa \in \Ninf \), and let \(\seqska \in \Kggqka \).
 By virtue of \rrem{ML415}, one can easily check then that \(s_j^\ad =s_j\) for each \(k\in \mn{0}{\kappa}\) and \(s_{\at k}^\ad = s_{\at k}\) for each \(k\in \mn{0}{\kappa-1}\).
 Furthermore, for each \(n\in \NO \) with \(2n\leq \kappa\), from \rrem{DFKMT2.1} one can see that the matrices \(s_{2n}\), \(H_n\), and \(L_n\) are \tnnH{} and, for each \(n\in \NO \) with \(2n+1\leq \kappa\), the matrices \(s_{\at 2n}\), \(H_{\at n}\), and \(L_{\at n}\) are \tnnH{} as well.
\end{rem}

\begin{rem} \label{ZS}
 Let \(\alpha \in \R\) and let \(\kappa \in \NOinf \).
 According to the definition of \(\Kggeqka \) and~\cite[\clem{4.7}]{MR2735313}, one can easily check that \(\Kggeqka \subseteq \Kggqka \cap \Hggeqka\).
 In particular, if \(\seqska \) belongs to \(\Kggeqka \), then, in view of \rrem{R1440} for each \(m\in \mn{0}{\kappa}\), the sequence \(\seqs{m}\) belongs to \(\Kggeq{m} \cap \Hggeqka\).
 Further, if \(\seqska \in \Kggeqka \), then the definition of the sets \(\Kggeqka \) and \(\Hggeqka \) and~\cite[\cprop{4.8} and \clem{4.11}]{MR2735313} show that, for each \(m\in\mn{0}{\kappa - 1}\), the sequence \(\seqsa{m}\) belongs to \(\Hggeq{m} \).
\end{rem}

\begin{rem} \label{R3.12}
 Let \(\alpha \in \R\), let \(\kappa \in \Ninf \), and let \(\seqska \in \Kggeqka \).
 For each \(m\in \mn{0}{\kappa}\), we have \(\seqs{m} \in \Kggeq{m}\).
 In view of \rremss{1582012-2}{ZS}, from ~\cite[\clemss{4.15}{4.16}]{MR2735313} one can see that 
\[
 \nul{L_0}
 \subseteq \nul{L_{\at{0}}}
 \subseteq \nul{L_1}
 \subseteq\dotsb
 \subseteq \nul{L_n}
 \subseteq \nul{L_{\at{n}}}
\]
 and that
\[
 \ran{L_0}
 \supseteq
 \ran{L_{\at{0}}}
 \supseteq\ran{L_1}
 \supseteq\dotsb
 \supseteq\ran{L_n}
 \supseteq\ran{L_{\at{n}}}
\]
 are valid for each \(n\in \NO \) with \(2n +1 \leq \kappa\) and 
\[
 \nul{L_0}
 \subseteq
 \nul{L_{\at{0}}}
 \subseteq\nul{L_1}
 \subseteq\dotsb
 \subseteq\nul{L_{\at{n-1}}}
 \subseteq\nul{L_n}
\]
 and 
\[
 \ran{L_0}
 \supseteq\ran{L_{\at{0}}}
 \supseteq\ran{L_1}
 \supseteq\dotsb
 \supseteq\ran{L_{\at{n-1}}}
 \supseteq\ran{L_n}
\]
 hold true for each \(n\in \N\) with \(2n \leq\kappa\).
\end{rem}

\section{Dubovoj  subspaces}\label{S1710} 
 If \(\cU\) and \(\cW\) are subspaces of \(\Cq \), then we write \(\cU +\cW\) for the sum of \(\cU\) and \(\cW\).
 To indicate that the sum \(\cU + \cW\) is a direct sum, \tie{}, that \(\cU\cap\cW = \set{\Ouu{q}{1}}\) is fulfilled, we use the notation \(\cU\msum \cW\).
 V.~K.~Dubovoj  studied in~\cite{Dub} particular invariant subspaces to discuss the matricial Schur problem.
 Having in mind this, we give the following definition:
 
\begin{defn}\label{D1039}
 We call a subspace \(\cD\) of \(\Cp\) a \emph{Dubovoj  subspace} corresponding to a given ordered pair \((H,T)\) of complex \tppa{matrices} if \(T^\ad (\cD)\subseteq \cD\) and \(\nul{H} \msum \cD =\Cp\) are fulfilled.
\end{defn}

 We are going to consider special Dubovoj subspaces.
 
\begin{nota} \label{N1543}
 Let \(\kappa \in \NOinf \) and let \(\seqska \) be a sequence of complex \tpqa{matrices}.
 For each \(n\in \NO \) with \(2n \leq \kappa\), let \(\cD_n\defeq \ran{\diag (L_0,\ L_1,\dotsc,L_n)}\).
 Furthermore, if \(\kappa \geq 1\), then, for every choice of \(\alpha \in \R\) and \(n\in \NO \) with \(2n+1 \leq \kappa\), let \(\cD_{\at{n}}\defeq \ran{\diag (L_{\at{0}},\ L_{\at{1}},\dotsc,L_{\at{n}} )}\).
\end{nota}

 Using the Kronecker delta, we set 
\begin{align*}
 V_{q,n}&\defeq\mat{\Kronu{j,k}\Iq}_{\substack{j = 0,\dotsc,n\\ k = 0,\dotsc,n-1}}&
&\text{and}&
 \sV_{q,n}&\defeq\mat{\Kronu{j,k+1}\Iq}_{\substack{j = 0,\dotsc,n\\ k = 0,\dotsc,n-1}}.
\end{align*}

\begin{lem} \label{R1103}
 Let \(\kappa \in \Ninf \), and let \(\seqska \in \Hggeqka \).
 For each \(n\in \N\) with \(2n \leq \kappa\), then \( \Tqn^\ad (\cD_{n})\subseteq\cD_n\), \(\sV_{q,n}^\ad (\cD_{n})\subseteq\cD_{n-1}\), and \(V_{q,n}^\ad (\cD_{n})\subseteq\cD_{n-1}\).
\end{lem}
\begin{proof}
 Because of \(T_{q,n}^\ad\cdot \diag (L_0, L_1,\dotsc, L_n) = \smat{\Ouu{nq}{q} & \diag (L_1,L_2,\dotsc, L_n)\\ \Oqq  & \Ouu{q}{nq}} \), we have \(T_{q,n}^\ad (\cD_n) \subseteq \ran{\diag (L_1, L_2,\dotsc, L_n, \Oqq )}\).
 From \rrem{DFKMT2.1} we see that \(\set{L_0,L_1,\dotsc, L_n}\subseteq\Cggq \).
 Thus, using \rrem{A.1.} and~\cite[\cprop{2.13}]{MR2570113}, we get
\begin{align} \label{DM2}
 \ran{L_j}&=\ek*{\nul{L_j}}^\orth \subseteq\ek*{\nul{L_{j-1}}}^\orth = \ran{L_{j-1}}&\text{for each }j&\in\mn{1}{n},
\end{align}
 which implies \(\ran{\diag (L_1, L_2,\dotsc, L_n, \Oqq )}\subseteq \ran{\diag (L_0, L_1,\dotsc, L_{n-1}, \Oqq )}\subseteq \cD_n\).
 Consequently, \( \Tqn^\ad (\cD_{n})\subseteq\cD_n\).
 Obviously, we have \(\sV_{q,n}^\ad\cdot \diag (L_0, L_1,\dotsc, L_n) =\mat{ \Ouu{nq}{q}, \diag (L_1, L_2,\dotsc,L_n)}\) and \(V_{q,n}^\ad \cdot \diag (L_0,L_1,\dotsc, L_n) = (\diag (L_0,L_1,\dotsc, L_{n-1}), \Ouu{nq}{q})\).
 Because of \eqref{DM2}, consequently, \(\sV_{q,n}^\ad (\cD_n) \subseteq\ran{\diag (L_1,L_2,\dotsc, L_n)}\subseteq \ran{\diag (L_0,L_1,\dotsc, L_{n-1})} = \cD_{n-1}\) and \(V_{q,n}^\ad (\cD_n)\subseteq \cD_{n-1}\).
\end{proof}

 If \(n\in \NO \) and if \(\seqs{2n} \in \Hggeq{2n}\), then the existence of a Dubovoj  subspace corresponding to \((H_n, \Tqn)\) was proved in~\cite[\clem{3.2}]{MR1395706},~\cite[\cSatz{1.24}]{Thi06}, and~\zita{Dyu01}.
 An explicit construction of such a subspace gives the following result: 

\begin{prop} \label{T1.24-1}
 Let \(\kappa \in \NOinf \) and let \(\seqska \in \Hggeqka\).
 For each \(n\in \NO \) with \(2n \leq \kappa\), then \(\cD_n\) is a Dubovoj  subspace for \((H_n, \Tqn)\), where in particular \(\dim \cD_n=\rank H_n\) and \(\dim \cD_n =\sum_{j=0}^{n} \rank L_j\).
\end{prop}
\begin{proof}
 Let \(n\in\NO \) be such that \(2n\le \kappa\).
 Then \(\seqs{2n}\in\Hggeq{2n}\).
 Furthermore, \rlem{R1103} shows that \(T_{q,n}^\ad (\cD_n)\subseteq \cD_n\).
 Now we check that 
\beql{N136N}
 \dim \cD_k
 =\dim \ran{H_k}
\eeq
 holds true for each \(k\in\mn{0}{n}\).
 Because of \(L_0 = s_0 = H_0\), equation \eqref{N136N} is valid for \(k=0\).
 Thus, there is an \(m\in\mn{0}{n}\) such \eqref{N136N} is fulfilled for each \(k\in\mn{0}{m}\).
 We consider the case that \(2(m+1)\le \kappa\).
 Then from \rnota{N1543} and \eqref{N136N} we obtain
\beql{IBG}
 \dim \cD_{m+1}
 =\dim \cD_m + \dim \cR (L_{m+1})
 =\dim\cR (H_m) + \dim \cR (L_{m+1}).
\eeq
 Since we know from \rrem{DFKMT2.1} that the right-hand side of \eqref{IBG} coincides with   \(\dim \cR (H_{m+1})\), we see that \eqref{N136N} is true for \(k=m+1\) as well.
 Consequently, \eqref{N136N} holds for each \(k\in\mn{0}{n}\).
 This implies \(\dim\cD_n + \dim\nul{H_n} =\dim\Co{(n+1)q}\).
 Furthermore, \eqref{N136N} and \rrem{DFKMT2.1} show that \(\dim \cD_n = \sum_{j=0}^n \rank L_j\) holds true.
 It remains to prove that \(\cD_n\cap\nul{H_n}\subseteq \set{\Ouu{(n+1)q}{1}}\).
 We consider an arbitrary \(x\in\cD_n\cap\nul{H_n}\).
 Let \(x = \col (x_j)_{j=0}^n\) be the \taaa{q}{1}{block} representation of \(x\).
 Because of \(x\in\nul{H_n}\), from~\cite[\clem{A.2}]{MR2570113} we see that \(x_n\) belongs to \(\cN (L_n)\).
 Since we know from \rrem{DFKMT2.1} that \(L_n\) is \tnnH{}, we conclude \(x_n\in\cR (L_n)^\orth\).
 On the other hand, we have \(x\in\cD_n\), which implies \(\col (x_j)_{j=0}^n\in\ran{\diag (L_0, L_1,\dotsc, L_n)}\) and, consequently, \(x_n\in\cR (L_n)\).
 Thus, \(x_n\in\cR (L_n)\cap\cR (L_n)^\orth = \set{ \Ouu{q}{1}}\), \tie{}, \(x_n = \Ouu{q}{1}\).
 Inductively, then \(x_{n-j} = \Ouu{q}{1}\) follows for each \(j\in\mn{0}{n}\).
 Therefore, \(\cD_n\cap \nul{H_n}\subseteq \set{ \Ouu{(n+1) q}{1}}\). 
\end{proof}

 For each \(n \in \NO \) and each \(\seqs{2n}\in \Hggeq{2n}\), we will call \(\cD_n\) defined in \rnota{N1543} the \emph{canonical Dubovoj  subspace corresponding to \((H_n, T_{q,n})\)}.
 
 In~\cite[Abschnitt~1.4]{Thi06}, H.~C.~Thiele showed that \((s_j)_{j=0}^2\) given by \(s_0 \defeq 0\), \(s_1 \defeq 0\), and \(s_2 \defeq 1\) is a sequence belonging to \(\Hgguu{1}{2}\) for which no Dubovoj  subspace corresponding to \((H_1, T_{1,1})\) exists.

\begin{rem}\label{r5.4}
 Let \(\kappa \in \NOinf \), let \(\seqska  \in \Hggeqka \), and let \(n \in \NO \) be such that \(2n\leq \kappa\).
 Let \(\cD_n\) be the canonical Dubovoj  subspace corresponding to \((H_n, \Tqn)\).
 In view of \rprop{T1.24-1}, one can easily see that \(\dim \cD_n \geq 1\) if and only if \(s_0 \neq \Oqq \).
 Furthermore, it is readily checked that \(\dim \cD_n <(q+1)n\) if and only if \(\det H_n=0\).
\end{rem} 

\begin{rem} \label{7.7-1}
 Let \(\alpha \in \R\), let \(\kappa \in \NOinf \), and let \(\seqska \in \Kggeqka \).
 From \rrem{ZS} and \rprop{T1.24-1} one can see then that the following statements hold true:
\begin{enui}
 \item\label{7.7-1.a} For each \(n \in \NO \) with \(2n\leq \kappa\), the subspace \(\cD_n\) of \(\Co{(n+1)q}\) is a Dubovoj  subspace corresponding to \((H_n,\Tqn)\).
 \item\label{7.7-1.b} If \(\kappa \geq 1\), then for each \(n \in \NO \) with \(2n+1\leq \kappa\), the subspace \(\cD_{\at{n}}\) of \(\Co{(n+1)q}\) is a Dubovoj  subspace corresponding to \((H_{\at{n}},\Tqn)\).
\end{enui}
\end{rem}

\begin{prop} \label{DU}
 Let \(\alpha \in \R\), let \(\kappa \in \Ninf \), and let \(\seqska \in \Kggeqka\).
 Then:
\begin{enui}
 \item\label{DU.a} For each \(n\in \NO \) with \(2n+1 \leq \kappa\), then \(\Tqn^\ad (\cD_n)\subseteq \cD_{\at{n}} \subseteq \cD_n\),
\begin{align} \label{N59-111}
 \nul{H_n} \msum \cD_n&=\Co{(n+1)q},&
&\text{and}&
 \cN(H_{\at{n}}) \msum \cD_{\at{n}}&=\Co{(n+1)q}.
\end{align}
 \item\label{DU.b} For each \(n\in \N\) with \(2n \leq \kappa\), furthermore \(\sV_{q,n}^\ad (\cD_n) \subseteq \cD_{\at{n-1}}\), \(V_{q,n} (\cD_{\at{n-1}}) \subseteq \cD_n\), 
\begin{align}\label{N60}
 \nul{H_n} \msum \cD_n&=\Co{(n+1)q},&
&\text{and}&
 \cN(H_{\at{n-1}}) \msum \cD_{\at{n-1}}&=\Co{nq}.
\end{align}
\end{enui}
\end{prop}
\begin{proof}
 According to \rrem{ZS}, we have \(\seqs{m} \in \Kggeq{m}\) for each \(m \in \mn{0}{\kappa}\).
 Consequently, \rrem{R3.12} yields
\begin{align} 
 \cR(L_{j+1})&\subseteq \cR(L_{\at{j}})&\text{for each }j&\in \NO\text{ with }2j+2\leq \kappa\label{RL2}
\intertext{and}
 \cR(L_{\at{j}})&\subseteq \cR(L_{j})&\text{for each }j&\in \NO\text{ with }2j+1\leq \kappa.\label{RL3} 
\end{align}

 \eqref{DU.a} Let \(n\in \NO \) be such that \(2n+1 \leq \kappa\).
 Because of \rrem{7.7-1} and the definition of a Dubovoj  subspace, we get \eqref{N59-111}.
 In view of \eqref{RL3}, we have \( \cD_{\at{n}}= \ran{\diag(L_{\at{j}})_{j=0}^n}\subseteq \ran{\diag\rk{L_j}_{j=0}^n})=\cD_n\).
 If \(n=0\), then \(\Tqn=\Oqq\) and, consequently, \(\Tqn^\ad (\cD_n) \subseteq \cD_{\at{n}}\).
 
 Now we assume that \(n \geq 1\).
 In view of \eqref{RL2}, then it is readily checked that 
\begin{align*} 
 \Tqn^\ad (\cD_n)
 =\Tqn^\ad\ek*{\Ran{\diag\rk{L_j}_{j=0}^n}}
 &\subseteq\Ran{\diag\ek*{ \diag(L_{j+1})_{j=0}^{n-1}, \Oqq}} \\
 &\subseteq \Ran{\diag(L_{\at{j}})_{j=0}^{n}} 
 =\cD_{\at{n}}.
 \end{align*}

 \eqref{DU.b} Let \(\kappa \geq 2\) and let \(n\in \N\) such that \(2n \leq \kappa\).
 Because of \rrem{7.7-1} and the definition of a Dubovoj  subspace, we get \eqref{N60}.
 From \( \sV_{q,n}^\ad \cdot \diag(L_{j})_{j=0}^{n} =\mat{ \Ouu{nq}{q},\diag(L_{j+1})_{j=0}^{n-1}}\) we conclude \(\sV_{q,n}^\ad (\cD_n)\subseteq \ran{ \diag(L_{j+1})_{j=0}^{n-1}}\).
 Using \eqref{RL2}, we obtain \(\ran{\diag(L_{j+1})_{j=0}^{n-1}}\subseteq \ran{\diag(L_{\at{j}})_{j=0}^{n-1}}=\cD_{\at{n-1}}\) and, consequently, \(\sV_{q,n}^\ad (\cD_n)\subseteq \cD_{\at{n-1}}\).
 Obviously, \(V_{q,n}\cdot \diag (L_{\alpha\triangleright j})_{j=0}^{n-1} =\diag \ek{ \diag(L_{\at{j}})_{j=0}^{n-1}, \Oqq }\cdot V_{q,n} \) and, hence, \(V_{q,n} (\cD_{\at{n-1}}) = \ran{\diag\ek{\diag(L_{\at{j}})_{j=0}^{n-1},\Oqq}}\).
 Since \eqref{RL3} implies \(\ran{\diag\ek{\diag(L_{\at{j}})_{j=0}^{n-1},\Oqq}}\subseteq \ran{\diag(L_{j})_{j=0}^{n}}=\cD_n\), we get \(V_{q,n} (\cD_{\at{n-1}}) \subseteq \cD_n\).
\end{proof}

\begin{rem}\label{8B}
 If \(A\in\Cpq \) and if \(\cU\) and \(\cV\) are subspaces of \(\Cq \) and \(\Cp \), respectively, such that \(\nul{A}\msum \cU =\Cq \) and \(\ran{A}\msum\cV =\Cp \) are fulfilled, then there is a unique \(X\in\Cqp \) such that
\begin{align*}
 AXA&= A,&
 XAX&=X,&
 \cR(X)&= \cU,&
 &\text{and}&
 \cN (X)&=\cV
\end{align*}
 (see, \teg{}~\cite[Chapter~2, \cthm{12(c)}]{MR587113}), and we will use \(A_{\cU,\cV}^{(1,2)}\) to denote this matrix \(X\).
 In particular, if \(A\) is a \tH{} complex \tqqa{matrix} and if \(\cU\) is a subspace of \(\Cq \) with \(\nul{A} \msum \cU =\Cq \), then \(\ran{A}\msum \cU^\orth =\Cq \) and we will also write \(A_\cU^\gi \) for \(A_{\cU,\cU^\orth}^{(1,2)}\).
 (In \rappe{A1556}, we turn our attention to the \tH{} case, in which special equations hold true.)
\end{rem}

 If \(\kappa \in \NOinf \) and a sequence \(\seqska  \in \Hggeqka \) is given, then, for each \(n \in \NO \) with \(2n\leq \kappa\), let \(H_n^\gi  \defeq H_{\cD_n, {\cD_n^\orth}}^{(1,2)}\), where \(\cD_n\) is given by \rnota{N1543}.
 (Note that \rrem{ZS} shows that \(\Kggeqka \subseteq \Hggeqka\) holds true for each \(\alpha \in \R\) and each \(\kappa \in \NOinf \).)
 If \(\alpha \in \R\), \(\kappa \in \Ninf \), and \(\seqska \in \Kggeqka \) are given, then, for each \(n\in \NO \) with \(2n+1\leq \kappa\), let \(H_{\at{n}}^\gi  \defeq H_{\cD_{\at n}, \cD_{\at n}^\orth}^{(1,2)}\), where \(\cD_{\at n}\) is also given by \rnota{N1543}.

\begin{rem} \label{310-1}
 Let \(\kappa \in \NOinf \) and let \(\seqska  \in \Hggeqka \).
 In view of \rlem{A.40}, for each \(n \in \NO \) with \(2n\leq \kappa\), then it is readily checked that \(H_n^\gi  \in \Cggo{(n+1)q}\),
\begin{align} \label{N73}
 H_n H_n^\gi  H_n&=H_n,&
&\text{and}&
 H_n^\gi  H_n H_n^\gi&=H_n^\gi .
\end{align} 
\end{rem}

\begin{lem} \label{311-1}
 Let \(\alpha \in \R\), let \(\kappa \in \Ninf \), and let \(\seqska  \in\Kggeqka \), and let \(n \in \NO \) be such that \(2n +1 \leq \kappa\).  Then the matrices \(H_n^\gi\) and \(H_{\at n}^\gi\) are both \tnnH{} and fulfill
\begin{align} \label{N39-RI}
 (H_n^\gi )^\ad&= \Hu{n}^\gi,&
 (H_{\at n}^\gi )^\ad&= H_{\at n}^\gi .
\end{align} 
 Furthermore, the equations in \eqref{N73} as well as the following four identities hold true: 
\begin{align} 
 H_{\at n} H_{\at n}^\gi  H_{\at n}&=H_{\at n},&
&&
 H_{\at n}^\gi  H_{\at n} H_{\at n}^\gi&=H_{\at n}^\gi,\label{PPI-2}\\
 H_n^\gi  H_n H_{\at{n}}^\gi&=H_{\at{n}}^\gi,&
&\text{and}&
 H_{\at{n}}^\gi  H_n H_n^\gi&=H_{\at{n}}^\gi .\label{Nr.ML6-2}
\end{align}
\end{lem}
\begin{proof}
 The matrices \(H_n\) and \(H_{\at{n}}\) are both \tnnH{}.
 \rlem{A.40} yields then that \(H_n^\gi\) and \(H_{\at n}^\gi\) are both \tnnH{} and that the equations in \eqref{N73}, \eqref{N39-RI}, and \eqref{PPI-2} hold true.
 In order to prove \eqref{Nr.ML6-2}, we consider an arbitrary \(n \in \NO \) such that \(2n+1\leq \kappa\).
 Taking into account \(H_n^\ad =H_n\), \rprop{DU}, \rlem{A.40}, and \rrem{A.44}, we conclude
\[
 \cR(H_{\at{n}}^\gi )
 = \cR(H_{\cD_{\at{n}}}^\gi )
 =\cD_{\at{n}}
 \subseteq \cD_n
 =\cN (\Iu{(n+1)q} -H_{\cD_n}^\gi  H_n)
 = \cN(\Iu{(n+1)q}-H_n^\gi  H_n).
\]
 For every choice of \(x\in \Co{(n+1)q}\), this implies \(\NM = (\Iu{(n+1)q}-H_n^\gi  H_n) H_{\at{n}}^\gi  x\) and, consequently, \(H_n^\gi  H_n H_{\at{n}}^\gi x=H_{\at{n}}^\gi x\).
 Thus, the first equation in \eqref{Nr.ML6-2} is verified.
 Hence \(H_n^\ad = H_n\), \eqref{N39-RI}, and the first equation in \eqref{Nr.ML6-2} yield \(H_{\at{n}}^\gi  H_n H_n^\gi  =(H_n^\gi  H_n H_{\at{n}}^\gi )^\ad =(H_{\at{n}}^\gi )^\ad =H_{\at{n}}^\gi \) as well.
\end{proof}

\begin{rem} \label{SBNNN}
 Let \(\alpha \in \R\), let \(\kappa \in \Ninf \), let \(\seqska \in \Kggeqka \), and let \(n \in \NO \) be such that \(2n+1 \leq \kappa\).
 Then \rremss{ZS}{1582012-2} yield \(H_n^\ad =H_n\) and \(H_{\at n}^\ad =H_{\at n}\).
 Thus,
\begin{align}\label{BL}
 H_n^\mpi  H_n&=H_n H_n^\mpi&
&\text{and}&
 H_{\at n}^\mpi  H_{\at n}&=H_{\at n} H_{\at n}^\mpi. 
\end{align}
 In view of \rlem{311-1}, consequently,
\begin{align*}
 (\Iu{(n+1)q} - H_n^\mpi  H_n)H_n&= \NM,&
 (\Iu{(n+1)q} - H_{\at n}^\mpi  H_{\at n})H_{\at n}&= \NM,
\end{align*}
\begin{multline*} 
 (\Iu{(n+1)q} - H_n H_n^\gi ) (\Iu{(n+1)q} -H_n^\mpi  H_n)
 = (\Iu{(n+1)q} - H_n H_n^\gi ) (\Iu{(n+1)q} - H_n H_n^\mpi )\\
 = \Iu{(n+1)q} -H_n H_n^\mpi  -H_n H_n^\gi  +H_n H_n^\gi  H_n H_n^\mpi
 = \Iu{(n+1)q}-H_n H_n^\gi,
\end{multline*}
 and
\[\begin{split}
 &(\Iu{(n+1)q} -H_{\at n} H_{\at n}^\gi ) (\Iu{(n+1)q} -H_{\at n}^\mpi  H_{\at n})\\
 &= (\Iu{(n+1)q}- H_{\at n} H_{\at n}^\gi ) (\Iu{(n+1)q} - H_{\at n} H_{\at n}^\mpi )\\
 &= \Iu{(n+1)q} - H_{\at n} H_{\at n}^\mpi  -H_{\at n} H_{\at n}^\gi  + H_{\at n} H_{\at n}^\gi H_{\at n} H_{\at n}^\mpi 
 = \Iu{(n+1)q} -H_{\at n} H_{\at n}^\gi .
\end{split}\]
\end{rem}

\begin{lem} \label{L1021-1}
 Let \(\alpha \in \R\), let \(\kappa\in\Ninf \), and let \(\seqska  \in\Kggeqka \).
 Further, let \(n\in\NO \) be such that \(2n+1 \leq \kappa\).
 For each \(k \in \NO \), then \( H_n^\gi  \Tqn^k (\Iu{(n+1)q}-H_nH_n^\gi ) =\NM\), \( H_{\at n}^\gi  \Tqn^k (\Iu{(n+1)q}-H_nH_n^\gi ) = \NM\), and \( H_{\at{n}}^\gi  \Tqn^k (\Iu{(n+1)q}-H_{\at{n}}H_{\at{n}}^\gi ) = \NM\) hold true.
 Furthermore, \(H_n^\gi  \Tqn^k (\Iu{(n+1)q}-H_{\at{n}}H_{\at{n}}^\gi ) =\NM\) for each \(k \in \N\).
\end{lem}
\begin{proof}
 Use \rlem{311-1}, \rprop{DU}, and \rlem{A.46}.
\end{proof}

\begin{lem} \label{ML8-1}
 Let \(\alpha \in \R\), let \(\kappa \in \Ninf \), and let \(\seqska  \in \Kggeqka \).
 For each \(n \in \NO \) with \(2n+1\leq \kappa\), then the following statements hold true:
\begin{enui} 
 \item\label{ML8-1.a} For every choice of \(\zeta \in \C\) and \(k \in \NO \),
\begin{align} 
 H_n^\gi  R_{\Tqn} (\zeta) \Tqn^k (\Iu{(n+1)q}-H_nH_n^\gi )&=0,\notag\\
 (\Iu{(n+1)q}-H_n^\gi  H_n) (\Tqn^\ad )^k \ek*{R_{\Tqn} (\zeta)}^\ad H_n^\gi&=0,\label{N-71}\\
 H_{\at n}^\gi  R_{\Tqn} (\zeta) \Tqn^k (\Iu{(n+1)q}-H_{\at{n}}H_{\at{n}}^\gi )& =0,\label{NUC}
\intertext{and}
 (\Iu{(n+1)q}-H_{\at{n}}^\gi  H_{\at{n}}) (\Tqn^\ad )^k \ek*{R_{\Tqn} (\zeta)}^\ad H_{\at n}^\gi&=0.\label{N-72}
\end{align} 
 \item\label{ML8-1.b} For each \(\zeta \in \C\) and each \(k \in \N\),
\begin{align} 
 H_n^\gi  R_{\Tqn} (\zeta) \Tqn^k (\Iu{(n+1)q}-H_{\at{n}}H_{\at{n}}^\gi )&=0\label{N65MG}
\intertext{and}
 (\Iu{(n+1)q}-H_{\at{n}}^\gi  H_{\at{n}}) (\Tqn^\ad )^k \ek*{R_{\Tqn} (\zeta)}^\ad H_n^\gi&=0.\label{N65LB}
\end{align}
 \end{enui} 
\end{lem} 
\begin{proof}
 Let \(n \in \NO \) be such that \(2n+1\leq \kappa\) and let \(\zeta \in \C\).
 Because of \(\Kggeqka \subseteq \Kggqka \), \rrem{1582012-2}, and \rlem{311-1}, we get
\begin{align} 
 H_n^\ad&=H_n,&(H_n^\gi )^\ad&=H_n^\gi,&
&&
 (H_n^\gi  H_n)^\ad&=H_n^\ad (H_n^\gi )^\ad = H_nH_n^\gi,\label{M77}\\
 H_{\at{n}}^\ad&=H_{\at{n}},&(H_{\at{n}}^\gi )^\ad&=H_{\at{n}}^\gi,&
&\text{and}&
 (H_{\at{n}}^\gi  H_{\at{n}})^\ad&= H_{\at{n}}H_{\at{n}}^\gi.\label{174-1}
\end{align}
 In view of \rrem{21112N} and \rlem{L1021-1}, for each \(k \in \NO \), we obtain
\beql{N65-111}\begin{split}
 H_n^\gi R_{\Tqn} (\zeta) \Tqn^k (\Iu{(n+1)q}-H_nH_n^\gi ) 
 &= H_n^\gi  \rk*{\sum_{j=0}^n \zeta ^j \Tqn ^j}\Tqn^k (\Iu{(n+1)q}-H_nH_n^\gi )\\
 &= \sum_{j=0}^n \zeta ^j H_n^\gi  \Tqn ^{j+k} (\Iu{(n+1)q}-H_nH_n^\gi ) 
 =0
\end{split}\eeq
 and, analogously, \eqref{NUC}.
 Furthermore, the same arguments imply that \eqref{N65MG} holds true for each \(k\in\N\).
 For all \(k\in\NO \), equation~\eqref{N-71} follows from \eqref{N65-111}, \eqref{M77}, and \eqref{174-1}.
 Moreover, for each \(k\in\NO \), equation~\eqref{N-72} is a consequence of \eqref{M77}, \eqref{174-1}, and \eqref{NUC}.
 Using \eqref{N65MG}, \eqref{M77}, and \eqref{174-1}, we see that \eqref{N65LB} holds true for each \(k\in\N\).
\end{proof}

\begin{lem} \label{bem-ML8}
 Let \(\alpha \in \R\), let \(\kappa \in \Ninf \), and let \(\seqska  \in \Kggeqka \).
 For each \(n \in \NO \) with \(2n+1\leq \kappa\) and every choice of \(\zeta \in \C\) and \(\eta \in \C\), then 
\begin{align} 
 H_n^\gi  R_{\Tqn} (\zeta) \ek*{R_{\Tqn^\ad } (\eta)}^\invad (\Iu{(n+1)q}-H_nH_n^\gi )&=0,\notag\\
 (\Iu{(n+1)q}-H_n^\gi  H_n) \ek*{R_{\Tqn^\ad } (\eta)}^\inv \ek{R_{\Tqn} (\zeta)}^\ad  H_n^\gi&=0,\label{N67}\\
 H_{\at n}^\gi  R_{\Tqn} (\zeta) \ek*{R_{\Tqn^\ad } (\eta)}^\invad (\Iu{(n+1)q}-H_{\at{n}}H_{\at{n}}^\gi )&=0,\label{UC}
\intertext{and}
 (\Iu{(n+1)q}-H_{\at{n}}^\gi  H_{\at{n}}) \ek*{R_{\Tqn^\ad } (\eta)}^\inv \ek{R_{\Tqn} (\zeta)}^\ad H_{\at n}^\gi&=0.\label{GG671}
\end{align} 
\end{lem}
\begin{proof}
 Let \(n \in \NO \) be such that \(2n+1\leq \kappa\), and let \(\zeta,\eta \in \C\).
 Because of \rrem{1582012-2} and \rlem{311-1}, we obtain \eqref{M77} and \eqref{174-1}.
 From \rrem{21112N} and \rlemp{ML8-1}{ML8-1.a} we conclude
\beql{MF}\begin{split}
 &H_n^\gi R_{\Tqn} (\zeta) \ek*{R_{\Tqn^\ad } (\eta)}^\invad (\Iu{(n+1)q}-H_nH_n^\gi )\\
 &=H_n^\gi  R_{\Tqn} (\zeta) ( \Iu{(n+1)q} - \ko{\eta} \Tqn) (\Iu{(n+1)q}-H_nH_n^\gi )\\
 &=H_n^\gi  R_{\Tqn} (\zeta) (\Iu{(n+1)q}-H_nH_n^\gi ) - \ko{\eta} H_n^\gi  R_{\Tqn} (\zeta) \Tqn (\Iu{(n+1)q}-H_nH_n^\gi ) 
 = 0
\end{split}\eeq
 and, analogously, \eqref{UC}.
 Obviously, \eqref{MF} and \eqref{M77} imply \eqref{N67}.
 Furthermore, \eqref{M77}, \eqref{174-1}, and \eqref{UC} show that \eqref{GG671} holds true as well.
\end{proof}

\begin{lem} \label{ML9-1b}
 Let \(\alpha \in \R\), let \(\kappa \in \Ninf \), and let \(\seqska  \in \Kggeqka \).
 For each \(n \in \NO \) with \(2n+1\leq \kappa\), then \(H_n^\gi  R_{\Tqn} (\alpha) (v_{q,n} v_{q,n}^\ad H_n H_{\at{n}}^\gi +\Tqn) =H_{\at{n}}^\gi \) and 
\beql{UCL}
 H_{\at n}^\gi  \ek*{ \Iu{(n+1)q} - H_n v_{q,n} v_{q,n}^\ad R_{\Tqn^\ad } (\alpha) H_n^\gi }
 = \Tqn^\ad R_{\Tqn^\ad } (\alpha) H_n^\gi.
\eeq
\end{lem} 
\begin{proof}
 Let \(n \in \NO \) be such that \(2n+1\leq \kappa\).
 Because of \( \Kggeqka  \subseteq \Kggqka \), \rrem{1582012-2}, and \rlem{311-1}, we have \eqref{M77} and \eqref{174-1}.
 \rrem{lemC21-1}, \rlem{311-1}, and \rlemp{ML8-1}{ML8-1.a} yield 
\[ \begin{split}
 &H_n^\gi  R_{\Tqn} (\alpha)(v_{q,n} v_{q,n}^\ad H_n H_{\at{n}}^\gi +\Tqn)\\
 &= H_n^\gi  R_{\Tqn} (\alpha) \ek*{ \rk*{ \ek*{R_{\Tqn} (\alpha)}^\inv H_n -\Tqn H_{\at{n}}} H_{\at{n}}^\gi  +\Tqn }\\
 &=H_n^\gi  H_n H_{\at{n}}^\gi  - H_n^\gi  R_{\Tqn} (\alpha) \Tqn H_{\at{n}} H_{\at{n}}^\gi  + H_n^\gi  R_{\Tqn} (\alpha) \Tqn \\
 &=H_{\at{n}}^\gi  + H_n^\gi  R_{\Tqn} (\alpha) \Tqn (\Iu{(n+1)q}-H_{\at{n}} H_{\at{n}}^\gi )
 = H_{\at{n}}^\gi.
\end{split}\]
 This implies 
\beql{Nr.ML6-79}
 \ek*{ \Iu{(n+1)q} - H_n^\gi  R_{\Tqn} (\alpha) v_{q,n} v_{q,n}^\ad H_n}H_{\at{n}}^\gi
 =H_n^\gi  R_{\Tqn} (\alpha) \Tqn. 
\eeq
 In view of \eqref{M77}, \eqref{174-1}, \(\ek{R_{\Tqn} ( \alpha)}^\ad = R_{\Tqn^\ad } (\alpha)\), and \eqref{Nr.ML6-79}, it follows \eqref{UCL}.
\end{proof}

\section{Discussion of \(\Jimq\)-forms of particular matrix polynomials}\label{S1711}
 In this section, we construct special matrix polynomials, which are useful to describe the solution set of the matricial truncated Stieltjes power moment problem \mprob{\rhl}{m}{\lleq}.
 Particular interest is focused to representations of \(\Jimq\)\nobreakdash-forms, where 
\[
 \Jimq
 \defeq 
 \bMat \Oqq  & -\iu\Iq \\ \iu \Iq & \Oqq  \eMat.
\]
 Obviously, \(\Jimq\) is a \taaa{2q}{2q}{signature} matrix, \tie{}, \(\Jimq^\ad =\Jimq\) and \(\Jimq^2=\Iu{2q}\) hold true.
 We modify Bolotnikov's~\cite{MR1362524} approach, who considered the particular case \(\alpha=0\).
 However, the calculations in the general case \(\alpha \in \R\) are much more complicated.
 
\begin{rem} \label{J-1} 
 For every choice of \(A,B\in \Cqq\), we have \(\tmatp{A}{B}^\ad (-\Jimq) \tmatp{A}{B} = -\iu(B^\ad A-A^\ad B)\).
 In particular, \(\tmatp{A}{\Iq}^\ad (-\Jimq)\tmatp{A}{\Iq} =2 \Im A\).
\end{rem}

\begin{rem} \label{LG1}
 Let \(A \in \CHq\).
 Then the matrices \(B\defeq \tmat{ \Iq & \Oqq\\ A & \Iq} \) and \( C\defeq \tmat{ \Iq & A\\ \Oqq & \Iq} \) fulfill \(B^\ad \Jimq B =\Jimq\), \(C^\ad \Jimq C =\Jimq\), \(B\Jimq B^\ad =\Jimq\), and \(C \Jimq C^\ad =\Jimq\).
\end{rem}

\begin{rem} \label{JN}
 For each \(n\in \NO \) and each \(A \in \Coo{(n+1)q}{(n+1)q}\), we have
\begin{align} 
 \mat{\Iu{(n+1)q}, A} (\Iu{2} \otimes v_{q,n} ) \Jimq
 &= \iu\mat{A, -\Iu{(n+1)q} } (\Iu{2} \otimes v_{q,n} ), \label{SKP1} \\ 
 \mat{A, -\Iu{(n+1)q}} (\Iu{2} \otimes v_{q,n} ) \Jimq
 &= - \iu \mat{\Iu{(n+1)q}, A }(\Iu{2} \otimes v_{q,n} ),\notag\\
 \Jimq (\Iu{2} \otimes v_{q,n} )^\ad\mat{\Iu{(n+1)q}, A}^\ad
 &= -\iu (\Iu{2} \otimes v_{q,n} )^\ad\mat{A, -\Iu{(n+1)q}}^\ad,\label{SKP2} \\
 \Jimq (\Iu{2} \otimes v_{q,n} )^\ad\mat{A, -\Iu{(n+1)q}}^\ad
 &= \iu(\Iu{2} \otimes v_{q,n} )^\ad\mat{ \Iu{(n+1)q}, A}^\ad,\notag\\ 
 \mat{\Iu{(n+1)q}, A}(\Iu{2} \otimes v_{q,n} ) \Jimq (\Iu{2} \otimes v_{q,n} )^\ad\mat{\Iu{(n+1)q}, A}^\ad
 &= \iu(A v_{q,n} v_{q,n}^\ad - v_{q,n} v_{q,n}^\ad A^\ad ), \label{SKP3} \\
 \mat{A, -\Iu{(n+1)q}}(\Iu{2} \otimes v_{q,n} ) \Jimq (\Iu{2} \otimes v_{q,n} )^\ad\mat{A, -\Iu{(n+1)q}}^\ad
 &= \iu(A v_{q,n} v_{q,n}^\ad - v_{q,n} v_{q,n}^\ad A^\ad ),\label{SKP3-2}
\intertext{and}
 \mat{\Iu{(n+1)q}, A} (\Iu{2} \otimes v_{q,n} ) (\Iu{2} \otimes v_{q,n} )^\ad\mat{A, -\Iu{(n+1)q}}^\ad 
 &= -(A v_{q,n} v_{q,n}^\ad - v_{q,n} v_{q,n}^\ad A^\ad ).\label{SKP4} 
\end{align} 
\end{rem}

\begin{rem} \label{MB20} 
 Let \(\alpha \in \R\), let \(\kappa \in \NOinf \), and let \(\seqska \in\Hggeqka \).
 For each \(n\in \NO \) with \(2n \leq \kappa\), \rrem{21112N} shows that \(U_{n,\alpha}\colon\C \to\Coo{2q}{2q}\) defined by
\begin{multline} \label{U1}
 U_{n,\alpha} (\zeta)
 \defeq\Iu{2q} + (\zeta-\alpha) (\Iu{2} \otimes v_{q,n} )^\ad \mat{ \Tqn H_n, -\Iu{(n+1)q}}^\ad\\
 \times R_{\Tqn ^\ad } (\zeta) H_n^\gi  R_{\Tqn} (\alpha) \mat{ \Iu{(n+1)q}, \Tqn H_n} (\Iu{2} \otimes v_{q,n} )
\end{multline}
 is a matrix polynomial of degree not greater than \(n+1\), where \(H_n^\ad = H_n\) implies that, for each \(\zeta \in \C\), the matrix \(U_{n,\alpha} (\zeta)\) admits the block representation
\beql{U3}
 U_{n,\alpha} (\zeta)
 =
 \bMat
  A_n(\zeta)&B_n(\zeta)\\
  C_n(\zeta)&D_n(\zeta)
 \eMat
\eeq
 with
\begin{align}
 A_n(\zeta)&\defeq\Iq + (\zeta-\alpha) v_{q,n} ^\ad H_n \Tqn^\ad R_{\Tqn ^\ad } (\zeta) H_n^\gi  R_{\Tqn} (\alpha) v_{q,n}, \label{U3A}\\
 B_n(\zeta)&\defeq\phantom{\Iq}+(\zeta-\alpha) v_{q,n} ^\ad H_n \Tqn^\ad R_{\Tqn ^\ad } (\zeta) H_n^\gi  R_{\Tqn} (\alpha) \Tqn H_n v_{q,n},\label{U3B}\\
 C_n(\zeta)&\defeq\phantom{\Iq}-(\zeta-\alpha) v_{q,n} ^\ad R_{\Tqn ^\ad } (\zeta) H_n^\gi  R_{\Tqn} (\alpha) v_{q,n},\label{U3C}\\
 D_n(\zeta)&\defeq\Iq - (\zeta-\alpha) v_{q,n} ^\ad R_{\Tqn ^\ad } (\zeta) H_n^\gi  R_{\Tqn} (\alpha)\Tqn H_n v_{q,n}.\label{U3D}
\end{align}
\end{rem}

\begin{lem} \label{ML28-1}
 Let \(\alpha \in \R\), let \(\kappa \in \NOinf \), and let \(\seqska \in \Kggeqka \).
 For all of \(n\in \NO \) with \(2n \leq \kappa\) and all \(z, w\in \C \), the function \(U_{n,\alpha}\colon\C \to \Coo{2q}{2q}\) given by \eqref{U1} fulfills
\begin{multline*} 
 \Jimq -U_{n,\alpha} (z) \Jimq U_{n,\alpha}^\ad (w)
 = -\iu (z-\ko{w}) (\Iu{2} \otimes v_{q,n} )^\ad \mat{ \Tqn H_n, -\Iu{(n+1)q}}^\ad R_{\Tqn ^\ad } (z) H_n^\gi\\
 \times\ek*{R_{\Tqn^\ad } (w)}^\ad \mat{ \Tqn H_n, -\Iu{(n+1)q}} (\Iu{2} \otimes v_{q,n} ).
\end{multline*}
\end{lem}
\begin{proof}
 Let \(n\in \NO \) be such that \(2n \leq \kappa\) and let \(z\) and \(w\) be arbitrary complex numbers.
 \rrem{1582012-2} yields \(H_n^\ad =H_n\).
 \rlem{311-1} provides us \eqref{N39-RI} and \eqref{N73}.
 Using \eqref{U1}, \(\Jimq^2 =\Iu{2q}\), and \eqref{N39-RI}, we conclude 
\beql{CHI-2}\begin{split}
 &\Jimq - U_{n,\alpha} (z) \Jimq U_{n,\alpha}^\ad (w)\\
 &= \Jimq-\biggl\{\Iu{2q} + (z-\alpha) (\Iu{2} \otimes v_{q,n} )^\ad\mat{ \Tqn H_n, -\Iu{(n+1)q}}^\ad R_{\Tqn ^\ad } (z)H_n^\gi\\
 & \qquad\times   R_{\Tqn} (\alpha) \mat{ \Iu{(n+1)q}, \Tqn H_n} (\Iu{2} \otimes v_{q,n} ) \biggr\} \Jimq \biggl\{\Iu{2q} + (\ko{w}-\alpha) (\Iu{2} \otimes v_{q,n} )^\ad\\
 & \qquad\qquad\times \mat{ \Iu{(n+1)q}, \Tqn H_n}^\ad \ek*{R_{\Tqn} (\alpha)}^\ad H_n^\gi  \ek*{R_{\Tqn ^\ad } (w)}^\ad 
 \mat{ \Tqn H_n, -\Iu{(n+1)q}} (\Iu{2} \otimes v_{q,n} )\biggr\}\\ 
 & = S_1(z)+S_2(w)+S_3(z,w)
\end{split}\eeq
 where 
\begin{align}
 S_1(z)
 &\defeq -(z-\alpha) (\Iu{2} \otimes v_{q,n} )^\ad \mat{ \Tqn H_n, -\Iu{(n+1)q}}^\ad R_{\Tqn ^\ad } (z) H_n^\gi\notag  \\
 &\qquad\times R_{\Tqn} (\alpha)\mat{ \Iu{(n+1)q}, \Tqn H_n } (\Iu{2} \otimes v_{q,n}) \Jimq,\label{FIP} \\
 S_2(w)
 &\defeq -(\ko{w}-\alpha) \Jimq (\Iu{2} \otimes v_{q,n} )^\ad \mat{ \Iu{(n+1)q},\Tqn H_n }^\ad \ek*{R_{\Tqn} (\alpha)}^\ad H_n^\gi\notag   \\
 &\qquad\times\ek*{R_{\Tqn^\ad } (w)}^\ad \mat{ \Tqn H_n, -\Iu{(n+1)q}} (\Iu{2} \otimes v_{q,n} ),\label{FIPWW}
\end{align}
 and
\begin{multline} \label{FIDIM}
 S_3(z,w)
 \defeq -(z-\alpha) (\ko{w}-\alpha) (\Iu{2} \otimes v_{q,n} )^\ad \mat{ \Tqn H_n, -\Iu{(n+1)q}}^\ad R_{\Tqn ^\ad } (z) H_n^\gi  
 R_{\Tqn} (\alpha)\\
 \times\mat{ \Iu{(n+1)q},\Tqn H_n} (\Iu{2} \otimes v_{q,n} ) \Jimq (\Iu{2} \otimes v_{q,n} )^\ad 
 \mat{ \Iu{(n+1)q}, \Tqn H_n} ^\ad \\
 \times \ek*{R_{\Tqn} (\alpha)}^\ad H_n^\gi  \ek*{R_{\Tqn^\ad } (w)}^\ad \mat{ \Tqn H_n, -\Iu{(n+1)q}} (\Iu{2} \otimes v_{q,n} ). 
\end{multline}
 Because of \eqref{FIP}, \eqref{FIPWW}, \eqref{SKP1}, \eqref{SKP2}, and \rrem{21112N}, we get then
\begin{align}
 S_1(z) 
 & = -\iu (z-\alpha) (\Iu{2} \otimes v_{q,n} )^\ad \mat{ \Tqn H_n, -\Iu{(n+1)q}}^\ad R_{\Tqn ^\ad } (z)\notag\\
 & \qquad\times H_n^\gi  R_{\Tqn} (\alpha) \ek*{R_{\Tqn^\ad } (w)}^\invad \ek*{R_{\Tqn^\ad } (w)}^\ad \mat{ \Tqn H_n, -\Iu{(n+1)q}} (\Iu{2} \otimes v_{q,n}), \label{CHI-3}\\
 S_2(w)
 & = \iu (\ko{w}-\alpha) (\Iu{2} \otimes v_{q,n} )^\ad \mat{ \Tqn H_n,-\Iu{(n+1)q}}^\ad R_{\Tqn^\ad } (z)\ek*{R_{\Tqn^\ad } (z)}^\inv\ek*{R_{\Tqn} (\alpha)}^\ad\notag\\
 & \qquad\times  H_n^\gi   \ek*{R_{\Tqn^\ad } (w)}^\ad \mat{ \Tqn H_n, -\Iu{(n+1)q}} (\Iu{2} \otimes v_{q,n} ),\label{CH1}
\end{align}
 and, according to \eqref{FIDIM}, \eqref{SKP3}, and \(H_n^\ad =H_n\), we have
\begin{multline} \label{CH2}
 S_3(z,w)
 = -\iu (z-\alpha) (\ko{w}-\alpha) (\Iu{2} \otimes v_{q,n} )^\ad \mat{ \Tqn H_n, -\Iu{(n+1)q}}^\ad R_{\Tqn ^\ad } (z) H_n^\gi \\
 \times R_{\Tqn} (\alpha)\rk{\Tqn H_n v_{q,n}v_{q,n}^\ad - v_{q,n}v_{q,n}^\ad H_n \Tqn^\ad}\ek*{R_{\Tqn} (\alpha)}^\ad \\
 \times H_n^\gi  \ek*{R_{\Tqn^\ad } (w)}^\ad \mat{ \Tqn H_n, -\Iu{(n+1)q}} (\Iu{2} \otimes v_{q,n} ).
\end{multline}
 In view of \(H_n^\ad =H_n\), \(R_{\Tqn^\ad } (\alpha) = \ek{R_{\Tqn} (\alpha)}^\ad \), \eqref{B204}, \eqref{S1-2}, \eqref{S1-1}, and \(R_{\Tqn^\ad } (\alpha)= \ek{R_{\Tqn} (\alpha)}^\ad \), it follows
\[\begin{split} %
 &(z-\alpha) (\ko{w}-\alpha) R_{\Tqn} (\alpha)\rk{\Tqn H_n v_{q,n}v_{q,n}^\ad - v_{q,n}v_{q,n}^\ad H_n \Tqn^\ad}\ek*{R_{\Tqn} (\alpha)}^\ad \\
 &=(z-\alpha) (\ko{w}-\alpha) R_{\Tqn} (\alpha)\rk*{\Tqn H_n \ek{R_{\Tqn}(\alpha)}^\invad - \ek*{R_{\Tqn}(\alpha )}^\inv H_n\Tqn^\ad}\ek*{R_{\Tqn} (\alpha)}^\ad \\
 &= (z-\alpha) (\ko{w}-\alpha) R_{\Tqn} (\alpha) \Tqn H_n - (z-\alpha) (\ko{w}-\alpha)H_n\Tqn^\ad \ek*{R_{\Tqn} (\alpha)}^\ad \\
 &=- (z-\alpha)\rk*{R_{\Tqn} (\alpha)\ek*{R_{\Tqn^\ad }(w)}^\invad -\Iu{(n+1)q}} H_n\\
 &\qquad+ (\ko{w}-\alpha)H_n\rk*{\ek*{R_{\Tqn^\ad } (z)}^\inv \ek*{R_{\Tqn} (\alpha)}^\ad - \Iu{(n+1)q}} \\ 
 &= -(z-\alpha)R_{\Tqn} (\alpha)\ek*{R_{\Tqn^\ad }(w)}^\invad H_n + (\ko{w}-\alpha)H_n\ek*{R_{\Tqn^\ad } (z)}^\inv \ek*{R_{\Tqn} (\alpha)}^\ad+( z -\ko{w}) H_n.
\end{split}\]
 Consequently, from \eqref{CH2} we get then
\begin{multline} \label{CH2-2.5}
 S_3(z,w) 
 = -\iu (\Iu{2} \otimes v_{q,n} )^\ad \mat{\Tqn H_n, -\Iu{(n+1)q}}^\ad R_{\Tqn ^\ad } (z) H_n^\gi  \\
 \times \biggl \{ -(z-\alpha)R_{\Tqn} (\alpha)\ek*{R_{\Tqn^\ad }(w)}^\invad H_n +(\ko{w}-\alpha)H_n\ek*{R_{\Tqn^\ad } (z)}^\inv \ek{R_{\Tqn} (\alpha)}^\ad \\ 
 +(z- \ko{w}) H_n \biggr \} H_n^\gi  \ek*{R_{\Tqn^\ad } (w)}^\ad \mat{\Tqn H_n, -\Iu{(n+1)q}} (\Iu{2} \otimes v_{q,n} ).
\end{multline}
 The combination of \eqref{CHI-2}, \eqref{CHI-3}, \eqref{CH1}, and \eqref{CH2-2.5} yields 
\begin{multline} \label{CSNY1}
 \Jimq - U_{n,\alpha} (z) \Jimq U_{n,\alpha}^\ad (w) 
 = -\iu (\Iu{2} \otimes v_{q,n} )^\ad \mat{\Tqn H_n, -\Iu{(n+1)q}}^\ad R_{\Tqn ^\ad } (z) S(z,w)\\
 \times\ek*{R_{\Tqn^\ad } (w)}^\ad \mat{\Tqn H_n, -\Iu{(n+1)q}} (\Iu{2} \otimes v_{q,n} )
\end{multline}
 where
\begin{multline} \label{SW1}
 S(z,w)
 \defeq (z-\alpha) H_n^\gi  R_{\Tqn} (\alpha)\ek*{R_{\Tqn^\ad } (w)}^\invad - (\ko{w}-\alpha) \ek*{R_{\Tqn^\ad } (z)}^\inv \ek*{R_{\Tqn} (\alpha)}^\ad H_n^\gi\\
 -(z-\alpha)H_n^\gi  R_{\Tqn} (\alpha)\ek*{R_{\Tqn^\ad }(w)}^\invad H_n H_n^\gi+(\ko{w}-\alpha) H_n^\gi  H_n\ek*{R_{\Tqn^\ad } (z)}^\inv \ek*{R_{\Tqn} (\alpha)}^\ad H_n^\gi\\
 +(z- \ko{w})H_n^\gi H_n H_n^\gi. 
\end{multline}
 Using \eqref{SW1}, \rlem{bem-ML8}, and \rrem {310-1}, we infer \(S(z,w)=(z- \ko{w})H_n^\gi \).
 Hence, because of \eqref{CSNY1}, the proof is complete.
\end{proof}
 
\begin{rem} \label{86-1} 
 Let \(\alpha \in \R\), let \(\kappa \in \Ninf \), and let \(\seqska \in \Kggeqka \).
 For each \(n\in \NO \) with \(2n+1 \leq \kappa\), \rrem{21112N} shows then that \(\tilde{U}_{n,\alpha}\colon\C \to\Coo{2q}{2q}\) given by
\begin{multline} \label{U2}
 \tilde{U}_{n,\alpha} (\zeta)
 \defeq \Iu{2q} + (\zeta-\alpha) (\Iu{2} \otimes v_{q,n} )^\ad \mat*{\ek*{R_{\Tqn}(\alpha)}^\inv H_n, -\Iu{(n+1)q}}^\ad R_{\Tqn ^\ad } (\zeta) H_{\at{n}}^\gi\\
 \times   R_{\Tqn} (\alpha) \mat*{ \Iu{(n+1)q},\ek*{R_{\Tqn}(\alpha)}^\inv H_n} (\Iu{2} \otimes v_{q,n} )
\end{multline}
 is a matrix polynomial of degree not greater that \(n+1\), where \(H_n^\ad =H_n\) shows that, for each \(\zeta \in \C\), the matrix \(\tilde{U}_{n,\alpha} (\zeta) \) admits the block representation 
 \[
 \tilde{U}_{n,\alpha} (\zeta)
 =
 \bMat
  \tilde A_n(\zeta)&\tilde B_n(\zeta)\\
  \tilde C_n(\zeta)&\tilde D_n(\zeta)
 \eMat
\]
 with
\begin{align*}
 \tilde A_n(\zeta)&\defeq\Iq + (\zeta-\alpha) v_{q,n} ^\ad H_n \ek*{R_{\Tqn ^\ad }(\alpha)}^\inv R_{\Tqn ^\ad } (\zeta) H_{\at{n}}^\gi  R_{\Tqn} (\alpha) v_{q,n},\\
 \tilde B_n(\zeta)&\defeq\phantom{\Iq}+(\zeta-\alpha) v_{q,n} ^\ad H_n \ek*{R_{\Tqn ^\ad }(\alpha)}^\inv R_{\Tqn ^\ad } (\zeta) H_{\at{n}}^\gi  H_n v_{q,n},\\
 \tilde C_n(\zeta)&\defeq\phantom{\Iq}-(\zeta-\alpha) v_{q,n} ^\ad R_{\Tqn ^\ad } (\zeta) H_{\at{n}}^\gi  R_{\Tqn} (\alpha) v_{q,n},\\
 \tilde D_n(\zeta)&\defeq\Iq - (\zeta-\alpha) v_{q,n} ^\ad R_{\Tqn ^\ad } (\zeta) H_{\at{n}}^\gi  H_n v_{q,n}.
\end{align*}
\end{rem}
 
\begin{lem} \label{ML28-2}
 Let \(\alpha \in \R\), let \(\kappa \in \Ninf \), and let \(\seqska \in \Kggeqka \).
 Let \(\tilde {U}_{n,\alpha}\colon\C \to \Coo{2q}{2q}\) be given by \eqref{U2}.
 For all  \(n\in \NO \) with \(2n+1 \leq \kappa\) and all \(z, w\in\mathbb C\), then 
\begin{multline} \label{ZB}
 \Jimq -\tilde{U}_{n,\alpha} (z) \Jimq \tilde{U}_{n,\alpha}^\ad (w)
 = -\iu (z-\ko{w}) (\Iu{2} \otimes v_{q,n} )^\ad 
 \mat*{ \ek*{R_{\Tqn}(\alpha)}^\inv H_n, -\Iu{(n+1)q}}^\ad R_{\Tqn ^\ad } (z) H_{\at{n}}^\gi\\
 \times  \ek*{R_{\Tqn^\ad } (w)}^\ad \mat*{\ek*{R_{\Tqn}(\alpha)}^\inv H_n, -\Iu{(n+1)q}} (\Iu{2} \otimes v_{q,n} ).
\end{multline}
\end{lem}
\begin{proof}
 Let \(n\in \NO \) be such that \(2n+1 \leq \kappa\).
 Because of \(\seqska \in \Kggeqka \subseteq \Kggqka \) and \rrem{1582012-2}, we have \(H_n^\ad =H_n\) and \(H_{\at{n}}^\ad =H_{\at{n}}\).
 \rlem{311-1} provides us \eqref{N39-RI} and \eqref{PPI-2}.
 Let \(z, w\in\C\).
 Taking into account \eqref{U2}, \(\Jimq^2 =\Iu{2q}\), and \eqref{N39-RI}, we get then
\beql{CHI-2N}\begin{split}
 &\Jimq - \tilde{U}_{n,\alpha} (z) \Jimq \tilde{U}_{n,\alpha}^\ad (w) 
 = \Jimq -\biggl\{ \Iu{2q} + (z-\alpha) (\Iu{2} \otimes v_{q,n} )^\ad 
 \mat*{ \ek*{R_{\Tqn}(\alpha)}^\inv H_n, -\Iu{(n+1)q} }^\ad \\
 &\qquad\times R_{\Tqn ^\ad } (z) H_{\at{n}}^\gi  R_{\Tqn} (\alpha) 
 \mat*{ \Iu{(n+1)q},\ek*{R_{\Tqn}(\alpha)}^\inv H_n } (\Iu{2} \otimes v_{q,n} ) \biggr\}\Jimq  \\
 & \qquad\times \biggl\{ \Iu{2q} + (\ko{w}-\alpha) (\Iu{2} \otimes v_{q,n} )^\ad \mat*{ \Iu{(n+1)q},\ek*{R_{\Tqn}(\alpha)}^\inv H_n}^\ad \\
 & \qquad\qquad\times \ek*{R_{\Tqn} (\alpha)}^\ad H_{\at{n}}^\gi  \ek*{R_{\Tqn ^\ad } (w)}^\ad \mat*{ \ek*{R_{\Tqn}(\alpha)}^\inv H_n, -\Iu{(n+1)q} } (\Iu{2} \otimes v_{q,n} )\biggr\} \\
 & = \tilde{S}_1(z)+\tilde{S}_2(w)+\tilde{S}_3(z,w)
\end{split}\eeq
 where
\begin{align} 
 \tilde{S}_1(z)
 &\defeq  - (z-\alpha) (\Iu{2} \otimes v_{q,n} )^\ad \mat*{\ek*{R_{\Tqn}(\alpha)}^\inv H_n, -\Iu{(n+1)q}}^\ad 
 R_{\Tqn ^\ad } (z) H_{\at{n}}^\gi\notag\\
 &\qquad\times  R_{\Tqn} (\alpha)  \mat*{ \Iu{(n+1)q}, \ek*{R_{\Tqn}(\alpha)}^\inv H_n } (\Iu{2} \otimes v_{q,n}) \Jimq,\label{FIPN}\\
 \tilde{S}_2(w)
 &\defeq  - (\ko{w}-\alpha) \Jimq (\Iu{2} \otimes v_{q,n} )^\ad 
 \mat*{ \Iu{(n+1)q},\ek*{R_{\Tqn}(\alpha)}^\inv H_n}^\ad \ek*{R_{\Tqn} (\alpha)}^\ad H_{\at{n}}^\gi\notag  \\ 
 &\qquad\times \ek*{R_{\Tqn^\ad } (w)}^\ad \mat*{\ek*{R_{\Tqn}(\alpha)}^\inv H_n, -\Iu{(n+1)q}} (\Iu{2} \otimes v_{q,n} ),\label{FIPN-2} 
\end{align} 
 and
\begin{multline} \label{FIDIMN}
 \tilde{S}_3(z,w)
 \defeq - (z-\alpha) (\ko{w}-\alpha)(\Iu{2} \otimes v_{q,n} )^\ad 
 \mat*{ \ek*{R_{\Tqn}(\alpha)}^\inv H_n, -\Iu{(n+1)q}}^\ad R_{\Tqn ^\ad } (z) H_{\at{n}}^\gi\\
 \times R_{\Tqn} (\alpha) \mat*{\Iu{(n+1)q},\ek*{R_{\Tqn}(\alpha)}^\inv H_n} (\Iu{2} \otimes v_{q,n} ) \Jimq\\
 \times(\Iu{2} \otimes v_{q,n} )^\ad \mat*{\Iu{(n+1)q}, \ek*{R_{\Tqn}(\alpha)}^\inv H_n}^\ad \ek*{R_{\Tqn} (\alpha)}^\ad H_{\at{n}}^\gi\\
 \times\ek*{R_{\Tqn^\ad } (w)}^\ad \mat*{\ek*{R_{\Tqn}(\alpha)}^\inv H_n, -\Iu{(n+1)q}} (\Iu{2} \otimes v_{q,n} ). 
\end{multline}
 In view of \eqref{SKP1}, \eqref{SKP2}, and \rrem{21112N}, from \eqref{FIPN} and \eqref{FIPN-2} we obtain
\begin{align}
 \tilde{S}_1(z)
 &= -\iu (z-\alpha) (\Iu{2} \otimes v_{q,n} )^\ad \mat*{\ek*{R_{\Tqn}(\alpha)}^\inv H_n, -\Iu{(n+1)q}}^\ad R_{\Tqn ^\ad } (z)
 H_{\at{n}}^\gi  R_{\Tqn} (\alpha)\notag \\
 &\times \ek*{R_{\Tqn^\ad } (w)}^\invad \ek*{R_{\Tqn^\ad } (w)}^\ad \mat*{ \ek*{R_{\Tqn}(\alpha)}^\inv H_n, -\Iu{(n+1)q}} (\Iu{2} \otimes v_{q,n}), \label{CHI-3N}\\ 
 \tilde{S}_2(w)
 &= \iu (\ko{w}-\alpha) (\Iu{2} \otimes v_{q,n} )^\ad  \mat*{ \ek*{R_{\Tqn}(\alpha)}^\inv H_n, -\Iu{(n+1)q}}^\ad R_{\Tqn^\ad }(z) \ek*{R_{\Tqn^\ad}(z)}^\inv\notag \\ 
 & \times\ek*{R_{\Tqn} (\alpha)}^\ad H_{\at{n}}^\gi  \ek*{R_{\Tqn^\ad } (w)}^\ad \mat*{\ek*{R_{\Tqn}(\alpha)}^\inv H_n, -\Iu{(n+1)q}} (\Iu{2} \otimes v_{q,n} ),\label{CH1N}
\end{align} 
 and, because of \eqref{FIDIMN}, \eqref{SKP3}, and \(H_n^\ad =H_n\), furthermore,
\begin{multline} \label{CH2N}
 \tilde{S}_3(z,w) 
 = -\iu (z-\alpha) (\ko{w}-\alpha) (\Iu{2} \otimes v_{q,n} )^\ad \mat*{\ek*{R_{\Tqn}(\alpha)}^\inv H_n, -\Iu{(n+1)q}}^\ad 
 R_{\Tqn ^\ad } (z) H_{\at{n}}^\gi  \\
 \times R_{\Tqn} (\alpha) \rk*{ \ek*{R_{\Tqn}(\alpha)}^\inv H_n v_{q,n}v_{q,n}^\ad - v_{q,n} v_{q,n}^\ad H_n \ek*{R_{\Tqn}(\alpha)}^\invad} \ek*{R_{\Tqn} (\alpha)}^\ad \\
 \times H_{\at{n}}^\gi  \ek*{R_{\Tqn^\ad } (w)}^\ad \mat*{\ek*{R_{\Tqn}(\alpha)}^\inv H_n, -\Iu{(n+1)q}} (\Iu{2} \otimes v_{q,n} ). 
\end{multline}
 Using \(H_n^\ad =H_n\), \eqref{FIDW3}, \eqref{S1-1}, \eqref{S1-2}, and \(R_{\Tqn^\ad } (\alpha)= \ek{R_{\Tqn} (\alpha)}^\ad \), we infer 
\[\begin{split} 
 &(z-\alpha) (\ko{w}-\alpha) R_{\Tqn} (\alpha)\rk*{\ek*{R_{\Tqn}(\alpha)}^\inv H_n v_{q,n}v_{q,n}^\ad - v_{q,n} v_{q,n}^\ad H_n \ek*{R_{\Tqn}(\alpha)}^\invad} \ek*{R_{\Tqn} (\alpha)}^\ad \\
 &= (z-\alpha) (\ko{w}-\alpha) R_{\Tqn} (\alpha) \rk*{ \Tqn H_{\at{n}}\ek*{R_{\Tqn}(\alpha)}^\invad - \ek*{R_{\Tqn}(\alpha)}^\inv H_{\at{n}}\Tqn^\ad} \ek*{R_{\Tqn} (\alpha)}^\ad \\
 &= (z-\alpha) (\ko{w}-\alpha)R_{\Tqn} (\alpha) \Tqn H_{\at{n}} - (z-\alpha) (\ko{w}-\alpha)H_{\at{n}} \Tqn ^\ad \ek*{R_{\Tqn} (\alpha)}^\ad \\
 &= -(z-\alpha)\rk*{R_{\Tqn} (\alpha)\ek*{R_{\Tqn^\ad }(w)}^\invad - \Iu{(n+1)q}} H_{\at{n}} \\
 & \qquad +(\ko{w}-\alpha)H_{\at{n}} \rk*{ \ek*{R_{\Tqn^\ad } (z)}^\inv \ek*{R_{\Tqn} (\alpha)}^\ad - \Iu{(n+1)q}}\\
 &= -(z-\alpha) R_{\Tqn} (\alpha)\ek*{R_{\Tqn^\ad }(w)}^\invad H_{\at{n}} +(\ko{w}-\alpha)H_{\at{n}} \ek*{R_{\Tqn^\ad } (z)}^\inv \ek*{R_{\Tqn} (\alpha)}^\ad \\ 
 & \qquad + (z- \ko{w})H_{\at{n}}.
\end{split}\]
 Consequently, from \eqref{CH2N} we get then 
\begin{multline} \label{LR1-1}
 \tilde{S}_3(z,w) 
 = -\iu (\Iu{2} \otimes v_{q,n} )^\ad \mat*{\ek*{R_{\Tqn}(\alpha)}^\inv H_n, -\Iu{(n+1)q}}^\ad R_{\Tqn ^\ad } (z)H_{\at{n}}^\gi \\
 \times\biggl\{-(z-\alpha) R_{\Tqn} (\alpha)\ek*{R_{\Tqn^\ad }(w)}^\invad H_{\at{n}}+(\ko{w}-\alpha)H_{\at{n}} \ek*{R_{\Tqn^\ad } (z)}^\inv \ek*{R_{\Tqn} (\alpha)}^\ad\\
 + (z- \ko{w})H_{\at{n}}\biggr\}H_{\at{n}}^\gi \ek*{R_{\Tqn^\ad } (w)}^\ad \mat*{\ek*{R_{\Tqn}(\alpha)}^\inv H_n, -\Iu{(n+1)q}} (\Iu{2} \otimes v_{q,n} ).
\end{multline}
 The combination of \eqref{CHI-2N}, \eqref{CHI-3N}, \eqref{CH1N}, and \eqref{LR1-1} provides us
\begin{multline} \label{NTMO}
 \Jimq -\tilde{U}_{n,\alpha} (z) \Jimq \tilde{U}_{n,\alpha}^\ad (w)
 = -\iu (\Iu{2} \otimes v_{q,n} )^\ad \mat*{\ek*{R_{\Tqn}(\alpha)}^\inv H_n, -\Iu{(n+1)q}}^\ad R_{\Tqn ^\ad } (z) \tilde{S}(z,w) \\ 
 \qquad\times \ek*{R_{\Tqn^\ad } (w)}^\ad \mat*{\ek*{R_{\Tqn}(\alpha)}^\inv H_n, -\Iu{(n+1)q}} (\Iu{2} \otimes v_{q,n} ) 
\end{multline}
 where
\begin{multline} \label{NCF}
 \tilde{S}(z,w)
 \defeq (z-\alpha) H_{\at{n}}^\gi  R_{\Tqn}(\alpha) \ek*{R_{\Tqn^\ad }(w)}^\invad - (\ko{w}-\alpha) \ek*{R_{\Tqn^\ad}(z)}^\inv \ek*{R_{\Tqn}(\alpha)}^\ad  H_{\at{n}}^\gi  \\
 -(z-\alpha) H_{\at{n}}^\gi  R_{\Tqn}(\alpha)\ek*{R_{\Tqn^\ad }(w)}^\invad H_{\at{n}} H_{\at{n}}^\gi  \\
 +(\ko{w}-\alpha) H_{\at{n}}^\gi  H_{\at{n}} \ek*{R_{\Tqn^\ad } (z)}^\inv \ek*{R_{\Tqn} (\alpha)}^\ad H_{\at{n}}^\gi  + (z- \ko{w})H_{\at{n}}^\gi  H_{\at{n}} H_{\at{n}}^\gi. 
\end{multline}
 From \eqref{NCF} and the \rlemss{bem-ML8}{311-1} we obtain \(\tilde{S}(z,w)=(z- \ko{w})H_{\at{n}}^\gi \).
 Hence, taking into account \eqref{NTMO}, we get \eqref{ZB}.
\end{proof}

\begin{rem} \label{ML26-1}
 Let \(\alpha \in \R\), let \(\kappa \in \Ninf \), and let \(\seqska \in \Kggeqka \).
 For each \(n\in \NO \) with \(2n+1 \leq \kappa\), in view of the \rremss{LG1}{1582012-2} and \rlem{311-1}, it is readily checked that the matrices 
\begin{align} 
 B_{n,\alpha}
 &\defeq
 \begin{pmat}[{|}]
 \Iq & v_{q,n}^\ad H_n H_{\at{n}}^\gi  H_n v_{q,n} \cr\-
 \Oqq & \Iq\cr
 \end{pmat}\label{CHI-1}
\intertext{and} 
 \tilde{B}_{n,\alpha}
 &\defeq
 \begin{pmat}[{|}]
 \Iq & \Oqq \cr\-
 -v_{q,n}^\ad R_{\Tqn^\ad }(\alpha) H_n^\gi  R_{\Tqn}(\alpha) v_{q,n} & \Iq\cr
 \end{pmat}\label{CHI-112}
\end{align}
 are \(\Jimq\)\nobreakdash-unitary, \tie{}, that \( B_{n,\alpha}\Jimq B_{n,\alpha}^\ad =\Jimq\), \(B_{n,\alpha}^\ad \Jimq B_{n,\alpha} =\Jimq\), \(\tilde{B}_{n,\alpha}\Jimq \tilde{B}_{n,\alpha}^\ad =\Jimq\), and \(\tilde{B}_{n,\alpha}^\ad \Jimq \tilde{B}_{n,\alpha} =\Jimq\) hold true.
\end{rem}

\begin{rem}\label{115} 
 Let \(\alpha \in \R\), let \(\kappa \in \Ninf \), and let \(\seqska  \in \Kggeqka \).
 In view of \eqref{CHI-1} and \eqref{CHI-112}, for each \(n\in\NO \) with \(2n+1 \leq \kappa\), then
\[
 \mat{\Iu{(n+1)q}, \Tqn H_n} (\Iu{2} \otimes v_{q,n} ) B_{n,\alpha}
 =  \mat*{ \Iu{(n+1)q},\rk{v_{q,n} v_{q,n}^\ad H_n H_{\at{n}}^\gi  +\Tqn} H_n } (\Iu{2} \otimes v_{q,n} )
\]
and
\begin{multline*} 
 \mat*{ \Iu{(n+1)q}, \ek*{R_{\Tqn}(\alpha)}^\inv H_n} (\Iu{2} \otimes v_{q,n} ) \tilde{B}_{n,\alpha} \\
 = \ek*{R_{\Tqn}(\alpha)}^\inv \mat*{\ek*{\Iu{(n+1)q} - H_n v_{q,n} v_{q,n}^\ad R_{\Tqn^\ad }(\alpha) H_n^\gi} R_{\Tqn}(\alpha), H_n } (\Iu{2} \otimes v_{q,n} ).
\end{multline*} 
\end{rem}

\begin{rem} \label{155} 
 Let \(\alpha \in \R\), let \(\kappa \in \Ninf \), and let \(\seqska  \in \Kggeqka \).
 In view of \rrem{115} and \rlem{ML9-1b}, it is readily checked that, for each \(n\in\NO \) with \(2n+1 \leq \kappa\), then
\[
 H_n^\gi  R_{\Tqn}(\alpha)\mat{ \Iu{(n+1)q}, \Tqn H_n} (\Iu{2} \otimes v_{q,n} ) B_{n,\alpha}
 =\mat*{ H_n^\gi  R_{\Tqn}(\alpha), H_{\at n}^\gi  H_n } (\Iu{2} \otimes v_{q,n} ).
\]
\end{rem}

\begin{rem} \label{R8.13}
 Let \(\alpha \in \R\) and let \(n \in \NO \).
 According to \rrem{21112N}, the matrix-valued functions \(\Omega_{q,n,\alpha}\colon\C \to\Coo{2(n+1)q}{2(n+1)q}\) and \(\tilde{\Omega}_{q,n,\alpha}\colon\C \to\Coo{2(n+1)q}{2(n+1)q}\) given by
\begin{align} 
 \Omega_{q,n,\alpha} (\zeta)
 &\defeq
 \begin{pmat}[{|}]
 (\zeta-\alpha)\Tqn^\ad & \ek{R_{\Tqn^\ad}(\alpha)}^\inv \cr\-
 -(\zeta-\alpha) \Iu{(n+1)q} & -(\zeta-\alpha) \Iu{(n+1)q}\cr
 \end{pmat}\ek*{\Iu{2} \otimes R_{\Tqn^\ad } (\zeta)}\label{MQN}
\intertext{and}
 \tilde{\Omega}_{q,n,\alpha} (\zeta)
 &\defeq
 \begin{pmat}[{|}]
 (\zeta-\alpha)\Tqn^\ad & (\zeta-\alpha) \ek{R_{\Tqn^\ad}(\alpha)}^\inv \cr\-
 -\Iu{(n+1)q} & -(\zeta-\alpha) \Iu{(n+1)q}\cr
 \end{pmat}\ek*{\Iu{2} \otimes R_{\Tqn^\ad } (\zeta)}\label{MQN-1}
\end{align}
 are both matrix polynomials of degree \(n+1\).
\end{rem}

\begin{lem} \label{RF1}
 Let \(\alpha \in \R\), let \(\kappa \in \Ninf \), and let \(\seqska  \in \Kggeqka \).
 Let \(n \in \NO \) be such that \(2n+1\leq \kappa\).
 In view of \eqref{MQN} and \eqref{MQN-1}, let \(\Theta_{n,\alpha}\colon\C \to\Coo{2q}{2q}\) and \(\tilde{\Theta}_{n,\alpha}\colon\C \to\Coo{2q}{2q}\) be given by
\begin{multline} \label{TD2-2}
 \Theta_{n,\alpha} (\zeta)
 \defeq \Iu{2q} + (\Iu{2} \otimes v_{q,n} )^\ad \cdot \diag\rk{H_n,\Iu{(n+1)q}} \cdot \Omega_{q,n,\alpha} (\zeta) \\
 \times \diag\rk{H_n^\gi,H_{\at n}^\gi } \cdot \diag\rk*{R_{\Tqn}(\alpha), H_n} \cdot(\Iu{2} \otimes v_{q,n} ) 
\end{multline}
 and 
\begin{multline} \label{TD2-1}
 \tilde{\Theta}_{n,\alpha} (\zeta)
 \defeq \Iu{2q} + (\Iu{2} \otimes v_{q,n} )^\ad \cdot \diag\rk{H_n,\Iu{(n+1)q}}\cdot \tilde{\Omega}_{q,n,\alpha} (\zeta) \\
 \times \diag\rk{H_n^\gi,H_{\at n}^\gi } \cdot\diag\rk*{R_{\Tqn}(\alpha),H_n } \cdot(\Iu{2} \otimes v_{q,n} ).
\end{multline}
 Then \(\Theta_{n,\alpha}\) and \(\tilde{\Theta}_{n,\alpha}\) are matrix polynomials of degree not greater than \(n+1\) and, for each \(\zeta \in \C\), the representations
\begin{align} \label{TUB-1}
 \Theta_{n,\alpha} (\zeta)&= U_{n,\alpha}(\zeta) B_{n,\alpha}&
&\text{and}&
 \tilde{\Theta}_{n,\alpha} (\zeta)&= \tilde{U}_{n,\alpha}(\zeta) \tilde{B}_{n,\alpha}
\end{align}
 hold true, where \(U_{n,\alpha}\colon\C \to\Coo{2q}{2q}\) and \(\tilde{U}_{n,\alpha}\colon\C \to\Coo{2q}{2q}\) are defined by \eqref{U1} and \eqref{U2}, and where \(B_{n,\alpha}\) and \(\tilde{B}_{n,\alpha}\) are given by \eqref{CHI-1} and \eqref{CHI-112}, respectively.
 If
\begin{align} \label{TT1}
 \Theta_{n,\alpha}&=\mat{\Theta_{n,\alpha}^{(j,k)}}_{j,k=1}^{2}&
&\text{and}&
 \tilde{\Theta}_{n,\alpha}&=\mat{\tilde{\Theta}_{n,\alpha}^{(j,k)}}_{j,k=1}^{2}
\end{align}
 are the \tqqa{block} representations of \(\Theta_{n,\alpha}\) and \(\tilde{\Theta}_{n,\alpha}\), respectively, for each \(\zeta \in \C\), then
\begin{align}
 \Theta_{n,\alpha}^{(1,1)} (\zeta)
 &= \Iq +(\zeta-\alpha)v_{q,n}^\ad H_n \Tqn^\ad R_{\Tqn^\ad }(\zeta)H_n^\gi  R_{\Tqn}(\alpha) v_{q,n}, \label{T-11}\\
 \Theta_{n,\alpha}^{(1,2)} (\zeta) 
 &= v_{q,n}^\ad H_n \ek*{R_{\Tqn^\ad }(\alpha)}^\inv R_{\Tqn^\ad }(\zeta)H_{\at n}^\gi H_n^\gi  v_{q,n}, \label{T-12} \\
 \Theta_{n,\alpha}^{(2,1)} (\zeta) 
 &= -(\zeta-\alpha)v_{q,n}^\ad R_{\Tqn^\ad }(\zeta)H_n^\gi  R_{\Tqn}(\alpha) v_{q,n}, \label{T-21}\\
 \Theta_{n,\alpha}^{(2,2)} (\zeta) 
 &= \Iq-(\zeta-\alpha)v_{q,n}^\ad R_{\Tqn^\ad }(\zeta) H_{\at n}^\gi  H_n v_{q,n},\label{T-22}\\
 \tilde{\Theta}_{n,\alpha}^{(1,1)} (\zeta) 
 &= \Iq +(\zeta-\alpha)v_{q,n}^\ad H_n \Tqn^\ad R_{\Tqn^\ad }(\zeta)H_n^\gi  R_{\Tqn}(\alpha) v_{q,n},\label{MD+11} \\
 \tilde{\Theta}_{n,\alpha}^{(1,2)} (\zeta) 
 &= (\zeta-\alpha) v_{q,n}^\ad H_n \ek*{R_{\Tqn^\ad}(\alpha)}^\inv R_{\Tqn^\ad }(\zeta)H_{\at n}^\gi H_n v_{q,n}, \label{MD+12} \\
 \tilde{\Theta}_{n,\alpha}^{(2,1)} (\zeta) 
 &= -v_{q,n}^\ad R_{\Tqn^\ad }(\zeta)H_n^\gi  R_{\Tqn}(\alpha) v_{q,n},\label{MD+21}
\intertext{and} 
 \tilde{\Theta}_{n,\alpha}^{(2,2)} (\zeta) 
 &= \Iq-(\zeta-\alpha)v_{q,n}^\ad R_{\Tqn^\ad }(\zeta) H_{\at n}^\gi  H_n v_{q,n}.\label{MD+22}
\end{align} 
\end{lem} 
\begin{proof}
 \rrem{R8.13} shows that \(\Theta_{n,\alpha}\) and \(\tilde{\Theta}_{n,\alpha}\) are matrix polynomials of degree not greater than \(n+1\).  Let \(\zeta\in\C\).
 Because of the \rremss{ZS}{1582012-2}, we have \(H_n^\ad =H_n\) and \(H_{\at{n}}^\ad =H_{\at{n}}\).
 Using \eqref{TD2-2} and \eqref{MQN}, one can easily check that \eqref{T-11}, \eqref{T-12}, \eqref{T-21}, and \eqref{T-22} hold true.    From \eqref{TD2-1}, \eqref{MQN-1}, and \eqref{TT1} we infer that \eqref{MD+11}, \eqref{MD+12}, \eqref{MD+21}, and \eqref{MD+22} are valid.
 Let
\begin{align} \label{PHBDS}
 \Phi_{n,\alpha}&=\mat{\Phi_{n,\alpha}^{(j,k)}}_{j,k=1}^{2}&
&\text{and}&
 \tilde{\Phi}_{n,\alpha}&=\mat{\tilde{\Phi}_{n,\alpha}^{(j,k)}}_{j,k=1}^{2}
\end{align} 
 be the \tqqa{block} representations of 
\begin{align} \label{DIXI}
 \Phi_{n,\alpha}&\defeq U_{n,\alpha} B_{n,\alpha}&
&\text{and}&
 \tilde{\Phi}_{n,\alpha}&\defeq \tilde{U}_{n,\alpha} \tilde{B}_{n,\alpha}.
\end{align}
 By virtue of \eqref{U3}--\eqref{U3D}, \eqref{CHI-1}, and \eqref{T-11}, then
\beql{BRW11}
 \Phi_{n,\alpha}^{(1,1)} (\zeta) 
 = \Iq + (\zeta-\alpha) v_{q,n} ^\ad H_n \Tqn^\ad R_{\Tqn ^\ad } (\zeta) H_n^\gi  R_{\Tqn} (\alpha) v_{q,n} 
 =\Theta_{n,\alpha}^{(1,1)} (\zeta), 
\eeq
 follows, whereas \eqref{U3}--\eqref{U3D}, \eqref{CHI-1}, and \eqref{T-21} show that
\beql{BRW21}
 \Phi_{n,\alpha}^{(2,1)} (\zeta) 
 = -(\zeta-\alpha) v_{q,n} ^\ad R_{\Tqn ^\ad } (\zeta)H_n^\gi  R_{\Tqn} (\alpha) v_{q,n} 
 =\Theta_{n,\alpha}^{(2,1)} (\zeta).
\eeq
 From \eqref{DIXI}, \eqref{U3}--\eqref{U3D}, \eqref{CHI-1}, \rlem{ML9-1b}, \rrem{lemmart}, and \eqref{T-12} we conclude
\[\begin{split}%
 &\Phi_{n,\alpha}^{(1,2)} (\zeta)
 = \ek*{\Iq + (\zeta-\alpha) v_{q,n} ^\ad H_n \Tqn^\ad R_{\Tqn ^\ad } (\zeta) H_n^\gi  R_{\Tqn} (\alpha) v_{q,n} } v_{q,n} ^\ad H_n H_{\at n}^\gi  H_n v_{q,n}\\
 &\qquad+ (\zeta-\alpha) v_{q,n} ^\ad H_n \Tqn^\ad R_{\Tqn ^\ad } (\zeta) H_n^\gi  R_{\Tqn} (\alpha) \Tqn H_n v_{q,n}\\ 
 &= v_{q,n} ^\ad H_n\ek*{H_{\at n}^\gi  +(\zeta-\alpha)\Tqn^\ad R_{\Tqn ^\ad } (\zeta) H_n^\gi  R_{\Tqn} (\alpha) \rk{v_{q,n}v_{q,n} ^\ad H_n H_{\at n}^\gi  +\Tqn }}H_n v_{q,n} \\
 &= v_{q,n} ^\ad H_n \ek*{\Iu{(n+1)q}+(\zeta-\alpha)\Tqn^\ad R_{\Tqn ^\ad } (\zeta) } H_{\at n}^\gi  H_n v_{q,n} \\
 &= v_{q,n} ^\ad H_n \ek*{R_{\Tqn^\ad } (\alpha)}^\inv R_{\Tqn ^\ad } (\zeta) H_{\at n}^\gi  H_n v_{q,n}
 =\Theta_{n,\alpha}^{(1,2)}(\zeta)
\end{split}\]
 and, using additionally \eqref{T-22} instead of \eqref{T-12}, furthermore
\[\begin{split}%
 \Phi_{n,\alpha}^{(2,2)} (\zeta) 
 &= - (\zeta-\alpha) v_{q,n} ^\ad R_{\Tqn ^\ad } (\zeta) H_n^\gi  R_{\Tqn} (\alpha) v_{q,n} v_{q,n} ^\ad H_n H_{\at n}^\gi  H_n v_{q,n} \\
 &\qquad+ \Iq - (\zeta-\alpha) v_{q,n} ^\ad R_{\Tqn ^\ad } (\zeta) H_n^\gi  R_{\Tqn} (\alpha) \Tqn H_n v_{q,n}\\ 
 &= \Iq -(\zeta-\alpha) v_{q,n} ^\ad R_{\Tqn ^\ad } (\zeta) H_n^\gi  R_{\Tqn} (\alpha) \rk{v_{q,n}v_{q,n} ^\ad H_n H_{\at n}^\gi  +\Tqn } H_n v_{q,n} \\
 &= \Iq- (\zeta-\alpha) v_{q,n} ^\ad R_{\Tqn ^\ad } (\zeta) H_{\at n}^\gi  H_n v_{q,n}
 =\Theta_{n,\alpha}^{(2,2)} (\zeta).
\end{split}\]
 Consequently, taking additionally into account \eqref{BRW11}, \eqref{BRW21}, \eqref{TT1}, \eqref{DIXI}, and \eqref{PHBDS}, we obtain the first equation in \eqref{TUB-1}.
 From \eqref{DIXI}, \rrem{86-1}, \eqref{CHI-112}, \eqref{PHBDS}, \rlem{ML9-1b}, \rrem{lemmart}, and \eqref{MD+11} we get
\beql{BRW11-2}\begin{split}
 &\tilde{\Phi}_{n,\alpha}^{(1,1)} (\zeta)\\
 &= \Iq + (\zeta-\alpha) v_{q,n} ^\ad H_n \ek*{R_{\Tqn^\ad } (\alpha)}^\inv R_{\Tqn ^\ad } (\zeta) H_{\at n }^\gi  R_{\Tqn} (\alpha) v_{q,n}\\
 &-(\zeta-\alpha) v_{q,n} ^\ad H_n \ek*{R_{\Tqn^\ad } (\alpha)}^\inv R_{\Tqn ^\ad } (\zeta) H_{\at n}^\gi  H_n v_{q,n} v_{q,n} ^\ad R_{\Tqn ^\ad } (\alpha) H_{ n }^\gi  R_{\Tqn} (\alpha) v_{q,n}\\
 &= \Iq + (\zeta-\alpha) v_{q,n} ^\ad H_n \ek*{R_{\Tqn^\ad } (\alpha)}^\inv R_{\Tqn ^\ad } (\zeta) H_{\at n }^\gi \\
 &\qquad\times\ek*{ \Iu{(n+1)q}- H_n v_{q,n} v_{q,n} ^\ad R_{\Tqn ^\ad } (\alpha) H_{ n }^\gi } R_{\Tqn} (\alpha) v_{q,n}\\
 &= \Iq + (\zeta-\alpha) v_{q,n} ^\ad H_n \ek*{R_{\Tqn^\ad } (\alpha)}^\inv R_{\Tqn ^\ad } (\zeta) \Tqn^\ad R_{\Tqn^\ad } (\alpha) H_n^\gi  R_{\Tqn} (\alpha) v_{q,n} \\
 &= \Iq + (\zeta-\alpha) v_{q,n} ^\ad H_n \Tqn^\ad R_{\Tqn ^\ad } (\zeta) H_n^\gi  R_{\Tqn} (\alpha) v_{q,n} 
 = \tilde{\Theta}_{n,\alpha}^{(1,1)} (\zeta),
\end{split}\eeq
 whereas \eqref{DIXI}, \rrem{86-1}, \eqref{CHI-112}, \eqref{PHBDS}, and \eqref{MD+12} show that 
\beql{BRW12-2}
 \tilde{\Phi}_{n,\alpha}^{(1,2)} (\zeta) 
 = (\zeta-\alpha) v_{q,n} ^\ad H_n \ek*{R_{\Tqn^\ad}(\alpha)}^\inv R_{\Tqn ^\ad } (\zeta) H_{\at{n}}^\gi  H_n v_{q,n} 
 =\tilde{\Theta}_{n,\alpha}^{(1,2)} (\zeta)
\eeq
 holds true.
 Using \eqref{DIXI}, \rrem{86-1}, \eqref{CHI-112}, \eqref{PHBDS}, \rlem{ML9-1b}, \rrem{21112N}, and \eqref{MD+21}, we conclude
\beql{BRW21-2}\begin{split}
 &\tilde{\Phi}_{n,\alpha}^{(2,1)} (\zeta)\\
 & =-(\zeta-\alpha) v_{q,n} ^\ad R_{\Tqn ^\ad } (\zeta) H_{\at{n}}^\gi  R_{\Tqn} (\alpha) v_{q,n} \\
 &\qquad- \ek*{\Iq - (\zeta-\alpha) v_{q,n} ^\ad R_{\Tqn ^\ad } (\zeta) H_{\at{n}}^\gi  H_n v_{q,n}} v_{q,n}^\ad R_{\Tqn^\ad }(\alpha) H_n^\gi  R_{\Tqn}(\alpha) v_{q,n} \\
 &= - v_{q,n} ^\ad R_{\Tqn ^\ad } (\zeta) \biggl\{(\zeta-\alpha) H_{\at{n}}^\gi  \ek*{\Iu{(n+1)q} - H_n v_{q,n} v_{q,n}^\ad R_{\Tqn^\ad }(\alpha) H_n^\gi  } \\
 &\qquad+ \ek*{R_{\Tqn ^\ad } (\zeta)}^\inv R_{\Tqn^\ad }(\alpha) H_n^\gi  \biggr\} R_{\Tqn}(\alpha) v_{q,n} \\
 &= - v_{q,n} ^\ad R_{\Tqn ^\ad } (\zeta)\rk*{(\zeta-\alpha) \Tqn^\ad R_{\Tqn^\ad }(\alpha) H_n^\gi  + \ek*{R_{\Tqn ^\ad } (\zeta)}^\inv R_{\Tqn^\ad }(\alpha) H_n^\gi} R_{\Tqn}(\alpha) v_{q,n} \\ 
 &= - v_{q,n} ^\ad R_{\Tqn ^\ad } (\zeta)\rk{\Iu{(n+1)q}-\alpha \Tqn^\ad}R_{\Tqn^\ad }(\alpha) H_n^\gi  R_{\Tqn}(\alpha) v_{q,n}\\ 
 &= - v_{q,n} ^\ad R_{\Tqn ^\ad } (\zeta) H_n^\gi  R_{\Tqn}(\alpha) v_{q,n} 
 =\tilde{\Theta}_{n,\alpha}^{(2,1)}(\zeta)
\end{split}\eeq
 and, in view of \eqref{MD+22}, furthermore 
\beql{BRW22-2}
 \tilde{\Phi}_{n,\alpha}^{(2,2)} (\zeta)
 = \Iq - (\zeta-\alpha) v_{q,n} ^\ad R_{\Tqn ^\ad } (\zeta) H_{\at{n}}^\gi  H_n v_{q,n}
 = \tilde{\Theta}_{n,\alpha}^{(2,2)} (\zeta).
\eeq
 From \eqref{BRW11-2}, \eqref{BRW12-2}, \eqref{BRW21-2}, \eqref{BRW22-2}, \eqref{TT1}, \eqref{DIXI}, and \eqref{PHBDS}, the second equation in \eqref{TUB-1} also follows.
\end{proof}

\begin{lem} \label{BL4.2-1}
 Let \(\alpha \in \R\), let \(\kappa \in \Ninf \), and let \(\seqska  \in \Kggeqka \).
 For each \(n \in \NO \) with \(2n+1\leq \kappa\), the functions \(\Theta_{n,\alpha}\colon\C \to\Coo{2q}{2q}\) given by \eqref{TD2-2} and \(\tilde{\Theta}_{n,\alpha}\colon\C \to\Coo{2q}{2q}\) given by \eqref{TD2-1} fulfill for each \(\zeta \in \C \setminus \set{\alpha}\) the identity 
\[
 \tilde{\Theta}_{n,\alpha} (\zeta)
 = \diag\rk*{(\zeta - \alpha)\Iq, \Iq }  \cdot \Theta_{n,\alpha} (\zeta) \cdot \diag\rk*{(\zeta - \alpha)^\inv \Iq, \Iq }.
\]
\end{lem} 
\begin{proof}
 Let \(n \in \NO \) with \(2n+1\leq \kappa\).
 For each \(\zeta \in \C \setminus \set{\alpha}\), we have
\begin{align} 
 &\diag \rk*{(\zeta - \alpha)\Iq, \Iq }\cdot\Iu{2q} \cdot \diag\rk*{(\zeta - \alpha)^\inv \Iq, \Iq }
 =\Iu{2q},\label{MQN-3}\\
 &\diag \rk*{(\zeta - \alpha)\Iq, \Iq }\cdot (\Iu{2} \otimes v_{q,n} )^\ad \cdot \diag \rk{H_n, \Iu{(n+1)q}}\notag\\
 &=(\Iu{2} \otimes v_{q,n} )^\ad \cdot \diag \rk{H_n, \Iu{(n+1)q}} \cdot \diag\rk*{(\zeta - \alpha)\Iu{(n+1)q}, \Iu{(n+1)q} },\label{MQN-4}
\end{align} 
\begin{multline} \label{MQN-5}
 \diag \rk{H_n^\gi, H_{\at n}^\gi } \cdot \diag\rk*{R_{\Tqn}(\alpha), H_n} \cdot (\Iu{2} \otimes v_{q,n} )\cdot \diag\rk*{(\zeta - \alpha)^\inv \Iq, \Iq } \\
 = \diag\rk*{(\zeta - \alpha)^\inv \Iu{(n+1)q}, \Iu{(n+1)q} } \cdot \diag\rk{H_n^\gi, H_{\at n}^\gi }\cdot\diag\rk*{R_{\Tqn}(\alpha), H_n} \cdot (\Iu{2} \otimes v_{q,n} )
\end{multline} 
 and, in view of \eqref{MQN} and \eqref{MQN-1}, furthermore, 
\beql{MQN-2}\begin{split}
 &\diag \mat{(\zeta - \alpha)\Iu{(n+1)q}, \Iu{(n+1)q} } \cdot \Omega_{q,n,\alpha} (\zeta) \cdot \diag\rk{(\zeta - \alpha)^\inv \Iu{(n+1)q}, \Iu{(n+1)q} } \\
 &=
 \bMat
 (\zeta-\alpha)\Tqn^\ad R_{\Tqn^\ad } (\zeta)& (\zeta-\alpha) \ek{R_{\Tqn^\ad}(\alpha)}^\inv R_{\Tqn^\ad } (\zeta)\\
 -R_{\Tqn^\ad } (\zeta) & -(\zeta-\alpha) R_{\Tqn^\ad } (\zeta)
 \eMat\\
 &=
 \bMat
 (\zeta-\alpha)\Tqn^\ad & (\zeta-\alpha) \ek{R_{\Tqn^\ad}(\alpha)}^\inv \\
 - \Iu{(n+1)q}& -(\zeta-\alpha)\Iu{(n+1)q} 
 \eMat(\Iu{2} \otimes R_{\Tqn^\ad } (\zeta))
 = \tilde{\Omega}_{q,n,\alpha} (\zeta).
\end{split}\eeq
 Thus, \eqref{TD2-2}, \eqref{MQN-3}, \eqref{MQN-4}, \eqref{MQN-5}, \eqref{MQN-2}, and \eqref{TD2-1} imply
\[\begin{split}
 &\diag\rk*{(\zeta - \alpha)\Iq, \Iq } \cdot \Theta_{n,\alpha} (\zeta) \cdot \diag\rk*{(\zeta - \alpha)^\inv \Iq, \Iq }\\
 &=\diag\rk*{(\zeta - \alpha)\Iq, \Iq } \cdot\Iu{2q} \cdot \diag\rk*{(\zeta - \alpha)^\inv \Iq, \Iq }\\
 &\qquad+\diag\rk*{(\zeta - \alpha)\Iq, \Iq } \cdot (\Iu{2} \otimes v_{q,n} )^\ad \cdot \diag\rk{H_n, \Iu{(n+1)q}} \cdot \Omega_{q,n,\alpha} (\zeta)\\
 &\qquad\qquad\times \diag\rk{H_n^\gi, H_{\at n}^\gi } \cdot \diag\rk*{R_{\Tqn}(\alpha), H_n} \cdot (\Iu{2} \otimes v_{q,n} )\cdot \diag\rk*{(\zeta - \alpha)^\inv \Iq, \Iq }\\
 &=\Iu{2q} + (\Iu{2} \otimes v_{q,n} )^\ad \cdot \diag\rk{H_n, \Iu{(n+1)q}} \cdot \diag\rk*{(\zeta - \alpha)\Iu{(n+1)q}, \Iu{(n+1)q} } \cdot \Omega_{q,n,\alpha} (\zeta)\\
 &\times\diag\rk*{(\zeta - \alpha)^\inv \Iu{(n+1)q}, \Iu{(n+1)q} } \cdot \diag\rk{H_n^\gi, H_{\at n}^\gi }\cdot\diag\rk*{R_{\Tqn}(\alpha), H_n}\cdot(\Iu{2} \otimes v_{q,n} )\\
 &=\Iu{2q} + (\Iu{2} \otimes v_{q,n} )^\ad \cdot \diag\rk{H_n, \Iu{(n+1)q}}\cdot\tilde{\Omega}_{q,n,\alpha} (\zeta)\\
 &\qquad\times\diag\rk{H_n^\gi, H_{\at n}^\gi } \cdot \diag\rk*{R_{\Tqn}(\alpha), H_n} \cdot (\Iu{2} \otimes v_{q,n} )
 = \tilde{\Theta}_{n,\alpha} (\zeta).\qedhere
\end{split}\]
\end{proof}

 Now we obtain a generalization of a result which is, for the special case \(\alpha = 0\), due to Bolotnikov~\cite[\clem{4.2}]{MR1362524}.

\begin{lem} \label{BL4.14-16}
 Let \(\alpha \in \R\), let \(\kappa \in \Ninf \), and let \(\seqska  \in \Kggeqka \).
 For each \(n \in \NO \) with \(2n+1\leq \kappa\) and every choice of \(z\in \C\) and \(w\in \C\), then
\begin{multline*}
 \Jimq -\Theta_{n,\alpha} (z) \Jimq\Theta_{n,\alpha}^\ad (w)
 = -\iu (z-\ko{w}) (\Iu{2} \otimes v_{q,n} )^\ad \mat{\Tqn H_n, -\Iu{(n+1)q}}^\ad R_{\Tqn ^\ad } (z) H_n^\gi\\
 \times\ek*{R_{\Tqn^\ad } (w)}^\ad \mat{\Tqn H_n, -\Iu{(n+1)q}} (\Iu{2} \otimes v_{q,n} ) 
\end{multline*}
 and 
\begin{multline*}
 \Jimq -\tilde{\Theta}_{n,\alpha} (z) \Jimq \tilde{\Theta}_{n,\alpha}^\ad (w)
 = -\iu (z-\ko{w}) (\Iu{2} \otimes v_{q,n} )^\ad \mat*{\ek*{R_{\Tqn}(\alpha)}^\inv H_n, -\Iu{(n+1)q}}^\ad R_{\Tqn ^\ad } (z) H_{\at{n}}^\gi\\
 \times\ek*{R_{\Tqn^\ad } (w)}^\ad \mat*{\ek*{R_{\Tqn}(\alpha)}^\inv H_n, -\Iu{(n+1)q}} (\Iu{2} \otimes v_{q,n} ).
\end{multline*}
\end{lem} 
\begin{proof}
 Let \(n\in \NO \) be such that \(2n+1\leq \kappa\) and let \(z, w \in \C\).
 Using \rlem{RF1}, \rrem{ML26-1}, and \rlem{ML28-1}, we get 
\[\begin{split}
 \Jimq -\Theta_{n,\alpha} (z)\Jimq\Theta_{n,\alpha}^\ad (w)
 &= \Jimq -U_{n,\alpha} (z) B_{n,\alpha} \Jimq \ek*{U_{n,\alpha} (w)B_{n,\alpha}}^\ad \\
 &= \Jimq -U_{n,\alpha} (z) B_{n,\alpha} \Jimq B_{n,\alpha} ^\ad U_{n,\alpha}^\ad (w)
 = \Jimq -U_{n,\alpha} (z) \Jimq ^\ad U_{n,\alpha}^\ad (w) \\ 
 &=-\iu (z-\ko{w}) (\Iu{2} \otimes v_{q,n} )^\ad \mat{\Tqn H_n, -\Iu{(n+1)q}}^\ad R_{\Tqn ^\ad } (z) H_n^\gi\\
 &\qquad\times\ek*{R_{\Tqn^\ad } (w)}^\ad \mat{\Tqn H_n, -\Iu{(n+1)q}} (\Iu{2} \otimes v_{q,n} ). 
\end{split}\]
 Analogously, from \rlem{RF1}, \rrem{ML26-1}, and \rlem{ML28-2} we conclude
\[\begin{split}
 &\Jimq -\tilde{\Theta}_{n,\alpha} (z) \Jimq \tilde{\Theta}_{n,\alpha}^\ad (w)
  =\Jimq -\tilde{U}_{n,\alpha}(z) \tilde{B}_{n,\alpha} \Jimq \ek*{\tilde{U}_{n,\alpha}(w) \tilde{B}_{n,\alpha}}^\ad \\
 &= \Jimq -\tilde{U}_{n,\alpha}(z) \tilde{B}_{n,\alpha} \Jimq \tilde{B}_{n,\alpha}^\ad \tilde{U}_{n,\alpha}^\ad (w)
 =\Jimq -\tilde{U}_{n,\alpha}(z) \Jimq \tilde{U}_{n,\alpha}^\ad (w) \\
 &= -\iu (z-\ko{w}) (\Iu{2} \otimes v_{q,n} )^\ad \mat*{\ek*{R_{\Tqn}(\alpha)}^\inv H_n, -\Iu{(n+1)q}}^\ad \\
 &\qquad\times R_{\Tqn ^\ad } (z) H_{\at{n}}^\gi \ek*{R_{\Tqn^\ad } (w)}^\ad \mat*{\ek*{R_{\Tqn}(\alpha)}^\inv H_n, -\Iu{(n+1)q}} (\Iu{2} \otimes v_{q,n} ).\qedhere
\end{split}\]
\end{proof}

\begin{lem} \label{816-1}
 Let \(\alpha \in \R\), let \(\kappa \in \Ninf \), and let \(\seqska  \in \Kggeqka \).
 For each \(n \in \NO \) with \(2n+1\leq \kappa\) and each \(z\in \C\setminus \R\), then
\begin{align*} 
 \frac{1}{2 \Im z}\ek*{\Jimq - \Theta_{n,\alpha} (z) \Jimq \Theta_{n,\alpha} ^\ad (z)}& \lgeq\NM&
&\text{and}&
 \frac{1}{2 \Im z}\ek*{\Jimq - \tilde{\Theta}_{n,\alpha} (z) \Jimq \tilde{\Theta}_{n,\alpha} ^\ad (z) }&\lgeq\NM.
\end{align*}
\end{lem}
\begin{proof} 
 Let \(n\in \NO \) be such that \(2n +1\leq \kappa\).
 From \rlem{311-1} we get \(\set{ H_n^\gi,  H_{\at n}^\gi }\subseteq  \Cggo{(n+1)q}\).
 Applying  \rlem{BL4.14-16} completes the proof.
\end{proof}

\begin{lem} \label{BLR4.3-1}
 Let \(\alpha \in \R\), let \(\kappa \in \Ninf \), and let \(\seqska  \in \Kggeqka \).
 For each \(n \in \NO \) with \(2n+1\leq \kappa\), then the following statements hold true:
\begin{enui} 
 \item\label{BLR4.3-1.a} For each \(x\in \R\),
\begin{align*} 
 \Jimq -\Theta_{n,\alpha}(x) \Jimq\Theta_{n,\alpha}^\ad (x)&=\NM&
&\text{and}&
 \Jimq -\tilde{\Theta}_{n,\alpha} (x) \Jimq\tilde{\Theta}_{n,\alpha}^\ad (x)&=\NM. 
\end{align*} 
 \item\label{BLR4.3-1.b} For all \(z\in \C\), the matrices \( \Theta_{n,\alpha}(z)\) and \( \tilde{\Theta}_{n,\alpha}(z) \) are both non-singular and fulfill
\begin{align} \label{BLR4.3-B}
 \Theta_{n,\alpha}^\inv (z)&= \Jimq\Theta_{n,\alpha}^\ad (\ko {z}) \Jimq &
&\text{and}&
 \tilde{\Theta}_{n,\alpha}^\inv (z)&= \Jimq\tilde{\Theta}_{n,\alpha}^\ad (\ko {z}) \Jimq.
\end{align} 
 \item\label{BLR4.3-1.c} For every choice of \(z\) and \(w\) in \(\C\),
\begin{align*} 
 \Jimq -\Theta_{n,\alpha}^\invad (z) \Jimq \Theta_{n,\alpha}^\inv (w) 
 &=\Jimq\ek*{\Jimq - \Theta_{n,\alpha}(\ko{z}) \Jimq \Theta_{n,\alpha}^\ad (\ko{w})}\Jimq
\intertext{and} 
 \Jimq -\tilde{\Theta}_{n,\alpha}^\invad (z) \Jimq \tilde{\Theta}_{n,\alpha}^\inv (w) 
 &=\Jimq\ek*{\Jimq - \tilde{\Theta}_{n,\alpha}(\ko{z}) \Jimq \tilde{\Theta}_{n,\alpha}^\ad (\ko{w})}\Jimq.
\end{align*} 
\end{enui} 
\end{lem} 
\begin{proof}
 \eqref{BLR4.3-1.a} Use \rlem{BL4.14-16}.
 
 \eqref{BLR4.3-1.b} We know from \rlem{RF1} that \(\Theta_{n,\alpha}\) and \(\tilde{\Theta}_{n,\alpha}\) are matrix polynomials.
 Consequently, \(\Theta_{n,\alpha}^{\vee}\colon\C\to \Coo{2q}{2q}\) and \(\tilde{\Theta}_{n,\alpha}^{\vee}\colon\C\to\Coo{2q}{2q}\) defined by \(\Theta_{n,\alpha}^{\vee}(\zeta)\defeq \Theta_{n,\alpha}^\ad (\ko{\zeta})\) and \(\tilde{\Theta}_{n,\alpha}^{\vee}(\zeta)\defeq \tilde{\Theta}_{n,\alpha}^\ad (\ko{\zeta})\) are matrix polynomials as well.
 Thus, \(F\defeq \Jimq- \Theta_{n,\alpha} \Jimq \Theta_{n,\alpha}^{\vee} \) and \(\tilde{F}\defeq \Jimq- \tilde{\Theta}_{n,\alpha} \Jimq \tilde{\Theta}_{n,\alpha}^{\vee}\) are holomorphic in \(\C\).
 For each \(x \in \R\), we see from \rpart{BLR4.3-1.a} that 
\[
 F(x)
 = \Jimq- \Theta_{n,\alpha}(x) \Jimq \Theta_{n,\alpha}^{\vee}(x)
 =\Jimq- \Theta_{n,\alpha}(x) \Jimq \Theta_{n,\alpha}^\ad (\ko{x})
 =\NM
\]
 and, analogously, that \( \tilde{F}(x) =\NM\).
 The identity theorem for holomorphic functions implies 
\[
 \Theta_{n,\alpha}(z) \Jimq \Theta_{n,\alpha}^\ad (\ko{z})
 =\Theta_{n,\alpha}(z) \Jimq \Theta_{n,\alpha}^{\vee}(z)
 =\Jimq-F(z)
 =\Jimq
\]
 and, analogously, \( \tilde{\Theta}_{n,\alpha}(z) \Jimq \tilde{\Theta}_{n,\alpha}^\ad (\ko{z}) =\Jimq\) for all \(z\in \C\).
 Because of \(\Jimq^2 = \Iu{2q}\), we get \(\Theta_{n,\alpha}(z) \Jimq \Theta_{n,\alpha}^\ad (\ko{z}) \Jimq =\Jimq^2 =\Iu{2q}\) and \(\tilde{\Theta}_{n,\alpha}(z) \Jimq \tilde{\Theta}_{n,\alpha}^\ad (\ko{z})\Jimq =\Jimq^2 =\Iu{2q}\) for each \(z\in \C\).
 For all \(z\in\C\), then \(\det \Theta_{n,\alpha}(z) \neq 0\) and \(\det \tilde \Theta_{n,\alpha}(z) \neq 0\) and both equations in \eqref{BLR4.3-B} follow.

 \eqref{BLR4.3-1.c} Use \rpart{BLR4.3-1.b} and \(\Jimq^2 = \Iu{2q} \).
\end{proof}

\begin{lem} \label{BL4.17}
 Let \(\alpha \in \R\), let \(\kappa \in \Ninf \), and let \(\seqska  \in \Kggeqka \).
 For each \(n \in \NO \) with \(2n+1\leq \kappa\) and every choice of \(z\) and \(w\) in \(\C\), then
\begin{multline*}
 \Jimq -\Theta_{n,\alpha}^\invad (z)\Jimq\Theta_{n,\alpha}^\inv (w)
 = -\iu (\ko{z}-w) (\Iu{2} \otimes v_{q,n} )^\ad  \mat{ \Iu{(n+1)q},\Tqn H_n }^\ad R_{\Tqn ^\ad } (\ko{z}) H_n^\gi\\
 \times R_{\Tqn} (w)  \mat{ \Iu{(n+1)q}, \Tqn H_n}  (\Iu{2} \otimes v_{q,n} ) 
\end{multline*}
 and 
\begin{multline*}
 \Jimq -\tilde{\Theta}_{n,\alpha}^\invad (z) \Jimq \tilde{\Theta}_{n,\alpha}^\inv (w)
 = -\iu (\ko{z}-w) (\Iu{2} \otimes v_{q,n} )^\ad  \mat{ \Iu{(n+1)q}, \ek{R_{\Tqn}(\alpha)}^\inv H_n }^\ad R_{\Tqn ^\ad } (\ko{z}) H_{\at{n}}^\gi\\
 \times R_{\Tqn} (w) \mat{ \Iu{(n+1)q}, \ek{R_{\Tqn}(\alpha)}^\inv H_n } (\Iu{2} \otimes v_{q,n} ).
\end{multline*}
\end{lem} 
\begin{proof}
 Let \(n\in \NO \) be such that \(2n+1\leq \kappa\) and let \(z, w \in \C\).
 Using \rlemp{BLR4.3-1}{BLR4.3-1.b}, \rlemp{BLR4.3-1}{BLR4.3-1.c}, \rlem{BL4.14-16}, and \rrem{JN}, we obtain
\[\begin{split}
 \Jimq - \Theta_{n,\alpha}^\invad (z) \Jimq\Theta_{n,\alpha}^\inv (w) 
 &= \Jimq\ek*{\Jimq - \Theta_{n,\alpha}(\ko{z}) \Jimq \Theta_{n,\alpha}^\ad (\ko{w})}\Jimq \\
 &=\Jimq\biggl\{-\iu (\ko{z}-w) (\Iu{2} \otimes v_{q,n} )^\ad\mat{\Tqn H_n, -\Iu{(n+1)q}}^\ad R_{\Tqn ^\ad } (\ko{z}) H_n^\gi\\
 &\qquad\times\ek*{R_{\Tqn^\ad } (\ko{w})}^\ad\mat{\Tqn H_n, -\Iu{(n+1)q}}(\Iu{2} \otimes v_{q,n} ) \biggr\}\Jimq \\
 &=-\iu (\ko{z}-w) \rk*{ \iu (\Iu{2} \otimes v_{q,n} )^\ad  \mat{ \Iu{(n+1)q},\Tqn H_n}^\ad } R_{\Tqn ^\ad } (\ko{z}) H_n^\gi\\
 &\qquad\times R_{\Tqn} (w) \rk*{(-\iu)\cdot \mat{ \Iu{(n+1)q}, \Tqn H_n} (\Iu{2} \otimes v_{q,n} ) } \\
 &=-\iu (\ko{z}-w) (\Iu{2} \otimes v_{q,n} )^\ad  \mat{ \Iu{(n+1)q},\Tqn H_n}^\ad R_{\Tqn ^\ad } (\ko{z}) H_n^\gi\\
 &\qquad\times R_{\Tqn} (w) \mat{ \Iu{(n+1)q}, \Tqn H_n} (\Iu{2} \otimes v_{q,n} ) 
\end{split}\]
and 
\[\begin{split}
 &\Jimq -\tilde{\Theta}_{n,\alpha}^\invad (z) \Jimq \tilde{\Theta}_{n,\alpha}^\inv (w) 
 =\Jimq\ek*{\Jimq - \tilde{\Theta}_{n,\alpha}(\ko{z}) \Jimq \tilde{\Theta}_{n,\alpha}^\ad (\ko{w})}\Jimq \\
 &= \Jimq\biggl\{-\iu (\ko{z}-w) (\Iu{2} \otimes v_{q,n} )^\ad\mat*{\ek*{R_{\Tqn}(\alpha)}^\inv H_n, -\Iu{(n+1)q}}^\ad R_{\Tqn ^\ad } (\ko{z}) H_{\at{n}}^\gi\\
 &\qquad\times\ek*{R_{\Tqn^\ad } (\ko{w})}^\ad\mat*{\ek*{R_{\Tqn}(\alpha)}^\inv H_n, -\Iu{(n+1)q}} (\Iu{2} \otimes v_{q,n} )\biggr\} \Jimq \\
 &= -\iu (\ko{z}-w) \rk*{ \iu(\Iu{2} \otimes v_{q,n} )^\ad  \mat*{ \Iu{(n+1)q}, \ek*{R_{\Tqn}(\alpha)}^\inv H_n}^\ad }  R_{\Tqn ^\ad } (\ko{z}) H_{\at{n}}^\gi\\
 &\qquad\times R_{\Tqn} (w) \rk*{(-\iu)\cdot \mat*{ \Iu{(n+1)q}, \ek*{R_{\Tqn}(\alpha)}^\inv H_n} (\Iu{2} \otimes v_{q,n} )}\\
 &= -\iu (\ko{z}-w) (\Iu{2} \otimes v_{q,n} )^\ad  \mat*{ \Iu{(n+1)q}, \ek*{R_{\Tqn}(\alpha)}^\inv H_n}^\ad  R_{\Tqn ^\ad } (\ko{z}) H_{\at{n}}^\gi\\
 &\qquad\times R_{\Tqn} (w)  \mat*{ \Iu{(n+1)q}, \ek*{R_{\Tqn}(\alpha)}^\inv H_n} (\Iu{2} \otimes v_{q,n} ).\qedhere
\end{split}\]
\end{proof}

\begin{lem} \label{ML28-12}
 Let \(\alpha \in \R\), let \(\kappa \in \NOinf \), and let \(\seqska \in \Kggeqka \).
 Let \(U_{n,\alpha}\colon\C \to \Coo{2q}{2q}\) be defined by \eqref{U1}.
 For  each \(n\in\NO \) with \(2n\le \kappa\) and all \(z, w\in\C\), then
\begin{multline} \label{LIG}
 \Jimq - U_{n,\alpha}^\ad (w) \Jimq U_{n,\alpha} (z)\\
 = \iu (\ko{w}-z) (\Iu{2} \otimes v_{q,n} )^\ad  \mat{ \Iu{(n+1)q}, \Tqn H_n}^\ad\ek*{R_{\Tqn} (\alpha)}^\ad  H_n^\gi  \ek*{R_{\Tqn ^\ad } (w)}^\ad\ek*{R_{\Tqn}(\alpha)}^\inv H_n \\
  \times \ek*{R_{\Tqn^\ad}(\alpha)}^\inv R_{\Tqn ^\ad } (z) H_n^\gi R_{\Tqn} (\alpha) \mat{ \Iu{(n+1)q}, \Tqn H_n} (\Iu{2} \otimes v_{q,n} ).
\end{multline}
\end{lem}
\begin{proof}
 Let \(n\in \NO \) be such that \(2n \leq \kappa\).
 Since \(\seqska \in \Kggeqka \subseteq\Kggqka \) holds true, \rrem{1582012-2} yields \(H_n^\ad =H_n\) and from \rlem{311-1} we get \eqref{N39-RI}.
 Now let \(z, w\in \C\).
 In view of \eqref{U1}, \(\Jimq^2 =\Iu{2q}\), and \(H_{\at{n}}^\ad =H_{\at{n}}\), we conclude then 
\beql{CHI-2-II}\begin{split}
 &\Jimq -U_{n,\alpha}^\ad (w) \Jimq U_{n,\alpha} (z) \\
 &= \Jimq -\biggl\{\Iu{2q} + (\ko{w}-\alpha) (\Iu{2} \otimes v_{q,n} )^\ad  \mat{ \Iu{(n+1)q}, \Tqn H_n }^\ad\ek*{R_{\Tqn} (\alpha)}^\ad H_n^\gi \\
 & \qquad\times \ek*{R_{\Tqn^\ad } (w)}^\ad  \mat{ \Tqn H_n, -\Iu{(n+1)q}} (\Iu{2} \otimes v_{q,n} ) \biggr\}\Jimq\\
 &\qquad\times\biggl\{\Iu{2q} + (z-\alpha) (\Iu{2} \otimes v_{q,n} )^\ad  \mat{ \Tqn H_n, -\Iu{(n+1)q}}^\ad R_{\Tqn ^\ad } (z) H_n^\gi \\
 & \qquad\times   R_{\Tqn} (\alpha) \mat{ \Iu{(n+1)q}, \Tqn H_n} (\Iu{2} \otimes v_{q,n} ) \biggr\} \\
 & = S_1(w)+S_2(z)+S_3(z,w)
\end{split}\eeq
 where 
\begin{align*}
 S_1(w)
 &\defeq  -(\ko{w}-\alpha) (\Iu{2} \otimes v_{q,n} )^\ad  \mat{ \Iu{(n+1)q}, \Tqn H_n}^\ad  \ek*{R_{\Tqn} (\alpha)}^\ad H_n^\gi \\
 &\qquad\times\ek*{R_{\Tqn^\ad } (w)}^\ad \mat{ \Tqn H_n, -\Iu{(n+1)q}} (\Iu{2} \otimes v_{q,n} ) \Jimq,\\
 S_2(z)
 &\defeq  -(z-\alpha) \Jimq (\Iu{2} \otimes v_{q,n} )^\ad  \mat{ \Tqn H_n, -\Iu{(n+1)q}}^\ad R_{\Tqn ^\ad } (z) H_n^\gi\\
 &\qquad\times   R_{\Tqn} (\alpha)  \mat{  \Iu{(n+1)q}, \Tqn H_n} (\Iu{2} \otimes v_{q,n}),
\end{align*} 
 and
\[\begin{split}
 S_3(z,w)
 &\defeq - (\ko{w}-\alpha) (z-\alpha) (\Iu{2} \otimes v_{q,n} )^\ad  \mat{ \Iu{(n+1)q}, \Tqn H_n} ^\ad\ek*{R_{\Tqn} (\alpha)}^\ad H_n^\gi\\
 &\qquad\times\ek*{R_{\Tqn^\ad } (w)}^\ad  \mat{ \Tqn H_n, -\Iu{(n+1)q}} (\Iu{2} \otimes v_{q,n} )\Jimq(\Iu{2} \otimes v_{q,n} )^\ad\\
 &\qquad\times  \mat{ \Tqn H_n, -\Iu{(n+1)q}}^\ad  R_{\Tqn ^\ad } (z) H_n^\gi  R_{\Tqn} (\alpha)\mat{ \Iu{(n+1)q}, \Tqn H_n} (\Iu{2} \otimes v_{q,n} ). 
\end{split}\]
 Keeping in mind the \rremss{JN}{21112N}, we have 
\begin{align} 
 S_1(w)
 &= \iu (\ko{w}-\alpha) (\Iu{2} \otimes v_{q,n} )^\ad  \mat{ \Iu{(n+1)q},\Tqn H_n}^\ad  \ek*{R_{\Tqn} (\alpha)}^\ad H_n^\gi\notag \\ &\qquad\times\ek*{R_{\Tqn^\ad } (w)}^\ad\ek*{R_{\Tqn} (\alpha)}^\inv R_{\Tqn} (\alpha)\mat{ \Iu{(n+1)q},\Tqn H_n} (\Iu{2} \otimes v_{q,n} ),\label{CH1-II} \\
 S_2(z)
 &=  -\iu (z-\alpha) (\Iu{2} \otimes v_{q,n} )^\ad  \mat{ \Iu{(n+1)q},\Tqn H_n}^\ad  \ek*{R_{\Tqn} (\alpha)}^\ad \ek*{R_{\Tqn} (\alpha)}^\invad R_{\Tqn ^\ad } (z) H_n^\gi\notag\\
 &\qquad\times     R_{\Tqn} (\alpha) \mat{ \Iu{(n+1)q},\Tqn H_n} (\Iu{2} \otimes v_{q,n}),\label{CHI-3-II}
\end{align}
 and, by virtue of \eqref{SKP3-2} and \(H_n^\ad =H_n\), furthermore
\beql{CH2-II}\begin{split}
 S_3(z,w)
 &=- \iu(\ko{w}-\alpha) (z-\alpha) (\Iu{2} \otimes v_{q,n} )^\ad  \mat{ \Iu{(n+1)q}, \Tqn H_n } ^\ad \ek*{R_{\Tqn} (\alpha)}^\ad H_n^\gi  \\
 &\qquad\times \ek*{R_{\Tqn^\ad } (w)}^\ad\rk{\Tqn H_n v_{q,n}v_{q,n}^\ad - v_{q,n}v_{q,n}^\ad H_n \Tqn^\ad}R_{\Tqn ^\ad } (z) \\ 
 &\qquad\times H_n^\gi  R_{\Tqn} (\alpha) \mat{ \Iu{(n+1)q}, \Tqn H_n } (\Iu{2} \otimes v_{q,n} ).
\end{split}\eeq
 \rrem{lemmart} shows that 
\beql{215-1}
 (z-\alpha) \Tqn^\ad R_{\Tqn^\ad }(z)
 = \ek*{R_{\Tqn^\ad}(\alpha)}^\inv R_{\Tqn^\ad }(z)-\Iu{(n+1)q}
\eeq
 is valid and, because of \(\ek{ R_{T_{q,n}^\ad} (w)}^\ad = R_{T_{q,n}} (\ko{w})\), that
\beql{215-2}
 (\ko{w}-\alpha) \ek*{R_{\Tqn^\ad }(w)}^\ad \Tqn 
 =\ek*{R_{\Tqn^\ad }(w)}^\ad \ek*{R_{\Tqn}(\alpha)}^\inv - \Iu{(n+1)q}
\eeq
 is also true.
 In view of \eqref{B204}, \( \ek{R_{\Tqn} (\alpha)}^\invad = \ek{R_{\Tqn^\ad } (\alpha)}^\inv \), \eqref{215-1}, and \eqref{215-2}, we obtain
\[\begin{split} 
 &(\ko{w}- \alpha)(z-\alpha) \ek*{R_{\Tqn^\ad } (w)}^\ad\rk{\Tqn H_n v_{q,n}v_{q,n}^\ad - v_{q,n}v_{q,n}^\ad H_n \Tqn^\ad} R_{\Tqn ^\ad } (z) \\
 &=(\ko{w}- \alpha)(z-\alpha) \ek*{R_{\Tqn^\ad } (w)}^\ad\rk*{\Tqn H_n \ek*{R_{\Tqn^\ad}(\alpha)}^\inv -\ek*{R_{\Tqn}(\alpha)}^\inv H_n\Tqn^\ad} R_{\Tqn ^\ad } (z) \\
 &= (\ko{w}- \alpha)(z-\alpha) \ek*{R_{\Tqn^\ad } (w)}^\ad \Tqn H_n \ek*{R_{\Tqn^\ad}(\alpha)}^\inv R_{\Tqn ^\ad } (z) \\
 & \qquad -(\ko{w}- \alpha)(z-\alpha) \ek*{R_{\Tqn^\ad } (w)}^\ad \ek*{R_{\Tqn}(\alpha)}^\inv H_n\Tqn^\ad R_{\Tqn ^\ad } (z) \\
 &= (z-\alpha)\rk*{\ek*{R_{\Tqn^\ad }(w)}^\ad \ek*{R_{\Tqn}(\alpha)}^\inv - \Iu{(n+1)q}} H_n \ek*{R_{\Tqn^\ad}(\alpha)}^\inv R_{\Tqn ^\ad } (z) \\
 & \qquad -(\ko{w}- \alpha)\ek*{R_{\Tqn^\ad } (w)}^\ad \ek*{R_{\Tqn}(\alpha)}^\inv H_n\rk*{\ek*{R_{\Tqn^\ad}(\alpha)}^\inv R_{\Tqn^\ad }(z)-\Iu{(n+1)q}} \\
 &= (z-\alpha)\ek*{R_{\Tqn^\ad }(w)}^\ad \ek*{R_{\Tqn}(\alpha)}^\inv H_n \ek*{R_{\Tqn^\ad}(\alpha)}^\inv R_{\Tqn ^\ad } (z) \\
 & \qquad - (z-\alpha) H_n \ek*{R_{\Tqn^\ad}(\alpha)}^\inv R_{\Tqn ^\ad } (z) \\
 & \qquad -(\ko{w}- \alpha)\ek*{R_{\Tqn^\ad } (w)}^\ad \ek*{R_{\Tqn}(\alpha)}^\inv H_n \ek*{R_{\Tqn^\ad}(\alpha)}^\inv R_{\Tqn^\ad }(z) \\
 & \qquad +(\ko{w}- \alpha)\ek*{R_{\Tqn^\ad } (w)}^\ad \ek*{R_{\Tqn}(\alpha)}^\inv H_n \\
 &= (\ko{w}- \alpha)\ek*{R_{\Tqn^\ad } (w)}^\ad \ek*{R_{\Tqn}(\alpha)}^\inv H_n - (z-\alpha) H_n \ek*{R_{\Tqn^\ad}(\alpha)}^\inv R_{\Tqn ^\ad } (z) \\
 & \qquad -(\ko{w}- z)\ek*{R_{\Tqn^\ad } (w)}^\ad \ek*{R_{\Tqn}(\alpha)}^\inv H_n \ek*{R_{\Tqn^\ad}(\alpha)}^\inv R_{\Tqn^\ad }(z),
\end{split}\]
 which together with \eqref{CH2-II} implies
\beql{CH2-II-2}\begin{split} 
 &S_3(z,w)\\
 &=- \iu (\Iu{2} \otimes v_{q,n} )^\ad  \mat{ \Iu{(n+1)q}, \Tqn H_n} ^\ad \ek*{R_{\Tqn} (\alpha)}^\ad H_n^\gi  \\
 &\qquad\times \biggl\{ (\ko{w}- \alpha)\ek*{R_{\Tqn^\ad } (w)}^\ad \ek*{R_{\Tqn}(\alpha)}^\inv H_n - (z-\alpha) H_n \ek*{R_{\Tqn^\ad}(\alpha)}^\inv R_{\Tqn ^\ad } (z) \\
 & \qquad\qquad -(\ko{w}- z)\ek*{R_{\Tqn^\ad } (w)}^\ad \ek*{R_{\Tqn}(\alpha)}^\inv H_n \ek*{R_{\Tqn^\ad}(\alpha)}^\inv R_{\Tqn^\ad }(z) \biggr\} \\ 
 &\qquad\times H_n^\gi  R_{\Tqn} (\alpha)  \mat{ \Iu{(n+1)q}, \Tqn H_n} (\Iu{2} \otimes v_{q,n} ).
\end{split}\eeq
 The combination of \eqref{CHI-2-II}, \eqref{CH1-II}, \eqref{CHI-3-II}, and \eqref{CH2-II-2} provides us
\begin{multline} \label{CSNY1-II}
 \Jimq - U_{n,\alpha}^\ad (w) \Jimq  U_{n,\alpha} (z)
 = \iu (\Iu{2} \otimes v_{q,n} )^\ad  \mat{ \Iu{(n+1)q}, \Tqn H_n}^\ad \ek*{R_{\Tqn} (\alpha)}^\ad S(z,w)\\
 \times R_{\Tqn} (\alpha) \mat{ \Iu{(n+1)q}, \Tqn H_n} (\Iu{2} \otimes v_{q,n} )
\end{multline}
 where 
\beql{SW1-II}\begin{split} 
 S(z,w)
 &\defeq  (\ko{w}-\alpha) H_n^\gi  \ek*{R_{\Tqn^\ad } (w)}^\ad \ek*{R_{\Tqn} (\alpha)}^\inv\\
 &\qquad-(z-\alpha) \ek*{R_{\Tqn} (\alpha)}^\invad R_{\Tqn ^\ad } (z) H_n^\gi  \\
 &\qquad -(\ko{w}- \alpha) H_n^\gi  \ek*{R_{\Tqn^\ad } (w)}^\ad \ek*{R_{\Tqn}(\alpha)}^\inv H_n H_n^\gi  \\
 & \qquad + (z-\alpha) H_n^\gi  H_n \ek*{R_{\Tqn^\ad}(\alpha)}^\inv R_{\Tqn ^\ad } (z) H_n^\gi  \\
 & \qquad+(\ko{w}- z)H_n^\gi  \ek*{R_{\Tqn^\ad } (w)}^\ad  \ek*{R_{\Tqn}(\alpha)}^\inv H_n \ek*{R_{\Tqn^\ad}(\alpha)}^\inv R_{\Tqn^\ad }(z) H_n^\gi . 
\end{split}\eeq
 Because of \( R_{\Tqn ^\ad } (\zeta)= \ek{R_{\Tqn } (\ko \zeta)}^\ad \), which is true for each \(\zeta\in\C\), \rlem{bem-ML8} shows that
\beql{8117-1}
 H_n^\gi  \ek*{R_{\Tqn^\ad } (\zeta)}^\ad \ek*{R_{\Tqn}(\alpha)}^\inv H_n H_n^\gi
 = H_n^\gi  \ek*{R_{\Tqn^\ad } (\zeta)}^\ad \ek*{R_{\Tqn}(\alpha)}^\inv 
\eeq
 and 
\beql{8117-2}
 H_n^\gi  H_n \ek*{R_{\Tqn^\ad}(\alpha)}^\inv R_{\Tqn ^\ad } (\zeta) H_n^\gi
 = \ek*{R_{\Tqn}(\alpha)}^\invad R_{\Tqn ^\ad } (\zeta) H_n^\gi 
\eeq
 are valid for all \(\zeta\in\C\).
 Thus, from \eqref {SW1-II}, \eqref{8117-1}, and \eqref{8117-2} we get
\beql{N111T}
 S(z,w)
 = (\ko{w}- z)H_n^\gi  \ek*{R_{\Tqn^\ad } (w)}^\ad \ek*{R_{\Tqn}(\alpha)}^\inv H_n \ek*{R_{\Tqn^\ad}(\alpha)}^\inv R_{\Tqn^\ad }(z) H_n^\gi.
\eeq
 Taking into account \eqref{CSNY1-II} and \eqref{N111T}, we obtain \eqref{LIG}.
\end{proof}

\begin{lem} \label{8.17-1}
 Let \(\alpha \in \R\), let \(\kappa \in \Ninf \), and let \(\seqska \in \Kggeqka \).
 For each \(n\in \NO \) with \(2n +1\leq \kappa\) and all \(w,z\in \C\), then 
\begin{multline} \label{M111M}
 \Jimq - \tilde{U}_{n,\alpha}^\ad (w) \Jimq  \tilde{U}_{n,\alpha} (z)\\
 = \iu (\ko{w}-z) (\Iu{2} \otimes v_{q,n} )^\ad  \mat*{ \Iu{(n+1)q}, \ek*{R_{\Tqn}(\alpha)}^\inv H_n}^\ad \ek*{R_{\Tqn} (\alpha)}^\ad H_{\at n}^\gi\\
 \times   \ek*{R_{\Tqn ^\ad } (w)}^\ad \ek*{R_{\Tqn}(\alpha)}^\inv H_{\at n} \ek*{R_{\Tqn}(\alpha)}^\invad R_{\Tqn ^\ad } (z)   \\ 
 \times H_{\at n}^\gi R_{\Tqn} (\alpha)  \mat*{ \Iu{(n+1)q}, \ek*{R_{\Tqn}(\alpha)}^\inv H_n} (\Iu{2} \otimes v_{q,n} ).
\end{multline}
\end{lem}
\begin{proof}
 Let \(n\in \NO \) be such that \(2n +1\leq \kappa\).
 Because of \(\seqska \in \Kggeqka \subseteq\Kggqka \) and \rrem{1582012-2}, we have \(H_n^\ad =H_n\) and \(H_{\at{n}}^\ad =H_{\at{n}}\).
 \rlem{311-1} shows that \eqref{N39-RI} holds true.
 Let \(w, z\in \C\).
 In view of \eqref{U2} and \eqref{N39-RI}, we get 
\beql{RT1-1}\begin{split} 
 &\Jimq -\tilde{U}_{n,\alpha}^\ad (w) \Jimq \tilde{U}_{n,\alpha} (z)\\
 &= \Jimq -\biggl\{\Iu{2q} + (\ko{w}-\alpha) (\Iu{2} \otimes v_{q,n} )^\ad  \mat*{ \Iu{(n+1)q}, \ek*{R_{\Tqn}(\alpha)}^\inv H_n }^\ad\ek*{R_{\Tqn} (\alpha)}^\ad \\
 & \qquad\qquad\times  H_{\at n}^\gi  \ek*{R_{\Tqn^\ad } (w)}^\ad  \mat*{ \ek*{R_{\Tqn}(\alpha)}^\inv H_n, -\Iu{(n+1)q}} (\Iu{2} \otimes v_{q,n} ) \biggr\}\Jimq  \\
 & \qquad\times \biggl\{\Iu{2q} + (z-\alpha) (\Iu{2} \otimes v_{q,n} )^\ad  \mat*{ \ek*{R_{\Tqn}(\alpha)}^\inv H_n, -\Iu{(n+1)q}}^\ad R_{\Tqn ^\ad } (z) H_{\at n}^\gi\\
 &\qquad\qquad\times   R_{\Tqn} (\alpha)  \mat*{ \Iu{(n+1)q}, \ek*{R_{\Tqn}(\alpha)}^\inv H_n} (\Iu{2} \otimes v_{q,n} ) \biggr\} \\
 &= \tilde{S}_1(w)+\tilde{S}_2(z)+\tilde{S}_3(z,w),
\end{split}\eeq
 where
\begin{align*}
 \tilde{S}_1(w
 &\defeq-(\ko{w}-\alpha) (\Iu{2} \otimes v_{q,n} )^\ad \mat*{ \Iu{(n+1)q}, \ek*{R_{\Tqn}(\alpha)}^\inv H_n }^\ad \ek*{R_{\Tqn} (\alpha)}^\ad H_{\at n}^\gi \\
 &\qquad\times  \ek*{R_{\Tqn^\ad } (w)}^\ad \mat*{ \ek*{R_{\Tqn}(\alpha)}^\inv H_n, -\Iu{(n+1)q}} (\Iu{2} \otimes v_{q,n} ) \Jimq, \\
 \tilde{S}_2(z)
 &\defeq-(z-\alpha) \Jimq (\Iu{2} \otimes v_{q,n} )^\ad  \mat*{ \ek*{R_{\Tqn}(\alpha)}^\inv H_n, -\Iu{(n+1)q}}^\ad R_{\Tqn ^\ad } (z) H_{\at n}^\gi\\
 &\qquad\times   R_{\Tqn} (\alpha)  \mat*{ \Iu{(n+1)q}, \ek*{R_{\Tqn}(\alpha)}^\inv H_n } (\Iu{2} \otimes v_{q,n}),
\end{align*} 
 and
\[\begin{split}
 \tilde{S}_3(z,w)
 &\defeq - (\ko{w}-\alpha) (z-\alpha) (\Iu{2} \otimes v_{q,n} )^\ad \mat*{ \Iu{(n+1)q}, \ek*{R_{\Tqn}(\alpha)}^\inv H_n } ^\ad \ek*{R_{\Tqn} (\alpha)}^\ad H_{\at n}^\gi\\
 &\qquad\times   \ek*{R_{\Tqn^\ad } (w)}^\ad  \mat*{ \ek*{R_{\Tqn}(\alpha)}^\inv H_n, -\Iu{(n+1)q}} (\Iu{2} \otimes v_{q,n} )\Jimq \\
 &\qquad\times  (\Iu{2} \otimes v_{q,n} )^\ad  \mat*{ \ek*{R_{\Tqn}(\alpha)}^\inv H_n, -\Iu{(n+1)q}}^\ad R_{\Tqn ^\ad } (z) H_{\at n}^\gi  \\
 &\qquad\times R_{\Tqn} (\alpha) \mat*{ \Iu{(n+1)q}, \ek*{R_{\Tqn}(\alpha)}^\inv H_n} (\Iu{2} \otimes v_{q,n} ). 
\end{split}\]
 From \rrem{JN} we obtain 
\begin{align} 
 \tilde{S}_1(w)
 &= \iu (\ko{w}-\alpha) (\Iu{2} \otimes v_{q,n} )^\ad  \mat*{ \Iu{(n+1)q},\ek*{R_{\Tqn}(\alpha)}^\inv H_n }^\ad \ek*{R_{\Tqn} (\alpha)}^\ad H_{\at n}^\gi \ek*{R_{\Tqn^\ad } (w)}^\ad \notag \\
 &\qquad\times  \ek*{R_{\Tqn}(\alpha)}^\inv R_{\Tqn}(\alpha)\mat*{ \Iu{(n+1)q},\ek*{R_{\Tqn}(\alpha)}^\inv H_n} (\Iu{2} \otimes v_{q,n} ),\label{LTR6} \\ 
 \tilde{S}_2(z)
 &=  -\iu (z-\alpha) (\Iu{2} \otimes v_{q,n} )^\ad \mat*{ \Iu{(n+1)q},\ek*{R_{\Tqn}(\alpha)}^\inv H_n }^\ad \ek*{R_{\Tqn}(\alpha)}^\ad  \ek*{R_{\Tqn}(\alpha)}^\invad  \notag \\
 &\qquad\times R_{\Tqn ^\ad } (z)H_{\at n}^\gi  R_{\Tqn} (\alpha) \mat*{ \Iu{(n+1)q},\ek*{R_{\Tqn}(\alpha)}^\inv H_n} (\Iu{2} \otimes v_{q,n}),\label{LTR7}
\end{align}
 and, in view of \eqref{SKP3-2} and \(H_n^\ad =H_n\), furthermore
\beql{MD1}\begin{split} 
 &\tilde{S}_3(z,w)\\
 &=- \iu(\ko{w}-\alpha) (z-\alpha) (\Iu{2} \otimes v_{q,n} )^\ad \mat*{\Iu{(n+1)q}, \ek*{R_{\Tqn}(\alpha)}^\inv H_n} ^\ad \ek*{R_{\Tqn} (\alpha)}^\ad H_{\at n}^\gi  \\
 &\qquad\times \ek*{R_{\Tqn^\ad } (w)}^\ad\rk*{ \ek*{R_{\Tqn}(\alpha)}^\inv H_n v_{q,n}v_{q,n}^\ad - v_{q,n}v_{q,n}^\ad H_n \ek*{R_{\Tqn}(\alpha)}^\invad } R_{\Tqn ^\ad } (z) \\ 
 &\qquad\times H_{\at n}^\gi  R_{\Tqn} (\alpha) \mat*{\Iu{(n+1)q}, \ek*{R_{\Tqn}(\alpha)}^\inv H_n} (\Iu{2} \otimes v_{q,n} ).
\end{split}\eeq
 With the aid of \rrem{lemmart} and \(\ek{R_{\Tqn^\ad}(\alpha)}^\invad =\ek{R_{\Tqn}(\alpha)}^\inv\), we see that \eqref{215-1} and \eqref{215-2} hold true.
 Using equation \eqref{FIDW3} in \rrem{ZF1}, \(\ek{R_{\Tqn^\ad}(\alpha)}^\inv =\ek{R_{\Tqn}(\alpha)}^\invad \), \eqref{215-1}, and \eqref{215-2}, we conclude 
\[\begin{split}
 &(\ko{w}-\alpha) (z-\alpha) \ek*{R_{\Tqn^\ad } (w)}^\ad\rk*{\ek*{R_{\Tqn}(\alpha)}^\inv H_n v_{q,n}v_{q,n}^\ad - v_{q,n}v_{q,n}^\ad H_n \ek*{R_{\Tqn}(\alpha)}^\invad} R_{\Tqn ^\ad } (z) \\ 
 &=(\ko{w}-\alpha) (z-\alpha) \ek*{R_{\Tqn^\ad } (w)}^\ad\rk*{\Tqn H_{\at{n}}\ek*{R_{\Tqn}(\alpha)}^\invad -\ek*{R_{\Tqn}(\alpha)}^\inv H_{\at{n}}\Tqn^\ad} R_{\Tqn ^\ad } (z) \\ 
 &= (\ko{w}-\alpha) (z-\alpha) \ek*{R_{\Tqn^\ad } (w)}^\ad \Tqn H_{\at{n}}\ek*{R_{\Tqn}(\alpha)}^\invad R_{\Tqn ^\ad } (z) \\ 
 & \qquad -(\ko{w}-\alpha) (z-\alpha) \ek*{R_{\Tqn^\ad } (w)}^\ad \ek*{R_{\Tqn}(\alpha)}^\inv H_{\at{n}}\Tqn^\ad R_{\Tqn ^\ad } (z) \\ 
 &= (z-\alpha)\rk*{ \ek*{R_{\Tqn^\ad } (w)}^\ad \ek*{R_{\Tqn}(\alpha)}^\inv -\Iu{(n+1)q}} H_{\at{n}} \ek*{R_{\Tqn}(\alpha)}^\invad R_{\Tqn ^\ad } (z) \\
 & \qquad -(\ko{w}-\alpha) \ek*{R_{\Tqn^\ad } (w)}^\ad \ek*{R_{\Tqn}(\alpha)}^\inv H_{\at{n}}\rk*{ \ek*{R_{\Tqn^\ad}(\alpha)}^\inv R_{\Tqn ^\ad } (z) -\Iu{(n+1)q}} \\ 
 &= (z-\alpha) \ek*{R_{\Tqn^\ad } (w)}^\ad \ek*{R_{\Tqn}(\alpha)}^\inv H_{\at{n}} \ek*{R_{\Tqn}(\alpha)}^\invad R_{\Tqn ^\ad } (z) \\ 
 & \qquad -(z-\alpha) H_{\at{n}} \ek*{R_{\Tqn}(\alpha)}^\invad R_{\Tqn ^\ad } (z) \\
 & \qquad -(\ko{w}-\alpha) \ek*{R_{\Tqn^\ad } (w)}^\ad \ek*{R_{\Tqn}(\alpha)}^\inv H_{\at{n}} \ek*{R_{\Tqn^\ad}(\alpha)}^\inv R_{\Tqn ^\ad } (z) \\ 
 & \qquad +(\ko{w}-\alpha) \ek*{R_{\Tqn^\ad } (w)}^\ad \ek*{R_{\Tqn}(\alpha)}^\inv H_{\at{n}} \\ 
 &= (\ko{w}-\alpha) \ek*{R_{\Tqn^\ad } (w)}^\ad \ek*{R_{\Tqn}(\alpha)}^\inv H_{\at{n}} -(z-\alpha) H_{\at{n}} \ek*{R_{\Tqn}(\alpha)}^\invad R_{\Tqn ^\ad } (z) \\
 & \qquad-(\ko{w}-z) \ek*{R_{\Tqn^\ad } (w)}^\ad \ek*{R_{\Tqn}(\alpha)}^\inv H_{\at{n}} \ek*{R_{\Tqn}(\alpha)}^\invad R_{\Tqn ^\ad } (z), 
\end{split}\]
 which, because of \eqref{MD1}, implies
\beql{MD2}\begin{split} 
 \tilde{S}_3(z,w)
 &=- \iu (\Iu{2} \otimes v_{q,n} )^\ad \mat*{\Iu{(n+1)q}, \ek*{R_{\Tqn}(\alpha)}^\inv H_n} ^\ad \ek*{R_{\Tqn} (\alpha)}^\ad H_{\at n}^\gi  \\ 
 &\qquad\times \biggl\{(\ko{w}-\alpha) \ek*{R_{\Tqn^\ad } (w)}^\ad \ek*{R_{\Tqn}(\alpha)}^\inv H_{\at{n}}\\
 & \qquad\qquad-(z-\alpha) H_{\at{n}} \ek*{R_{\Tqn}(\alpha)}^\invad R_{\Tqn ^\ad } (z) \\
 & \qquad\qquad-(\ko{w}-z) \ek*{R_{\Tqn^\ad } (w)}^\ad \ek*{R_{\Tqn}(\alpha)}^\inv H_{\at{n}} \ek*{R_{\Tqn}(\alpha)}^\invad R_{\Tqn ^\ad } (z) \biggr\} \\ 
 &\qquad\times H_{\at n}^\gi  R_{\Tqn} (\alpha) \mat*{\Iu{(n+1)q}, \ek*{R_{\Tqn}(\alpha)}^\inv H_n} (\Iu{2} \otimes v_{q,n} ).
\end{split}\eeq
 Combining \eqref{RT1-1}, \eqref{LTR6}, \eqref{LTR7}, and \eqref{MD2}, we see that
\begin{multline} \label{BRT} 
 \Jimq -\tilde{U}_{n,\alpha}^\ad (w) \Jimq \tilde{U}_{n,\alpha} (z)\\
 = \iu (\Iu{2} \otimes v_{q,n} )^\ad \mat*{\Iu{(n+1)q}, \ek*{R_{\Tqn}(\alpha)}^\inv H_n }^\ad \ek*{R_{\Tqn} (\alpha)}^\ad \tilde{S}(z,w)\\
 \times R_{\Tqn}(\alpha) \mat*{ \Iu{(n+1)q}, \ek*{R_{\Tqn}(\alpha)}^\inv H_n} (\Iu{2} \otimes v_{q,n} ) 
\end{multline}
 is valid, where
\beql{BTM1}\begin{split} 
 &\tilde S(z,w)\\
 &\defeq(\ko{w}-\alpha) H_{\at n }^\gi  \ek*{R_{\Tqn^\ad } (w)}^\ad \ek*{R_{\Tqn} (\alpha)}^\inv -(z-\alpha) \ek*{R_{\Tqn} (\alpha)}^\invad R_{\Tqn ^\ad } (z) H_{\at n}^\gi  \\
 &\qquad-(\ko{w}-\alpha) H_{\at n}^\gi  \ek*{R_{\Tqn^\ad } (w)}^\ad \ek*{R_{\Tqn}(\alpha)}^\inv H_{\at{n}} H_{\at n}^\gi  \\
 &\qquad +(z-\alpha) H_{\at n}^\gi  H_{\at{n}} \ek*{R_{\Tqn}(\alpha)}^\invad R_{\Tqn ^\ad } (z) H_{\at n}^\gi  \\
 &\qquad +(\ko{w}-z) H_{\at n}^\gi  \ek*{R_{\Tqn^\ad } (w)}^\ad \ek*{R_{\Tqn}(\alpha)}^\inv H_{\at{n}} \ek*{R_{\Tqn}(\alpha)}^\invad R_{\Tqn ^\ad } (z)H_{\at n}^\gi . 
\end{split}\eeq
 Since \(\ek{R_{\Tqn^\ad } (\eta)}^\ad = R_{\Tqn} (\ko{\eta})\) holds true for all \(\eta\in\C\), from \rlem{bem-ML8} we infer that
\beql{LC1}
 H_{\at n}^\gi  \ek*{R_{\Tqn^\ad } (\zeta)}^\ad \ek*{R_{\Tqn}(\alpha)}^\inv H_{\at n} H_{\at n}^\gi
 = H_{\at n}^\gi  \ek*{R_{\Tqn^\ad } (\zeta)}^\ad \ek*{R_{\Tqn}(\alpha)}^\inv 
\eeq
 and 
\beql{LC2}
 H_{\at n}^\gi   H_{\at n} \ek*{R_{\Tqn}(\alpha)}^\invad R_{\Tqn ^\ad } (\zeta) H_{\at n}^\gi
 = \ek*{R_{\Tqn}(\alpha)}^\invad R_{\Tqn ^\ad } (\zeta) H_{\at n}^\gi 
\eeq
 are fulfilled for each \(\zeta\in\C\).
 Using \eqref{LC1}, \eqref{LC2}, and \eqref{BTM1}, we get 
\[
 \tilde S(z,w)
 = (\ko{w}-z) H_{\at n}^\gi  \ek*{R_{\Tqn^\ad } (w)}^\ad \ek*{R_{\Tqn}(\alpha)}^\inv H_{\at{n}} \ek*{R_{\Tqn}(\alpha)}^\invad R_{\Tqn ^\ad } (z)H_{\at n}^\gi  
\]
 and, because of \eqref{BRT}, then \eqref{M111M} follows. 
\end{proof}

\begin{lem} \label{QB-BQ}
 Let \(\alpha \in \R\), let \(\kappa \in \Ninf \), and let \(\seqska  \in \Kggeqka \).
 For each \(n \in \NO \) with \(2n+1\leq \kappa\) and all \(w, z \in \C\), then 
\beql{Nr.QB-1}\begin{split}
 \Jimq - \Theta_{n,\alpha} ^\ad (w) \Jimq \Theta_{n,\alpha} (z)
 &  = \iu (\ko{w}-z) B_{n,\alpha}^\ad (\Iu{2} \otimes v_{q,n} )^\ad \mat{ \Iu{(n+1)q}, \Tqn H_n }^\ad \ek*{R_{\Tqn} (\alpha)}^\ad  H_n^\gi\\
 &\qquad\times  \ek*{R_{\Tqn ^\ad } (w)}^\ad \ek*{R_{\Tqn}(\alpha)}^\inv H_n \ek*{R_{\Tqn^\ad}(\alpha)}^\inv R_{\Tqn ^\ad } (z)  \\ 
 &\qquad\times H_n^\gi R_{\Tqn} (\alpha) \mat{ \Iu{(n+1)q}, \Tqn H_n } (\Iu{2} \otimes v_{q,n} ) B_{n,\alpha}
\end{split}\eeq 
 and
\beql{Nr.tQB-1}\begin{split}
 &\Jimq - \tilde{\Theta}_{n,\alpha} ^\ad (w) \Jimq \tilde{\Theta}_{n,\alpha} (z)\\
 &= \iu (\ko{w}-z) \tilde{B}_{n,\alpha}^\ad (\Iu{2} \otimes v_{q,n} )^\ad \mat*{\Iu{(n+1)q}, \ek*{R_{\Tqn}(\alpha)}^\inv H_n}^\ad \ek*{R_{\Tqn} (\alpha)}^\ad H_{\at n}^\gi\\
 &\qquad\times   \ek*{R_{\Tqn ^\ad } (w)}^\ad \ek*{R_{\Tqn}(\alpha)}^\inv H_{\at n} \ek*{R_{\Tqn^\ad}(\alpha)}^\inv R_{\Tqn ^\ad } (z) \\ 
 &\qquad\times H_{\at n}^\gi  R_{\Tqn} (\alpha) \mat*{\Iu{(n+1)q}, \ek*{R_{\Tqn}(\alpha)}^\inv H_n} (\Iu{2} \otimes v_{q,n} ) \tilde{B}_{n,\alpha}.
\end{split}\eeq 
\end{lem}
\begin{proof}
 Let \(n\in \NO \) be such that \(2n +1\leq \kappa\).
 From \rrem{ML26-1} and \rlem{RF1} we get
\beql{M114Z}\begin{split} 
 \Jimq - \Theta_{n,\alpha} ^\ad (w)\Jimq \Theta_{n,\alpha} (z)
 &= B_{n,\alpha}^\ad \Jimq B_{n,\alpha}- \ek*{U_{n,\alpha} (w) B_{n,\alpha}}^\ad \Jimq U_{n,\alpha} (z) B_{n,\alpha} \\
 & = B_{n,\alpha}^\ad \ek*{\Jimq - U_{n,\alpha}^\ad (w) \Jimq U_{n,\alpha} (z) }B_{n,\alpha}
\end{split}\eeq
 and, analogously, 
\beql{M111W}
 \Jimq - \tilde{\Theta}_{n,\alpha} ^\ad (w) \Jimq \tilde{\Theta}_{n,\alpha} (z)
 = \tilde{B}_{n,\alpha} ^\ad \ek*{\Jimq - \tilde{U}_{n,\alpha}^\ad (w) \Jimq \tilde{U}_{n,\alpha} (z)} \tilde{B}_{n,\alpha}
\eeq
 for every choice of  \(z\) and \(w\) in \(\C\).
 Using \eqref{M114Z} and \rlem{ML28-12}, we obtain \eqref{Nr.QB-1}.
 Because of \eqref{M111W} and \rlem{8.17-1}, then \eqref{Nr.tQB-1} follows.
\end{proof}

\begin{lem} \label{QB-BQq}
 Let \(\alpha \in \R\), let \(\kappa \in \Ninf \), and let \(\seqska  \in \Kggeqka \).
 For each \(n \in \NO \) with \(2n+1\leq \kappa\) and each \(z\in\C\setminus \R\), then
\begin{align*} 
 \frac{1}{2 \Im z}\ek*{\Jimq - \Theta_{n,\alpha} ^\ad (z) \Jimq \Theta_{n,\alpha} (z)}&\ge 0&
&\text{and}&
 \frac{1}{2 \Im z}\ek*{\Jimq - \tilde{\Theta}_{n,\alpha} ^\ad (z) \Jimq \tilde{\Theta}_{n,\alpha} (z) }&\ge 0.
 \end{align*}
\end{lem}
\begin{proof}
 Let \(n\in \NO \) be such that \(2n +1\leq \kappa\).
 Because of \(\seqska \in \Kggeqka  \subseteq \Kggqka \) and \rrem{1582012-2}, we have \(H_n \in \Cggo{(n+1)q}\) and \(H_{\at n} \in \Cggo{(n+1)q}\).
 Using \rlem{QB-BQ} and \rlem{311-1}, for all \(z \in \C \setminus \R\), we get then 
\[\begin{split} 
 & \frac{1}{2 \Im z} \ek*{\Jimq - \Theta_{n,\alpha}^\ad (z) \Jimq \Theta_{n,\alpha} (z)}\\
 &=\frac{1}{2 \Im z}\biggl\{\iu (\ko{z}-z) B_{n,\alpha}^\ad (\Iu{2} \otimes v_{q,n} )^\ad\mat{ \Iu{(n+1)q}, \Tqn H_n }^\ad \ek*{R_{\Tqn} (\alpha)}^\ad H_n^\gi\\
 &\qquad\times  \ek*{R_{\Tqn ^\ad } (z)}^\ad  \ek*{R_{\Tqn}(\alpha)}^\inv H_n \ek*{R_{\Tqn^\ad}(\alpha)}^\inv R_{\Tqn ^\ad } (z)\\
 &\qquad\times H_n^\gi  R_{\Tqn} (\alpha) \mat{ \Iu{(n+1)q}, \Tqn H_n } (\Iu{2} \otimes v_{q,n} ) B_{n,\alpha}\biggr\} \\
 & =\rk*{\ek*{R_{\Tqn^\ad}(\alpha)}^\inv R_{\Tqn ^\ad } (z) H_n^\gi  R_{\Tqn} (\alpha) \mat{ \Iu{(n+1)q}, \Tqn H_n } (\Iu{2} \otimes v_{q,n} ) B_{n,\alpha}}^\ad \\
 & \qquad\times H_n\rk*{\ek*{R_{\Tqn^\ad}(\alpha)}^\inv R_{\Tqn ^\ad } (z) H_n^\gi  R_{\Tqn} (\alpha)  \mat{ \Iu{(n+1)q}, \Tqn H_n } (\Iu{2} \otimes v_{q,n} ) B_{n,\alpha} }
 \lgeq\NM
\end{split}\]
 and
\[\begin{split} 
 &\frac{1}{2 \Im z} \ek*{\Jimq - \tilde{\Theta}_{n,\alpha} ^\ad (z) \Jimq \tilde{\Theta}_{n,\alpha} (z) }\\
 &=\frac{1}{2 \Im z}\biggl\{ \iu (\ko{z}-z) \tilde{B}_{n,\alpha}^\ad (\Iu{2} \otimes v_{q,n} )^\ad\mat*{\Iu{(n+1)q}, \ek*{R_{\Tqn}(\alpha)}^\inv H_n}^\ad \ek*{R_{\Tqn} (\alpha)}^\ad H_{\at n}^\gi \\ 
 & \quad\times   \ek*{R_{\Tqn ^\ad } (z)}^\ad \ek*{R_{\Tqn}(\alpha)}^\inv H_{\at n}\ek*{R_{\Tqn^\ad}(\alpha)}^\inv R_{\Tqn ^\ad } (z) \\ 
 & \quad\times  H_{\at n}^\gi  R_{\Tqn} (\alpha) \mat*{\Iu{(n+1)q}, \ek*{R_{\Tqn}(\alpha)}^\inv H_n} (\Iu{2} \otimes v_{q,n} ) \tilde{B}_{n,\alpha} \biggr\} \\ 
 &=\rk*{ \ek*{R_{\Tqn^\ad}(\alpha)}^\inv R_{\Tqn ^\ad } (z) H_{\at n}^\gi  R_{\Tqn} (\alpha)  \mat*{\Iu{(n+1)q}, \ek*{R_{\Tqn}(\alpha)}^\inv H_n}(\Iu{2} \otimes v_{q,n} ) \tilde{B}_{n,\alpha} }^\ad H_{\at n} \\
 & \quad\times \rk*{ \ek*{R_{\Tqn^\ad}(\alpha)}^\inv R_{\Tqn ^\ad } (z) H_{\at n}^\gi  R_{\Tqn} (\alpha) \mat*{\Iu{(n+1)q}, \ek*{R_{\Tqn}(\alpha)}^\inv H_n} (\Iu{2} \otimes v_{q,n} )\tilde{B}_{n,\alpha}}
 \lgeq\NM.\qedhere
\end{split}\]
\end{proof}
 
 Let \(\cG\) be a \tne{} open subset of \(\C\) and let \(f =\mat{f_{jk}}_{\substack{j=1,\dotsc,p\\k=1,\dotsc,q}}\) be a \tpqa{matrix-valued} function which is meromorphic in \(\cG\).
 For every choice of \(j\) in \(\mn{1}{p}\) and \(k\) in \(\mn{1}{q}\), then let \(\hol{f_{jk}}\) be the set of all \(z\in\cG\) in which \(f_{jk}\) is holomorphic and let \(\pol{f_{jk}}\) be the set of all poles of \(f_{jk}\).
 Furthermore, let \(\hol{f} \defeq \bigcap_{j=1}^p\bigcap_{k=1}^q\hol{f_{jk}}\) and let \(\pol{f} \defeq \bigcup_{j=1}^p\bigcup_{k=1}^q\pol{f_{jk}}\).

\begin{nota} \label{bez-W} 
 Let \(\alpha \in \R\).
 By \(\PJif\) we denote the set of all \taaa{2q}{2q}{matrix-valued} functions \(\Theta\) which are holomorphic in \(\Cs \) and for which there exists a discrete subset \(\cD\) of \(\Cs \) such that the following three conditions are fulfilled:
\begin{Aeqi}{0}
 \item\label{bez-W.i} \(\Theta\) is holomorphic in \(\C \setminus (\rhl \cup \cD)\).
 \item\label{bez-W.ii} \(\Theta (z) \Jimq \Theta^\ad (z) \lleq \Jimq\) for each \(z \in \uhp \setminus \cD\).
 \item\label{bez-W.iii} \(\Theta (x) \Jimq \Theta^\ad (x) =\Jimq\) for each \(x \in \crhl  \setminus \cD\).
\end{Aeqi} 
\end{nota}

 Observe that continuity arguments show that conditions~\ref{bez-W.ii} and~\ref{bez-W.iii} in \rnota{bez-W} can be replaced equivalently by the following conditions~\ref{bez-W.ii'} and~\ref{bez-W.iii'}, respectively: 
\begin{Aeqip}{1}
 \item\label{bez-W.ii'} \(\Theta (z) \Jimq \Theta^\ad (z) \lleq \Jimq\) for each \(z \in \uhp \cap \hol{\Theta}\).
 \item\label{bez-W.iii'} \(\crhl \subseteq \hol{\Theta}\) and \( \Theta (x) \Jimq \Theta^\ad (x) =\Jimq\) for each \(x \in \crhl  \).
\end{Aeqip} 

\begin{rem} \label{WT}
 Let \(\alpha \in \R\), let \(\kappa \in \Ninf \), and let \(\seqska  \in \Kggeqka \).
 From the \rlemsss{RF1}{816-1}{BLR4.3-1} we see then that, for each \(n\in \NO \) with \(2n+1\leq \kappa\), the functions \(\hat \Theta_{n,\alpha}\defeq \Rstr_{\Cs }\Theta_{n,\alpha}\) and \(\hat{\tilde{\Theta}}_{n,\alpha}\defeq \Rstr_{\Cs }\tilde{\Theta}_{n,\alpha}\) given by \eqref{TD2-2} and \eqref{TD2-1} are holomorphic in \(\Cs \) and belong both to \(\PJif \).
\end{rem}

\begin{nota} \label{W-Jq-a} 
 Let \(\alpha\in \R\) and let the matrix-valued function \(P_\alpha \colon\Cs\to\Coo{2q}{2q} \) be defined by \(P_\alpha (z) \defeq \diag ((z- \alpha) \Iq, \Iq )\).
 Then let \(\PJisf\) be the set of all \(\Theta \in \PJif \) for which \(\tilde{\Theta} \defeq P_{\alpha} \Theta P_{\alpha}^\inv \) belongs to \(\PJif \).
\end{nota}

\begin{rem} \label{L238-1}
 Let \(\alpha \in \R\), let \(\kappa \in \Ninf \), and let \(\seqska  \in \Kggeqka \).
 From \rrem{WT} and \rlem{BL4.2-1} one can easily see  that, for each \(n \in \NO \) with \(2n+1\leq \kappa\), the matrix-valued function \(\Rstr_{\Cs }\Theta_{n,\alpha}\) given by \eqref{TD2-2} is holomorphic in \(\Cs \) and belongs to \(\PJisf \).
\end{rem}

\begin{lem} \label{NIN}
 Let \(\alpha \in \R\) and let \(\Theta \in \PJisf \).
 Then there is a discrete subset \(\cD\) of \(\Cs \) such that \(\Theta\) is holomorphic in \(\C \setminus (\rhl \cup \cD)\) and that \(\det \Theta (z)\neq 0\) holds true for each \(z \in \C \setminus (\rhl \cup \cD)\).
 Furthermore, \(\Theta^\inv \) is meromorphic in \(\Cs \) with \(\hol{\Theta^\inv } \supseteq \C \setminus (\rhl \cup \cD)\) and the identity \(\Theta^\inv (z)= \Jimq \Theta^\ad (\ko{z}) \Jimq\) holds true for each \(z \in \C \setminus (\rhl \cup \cD)\).
\end{lem}
\begin{proof}
 Let \(\Rref{\hol{\Theta}}\defeq \setaca{ z \in \Cs}{\ko{z} \in \hol{\Theta}}\).
 Then \(\Rref{\Theta}\colon\Rref{\hol{\Theta}} \to \Coo{2q}{2q}\) given by \(\Rref{\Theta}(z)\defeq \Theta^\ad (\ko{z})\) is meromorphic in \(\Cs \) with \(\hol{\Rref{\Theta}}= (\hol{\Theta})^\lor\).
 Thus, \(\Omega\defeq \Jimq- \Theta \Jimq \Rref{\Theta}\) is meromorphic in \(\Cs \) with \(\hol{\Omega} \supseteq \hol{\Theta} \cap \hol{\Rref{\Theta}}=\hol{\Theta}\).
 Because of \(\Theta \in \PJisf \), there is a discrete subset \(\cD\) of \(\Cs \) with \(\C \setminus (\rhl \cup \cD)\subseteq \hol{\Theta} \) such that \(\Omega\) is holomorphic in \(\C \setminus (\rhl \cup \cD)\) and that
\[
 \Omega(x)
 = \Jimq -\Theta(x) \Jimq \Rref{\Theta}(x)
 = \Jimq -\Theta(x) \Jimq \Theta^\ad (\ko{x})
 =\Jimq- \Theta(x) \Jimq \Theta^\ad (x)
 =\NM
\]
 holds true for each \(x \in \crhl  \setminus \cD\).
 Consequently, the identity theorem for holomorphic functions shows that \(\Theta(z) \Jimq \Rref{\Theta}(z) = \Jimq\) is valid for each \(z \in \hol{\Theta} \cap \hol{\Rref{\Theta}}\), which implies \( \Theta(z) \Jimq \Theta^\ad (\ko{z})\Jimq=\Theta(z) \Jimq \Rref{\Theta}(z) \Jimq=\Jimq^2 =\Iu{2q} \) for each \(z \in \hol{\Theta} \cap \hol{\Rref{\Theta}}\) and, in particular, for each \(z \in \C \setminus (\rhl \cup \cD)\).
 The rest is plain.
\end{proof}

\begin{lem} \label{OQ-1}
 Let \(\alpha \in \R\), let \(\kappa \in \Ninf \), and let \(\seqska  \in \Kggeqka \).
 Let \(\cG\) be a subset of \(\C\) with \(\cG\setminus \R \neq \emptyset\) and let \(f\colon\cG \to \Cqq\) be a \tqqa{matrix-valued} function.
 For each \(n\in\NO \) with \(2n + 1\le \kappa\), then \(\Sigma_{2n}^{[f]}\colon\cG \setminus \R\to \Cqq\) and \(\Sigma_{2n+1}^{[f]}\colon\cG\setminus\R \to \Cqq\) given by \eqref{N54} and \eqref{N55}
 admit, for each \(z\in \cG \setminus \R\), the representations 
\[
 \Sigma_{2n}^{[f]}(z)
 = \matp{f(z)}{\Iq}^\ad \Theta_{n,\alpha}^\invad (z) \rk*{\frac{-\Jimq}{2 \Im z}} \Theta_{n,\alpha}^\inv (z) \matp{f(z)}{\Iq}
\]
 and
\[
 \Sigma_{2n+1}^{[f]}(z)
 = \matp{(z-\alpha)f(z)}{\Iq }^\ad \tilde{\Theta}_{n,\alpha}^\invad (z) \rk*{\frac{-\Jimq}{2 \Im z}} \tilde{\Theta}_{n,\alpha}^\inv (z) \matp{ (z-\alpha)f(z)}{\Iq }.
\]
\end{lem}
\begin{proof}
 Use \rremsss{remark4.5}{LG1}{JN} and \rlem{BL4.17}.
\end{proof}

\begin{lem} \label{125} 
 Let \(\alpha \in \R\), let \(\kappa \in \Ninf \), and let \(\seqska  \in \Kggeqka \).
 For each \(n\in\NO \) with \(2n+1 \leq \kappa\) and all \(z \in \C\), then 
\begin{multline} \label{N264-1}
 (\Iu{(n+1)q} -  H_n^\mpi  H_n) R_{\Tqn}(z) \mat{ \Iu{(n+1)q}, \Tqn H_n } (\Iu{2} \otimes v_{q,n} ) \Theta_{n,\alpha}(z) \\
 =\ek*{ \Iu{(n+1)q}+(z-\alpha) \rk{\Iu{(n+1)q}-H_n^\mpi  H_n} \Tqn R_{\Tqn}(z) (\Iu{(n+1)q} -H_nH_n^\gi  )} \\
 \qquad\times \rk{\Iu{(n+1)q}-H_n^\mpi  H_n} R_{\Tqn} (\alpha) \mat{ \Iu{(n+1)q}, \Tqn H_n } (\Iu{2} \otimes v_{q,n} ) B_{n,\alpha}
\end{multline}
 and 
\begin{multline} \label{N264-2}
 (\Iu{(n+1)q} - H_{\at n}^\mpi H_{\at n}) R_{\Tqn}(z) \mat*{ \Iu{(n+1)q}, \ek*{R_{\Tqn}(\alpha)}^\inv H_n } (\Iu{2} \otimes v_{q,n} ) \tilde{\Theta}_{n,\alpha}(z) \\
 = \ek*{ \Iu{(n+1)q}+(z-\alpha) (\Iu{(n+1)q} -H_{\at n}^\mpi H_{\at n})\Tqn R_{\Tqn}(z) (\Iu{(n+1)q} - H_{\at n} H_{\at{n}}^\gi  )} \\ 
 \qquad\times (\Iu{(n+1)q} -H_{\at n}^\mpi H_{\at n}) R_{\Tqn} (\alpha) \mat*{ \Iu{(n+1)q}, \ek*{R_{\Tqn}(\alpha)}^\inv H_n } (\Iu{2} \otimes v_{q,n} ) \tilde{B}_{n,\alpha}.
\end{multline} 
\end{lem}
\begin{proof} 
 Let \(n \in \NO \) be such that \(2n+1 \leq \kappa\) and let \(z \in \C\).
 \rremss{ZS}{1582012-2} yield \(H_n^\ad =H_n\) and \(H_{\at{n}}^\ad =H_{\at{n}}\).
 Using \eqref{TUB-1}, \eqref{U1}, and \(R_{\Tqn^\ad } (z)=\ek{R_{\Tqn } (\ko{z})}^\ad \), we have
\beql{N265-1}\begin{split}
 &(\Iu{(n+1)q}-H_n^\mpi  H_n) R_{\Tqn}(z) \mat{ \Iu{(n+1)q}, \Tqn H_n } (\Iu{2} \otimes v_{q,n} ) \Theta_{n,\alpha}(z) \\
 &=(\Iu{(n+1)q} -H_n^\mpi  H_n) R_{\Tqn}(z) \mat{ \Iu{(n+1)q}, \Tqn H_n } (\Iu{2} \otimes v_{q,n} ) U_{n,\alpha}(z) B_{n,\alpha} \\
 &=(\Iu{(n+1)q} -H_n^\mpi  H_n) R_{\Tqn}(z) \mat{ \Iu{(n+1)q}, \Tqn H_n } (\Iu{2} \otimes v_{q,n} ) \\
 &\qquad\times \biggl\{\Iu{2q} + (z-\alpha) (\Iu{2} \otimes v_{q,n} )^\ad \mat{\Tqn H_n, -\Iu{(n+1)q}}^\ad \\
 &\qquad\qquad\times \ek*{R_{\Tqn } (\ko{z})}^\ad H_n^\gi  R_{\Tqn} (\alpha) \mat{ \Iu{(n+1)q}, \Tqn H_n} (\Iu{2} \otimes v_{q,n} )\biggr\} B_{n,\alpha} \\
 &=(\Iu{(n+1)q} -H_n^\mpi  H_n) \Phi(z) R_{\Tqn} (\alpha) \mat{ \Iu{(n+1)q}, \Tqn H_n } (\Iu{2} \otimes v_{q,n} ) B_{n,\alpha} 
\end{split}\eeq
 where
\begin{multline} \label{phi-t} 
 \Phi(z)
 \defeq R_{\Tqn}(z)\ek*{R_{\Tqn} (\alpha)}^\inv + (z-\alpha) R_{\Tqn}(z) \mat{ \Iu{(n+1)q}, \Tqn H_n }(\Iu{2} \otimes v_{q,n} ) \\
 \times  (\Iu{2} \otimes v_{q,n} )^\ad \mat{\Tqn H_n, -\Iu{(n+1)q}}^\ad \ek*{R_{\Tqn } (\ko{z})}^\ad H_n^\gi .
\end{multline} 
 Taking into account equation~\eqref{SKP4} in \rrem{JN}, \(H_n^\ad =H_n\), and \eqref{B204}, we obtain 
\beql{NM12Z} \begin{split} 
  & R_{\Tqn}(z)\mat{ \Iu{(n+1)q}, \Tqn H_n } (\Iu{2} \otimes v_{q,n} ) (\Iu{2} \otimes v_{q,n} )^\ad \mat{\Tqn H_n, -\Iu{(n+1)q}}^\ad \ek*{R_{\Tqn } (\ko{z})}^\ad \\
 & = R_{\Tqn}(z) \rk{v_{q,n}v_{q,n}^\ad H_n\Tqn^\ad -\Tqn H_n v_{q,n}v_{q,n}^\ad } \ek*{R_{\Tqn } (\ko{z})}^\ad \\
 & = R_{\Tqn}(z) \rk*{\ek*{R_{\Tqn}(z)}^\inv H_n \Tqn^\ad - \Tqn H_n \ek*{R_{\Tqn}(\ko{z})}^\invad } \ek*{R_{\Tqn } (\ko{z})}^\ad \\
 & = H_n \Tqn^\ad \ek*{R_{\Tqn } (\ko{z})}^\ad - R_{\Tqn}(z) \Tqn H_n.
\end{split}\eeq
 From \eqref{NM12Z}, \eqref{phi-t}, \eqref{N52}, and the identity \(R_{\Tqn}(z) \Tqn= \Tqn R_{\Tqn}(z) \) we get
\beql{125-t}\begin{split} 
 \Phi(z)
 &= R_{\Tqn}(z)\ek*{R_{\Tqn} (\alpha)}^\inv + (z-\alpha) \rk*{H_n \Tqn^\ad \ek*{R_{\Tqn } (\ko{z})}^\ad - R_{\Tqn}(z) \Tqn H_n } H_n^\gi  \\
 & = \Iu{(n+1)q}+(z-\alpha) \Tqn R_{\Tqn}(z) +(z-\alpha) H_n \Tqn^\ad \ek*{R_{\Tqn } (\ko{z})}^\ad H_n^\gi\\
 &\qquad- (z-\alpha) \Tqn R_{\Tqn}(z) H_nH_n^\gi  \\
 & = \Iu{(n+1)q}+(z-\alpha) \Tqn R_{\Tqn}(z) (\Iu{(n+1)q} -H_nH_n^\gi  ) +(z-\alpha) H_n \Tqn^\ad \ek*{R_{\Tqn } (\ko{z})}^\ad H_n^\gi .
\end{split}\eeq
 In view of \rrem{SBNNN}, consequently,
\beql{125-t-3}
 (z-\alpha)\rk{\Iu{(n+1)q}-H_n^\mpi  H_n} H_n \Tqn^\ad \ek*{R_{\Tqn } (\ko{z})}^\ad H_n^\gi  = 0.
\eeq
 By virtue of \eqref{125-t}, \rrem{SBNNN}, and \eqref{125-t-3}, we conclude
\beql{125-t-2}\begin{split} 
 &\rk{\Iu{(n+1)q}-H_n^\mpi  H_n} \Phi(z) R_{\Tqn} (\alpha) \mat{ \Iu{(n+1)q}, \Tqn H_n } (\Iu{2} \otimes v_{q,n} ) B_{n,\alpha} \\
 &= \rk{\Iu{(n+1)q}-H_n^\mpi  H_n} \\
 & \qquad\times\ek*{\Iu{(n+1)q}+(z-\alpha) \Tqn R_{\Tqn}(z) (\Iu{(n+1)q} -H_nH_n^\gi  ) +(z-\alpha) H_n \Tqn^\ad \ek{R_{\Tqn } (\ko{z})}^\ad H_n^\gi } \\
 & \qquad\times R_{\Tqn} (\alpha) \mat{ \Iu{(n+1)q}, \Tqn H_n } (\Iu{2} \otimes v_{q,n} ) B_{n,\alpha} \\
 &=\biggl \{\rk{\Iu{(n+1)q}-H_n^\mpi  H_n} \\
 & \qquad\qquad +(z-\alpha) \rk{\Iu{(n+1)q}-H_n^\mpi  H_n} \Tqn R_{\Tqn}(z) (\Iu{(n+1)q} -H_nH_n^\gi ) \rk{\Iu{(n+1)q}-H_n^\mpi  H_n} \\
 & \qquad\qquad+(z-\alpha) \rk{\Iu{(n+1)q}-H_n^\mpi  H_n} H_n \Tqn^\ad \ek*{R_{\Tqn } (\ko{z})}^\ad H_n^\gi\biggr \} \\
 & \qquad\times R_{\Tqn} (\alpha) \mat{ \Iu{(n+1)q}, \Tqn H_n } (\Iu{2} \otimes v_{q,n} ) B_{n,\alpha} \\ 
 &=\ek*{ \Iu{(n+1)q}+(z-\alpha) \rk{\Iu{(n+1)q}-H_n^\mpi  H_n} \Tqn R_{\Tqn}(z) (\Iu{(n+1)q} -H_nH_n^\gi  )} \\
 & \qquad\times \rk{\Iu{(n+1)q}-H_n^\mpi  H_n} R_{\Tqn} (\alpha) \mat{ \Iu{(n+1)q}, \Tqn H_n } (\Iu{2} \otimes v_{q,n} ) B_{n,\alpha}.
\end{split}\eeq
 The combination of \eqref{N265-1} and \eqref{125-t-2} yields \eqref{N264-1}.
 Furthermore, using \eqref{TUB-1}, \eqref{U2}, and \(R_{\Tqn^\ad } (z)=\ek{R_{\Tqn } (\ko{z})}^\ad \), we infer
\beql{TJ2-27}\begin{split} 
 &( \Iu{(n+1)q} - H_{\at n}^\mpi H_{\at n}) R_{\Tqn}(z) \mat*{ \Iu{(n+1)q}, \ek*{R_{\Tqn}(\alpha)}^\inv H_n } (\Iu{2} \otimes v_{q,n} ) \tilde{\Theta}_{n,\alpha}(z) \\
 &= (\Iu{(n+1)q} -H_{\at n}^\mpi H_{\at n}) R_{\Tqn}(z) \mat*{ \Iu{(n+1)q}, \ek*{R_{\Tqn}(\alpha)}^\inv H_n } (\Iu{2} \otimes v_{q,n} ) \tilde{U}_{n,\alpha}(z) \tilde{B}_{n,\alpha} \\
 &= (\Iu{(n+1)q} -H_{\at n}^\mpi H_{\at n}) R_{\Tqn}(z) \mat*{ \Iu{(n+1)q}, \ek*{R_{\Tqn}(\alpha)}^\inv H_n } (\Iu{2} \otimes v_{q,n} ) \\
 &\qquad\times \biggl\{\Iu{2q} + (z-\alpha) (\Iu{2} \otimes v_{q,n} )^\ad \mat*{\ek*{R_{\Tqn}(\alpha)}^\inv H_n, -\Iu{(n+1)q}}^\ad\ek*{R_{\Tqn} (\ko{z})}^\ad H_{\at{n}}^\gi \\
 & \qquad \qquad \times   R_{\Tqn} (\alpha) \mat*{ \Iu{(n+1)q},\ek*{R_{\Tqn}(\alpha)}^\inv H_n} (\Iu{2} \otimes v_{q,n} )\biggr\}\tilde{B}_{n,\alpha} \\
 &= (\Iu{(n+1)q} -H_{\at n}^\mpi H_{\at n}) \tilde{\Phi}(z) R_{\Tqn} (\alpha) \mat*{ \Iu{(n+1)q}, \ek*{R_{\Tqn}(\alpha)}^\inv H_n } (\Iu{2} \otimes v_{q,n} ) \tilde{B}_{n,\alpha}
\end{split}\eeq
 where 
\begin{multline} \label{ltr}
 \tilde{\Phi}(z)
 \defeq R_{\Tqn}(z)\ek*{R_{\Tqn} (\alpha)}^\inv + (z-\alpha) R_{\Tqn}(z) \mat*{ \Iu{(n+1)q}, \ek*{R_{\Tqn}(\alpha)}^\inv H_n } (\Iu{2} \otimes v_{q,n} ) \\
 \times(\Iu{2} \otimes v_{q,n} )^\ad \mat*{\ek*{R_{\Tqn}(\alpha)}^\inv H_n, -\Iu{(n+1)q}}^\ad \ek*{R_{\Tqn} (\ko{z})}^\ad H_{\at{n}}^\gi .
\end{multline}
 Because of identity~\eqref{SKP4} in \rrem{JN}, \(H_n^\ad =H_n\), and equation~\eqref{FIDW3} in \rrem{ZF1}, we obtain 
\[\begin{split}
 &R_{\Tqn}(z)  \mat*{ \Iu{(n+1)q}, \ek*{R_{\Tqn}(\alpha)}^\inv H_n }(\Iu{2} \otimes v_{q,n} )\\
 &\qquad\times(\Iu{2} \otimes v_{q,n} )^\ad\mat*{\ek*{R_{\Tqn}(\alpha)}^\inv H_n, -\Iu{(n+1)q}}^\ad \ek*{R_{\Tqn} (\ko{z})}^\ad \\
 & = R_{\Tqn}(z) \rk*{ v_{q,n} v_{q,n}^\ad H_n \ek*{R_{\Tqn}(\alpha)}^\invad - \ek*{R_{\Tqn}(\alpha)}^\inv H_n v_{q,n} v_{q,n}^\ad } \ek*{R_{\Tqn} (\ko{z})}^\ad \\ 
 &= R_{\Tqn}(z) \rk*{ \ek*{R_{\Tqn}(z)}^\inv H_{\at n} \Tqn ^\ad -\Tqn H_{\at n} \ek*{R_{\Tqn}(\ko{z})}^\invad } \ek*{R_{\Tqn} (\ko{z})}^\ad \\ 
 &= H_{\at n} \Tqn ^\ad \ek*{R_{\Tqn} (\ko{z})}^\ad - R_{\Tqn}(z)\Tqn H_{\at n},
\end{split}\]
 which, in view of \eqref{ltr}, \eqref{N52}, and the identity \(R_{\Tqn}(z) \Tqn= \Tqn R_{\Tqn}(z)\), implies
\beql{ltr-2}\begin{split} 
 &\tilde{\Phi}(z)\\
 & = R_{\Tqn}(z)\ek*{R_{\Tqn} (\alpha)}^\inv +(z-\alpha) \rk*{ H_{\at n} \Tqn ^\ad \ek*{R_{\Tqn} (\ko{z})}^\ad - R_{\Tqn}(z)\Tqn H_{\at n} } H_{\at{n}}^\gi  \\
 & = \Iu{(n+1)q}+(z-\alpha) \Tqn R_{\Tqn}(z) +(z-\alpha) H_{\at n} \Tqn ^\ad \ek*{R_{\Tqn} (\ko{z})}^\ad H_{\at{n}}^\gi\\
 &\qquad-(z-\alpha) \Tqn R_{\Tqn}(z) H_{\at n} H_{\at{n}}^\gi  \\
 & = \Iu{(n+1)q}+(z-\alpha) \Tqn R_{\Tqn}(z) (\Iu{(n+1)q} - H_{\at n} H_{\at{n}}^\gi  )\\
 &\qquad+(z-\alpha) H_{\at n} \Tqn ^\ad \ek*{R_{\Tqn} (\ko{z})}^\ad H_{\at{n}}^\gi .
\end{split}\eeq
 From \rrem{SBNNN} we see that
\beql{125-t4}
 (z-\alpha)(\Iu{(n+1)q}-H_{\at{n}}^\mpi  H_{\at{n}}) H_{\at n} \Tqn ^\ad \ek*{R_{\Tqn} (\ko{z})}^\ad H_{\at{n}}^\gi
 =\NM
\eeq
 is true.
 Using \eqref{ltr-2}, \rrem{SBNNN}, and \eqref{125-t4}, we get 
\beql{TJ2-30} \begin{split} 
 &(\Iu{(n+1)q} -H_{\at n}^\mpi H_{\at n}) \tilde{\Phi}(z) R_{\Tqn} (\alpha) \mat*{ \Iu{(n+1)q}, \ek*{R_{\Tqn}(\alpha)}^\inv H_n } (\Iu{2} \otimes v_{q,n} ) \tilde{B}_{n,\alpha} \\
 &=(\Iu{(n+1)q} -H_{\at n}^\mpi H_{\at n})\biggl\{\Iu{(n+1)q}+(z-\alpha) \Tqn R_{\Tqn}(z) (\Iu{(n+1)q} - H_{\at n} H_{\at{n}}^\gi  )\\
 &\qquad\qquad+(z-\alpha) H_{\at n} \Tqn ^\ad \ek*{R_{\Tqn} (\ko{z})}^\ad H_{\at{n}}^\gi\biggr \} \\ 
 &\qquad\times R_{\Tqn} (\alpha) \mat*{ \Iu{(n+1)q}, \ek*{R_{\Tqn}(\alpha)}^\inv H_n } (\Iu{2} \otimes v_{q,n} ) \tilde{B}_{n,\alpha} \\ 
 &=\biggl \{(\Iu{(n+1)q} -H_{\at n}^\mpi H_{\at n}) +(z-\alpha) (\Iu{(n+1)q} -H_{\at n}^\mpi H_{\at n}) \Tqn R_{\Tqn}(z) \\ &\qquad\qquad\qquad\times(\Iu{(n+1)q} - H_{\at n} H_{\at{n}}^\gi  )(\Iu{(n+1)q} -H_{\at n}^\mpi H_{\at n}) \\ 
 & \qquad\qquad +(z-\alpha) (\Iu{(n+1)q} -H_{\at n}^\mpi H_{\at n}) H_{\at n} \Tqn ^\ad \ek*{R_{\Tqn} (\ko{z})}^\ad H_{\at{n}}^\gi \biggr \} \\ 
 & \qquad\times R_{\Tqn} (\alpha) \mat{ \Iu{(n+1)q}, \ek{R_{\Tqn}(\alpha)}^\inv H_n } (\Iu{2} \otimes v_{q,n} ) \tilde{B}_{n,\alpha} \\ 
 & =\ek*{\Iu{(n+1)q}+(z-\alpha) (\Iu{(n+1)q} -H_{\at n}^\mpi H_{\at n})\Tqn R_{\Tqn}(z) (\Iu{(n+1)q} - H_{\at n} H_{\at{n}}^\gi  )} \\ 
 & \qquad\times (\Iu{(n+1)q} -H_{\at n}^\mpi H_{\at n}) R_{\Tqn} (\alpha) \mat*{ \Iu{(n+1)q}, \ek*{R_{\Tqn}(\alpha)}^\inv H_n } (\Iu{2} \otimes v_{q,n} ) \tilde{B}_{n,\alpha}.
\end{split}\eeq
 The combination of \eqref{TJ2-27} and \eqref{TJ2-30} provides us \eqref{N264-2}.
 \end{proof}

\begin{lem} \label{133} 
 Let \(\alpha \in \R\), let \(\kappa \in \Ninf \), and let \(\seqska  \in \Kggeqka \).
 For each \(n\in\NO \) such that \(2n+1 \leq \kappa\), then 
\begin{multline*} 
 (\Iu{(n+1)q} -H_n^\mpi  H_n) R_{\Tqn}(\alpha) \mat{ \Iu{(n+1)q}, \Tqn H_n} (\Iu{2} \otimes v_{q,n} ) B_{n,\alpha} \\
 = (\Iu{(n+1)q} -H_n^\mpi  H_n)R_{\Tqn}(\alpha) \mat*{ \Iu{(n+1)q}, \Tqn (\Iu{(n+1)q} -H_{\at n}H_{\at n}^\gi ) H_n } (\Iu{2} \otimes v_{q,n} )
\end{multline*} 
 and 
\begin{multline*} 
 (\Iu{(n+1)q} -H_{\at n}^\mpi  H_{\at n}) R_{\Tqn}(\alpha) \mat*{ \Iu{(n+1)q}, \ek*{R_{\Tqn}(\alpha)}^\inv H_n} (\Iu{2} \otimes v_{q,n} ) \tilde{B}_{n,\alpha} \\
 = (\Iu{(n+1)q} -H_{\at n}^\mpi  H_{\at n}) \mat*{(\Iu{(n+1)q} -H_n H_n^\gi ) R_{\Tqn}(\alpha), H_n } (\Iu{2} \otimes v_{q,n} ).
\end{multline*} 
\end{lem}
\begin{proof}
 Let \(n\in \NO \) be such that \(2n+1\leq \kappa\).
 From \rremss{ZS}{1582012-2} we get \(H_n^\ad =H_n\) and \(H_{\at{n}}^\ad =H_{\at{n}}\).
 Because of the \rremss{lemC21-1}{SBNNN}, we have 
\beql{N247-1} \begin{split} 
 &(\Iu{(n+1)q}  - H_n^\mpi  H_n) R_{\Tqn}(\alpha)\rk{v_{q,n} v_{q,n}^\ad H_n H_{\at{n}}^\gi +\Tqn} \\
 &= (\Iu{(n+1)q} - H_n^\mpi  H_n) R_{\Tqn}(\alpha) \ek*{ \rk*{ \ek{R_{\Tqn} (\alpha)}^\inv H_n -\Tqn H_{\at{n}}} H_{\at{n}}^\gi  +\Tqn } \\
 &=(\Iu{(n+1)q} - H_n^\mpi  H_n) H_n H_{\at{n}}^\gi  - (\Iu{(n+1)q} - H_n^\mpi  H_n) R_{\Tqn}(\alpha) \Tqn H_{\at{n}} H_{\at{n}}^\gi\\
 &\qquad+ (\Iu{(n+1)q} - H_n^\mpi  H_n) R_{\Tqn}(\alpha) \Tqn \\
 &=(\Iu{(n+1)q} - H_n^\mpi  H_n) R_{\Tqn}(\alpha)\Tqn (\Iu{(n+1)q}-H_{\at{n}} H_{\at{n}}^\gi ).
\end{split}\eeq
 Applying \rrem{115} and \eqref{N247-1}, we conclude
\[\begin{split}
 & (\Iu{(n+1)q} - H_n^\mpi  H_n) R_{\Tqn}(\alpha) \mat{ \Iu{(n+1)q}, \Tqn H_n} (\Iu{2} \otimes v_{q,n} ) B_{n,\alpha} \\
 & = (\Iu{(n+1)q} -H_n^\mpi  H_n) R_{\Tqn}(\alpha) \mat*{ \Iu{(n+1)q}, \rk{v_{q,n} v_{q,n}^\ad H_n H_{\at{n}}^\gi  +\Tqn} H_n } (\Iu{2} \otimes v_{q,n} ) \\
 & =\mat*{(\Iu{(n+1)q} -H_n^\mpi  H_n) R_{\Tqn}(\alpha),(\Iu{(n+1)q} - H_n^\mpi  H_n) R_{\Tqn}(\alpha) \rk{v_{q,n} v_{q,n}^\ad H_n H_{\at{n}}^\gi +\Tqn} H_n}\\
 &\qquad\times(\Iu{2} \otimes v_{q,n} ) \\ 
 & =\mat*{(\Iu{(n+1)q} -H_n^\mpi  H_n) R_{\Tqn}(\alpha), (\Iu{(n+1)q} - H_n^\mpi  H_n) R_{\Tqn}(\alpha) \Tqn (\Iu{(n+1)q}-H_{\at{n}} H_{\at{n}}^\gi ) H_n}\\
 &\qquad\times(\Iu{2} \otimes v_{q,n} ) \\
 & = (\Iu{(n+1)q} -H_n^\mpi  H_n) R_{\Tqn}(\alpha) \mat*{ \Iu{(n+1)q}, \Tqn (\Iu{(n+1)q} -H_{\at n}H_{\at n}^\gi ) H_n } (\Iu{2} \otimes v_{q,n} ).
\end{split}\]
 Taking into account \(H_n^\ad =H_n\), \(H_{\at{n}}^\ad =H_{\at{n}}\), and \rrem{lemC21-1}, we obtain \(H_n v_{q,n} v_{q,n}^\ad =H_n \ek{R_{\Tqn}(\alpha)}^\invad - H_{\at n} \Tqn ^\ad \) and, hence,
\begin{multline} \label{NK122} 
 \Iu{(n+1)q} - H_n v_{q,n} v_{q,n}^\ad \ek*{R_{\Tqn}(\alpha)}^\ad  H_n^\gi\\
 =\Iu{(n+1)q} -\rk*{H_n \ek*{R_{\Tqn}(\alpha)}^\invad - H_{\at n} \Tqn ^\ad } \ek*{R_{\Tqn}(\alpha)}^\ad  H_n^\gi  \\
 =\Iu{(n+1)q} -H_n H_n^\gi  + H_{\at n} \Tqn ^\ad \ek*{R_{\Tqn}(\alpha)}^\ad  H_n^\gi.
\end{multline}
 Let \(P\defeq\Iu{(n+1)q} -H_{\at n}^\mpi  H_{\at n}\).
 From \eqref{NK122} and \rrem{SBNNN} we see that
\begin{multline} \label{TT-115}
 P \rk*{ \Iu{(n+1)q} - H_n v_{q,n} v_{q,n}^\ad \ek*{R_{\Tqn}(\alpha)}^\ad  H_n^\gi  } R_{\Tqn}(\alpha) \\
 =P \rk*{\Iu{(n+1)q} -H_n H_n^\gi  + H_{\at n} \Tqn^\ad \ek*{R_{\Tqn}(\alpha)}^\ad  H_n^\gi } R_{\Tqn}(\alpha)\\
 =P (\Iu{(n+1)q} -H_n H_n^\gi )R_{\Tqn}(\alpha) 
\end{multline}
 holds true.
 Using \rrem{115} and \eqref{TT-115}, we obtain finally
\[\begin{split} 
 &P R_{\Tqn}(\alpha) \mat*{ \Iu{(n+1)q}, \ek*{R_{\Tqn}(\alpha)}^\inv H_n} (\Iu{2} \otimes v_{q,n} ) \tilde{B}_{n,\alpha}\\
 &=P R_{\Tqn}(\alpha)\ek*{R_{\Tqn}(\alpha)}^\inv \mat*{\ek*{\Iu{(n+1)q} - H_n v_{q,n} v_{q,n}^\ad R_{\Tqn^\ad }(\alpha) H_n^\gi} R_{\Tqn}(\alpha), H_n } (\Iu{2} \otimes v_{q,n} )\\
 &=\mat*{P\ek*{\Iu{(n+1)q} - H_n v_{q,n} v_{q,n}^\ad R_{\Tqn^\ad }(\alpha) H_n^\gi} R_{\Tqn}(\alpha),P H_n}(\Iu{2} \otimes v_{q,n} )\\
 &= \mat*{ P(\Iu{(n+1)q} -H_n H_n^\gi )R_{\Tqn}(\alpha), P H_n}(\Iu{2} \otimes v_{q,n} )\\
 &= P \mat*{(\Iu{(n+1)q} -H_n H_n^\gi ) R_{\Tqn}(\alpha), H_n } (\Iu{2} \otimes v_{q,n} ).\qedhere
\end{split}\] 
\end{proof}

\begin{lem} \label{140} 
 Let \(\alpha \in \R\), let \(\kappa \in \Ninf \), and let \(\seqska  \in \Kggeqka \).
 Let \(n\in\NO \) be such that \(2n+1 \leq \kappa\) and let the matrix-valued functions \(P_{n,\alpha}\), \(Q_{n,\alpha}\), and \(S_{n,\alpha}\) be defined on \(\C\) and, for every choice of \(z\in\C\), be given by 
\begin{align} 
 P_{n,\alpha}(z)&\defeq \Iu{(n+1)q} +(z- \alpha )(\Iu{(n+1)q} -H_n^\mpi  H_n) \Tqn R_{\Tqn}(z) (\Iu{(n+1)q} -H_n H_n^\gi ), \label{V1} \\
 Q_{n,\alpha}(z)&\defeq \Iu{(n+1)q} +(z- \alpha )(\Iu{(n+1)q} -H_{\at n}^\mpi  H_{\at n}) \Tqn R_{\Tqn}(z) (\Iu{(n+1)q} -H_{\at n} H_{\at n}^\gi ) \label{V2}, 
\intertext{and} 
 S_{n,\alpha}(z)&\defeq \Iu{(n+1)q} -(z- \alpha )(\Iu{(n+1)q} -H_{\at n}^\mpi  H_{\at n}) R_{\Tqn} (\alpha) \Tqn (\Iu{(n+1)q} -H_{\at n} H_{\at n}^\gi ).\label{V3} 
\end{align}
 For each \(z \in \C\), then
\beql{MU124}\begin{split} 
 &\diag\mat*{P_{n,\alpha}(z), Q_{n,\alpha}(z) } \\
 & \qquad\times
 \begin{pmat}[{|}]
  \Iu{(n+1)q} & \Ouu{(n+1)q}{(n+1)q} \cr\-
  (z- \alpha )(\Iu{(n+1)q} -H_{\at n}^\mpi  H_{\at n})(\Iu{(n+1)q} -H_n H_n^\gi ) & S_{n,\alpha}(z)\cr 
 \end{pmat}\\ 
 & \qquad\times
 \begin{pmat}[{|}]
  \Iu{(n+1)q} & (\Iu{(n+1)q} -H_n^\mpi  H_n)R_{\Tqn}(\alpha) \Tqn (\Iu{(n+1)q} -H_{\at n } H_{\at n }^\gi ) \cr\-
  \Ouu{(n+1)q}{(n+1)q} & \Iu{(n+1)q}\cr
 \end{pmat}\\
 & \qquad\times \diag\rk*{(\Iu{(n+1)q} -H_n^\mpi  H_n) R_{\Tqn}(\alpha) v_{q,n}, (\Iu{(n+1)q} -H_{\at n }^\mpi  H_{\at n }) H_n v_{q,n} } \\
 &=
 \begin{pmat}[{}]
  (\Iu{(n+1)q} -H_n^\mpi  H_n) R_{\Tqn}(z) \mat{ \Iu{(n+1)q}, \Tqn H_n} (\Iu{2} \otimes v_{q,n} ) \Theta_{n,\alpha} (z)\cr\-
  \begin{gathered}
   (\Iu{(n+1)q} -H_{\at n}^\mpi  H_{\at n}) R_{\Tqn}(z) \mat{ \Iu{(n+1)q}, \ek{R_{\Tqn}(\alpha)}^\inv H_n} (\Iu{2} \otimes v_{q,n} )\\
   \times\tilde{\Theta}_{n,\alpha}(z)\diag\rk{ (z-\alpha)\Iq,\Iq }
  \end{gathered}\cr
 \end{pmat}. 
\end{split}\eeq
\end{lem}
\begin{proof} 
 Let \(z \in \C\).
 Obviously, the matrix on the left-hand side of \eqref{MU124} coincides with 
\[
 R_{n,\alpha} (z)
 \defeq \diag\rk*{ P_{n,\alpha}(z), Q_{n,\alpha}(z) } 
 \bMat
  \Psi_{n, \alpha}^{(1,1)}(z) &\Psi_{n, \alpha}^{(1,2)}(z) \\
  \Psi_{n, \alpha}^{(2,1)}(z) &\Psi_{n, \alpha}^{(2,2)}(z)
 \eMat(\Iu{2} \otimes v_{q,n} )
\]
 where
\begin{align} 
 \Psi_{n, \alpha}^{(1,1)}(z)
 &\defeq (\Iu{(n+1)q} -H_n^\mpi  H_n) R_{\Tqn}(\alpha), \label{F11} \\
 \Psi_{n, \alpha}^{(1,2)}(z)
 &\defeq (\Iu{(n+1)q} -H_n^\mpi  H_n) R_{\Tqn}(\alpha) \Tqn\notag\\
 &\qquad\times(\Iu{(n+1)q} -H_{\at n } H_{\at n }^\gi ) (\Iu{(n+1)q} -H_{\at n }^\mpi  H_{\at n }) H_n, \label{F12} \\
 \Psi_{n, \alpha}^{(2,1)}(z)
 &\defeq (z- \alpha )(\Iu{(n+1)q} -H_{\at n}^\mpi  H_{\at n})\notag\\
 &\qquad\times(\Iu{(n+1)q} -H_n H_n^\gi ) (\Iu{(n+1)q} -H_n^\mpi  H_n) R_{\Tqn}(\alpha), \label{F21}
\intertext{and}
 \Psi_{n, \alpha}^{(2,2)}(z)
 &\defeq (z- \alpha )(\Iu{(n+1)q} -H_{\at n}^\mpi  H_{\at n})(\Iu{(n+1)q} -H_n H_n^\gi ) (\Iu{(n+1)q} -H_n^\mpi  H_n)\notag \\
 &\qquad\qquad\times R_{\Tqn} (\alpha) \Tqn (\Iu{(n+1)q} -H_{\at n } H_{\at n }^\gi ) (\Iu{(n+1)q} -H_{\at n }^\mpi  H_{\at n }) H_n\notag \\
 &\qquad+S_{n,\alpha}(z)(\Iu{(n+1)q} -H_{\at n }^\mpi  H_{\at n }) H_n\label{F22}. 
\end{align}
 Because of \eqref{F12} and \rrem{SBNNN}, we have
\beql{F12-2} 
 \Psi_{n, \alpha}^{(1,2)}(z)
 = (\Iu{(n+1)q} -H_n^\mpi  H_n) R_{\Tqn}(\alpha) \Tqn (\Iu{(n+1)q} -H_{\at n } H_{\at n }^\gi ) H_n. 
\eeq
 Furthermore, \eqref{F21} and \rrem{SBNNN} yield
\beql{F21-2} 
 \Psi_{n, \alpha}^{(2,1)}(z)
 = (z- \alpha )(\Iu{(n+1)q} -H_{\at n}^\mpi  H_{\at n}) (\Iu{(n+1)q} -H_n H_n^\gi ) R_{\Tqn}(\alpha).
\eeq
 From \rrem{SBNNN}, \rlem{ML8-1}, and \eqref{V3} we conclude 
\beql{F22-2-t}\begin{split}
 &(z - \alpha ) (\Iu{(n+1)q} -H_{\at n}^\mpi  H_{\at n}) (\Iu{(n+1)q} -H_n H_n^\gi )(\Iu{(n+1)q} -H_n^\mpi  H_n) \\
 & \qquad\times R_{\Tqn} (\alpha) \Tqn (\Iu{(n+1)q} -H_{\at n } H_{\at n }^\gi ) \\
 &=(z- \alpha ) (\Iu{(n+1)q} -H_{\at n}^\mpi  H_{\at n}) (\Iu{(n+1)q} -H_n H_n^\gi ) \\
 & \qquad\times  R_{\Tqn} (\alpha) \Tqn (\Iu{(n+1)q} -H_{\at n } H_{\at n }^\gi ) \\ 
 &= (z- \alpha ) (\Iu{(n+1)q} -H_{\at n}^\mpi  H_{\at n}) R_{\Tqn} (\alpha) \Tqn (\Iu{(n+1)q} -H_{\at n } H_{\at n }^\gi )\\
 & \qquad -(z- \alpha ) (\Iu{(n+1)q} -H_{\at n}^\mpi  H_{\at n}) H_n H_n^\gi  R_{\Tqn} (\alpha) \Tqn (\Iu{(n+1)q} -H_{\at n } H_{\at n }^\gi ) \\ 
 & = (z- \alpha ) (\Iu{(n+1)q} -H_{\at n}^\mpi  H_{\at n}) R_{\Tqn} (\alpha) \Tqn (\Iu{(n+1)q} -H_{\at n } H_{\at n }^\gi )\\ 
 & = \Iu{(n+1)q}-S_{n,\alpha}(z). 
\end{split}\eeq
 Combining \eqref{F22} and \eqref{F22-2-t}, we obtain 
\beql{F22-2}\begin{split}
  &\Psi_{n, \alpha}^{(2,2)}(z)\\
 &= (\Iu{(n+1)q}-S_{n,\alpha}(z))(\Iu{(n+1)q} -H_{\at n }^\mpi  H_{\at n }) H_n +S_{n,\alpha}(z)(\Iu{(n+1)q} -H_{\at n }^\mpi  H_{\at n }) H_n\\
 &= (\Iu{(n+1)q} -H_{\at n }^\mpi  H_{\at n })H_n.
\end{split}\eeq  
 By virtue of \rlem{125}, \eqref{V1}, \rlem{133}, \eqref{F11}, and \eqref{F12-2}, we get
\beql{F31}\begin{split}
 &(\Iu{(n+1)q} -  H_n^\mpi  H_n) R_{\Tqn}(z) \mat{ \Iu{(n+1)q}, \Tqn H_n } (\Iu{2} \otimes v_{q,n} ) \Theta_{n,\alpha}(z) \\
 &=\ek*{\Iu{(n+1)q}+(z- \alpha )(\Iu{(n+1)q} -H_n^\mpi  H_n) \Tqn R_{\Tqn}(z) (\Iu{(n+1)q} -H_n H_n^\gi )} \\
 &\qquad\times (\Iu{(n+1)q} -H_n^\mpi  H_n) R_{\Tqn}(\alpha) \mat{ \Iu{(n+1)q}, \Tqn H_n } (\Iu{2} \otimes v_{q,n} ) B_{n,\alpha} \\ 
 &=P_{n,\alpha}(z) (\Iu{(n+1)q} -H_n^\mpi  H_n) R_{\Tqn}(\alpha) \mat{ \Iu{(n+1)q}, \Tqn H_n } (\Iu{2} \otimes v_{q,n} ) B_{n,\alpha} \\ 
 &=P_{n,\alpha}(z) (\Iu{(n+1)q} -H_n^\mpi  H_n)R_{\Tqn}(\alpha) \mat{ \Iu{(n+1)q}, \Tqn (\Iu{(n+1)q} -H_{\at n}H_{\at n}^\gi ) H_n }(\Iu{2} \otimes v_{q,n} )\\
 &=P_{n,\alpha}(z) \mat*{ \Psi_{n,\alpha}^{(1,1)}(z), \Psi_{n,\alpha}^{(1,2)}(z) } (\Iu{2} \otimes v_{q,n} ).
\end{split}\eeq
 Similarly, \rlem{125}, \eqref{V2}, \rlem{133}, \eqref{F21-2}, and \eqref{F22-2} provide us
\beql{F32}\begin{split} 
 &(\Iu{(n+1)q} -H_{\at n}^\mpi  H_{\at n}) R_{\Tqn}(z) \mat*{ \Iu{(n+1)q}, \ek*{R_{\Tqn}(\alpha)}^\inv H_n}\\
 &\qquad\times(\Iu{2} \otimes v_{q,n} )\tilde{\Theta}_{n,\alpha}(z)\diag\rk*{(z-\alpha)\Iq,\Iq} \\
 &=\ek*{\Iu{(n+1)q}+(z- \alpha )(\Iu{(n+1)q} -H_{\at n}^\mpi  H_{\at n}) \Tqn R_{\Tqn}(z) (\Iu{(n+1)q} -H_{\at n} H_{\at n}^\gi )} \\
 &\qquad\times (\Iu{(n+1)q} -H_{\at n}^\mpi  H_{\at n}) R_{\Tqn}(\alpha) \mat*{ \Iu{(n+1)q}, \ek*{R_{\Tqn}(\alpha)}^\inv H_n }\\
 &\qquad\times(\Iu{2} \otimes v_{q,n} )\tilde{B}_{n,\alpha} \diag\rk*{(z-\alpha)\Iq,\Iq} \\
 &=Q_{n,\alpha}(z) (\Iu{(n+1)q} -H_{\at n}^\mpi  H_{\at n}) R_{\Tqn}(\alpha) \mat*{ \Iu{(n+1)q}, \ek*{R_{\Tqn}(\alpha)}^\inv H_n }\\
 &\qquad\times(\Iu{2} \otimes v_{q,n} )\tilde{B}_{n,\alpha}  \diag\rk*{(z-\alpha)\Iq,\Iq} \\
 &=Q_{n,\alpha}(z) (\Iu{(n+1)q} -H_{\at n}^\mpi  H_{\at n}) \mat*{(\Iu{(n+1)q} -H_n H_n^\gi ) R_{\Tqn}(\alpha), H_n }\\
 &\qquad\times(\Iu{2} \otimes v_{q,n} )\diag\rk*{(z-\alpha)\Iq,\Iq} \\ 
 &=Q_{n,\alpha}(z) (\Iu{(n+1)q} -H_{\at n}^\mpi  H_{\at n}) \mat*{(z-\alpha)(\Iu{(n+1)q} -H_n H_n^\gi ) R_{\Tqn}(\alpha), H_n }(\Iu{2} \otimes v_{q,n} ) \\
 &=Q_{n,\alpha}(z) \mat*{ \Psi_{n,\alpha}^{(2,1)}(z), \Psi_{n,\alpha}^{(2,2)}(z) } (\Iu{2} \otimes v_{q,n} ). 
\end{split}\eeq
 Since \(R_{n,\alpha} (z)\) coincides with the matrix on the left-hand side of \eqref{MU124} from \eqref{F31} and \eqref{F32}, equation~\eqref{MU124} follows.
\end{proof}

 If \(\cG\) is a \tne{} subset of \(\C\) and if \(f\colon\cG\to\C\) is a function, then let \(\zer{f}\defeq \setaca{z\in\cG}{f(z) = 0}\).

\begin{lem} \label{145} 
 Let \(\alpha \in \R\), let \(\kappa \in \Ninf \), let \(\seqska  \in \Kggeqka \), and let \(n\in\NO \) be such that \(2n+1 \leq \kappa\).
 Then:
\begin{enui}
 \item\label{145.a} The set \(\zer{\det P_{n,\alpha}} \cup \zer{\det Q_{n,\alpha}} \cup \zer{\det S_{n,\alpha}}\) is finite.
 \item\label{145.b} Let \(x \in \Cqq\) and let \(y \in \Cqq\).
 Then the following statements are equivalent: 
\begin{aeqii}{0}
 \item\label{145.i} For each \(z \in \C \setminus (\zer{\det P_{n,\alpha}} \cup \zer{\det Q_{n,\alpha}} \cup \zer{\det S_{n,\alpha}})\), the equations
\beql{L145.ia}
 (\Iu{(n+1)q} -H_n^\mpi  H_n) R_{\Tqn}(z) \mat{ \Iu{(n+1)q}, \Tqn H_n }(\Iu{2} \otimes v_{q,n} ) \Theta_{n,\alpha} (z) \matp{ x}{ y } 
 =\NM
\eeq
 and
\begin{multline}\label{L145.ib}
 (\Iu{(n+1)q} - H_{\at n}^\mpi  H_{\at n}) R_{\Tqn}(z) \mat*{ \Iu{(n+1)q}, \ek*{R_{\Tqn}(\alpha)}^\inv H_n }\\
 \times(\Iu{2} \otimes v_{q,n} )\tilde{\Theta}_{n,\alpha}(z)\diag\rk*{ (z-\alpha)\Iq,\Iq } \matp{x}{y }
 =\NM
\end{multline} 
 hold true.
 \item\label{145.ii} There exists a number \(z \in \C \setminus (\zer{\det P_{n,\alpha}} \cup \zer{\det Q_{n,\alpha}} \cup \zer{\det S_{n,\alpha}})\) such that \eqref{L145.ia} and \eqref{L145.ib} are valid.
 \item\label{145.iii} The equations 
\begin{align} 
 (\Iu{(n+1)q} -H_n^\mpi  H_n) R_{\Tqn}(\alpha) v_{q,n} x&=\NM\label{L145.iia} 
\intertext{and} 
 (\Iu{(n+1)q} -H_{\at n }^\mpi  H_{\at n }) H_n v_{q,n} y&=\NM\label{L145.iib}
\end{align} 
 are fulfilled.
\end{aeqii}
\end{enui} 
\end{lem}
\begin{proof} 
 By virtue of \rrem{21112N}, \eqref{V1}, \eqref{V2}, and \eqref{V3}, we see that \(P_{n,\alpha}\), \(Q_{n,\alpha}\), and \(S_{n,\alpha}\) are matrix polynomials with \(P_{n,\alpha}(\alpha)= \Iu{(n+1)q}\), \(Q_{n,\alpha}(\alpha)= \Iu{(n+1)q}\), and \(S_{n,\alpha}(\alpha)= \Iu{(n+1)q}\).
 In particular, \(\det P_{n,\alpha}\), \(\det Q_{n,\alpha}\), and \(\det S_{n,\alpha}\) are polynomials which do not vanish identically.
 In view of the fundamental theorem of algebra, the proof of \rpart{145.a} is complete.
 For each \(z \in \C\), \rlem{140} provides us 
\beql{N286-1}\begin{split} 
 &
 \begin{pmat}[{}]
  (\Iu{(n+1)q} -H_n^\mpi  H_n) R_{\Tqn}(z) \mat{ \Iu{(n+1)q}, \Tqn H_n } (\Iu{2} \otimes v_{q,n} ) \Theta_{n,\alpha} (z) \tmatp{x}{y } \cr\-
  \begin{gathered}
   (\Iu{(n+1)q} -H_{\at n}^\mpi  H_{\at n}) R_{\Tqn}(z) \mat{ \Iu{(n+1)q}, \ek{R_{\Tqn}(\alpha)}^\inv H_n }\\
   \times(\Iu{2} \otimes v_{q,n} )\tilde{\Theta}_{n,\alpha}(z)\diag\rk{ (z-\alpha)\Iq,\Iq  }\tmatp{x}{y }
  \end{gathered}\cr
 \end{pmat}\\
 &=
 \begin{pmat}[{}]
  (\Iu{(n+1)q} -H_n^\mpi  H_n) R_{\Tqn}(z) \mat{ \Iu{(n+1)q}, \Tqn H_n} (\Iu{2} \otimes v_{q,n} ) \Theta_{n,\alpha} (z) \cr\-
 \begin{gathered}
  (\Iu{(n+1)q} -H_{\at n}^\mpi  H_{\at n}) R_{\Tqn}(z) \mat{ \Iu{(n+1)q}, \ek{R_{\Tqn}(\alpha)}^\inv H_n}  \\
 \times (\Iu{2} \otimes v_{q,n} )\tilde{\Theta}_{n,\alpha}(z)\diag\rk{ (z-\alpha)\Iq,\Iq } 
 \end{gathered}\cr
 \end{pmat}
 \matp{x}{y }\\ 
 &=
 \diag\rk*{P_{n,\alpha}(z), Q_{n,\alpha}(z)} \\
 & \qquad\times
 \begin{pmat}[{|}]
  \Iu{(n+1)q} & \Ouu{(n+1)q}{(n+1)q} \cr\-
 (z- \alpha )(\Iu{(n+1)q} -H_{\at n}^\mpi  H_{\at n})(\Iu{(n+1)q} -H_n H_n^\gi ) & S_{n,\alpha}(z)\cr
 \end{pmat}\\
 & \qquad\times
 \begin{pmat}[{|}]
  \Iu{(n+1)q} & (\Iu{(n+1)q} -H_n^\mpi  H_n)R_{\Tqn}(\alpha) \Tqn (\Iu{(n+1)q} -H_{\at n } H_{\at n }^\gi ) \cr\-
 \Ouu{(n+1)q}{(n+1)q} & \Iu{(n+1)q}\cr
 \end{pmat}\\
 & \qquad\times
 \begin{pmat}[{}]
 (\Iu{(n+1)q} -H_n^\mpi  H_n) R_{\Tqn}(\alpha) v_{q,n} x \cr\-
 (\Iu{(n+1)q} -H_{\at n }^\mpi  H_{\at n }) H_n v_{q,n} y \cr
 \end{pmat}.
\end{split}\eeq

\begin{imp}{145.i}{145.ii}
 This implication is trivial.
\end{imp}

\begin{imp}{145.ii}{145.iii}
 According to~\ref{145.ii}, there exists a  number
\beql{N284-33}
 z
 \in \C \setminus (\zer{\det P_{n,\alpha}} \cup \zer{\det Q_{n,\alpha}} \cup \zer{\det S_{n,\alpha}})
\eeq
 such that \eqref{L145.ia} and \eqref{L145.ib} hold true.
 Using \eqref{L145.ia}, \eqref{L145.ib}, and \eqref{N286-1}, we get 
\beql{N284-2}\begin{split} 
 &\matp{\Ouu{(n+1)q}{q}}{\Ouu{(n+1)q}{q} }\\
 &= \diag\rk*{P_{n,\alpha}(z), Q_{n,\alpha}(z) } \\
 & \qquad\times
 \begin{pmat}[{|}]
  \Iu{(n+1)q} & \Ouu{(n+1)q}{(n+1)q} \cr\-
 (z- \alpha )(\Iu{(n+1)q} -H_{\at n}^\mpi  H_{\at n})(\Iu{(n+1)q} -H_n H_n^\gi ) & S_{n,\alpha}(z) \cr
 \end{pmat}\\ 
 & \qquad\times
 \begin{pmat}[{|}]
  \Iu{(n+1)q} & (\Iu{(n+1)q} -H_n^\mpi  H_n)R_{\Tqn}(\alpha) \Tqn (\Iu{(n+1)q} -H_{\at n } H_{\at n }^\gi ) \cr\-
\Ouu{(n+1)q}{(n+1)q} & \Iu{(n+1)q} \cr
 \end{pmat}\\
 & \qquad\times
 \begin{pmat}[{}]
  (\Iu{(n+1)q} -H_n^\mpi  H_n) R_{\Tqn}(\alpha) v_{q,n} x \cr\-
  (\Iu{(n+1)q} -H_{\at n }^\mpi  H_{\at n }) H_n v_{q,n} y \cr
 \end{pmat}. 
\end{split}\eeq
 Because of \eqref{N284-33}, the first three factors of the matrix product on the right-hand side of equation~\eqref{N284-2} are non-singular matrices.
 Thus, \eqref{N284-2} implies \eqref{L145.iia} and \eqref{L145.iib}.
\end{imp}

\begin{imp}{145.iii}{145.i}
 Taking into account \eqref{L145.iia}, \eqref{L145.iib}, and \eqref{N286-1}, we conclude that
\[\begin{split}
 &\matp{\Ouu{(n+1)q}{q}}{\Ouu{(n+1)q}{q} }\\
 &= \diag\rk*{ P_{n,\alpha}(z), Q_{n,\alpha}(z) } \\ 
 &\qquad\times
 \begin{pmat}[{|}]
  \Iu{(n+1)q} & \Ouu{(n+1)q}{(n+1)q} \cr\-
  (z- \alpha )(\Iu{(n+1)q} -H_{\at n}^\mpi  H_{\at n})(\Iu{(n+1)q} -H_n H_n^\gi ) & S_{n,\alpha}(z)\cr
 \end{pmat}\\ 
 &\qquad\times
 \begin{pmat}[{|}]
  \Iu{(n+1)q} & (\Iu{(n+1)q} -H_n^\mpi  H_n)R_{\Tqn}(\alpha) \Tqn (\Iu{(n+1)q} -H_{\at n } H_{\at n }^\gi ) \cr\-
  \Ouu{(n+1)q}{(n+1)q} & \Iu{(n+1)q}\cr
 \end{pmat}\\
 &\qquad\times
 \begin{pmat}[{}]
  (\Iu{(n+1)q} -H_n^\mpi  H_n) R_{\Tqn}(\alpha) v_{q,n} x \cr\-
 (\Iu{(n+1)q} -H_{\at n }^\mpi  H_{\at n }) H_n v_{q,n} y \cr
 \end{pmat}\\
 &=
 \begin{pmat}[{}]
  (\Iu{(n+1)q} -H_n^\mpi  H_n) R_{\Tqn}(z) \mat{ \Iu{(n+1)q}, \Tqn H_n } (\Iu{2} \otimes v_{q,n} ) \Theta_{n,\alpha} (z) \tmatp{x}{y } \cr\-
  \begin{gathered}
   (\Iu{(n+1)q} -H_{\at n}^\mpi  H_{\at n}) R_{\Tqn}(z) \mat{ \Iu{(n+1)q}, \ek{R_{\Tqn}(\alpha)}^\inv H_n }\\
 \times(\Iu{2} \otimes v_{q,n} )\tilde{\Theta}_{n,\alpha}(z)\diag\rk{ (z-\alpha)\Iq,\Iq }\tmatp{x}{y } 
  \end{gathered}\cr
 \end{pmat}
\end{split}\]
 and, consequently, \eqref{L145.ia} and \eqref{L145.ib} hold true for each \(z \in \C\).
\end{imp}
\end{proof}

 If \(\cU\) is a subspace of \(\Cq\), then by \(\OPu{\cU}\) we denote the complex \tqqa{matrix} which represents the orthogonal projection onto \(\cU\), with respect to the standard basis of \(\Cq\) \tie{}, \(\OPu{\cU}\) is the unique complex \tqqa{matrix} which fulfills the three conditions \(\OPu{\cU}^2 = \OPu{\cU}\), \(\OPu{\cU}^\ad = \OPu{\cU}\), and \(\cR (\OPu{\cU}) = \cU\).
 In this case, we have \(\cN (\OPu{\cU}) = \cU^\orth\).
 
\begin{lem} \label{170}
 Let \(\alpha \in \R\), let \(\kappa \in \Ninf \), let \(\seqska  \in \Kggeqka \), and let \(n\in\NO \) be such that \(2n+1 \leq \kappa\).
 Then:
\benui
 \item\label{170.a} The sets
\begin{align} 
 \cU_{n, \alpha}&\defeq \ek*{\Nul{\rk{\Iu{(n+1)q}-H_n^\mpi  H_n} R_{\Tqn}(\alpha) v_{q,n} }}^\orth \label{170-v1}
\intertext{and}
 \cV_{n, \alpha}&\defeq \ek*{\Nul{\rk{\Iu{(n+1)q}-H_{\at n}^\mpi H_{\at n}}H_nv_{q,n} }}^\orth .\label{170-v2}
\end{align} 
 are orthogonal subspaces of \(\Cq\) with
\begin{align} 
 \dim \cU_{n, \alpha}&= \rank\ek*{\rk{\Iu{(n+1)q}-H_n^\mpi  H_n} R_{\Tqn}(\alpha) v_{q,n}}\label{N174}
\intertext{and} 
 \dim \cV_{n, \alpha}&= \rank\ek*{\rk{\Iu{(n+1)q}-H_{\at n}^\mpi H_{\at n}}H_nv_{q,n}}.\label{N174-1}
\end{align} 
 \item\label{170.b} Let \(A \in \Cqp\).
 Then \(\rk{\Iu{(n+1)q}-H_n^\mpi  H_n} R_{\Tqn}(\alpha) v_{q,n} A =\NM\) if and only if \(\OPu{\cU_{n, \alpha}}A=\NM\).
 Moreover, \(\rk{\Iu{(n+1)q}-H_{\at n}^\mpi H_{\at n}}H_nv_{q,n}A =\NM\) if and only if \(\OPu{\cV_{n, \alpha}}A=\NM\).
\eenui
\end{lem}
\begin{proof} 
 \eqref{170.a} Because of \rremss{ZS}{1582012-2}, we have \(H_n^\ad =H_n\) and \(H_{\at{n}}^\ad =H_{\at{n}}\).
 In particular, \(H_n^\mpi  H_n = H_nH_n^\mpi \).
 Obviously, \((H_n^\mpi  H_n)^\ad = H_n^\mpi  H_n\) and \((H_{\at n}^\mpi  H_{\at n})^\ad = H_{\at n}^\mpi  H_{\at n}\).
 Thus,
\beql{170-1b}
 \cU_{n, \alpha}
 = \Ran{\ek*{\rk{\Iu{(n+1)q}-H_n^\mpi  H_n} R_{\Tqn}(\alpha) v_{q,n}} ^\ad  } 
\eeq
 and
\beql{170-2b}
 \cV_{n, \alpha}
 = \Ran{\ek*{\rk{\Iu{(n+1)q}-H_{\at n}^\mpi H_{\at n}}H_nv_{q,n} } ^\ad  }
 = \Ran{ v_{q,n}^\ad  H_n \rk{\Iu{(n+1)q}-H_{\at n}^\mpi H_{\at n}} }.
\eeq
 In particular, \eqref{N174} and \eqref{N174-1} hold true.
 Let \(f\in \cU_{n,\alpha}\) and \(g\in\cV_{n,\alpha}\) be arbitrary chosen.
 According to \eqref{170-1b} and \eqref{170-2b}, there are \(x,y\in\Co{(n+1)q}\) such that \(f = \ek{(\Iu{(n+1)q} - H_n^\mpi  H_n) R_{\Tqn} (\alpha) v_{q,n}}^\ad  x\) and \(g= v_{q,n}^\ad  H_n (\Iu{(n+1)q} - H_{\at n}^\mpi  H_{\at n}) y\).
 By virtue of the \rremss{lemC21-1}{SBNNN}, we have 
\[\begin{split} 
 f^\ad g
 & = x^\ad (\Iu{(n+1)q}- H_n^\mpi  H_n) R_{\Tqn}(\alpha) v_{q,n} v_{q,n}^\ad  H_n (\Iu{(n+1)q}- H_{\at n}^\mpi  H_{\at n}) y \\
 &= x^\ad (\Iu{(n+1)q}- H_n^\mpi  H_n) R_{\Tqn}(\alpha) \rk*{ \ek*{R_{\Tqn}(\alpha)}^\inv H_n - \Tqn H_{\at n}} (\Iu{(n+1)q}- H_{\at n}^\mpi  H_{\at n}) y \\
 &= x^\ad (\Iu{(n+1)q}- H_n^\mpi  H_n) H_n (\Iu{(n+1)q}- H_{\at n}^\mpi  H_{\at n}) y
 =\NM.
\end{split}\]
 Consequently, the subspaces \(\cU_{n, \alpha}\) and \(\cV_{n, \alpha}\) are orthogonal.

 \eqref{170.b} Use the equations \(\nul{\rk{\Iu{(n+1)q}-H_n^\mpi  H_n} R_{\Tqn}(\alpha) v_{q,n}} =\cU_{n, \alpha}^\orth  = \cN(\OPu{\cU_{n, \alpha}})\) and \( \nul{\rk{\Iu{(n+1)q}-H_{\at n}^\mpi H_{\at n}}H_nv_{q,n} } =\cV_{n, \alpha}^\orth  = \cN(\OPu{\cV_{n, \alpha}})\).
\end{proof}

\section{Nevanlinna and Stieltjes pairs of meromorphic matrix-valued functions}\label{S1223}
 In this section, we consider special classes of pairs of meromorphic matrix-valued functions.
 This leads us to the set of parameters which occur in our description of the solution set of the power moment problem in question.

\begin{nota} \label{def-nev-paar}
 A pair \(\tmatp{\phi}{\psi}\) of \tqqa{matrix-valued} functions \(\phi\) and \(\psi\) meromorphic in \(\uhp\) is called a \emph{\tqNp{}} if there exists a discrete subset \(\cD\) of \(\uhp\) such that the following three conditions are fulfilled:
 \begin{aeqi}{0} 
 \item\label{def-nev-paar.i} \(\phi\) and \(\psi\) are holomorphic in \(\uhp \setminus \cD\).
 \item\label{def-nev-paar.ii} \(\rank \tmatp{\phi (w)}{\psi(w)} =q\) for all \(w\in \uhp \setminus \cD\).
 \item\label{def-nev-paar.iii} \(\tmatp{\phi (w)}{\psi(w)}^\ad (-\Jimq) \tmatp{\phi (w)}{\psi(w)} \geq \Oqq \) for each \(w\in \uhp \setminus \cD\).
 \end{aeqi} 
 The set of all \tqNps{} will be denoted by \(\qNp\).
\end{nota}

 Note the well-known fact that, for each \(S \in \RFq \), the pair \(\tmatp{S}{\Iq}\) belongs to \(\qNp \).

\begin{rem} \label{Ma09_1.11}
 If \(\tmatp{\phi}{\psi} \in \qNp \), then it is readily checked that, for each \tqqa{matrix-valued} function \(g\) which is meromorphic in \(\uhp\) and for which the function \(\det g\) does not vanish identically, the pair \(\tmatp{\phi g}{\psi g }\) belongs to \(\qNp \) as well.
 Two \tqN{s} \(\tmatp{\phi_1}{\psi_1}\) and \(\tmatp{\phi_2}{\psi_2}\) in \(\uhp\) are said to be \emph{equivalent} if there are a \tqqa{matrix-valued} function \(g\) which is meromorphic in \(\uhp\) and a discrete subset \(\cD\) of \(\uhp\) such that \(\phi_{1}\), \(\psi_{1}\), \(\phi_{2}\), \(\psi_{2}\), and \(g\) are holomorphic in \(\pid\) and that \(\det g(w) \neq 0\) and \(\tmatp{\phi_{2}(w)}{\psi_{2}(w)}= \tmatp{\phi_{1}g(w)}{\psi_{1}g(w)} \) hold true for all \(w\in\pid\).
 One can easily see that this implies an equivalence relation on \(\qNp \).
 For each \(\tmatp{\phi}{\psi} \in \qNp \), let \( \tmatpc{\phi}{\psi}\) be the equivalence class generated by \(\tmatp{\phi}{\psi}\).
\end{rem}

 Let us recall a well-known interrelation between the classes \(\qNp \) and \(\SchF{q}{q}{\uhp} \):

\begin{lem} \label{Bi13_5.23} 
\begin{enui} 
 \item\label{Bi13_5.23.a} For each \(\tmatp{\phi}{\psi}\in \qNp \), the function \(\det (\psi-\iu\phi)\) does not vanish identically and \(S\defeq (\psi+\iu\phi)(\psi-\iu\phi)^\inv \) belongs to the Schur class \(\SchF{q}{q}{\uhp} \).
 \item\label{Bi13_5.23.b} For each \(S \in \SchF{q}{q}{\uhp} \), the pair \(\tmatp{\phi}{\psi}\) given by \(\phi\defeq \iu(\Iq-S)\) and \(\psi \defeq \Iq+S\) belongs to the class \(\qNp \), the functions \(\phi\) and \(\psi\) are holomorphic in \(\uhp\), and \(\det\ek{\psi (w)-\iu\phi(w)} \neq 0\) and \(S(w)=\ek{\psi(w) +\iu\phi(w)}\ek{\psi(w) -\iu\phi(w)}^\inv \) hold true for each \(w\in\uhp\).
 \item\label{Bi13_5.23.c} Two \tqN{s} \(\tmatp{\phi_1}{\psi_1}\) and \(\tmatp{\phi_2}{\psi_2}\) in \(\uhp\) are equivalent if and only if 
\((\psi_1+\iu\phi_1)(\psi_1 -\iu\phi_1)^\inv =(\psi_2+\iu\phi_2)(\psi_2 -\iu\phi_2)^\inv \).
\end{enui}
\end{lem}

 A detailed proof of \rlem{Bi13_5.23} can be found, \teg{}, in~\cite[Lemma~1.7]{Thi06}.

\begin{prop} \label{NL1-1}
 Let \(\tmatp {\phi}{\psi} \in \qNp \).
 Then there exists a discrete subset \(\cD\) of \(\uhp\) such that \(\phi\) and \(\psi\) are holomorphic in \(\uhp \setminus \cD\) and that, for every choice of \(w\) and \(z\) in \(\uhp \setminus \cD\), the following four equations hold true: 
\begin{align} 
 \Ran{\phi(w)}&= \Ran{\phi(z)},&
&& 
 \Ran{\psi(w)}&= \Ran{\psi(z)}, \label{TJ2}\\ 
 \psi(w)\Nul{\phi(w)}&= \psi(z) \Nul{\phi(z)},&
&\text{and}&
 \phi(w)\Nul{\psi(w)}&= \phi(z) \Nul{\psi(z)}.\label{TJ3}
\end{align} 
\end{prop}
\begin{proof}
 Because of \rlem{Bi13_5.23}, the function \(\det (\psi-\iu\phi)\) does not vanish identically and \(S\defeq (\psi+\iu\phi)(\psi-\iu\phi)^\inv \) belongs to \(\SchF{q}{q}{\uhp} \).
 Consequently, there is a discrete subset \(\cD\) of \(\uhp\) such that \(\phi\) and \(\psi\) are holomorphic in \(\uhp \setminus \cD\) and that \(\det\ek{\psi(w)-\iu\phi(w)}\neq 0\) holds true for each \(w \in \uhp \setminus \cD\).
 For each \(w \in \uhp \setminus \cD\), we obtain then
\beql{N257-11}
 \phi(w)
 = \frac{\iu}{2} \rk*{\psi(w) -\iu\phi(w) -\ek*{\psi(w) +\iu\phi(w)} }
 = \frac{\iu}{2} \ek*{\Iq-S(w)}\ek*{\psi(w) -\iu\phi(w)} 
\eeq
 and, analogously,
\beql{N257-12}
 \psi(w)
 = \frac{1}{2}\ek*{\Iq+S(w)}\ek*{\psi(w) -\iu\phi(w)}.
\eeq
 In particular, for each \(w \in \uhp \setminus \cD\), we have
\begin{align} \label{N259-1} 
 \Ran{\phi(w)}&=\Ran{\Iq -S(w)}&
&\text{and}&
 \Ran{\psi(w)}&=\Ran{\Iq +S(w)}. 
\end{align}
 Thus, in view of \eqref{N259-1}, we see from \rlem{TS417} that \eqref{TJ2} holds true for every choice of \(w\) and \(z\) in \(\uhp \setminus \cD\).
 For each \(w \in \uhp \setminus \cD\), from \(\det\ek{\psi(w)-\iu\phi(w)}\neq 0\), \eqref{N257-11}, \eqref{N257-12}, and \rrem{PBT} we get \(\nul{\Iq -S(w)}= [\psi(w)-\iu\phi(w)]\nul{\phi(w)}\) and \(\cN(\Iq +S(w))= [\psi(w)-\iu\phi(w)]\cN(\psi(w))\).
 Thus, \rrem{ZLM} implies \(\cN(\Iq -S(w))= \psi(w)\cN(\phi(w))\) and \(\cN(\Iq +S(w))= \phi(w)\cN(\psi(w))\) for each \(w\in\uhp \setminus \cD\).
 Hence, \(S \in \SchF{q}{q}{\uhp} \) and \rlem{TS417} show then that \eqref{TJ3} holds true for every choice \(w\) and \(z\) in \(\uhp \setminus \cD\).
\end{proof}

 Now we introduce the class of pairs of meromorphic matrix-valued functions, which will be the set of parameters in the description of the solution set of the Stieltjes moment problem under consideration.
 
\begin{defn} \label{def-sp}
 Let \(\alpha \in \R\).
 Let \(\phi\) and \(\psi\) be \tqqa{matrix-valued} functions meromorphic in \(\Cs \).
 Then \(\tmatp{\phi}{\psi}\) is called a \emph{\tqSp{}} if there exists a discrete subset \(\cD\) of \(\Cs \) such that the following three conditions are fulfilled: 
 \begin{aeqi}{0} 
 \item\label{def-sp.i} \(\phi\) are \(\psi\) are holomorphic in \(\C \setminus (\rhl \cup \cD)\). 
 \item\label{def-sp.ii} \(\rank \tmatp{\phi (z)}{\psi(z)} =q\) for each \(z\in \C \setminus (\rhl \cup \cD)\).
 \item\label{def-sp.iii} For each \(z\in \C \setminus (\R \cup \cD)\),
\begin{align} 
 \matp{\phi (z)}{\psi(z)}^\ad\rk*{\frac{-\Jimq}{2 \Im z}}\matp{\phi (z)}{\psi(z)}
 &\lgeq\NM\label{KD1}
\intertext{and} 
 \matp{(z-\alpha) \phi (z)}{\psi(z)}^\ad\rk*{\frac{-\Jimq}{2 \Im z}}\matp{(z-\alpha) \phi (z)}{ \psi(z)}
 &\lgeq\NM.\label{KD2} 
 \end{align}
 \end{aeqi} 
 The set of all \tqSps{} will be denoted by \(\qSp\).
 
 A pair \(\tmatp{\phi}{\psi} \in \qSp\) is said to be a \emph{proper \tqSp{}} if \(\det\psi\) does not vanish identically in \(\Cs\).
\end{defn}

\begin{rem} \label{B238B}
 Let \(\alpha \in \R\) and let \(\tmatp{\phi}{\psi} \in \qSp \).
 Then one can easily see that \(\tmatp{\tilde \phi}{\tilde \psi} \) given by \(\tilde \phi\defeq \Rstr_{\uhp \cap \hol{\phi}} \phi\) and \(\tilde \psi\defeq \Rstr_{\uhp \cap \hol{\psi}} \psi\) belongs to \(\qNp \).
\end{rem}

\begin{rem} \label{SP.1}
 Let \(\alpha \in \R\), let \(\tmatp{\phi}{\psi} \in \qSp \), and let \(g\) be a \tqqa{matrix-valued} function which is meromorphic in \(\Cs \) such that \(\det g\) does not vanish identically.
 Then it is readily checked that \(\tmatp{\phi g}{\psi g} \in \qSp \).
\end{rem} 

\begin{rem} \label{def-aq}
 Two \tqS{s} \(\tmatp{\phi_1}{\psi_1}\) and \(\tmatp{\phi_2}{\psi_2}\) in \(\Cs \) are said to be \emph{equivalent} if there exist a \tqqa{matrix-valued} function \(g\) which is meromorphic in \(\Cs\) and a discrete subset \(\cD\) of \(\Cs\) such that \(\phi_1\), \(\phi_2\), \(\psi_1\), \(\psi_2\), and \(g\) are holomorphic in \(\C\setminus (\rhl \cup \cD)\) and that \( \det g(z) \neq 0\) and \( \tmatp{\phi_2 (z)}{\psi_2(z)}= \tmatp{\phi_1 (z) g (z)}{\psi_1(z) g (z)}\) hold true for each \(z\in \C \setminus (\rhl \cup \cD)\).
 It is readily checked that this implies an equivalence relation on \(\qSp \).
 For each \(\tmatp{\phi}{\psi} \in \qSp \), by \(\tmatpc{\phi}{\psi}\) we denote the equivalence class generated by \(\tmatp{\phi}{\psi}\).
\end{rem}

\begin{rem} \label{NR86.} 
 Let \(\alpha \in \R\).
 For each \(j\in\set{1,2}\),  let \(\tmatp{\phi_j}{\psi_j} \in \qSp \), let \(\tilde{\phi}_j \defeq  \Rstr_{\uhp  }\phi_j\) and let \(\tilde{\psi}_j \defeq  \Rstr_{\uhp  }\psi_j\).
 Then it is readily checked that \(\tmatpc{\phi_1}{\psi_1}=\tmatpc{\phi_2}{\psi_2}\) if and only if  \(\tmatpc{\tilde \phi_1}{\tilde \psi_1}=\tmatpc{\tilde \phi_2}{\tilde \psi_2}\).
\end{rem}

\begin{exam} \label{SP.2} 
 Let \(\alpha \in \R\) and let \(f \in \SFq \).
 Using~\cite[\cprop{4.4} and \clem{4.2}]{MR3380267}, one can easily check then that \(\tmatp{f}{\Iq} \) belongs to \(\qSp \).
 In particular,~\cite[\cexam{2.2}]{MR3380267} shows that \(\qSp \ne \emptyset\).
 Furthermore, if \(f,g\in\SFq\) are such that the pairs \(\tmatp{f}{\Iq}\) and \(\tmatp{g}{\Iq}\) are equivalent, then \(f=g\).
\end{exam} 

\begin{rem} \label{BW} 
 Let \(\alpha \in \R\) and let \(\tmatp{\phi}{\psi} \in \qSp \).
 Using a classical result of complex analysis (see, \teg{}~\cite[\cthm{11.46}, \cpage{395}]{MR555733}), one can prove that there is a \((\C\setminus\set{0})\)~valued function \(g\) holomorphic in \(\Cs \) such that \(\tilde \phi\defeq g \phi\) and \(\tilde \psi \defeq g \psi\) are holomorphic in \(\Cs \).
 In particular, \(\tmatp{\tilde \phi}{\tilde \psi}\) belongs to \(\qSp \) with \(\tmatpc{\tilde \phi}{\tilde \psi}=\tmatpc{ \phi}{ \psi}\).
\end{rem}

\begin{rem}\label{R1544}
 Let \(\alpha \in \R\) and let \(\tmatp{\phi}{\psi} \in \qSp \) be proper.
 Applying~\zitaa{MR3644521}{\cprop{4.3}} one can show then, that \(S\defeq\phi\psi^\inv\) belongs to \(\SFq\) and that \(\tmatp{S}{\Iq}\) is a proper \tqSp{} which is equivalent to \(\tmatp{\phi}{\psi}\).
 (Since we do not use this result in the following, we omit a detailed proof.)
\end{rem}
 
\begin{lem} \label{LHA} 
 Let \(\alpha \in \R\) and let \(\tmatp{\phi}{\psi} \in \qSp \).
 Then there exists a discrete subset \(\cD\) of \(\Cs \) such that the conditions~\ref{def-sp.i},~\ref{def-sp.ii}, and~\ref{def-sp.iii} of \rdefn{def-sp} are fulfilled and that the following statements hold true:
\begin{aeqi}{3} 
 \item\label{LHA.iv} \(\frac{1}{\Im z} \Im\ek{\psi^\ad (z) \phi(z)}\in \Cggq\) for each \(z \in \C \setminus (\R \cup \cD)\).
 \item\label{LHA.v} \(\frac{1}{\Im z} \Im [(z-\alpha)\psi^\ad (z) \phi(z) ] \in \Cggq\) for each \(z \in \C \setminus (\R \cup \cD)\).
 \item\label{LHA.vi} \( \Re [\psi^\ad (z) \phi(z) ] \in \Cggq\) for each \(z \in \lehpa  \setminus \cD\).
 \end{aeqi} 
\end{lem}
\begin{proof}
 In view of \rdefn{def-sp}, there is a discrete subset \(\cD\) of \(\Cs \) such that~\ref{def-sp.i},~\ref{def-sp.ii}, and~\ref{def-sp.iii} of \rdefn{def-sp} hold true.
 For each \(z \in \C \setminus (\R \cup \cD)\), from \rrem{J-1} and \rdefnp{def-sp}{def-sp.iii} we get
\[
 \frac{1}{\Im z}\Im\ek*{\psi^\ad (z) \phi(z)}
 = \matp{\phi(z)}{\psi(z) }^\ad \rk*{\frac{-\Jimq}{2 \Im z}} \matp{ \phi(z)}{\psi(z)}
 \in \Cggq
\]
 and, consequently,~\ref{LHA.iv}.
 For each \(z\in \C \setminus (\R \cup \cD)\), \rrem{J-1} and \rdefnp{def-sp}{def-sp.iii} yield
\[
 \frac{1}{\Im z}\Im\ek*{(z-\alpha)\psi^\ad (z) \phi(z)}
 = \matp{(z-\alpha)\phi(z)}{\psi(z)}^\ad \rk*{\frac{-\Jimq}{2 \Im z}} \matp{(z-\alpha)\phi(z)}{\psi(z)} 
 \in \Cggq.
\]
 Hence,~\ref{LHA.v} is true.
 To verify~\ref{LHA.vi}, we consider an arbitrary \(z\in \lehpa  \setminus (\R \cup \cD)\).
 Then \( \alpha - \Re z >0\) and~\ref{LHA.iv} implies \(\frac{\alpha - \Re z }{\Im z} \Im [\psi ^\ad (z) \phi(z)] \in \Cggq\).
 Because of \rrem{A11} and~\ref{LHA.v}, this implies 
\beql{SBW793}\begin{split}
 \Re\ek*{\psi^\ad (z) \phi(z)}
 &= \frac{1}{\Im z} \Im\ek*{z \psi^\ad (z) \phi(z)}- \frac{\Re z }{\Im z} \Im\ek*{\psi^\ad (z) \phi(z)}\\
 &= \frac{1}{\Im z} \Im\ek*{z \psi^\ad (z) \phi(z)} + \frac{\alpha -\Re z }{\Im z} \Im\ek*{\psi^\ad (z) \phi(z)}- \frac{\alpha}{\Im z} \Im\ek*{\psi^\ad (z) \phi(z)}\\
 &\lgeq \frac{1}{\Im z} \Im\ek*{(z- \alpha) \psi^\ad (z) \phi(z)}
 \in \Cggq.
\end{split}\eeq
 Now we consider an arbitrary \(z \in (\lehpa  \setminus \cD) \cap \R\).
 Since \(\cD\) is a discrete subset of \(\Cs \), there is a sequence \((z_n)_{n=1}^{\infi}\) of numbers belonging to \(\lehpa  \setminus\rk{\R\cup\cD}\) with \(\lim_{n \to \infi} z_n=z\).
 For each \(n \in \N\), in view of \eqref{SBW793}, then \(\Re [\psi^\ad (z_n) \phi(z_n) ]\in \Cggq\).
 Thus, in view of \rdefnp{def-sp}{def-sp.i}, then 
\beql{SBWLN}
 \Re\ek*{\psi^\ad (z) \phi(z)}
 =\Re \rk*{\lim_{n\to \infi}\ek*{\psi^\ad (z_n) \phi(z_n)}}
 = \lim_{n\to \infi} \Re\ek*{\psi^\ad (z_n) \phi(z_n)}
 \in \Cggq
\eeq
 follows.
 Taking into account \(\lehpa  \setminus \cD = [ \lehpa  \setminus (\R \cup \cD)] \cup [(\lehpa  \setminus \cD) \cap \R ]\), \eqref{SBW793}, and \eqref{SBWLN}, statement~\ref{LHA.vi} is proved as well.
 \end{proof}

\begin{lem} \label{T01}
 Let \(\tmatp{\phi}{\psi} \in \qSp \).
 Then the function \(\det (\psi-\iu\phi)\) does not vanish identically and the function \(F\defeq (\psi+\iu\phi)(\psi-\iu\phi)^\inv \) is meromorphic in \(\Cs \) and fulfills \(\Rstr _{\uhp} F \in \SchF{q}{q}{\uhp} \).
 Furthermore, there exists a discrete subset \(\cD\) of \(\Cs \) such that \(\phi\), \(\psi\), \((\psi-\iu\phi)^\inv \), and \(F\) are holomorphic in \(\uhp \cup[\C \setminus (\rhl \cup \cD)]\) and that \( \det\ek{\psi(z) - \iu \phi(z)}\neq 0 \) and
\beql{T-N257-1}
 F(z)
 =\ek*{\psi(z)+\iu\phi(z)}\ek*{\psi(z)-\iu\phi(z)}^\inv 
\eeq
 hold true for each \(z \in \C \setminus (\rhl \cup \cD)\).
 Moreover, for each \(z \in \C \setminus (\rhl \cup \cD)\), the matrix-valued functions \(\phi\) and \(\psi\) admit the representations 
\begin{align} \label{T-B55-252}
 \phi(z)&= \frac{\iu}{2} \ek{\Iq-F(z)}\ek{\psi(z)-\iu\phi(z)}&
&\text{and}&
 \psi(z)&= \frac{1}{2} \ek*{\Iq+F(z)}\ek{\psi(z)-\iu\phi(z)}.
\end{align} 
\end{lem}

 In view of \rrem{B238B}, \rlem{Bi13_5.23}, and \rdefn{def-sp}, one can easily prove \rlem{T01}.
 We omit the details.

\begin{prop}\label{NL1-2}
 Let \(\tmatp {\phi}{\psi} \in \qSp \).
 Then there exists a discrete subset \(\cD\) of \(\Cs\) such that \(\phi\) and \(\psi\) are holomorphic in \(\C \setminus (\rhl \cup \cD)\) and that \eqref{TJ2} and \eqref{TJ3} hold true for every choice of \(z\) and \(w\) in \(\C \setminus (\rhl \cup \cD)\).
\end{prop}
\begin{proof}
 The proof is divided into five steps.

 (I) Because of \(\tmatp {\phi}{\psi} \in \qSp \) and \rlem{LHA}, there is a discrete subset \(\tilde\cD\) of \(\Cs \) such that the following four conditions hold true:
\begin{aeqi}{0}
 \item\label{NL1-2.i} \(\phi\) and \(\psi\) are holomorphic in \(\C \setminus (\rhl \cup \tilde\cD)\).
 \item\label{NL1-2.ii} \(\rank \tmatp {\phi(z)}{\psi(z)}=q \) for each \(z \in \C \setminus (\rhl \cup \tilde\cD)\).
 \item\label{NL1-2.iii} The inequalities \eqref{KD1} and \eqref{KD2} hold true for each \(z \in \C \setminus (\R \cup \tilde\cD)\).
 \item\label{NL1-2.iv} \(\Re[\psi^\ad (z) \phi(z)] \in \Cggq\) for each \(z \in \lehpa  \setminus \tilde\cD\).
\end{aeqi}

 (II) Let \(\Pi_1\defeq \uhp \cap \hol{\phi} \cap \hol{\psi}\).
 Then \(\tilde\cD_1\defeq \tilde\cD \cap \uhp\) is a discrete subset of \(\uhp\) with \(\Pi_1 \supseteq \uhp \setminus \tilde\cD_1\).
 Because of~\ref{NL1-2.i}, the functions \(\phi_1\defeq \Rstr _{\Pi_1} \phi\) and \(\psi_1\defeq \Rstr _{\Pi_1} \psi\) are holomorphic in \(\Pi_1\) as well as in \(\uhp \setminus \tilde\cD_1\).
 If \(k=1\), then~\ref{NL1-2.ii} and~\ref{NL1-2.iii} imply 
\beql{HF}
 \rank \matp {\phi_k(z)}{\psi_k(z)}
 =q 
\eeq
 and 
\beql{WL17}
 \matp{\phi_k (z)}{\psi_k(z)}^\ad \rk*{\frac{-\Jimq}{2 \Im z}}\matp{\phi_k (z)}{\psi_k(z)}
 \lgeq \Oqq 
\eeq
 for each \(z\in \uhp \setminus \tilde\cD_1\).
 Thus, \rnota{def-nev-paar} shows that \(\tmatp {\phi_1}{\psi_1} \in \qNp \).
 From \rprop{NL1-1} we know that there is a discrete subset \(\cD_1\) of \(\uhp\) such that \(\phi_1\) and \(\psi_1\) are holomorphic in \(\uhp \setminus \cD_1\) and that 
\begin{align} 
 \Ran{\phi_k(w)}&= \Ran{\phi_k(z)},&
&&
 \Ran{\psi_k(w)}&= \Ran{\psi_k(z)},\label{VI1}\\ 
 \psi_k(w)\Nul{\phi_k(w)}&= \psi_k(z) \Nul{\phi_k(z)},&
&\text{and}& 
 \phi_k(w)\Nul{\psi_k(w)}&= \phi_k(z) \Nul{\psi_k(z)}\label{VI2}
\end{align} 
 hold true for \(k=1\) and for every choice of \(w\) and \(z\) in \(\uhp \setminus \cD_1\).
 Consequently, \eqref{TJ2} and \eqref{TJ3} are fulfilled for each \(w \in \uhp \setminus \cD_1\) and each \(z \in \uhp \setminus \cD_1\).
 
 (III) Let \(\Pi_2\defeq \setaca{z \in\uhp}{-z \in \hol{\phi} \cap \hol{\psi}}\) and let \(\tilde\cD_2\defeq \setaca{ z \in \uhp}{ -z \in \tilde\cD}\).
 Obviously, \(\tilde\cD_2\) is a discrete subset of \(\uhp\) with \(\Pi_2 \supseteq \uhp \setminus \tilde\cD_2\).
 The functions \(\phi_2\colon\Pi_2 \to \Cqq\) defined by \(\phi_2(z)\defeq -\phi(-z)\) and \(\psi_2\colon\Pi_2 \to \Cqq\) defined by \(\psi_2(z)\defeq \psi(-z)\) are holomorphic in \(\uhp \setminus \tilde\cD_2\).
 For each \(z \in \uhp \setminus \tilde\cD_2\), we have \(-z \in \C \setminus (\rhl \cup \tilde\cD)\) and \( \tmatp{\phi_2(z)}{\psi_2(z)} = \diag(-\Iq,\Iq) \cdot \tmatp{\phi(-z)}{\psi(-z)}\).
 Consequently, \eqref{HF} holds true for \(k=2\) and each \(z \in \uhp \setminus \tilde\cD_2\) and, in view of~\ref{NL1-2.iii}, furthermore, 
\[
 \matp{\phi_2(z)}{\psi_2(z)}^\ad (-\Jimq) \matp{\phi_2(z)}{\psi_2(z)} 
 = -2 \Im (-z) \matp{\phi (-z)}{\psi(-z)}^\ad \ek*{\frac{-\Jimq}{2 \Im (- z)}}\matp{\phi (-z)}{\psi(-z)}
 \in \Cggq
\]
 is fulfilled for each \(z \in\uhp \setminus \tilde\cD_2\).
 Thus, according to \rnota{def-nev-paar}, the pair \(\tmatp{\phi_2}{\psi_2}\) belongs to \(\qNp \).
 \rprop{NL1-1} shows that there exists a discrete subset \(\cD_2\) of \(\uhp\) such that \(\phi_2\) and \(\psi_2\) are holomorphic in \(\uhp \setminus \cD_2\) and that \eqref{VI1} and \eqref{VI2} are valid for \(k=2\) and every choice of \(w\) and \(z\) in \(\uhp \setminus \cD_2\).
 Hence, \(\cR(-\phi(-w))= \cR(-\phi(-z))\) and \(\cR(\psi(-w))= \cR(\psi(-z))\), and \( \psi(-w)\cN(-\phi(-w))= \psi(-z) \cN(-\phi(-z)) \) and \( -\phi(-w)\cN(\psi(-w))= -\phi(-z) \cN(\psi(-z)) \) for each \(w\in\uhp \setminus \cD_2\) and each \(z\in\uhp \setminus \cD_2\).
 Therefore, \(\hat \cD_2\defeq \setaca{z \in \lhp}{ -z \in \cD_2}\) is a discrete subset of \(\lhp\) such that \eqref{TJ2} and \eqref{TJ3} are fulfilled for every choice of \(w\) and \(z\) in \(\lhp \setminus \hat \cD_2\). 

 (IV) Obviously, \(\tilde\cD_3\defeq \setaca{ z \in \uhp}{ \alpha+\iu z \in \tilde\cD}\) is a discrete subset of \(\uhp\) and \( \uhp \setminus \tilde\cD_3\) is a subset of \(\Pi_3\defeq \setaca{z \in\uhp}{ \alpha+\iu z \in \hol{\phi} \cap \hol{\psi}}\).
 Because of~\ref{NL1-2.i}, the functions \(\phi_3\colon\Pi_3 \to \Cqq\) defined by \(\phi_3(z)\defeq\iu\phi(\alpha+\iu z)\) and \(\psi_3\colon\Pi_3 \to \Cqq\) defined by \(\psi_3(z) \defeq \psi(\alpha+\iu z)\) are holomorphic in \(\uhp \setminus \tilde\cD_3\).
 For each \(z \in \uhp \setminus \tilde\cD_3\), we have \(\alpha+\iu z \in \C \setminus (\rhl \cup \tilde\cD)\).
 Thus, for each \(z \in \uhp \setminus \tilde\cD_3\), from \(\tmatp{\phi_3 (z)}{\psi_3 (z)} = \diag (\iu\Iq,\Iq) \cdot \tmatp{\phi(\alpha+\iu z) }{\psi(\alpha+\iu z)}\) and~\ref{NL1-2.ii} we see that \eqref{HF} holds true for \(k=3\) and each \(z \in \uhp \setminus \tilde\cD_3\).
 Obviously, \( [\diag (\iu\Iq, \Iq )]^\ad(-\Jimq)\ek{\diag (\iu\Iq, \Iq )} =\tmat{ \Oqq&\Iq\\ \Iq&\Oqq}\).
 For each \(z \in \uhp \setminus \tilde\cD_3\), then 
\beql{RSWN9}\begin{split}
 &\matp{\phi_3(z)}{\psi_3(z)}^\ad \rk*{\frac{-\Jimq}{2 \Im z}} \matp{\phi_3(z)}{\psi_3(z)}\\
 &= \frac{1}{2 \Im z} \matp{\phi(\alpha+\iu z)}{\psi(\alpha+\iu z)}^\ad \ek*{\diag (\iu\Iq,\Iq)}^\ad(-\Jimq)\ek*{\diag (\iu\Iq, \Iq )}\matp{\phi(\alpha+\iu z)}{\psi(\alpha+\iu z)}\\
 &= \frac{1}{ \Im z} \Re\ek*{\psi^\ad (\alpha+\iu z)\phi(\alpha+\iu z)}.
\end{split}\eeq
 For each \(z\in \uhp\setminus \tilde\cD_3\), we have \(\alpha+\iu z \in \lehpa \setminus \tilde\cD\) and, consequently, \(\Re\ek{ \psi^\ad (\alpha +\iu z) \phi (\alpha +\iu z)} \in \Cggq\).
 Thus, \eqref{RSWN9} implies \eqref{WL17} for \(k=3\) and each \(z\in \uhp\setminus \tilde\cD_3\).
 Hence, in view of \rnota{def-nev-paar}, the pair \(\tmatp{\phi_3}{\psi_3}\) belongs to \(\qNp \).
 \rprop{NL1-1} shows then that there is a discrete subset \(\cD_3\) of \(\uhp\) such that \(\phi_3\) and \(\psi_3\) are holomorphic in \(\uhp \setminus \cD_3\) and that the equations in \eqref{VI1} and \eqref{VI2} hold true for \(k=3\) and every choice of \(z\) in \(\uhp \setminus \cD_3\).
 Therefore, for all \(w, z \in \uhp \setminus \cD_3\), we obtain 
\begin{align} 
 \Ran{\phi(\alpha+\iu w)}   &= \Ran{\iu\phi(\alpha+\iu w)}   =\Ran{\iu\phi(\alpha+\iu z)}  =\Ran{\phi(\alpha+\iu z)}, \label{TBK3}\\
 \Ran{\psi(\alpha+\iu w)}   &=\Ran{\psi(\alpha+\iu z)}, \notag\\
 \psi(\alpha+\iu w)\Nul{\phi(\alpha+\iu w)}  &=\psi(\alpha+\iu w)\Nul{\iu\phi(\alpha+\iu w)} =\psi(\alpha+\iu z)\Nul{\iu\phi(\alpha+\iu z)}\notag\\
 & =\psi(\alpha+\iu z)\Nul{\phi(\alpha+\iu z)}, \notag
\intertext{and}
 \phi(\alpha+\iu w)\Nul{\psi(\alpha+\iu w)}
 &= \phi(\alpha+\iu z)\Nul{\psi(\alpha+\iu z)}.\label{TBK2}
\end{align}
 Since \(\cD_3\) is a discrete subset of \(\uhp\), we know that \(\hat \cD_3\defeq \setaca{\alpha +\iu z }{ z \in \cD_3 }\) is a discrete subset of \(\lehpa \).
 Obviously, \(z \in \lehpa  \setminus \hat \cD_3\) if and only if \(\alpha+\iu z\) belongs to \(\uhp \setminus \cD_3\).
 Thus, \eqref{TBK3} and \eqref{TBK2} imply \eqref{TJ2} and \eqref{TJ3} for all \(w, z\in \lehpa  \setminus \hat \cD_3\).
 
 (V) We easily see that \(\cD\defeq \cD_1 \cup \hat \cD_2 \cup \hat \cD_3\) is a discrete subset of \(\C\), that \((\uhp \setminus \cD_1) \cup (\lhp \setminus \hat \cD_2) \cup (\lehpa  \setminus \hat \cD_3)=\C \setminus(\rhl \cup \cD)\), that \((\uhp \setminus \cD_1) \cap (\lehpa  \setminus \hat \cD_3) \neq \emptyset\), and that \((\lhp \setminus \hat \cD_2) \cap (\lehpa  \setminus \hat \cD_3) \neq \emptyset \).
 Hence, the equations in \eqref{TJ2} and \eqref{TJ3} hold true for every choice of \(w\) and \(z\) in \(\C \setminus (\rhl \cup \cD)\).
\end{proof}

\begin{prop} \label{S51-1}
 Let \(\alpha \in \R\), let \(\Theta \in \PJisf \), and let
\beql{TBD}
 \Theta
 =\mat{\Theta_{jk}}_{j,k=1}^2
\eeq
 be the \tqqa{block} representation of \(\Theta\).
 Then:
\begin{enui}
 \item\label{S51-1.a} The function \(\det \Theta\) does not vanish identically and the matrix-valued function \(\Theta^\inv \) is meromorphic in \(\Cs \).
 \item\label{S51-1.b} Let \(f\) be a \tqqa{matrix-valued} function meromorphic in \(\Cs \).
 Suppose that there is a discrete subset \(\cD\) of \(\Cs \) such that \(f\) and \(\Theta\) are holomorphic in \(\C \setminus (\rhl \cup \cD)\), that \(\det \Theta(z) \neq 0\) holds true for each \(z \in \C \setminus (\rhl \cup \cD)\), and that 
\beql{B5-1}
 \matp{f(z)}{\Iq}^\ad \Theta^\invad (z) \rk*{\frac{-\Jimq}{2\Im z}} \Theta^\inv (z) \matp{f(z)}{\Iq}
 \lgeq \Oqq 
\eeq
 and
\begin{multline}\label{B5-2}
 \matp{f(z)}{\Iq}^\ad \Theta^\invad (z)\rk*{\diag\ek*{(z-\alpha) \Iq,\Iq} }^\ad\\
 \times\rk*{\frac{-\Jimq}{2\Im z}}\rk*{\diag\ek*{(z-\alpha) \Iq,\Iq}}\Theta^\inv (z) \matp{f(z)}{\Iq}
 \lgeq \Oqq
\end{multline}
 are fulfilled for each \(z \in \C \setminus(\R \cup \cD)\).
 For every such discrete subset \(\cD\) of \(\C\setminus \rhl \),  there exists a pair \(\tmatp{\phi}{\psi} \in \qSp \) such that \(\phi\) and \(\psi\) are holomorphic in \(\C \setminus (\rhl \cup \cD)\) and that
\beql{PKW}
 \det\ek*{\Theta_{21}(z)\phi(z)+\Theta_{22}(z)\psi (z)}
 \neq 0
\eeq
 and 
\beql{GFZ}
 f(z)
 = \ek*{\Theta_{11}(z)\phi(z)+\Theta_{12}(z)\psi(z)} \ek*{\Theta_{21}(z)\phi(z)+\Theta_{22}(z)\psi (z)}^\inv 
\eeq
 hold true for each \(z \in \C \setminus (\rhl \cup \cD)\).
 \item\label{S51-1.c} Let \(\tmatp{\phi}{\psi} \in \qSp \) be such that \(\det (\Theta_{21} \phi + \Theta_{22} \psi)\) does not vanish identically.
 Then there exists a discrete subset \(\cD\) of \(\Cs \) such that the following three statements are valid:
\begin{Aeqii}{0}
 \il{S51-1.I} The the matrix-valued functions \(\Theta\), \(\phi\), and \(\psi\) are holomorphic in \(\C\setminus (\rhl \cup \cD)\).
 \il{S51-1.II} The inequalities  \(\det\Theta(z)\neq0\) and \eqref{PKW} hold true for each \(z\in \C \setminus (\rhl \cup \cD)\).
 \il{S51-1.III} The function
\beql{DF}
 f
 \defeq \rk{\Theta_{11} \phi +\Theta_{12} \psi} \rk{\Theta_{21} \phi +\Theta_{22} \psi}^\inv 
\eeq
 is holomorphic in \(\C \setminus (\rhl \cup \cD)\), the inequalities \eqref{B5-1} and \eqref{B5-2} hold true for each \(z \in \C \setminus (\R \cup \cD)\) and \eqref{GFZ} is fulfilled for each \(z \in \C \setminus(\rhl \cup \cD)\).
\end{Aeqii}
 \item\label{S51-1.d} For each \(k \in \set{1,2}\), let \(\tmatp{\phi_k}{\psi_k} \in \qSp \) be such that \(\det (\Theta_{21}\phi_k+\Theta_{22} \psi_k)\) does not vanish identically.
 Then \(\tmatpc{\phi_1}{\psi_1}=\tmatpc{\phi_2}{\psi_2}\) if and only if 
\[
 \rk{\Theta_{11} \phi_1 + \Theta_{12} \psi_1}\rk{\Theta_{21} \phi_1 + \Theta_{22} \psi_1}^\inv
 =\rk{\Theta_{11} \phi_2 + \Theta_{12} \psi_2}\rk{\Theta_{21} \phi_2 + \Theta_{22} \psi_2}^\inv.
\]
 \end{enui}
\end{prop}
\begin{proof} 
 \eqref{S51-1.a} Use \rlem{NIN}.
 
 \eqref{S51-1.b} Let \(\cD\) be a discrete subset of \(\Cs \) such that \(f\) and \(\Theta\) are holomorphic in \(\C \setminus (\rhl \cup \cD)\), that \(\det \Theta (z)\ne 0\) is valid for each \(z\in\C\setminus(\rhl \cup\cD)\), and that \eqref{B5-1} and \eqref{B5-2} are fulfilled for each \(z \in \C \setminus (\R \cup \cD)\).
 Then \(\Theta^\inv \), \(\phi\defeq \mat{\Iq,\Oqq} \Theta^\inv \tmatp{f}{\Iq}\), and \(\psi\defeq \mat{\Oqq,\Iq} \Theta^\inv \tmatp{f}{\Iq}\) are holomorphic in \(\C \setminus (\rhl \cup \cD)\) and, for each \(z\in\C \setminus (\rhl \cup \cD)\), we have
\beql{SMG} 
 \matp{\phi(z)}{\psi(z)}
 = \Theta^\inv (z) \matp{f(z)}{\Iq},
\eeq
 consequently, 
\beql{TF1-1} 
 \Theta_{11}(z)\phi(z)+\Theta_{12}(z) \psi(z)
 = \mat{\Iq,\Oqq} \Theta(z) \matp{\phi(z)}{\psi(z)}
 = f(z) 
\eeq
 and, analogously,
\beql{TF1-2} 
 \Theta_{21}(z) \phi(z) +\Theta_{22}(z) \psi(z)
 =\Iq. 
\eeq
 In particular, \eqref{TF1-2} implies \eqref{PKW} as well as
\[
 q 
 \geq \rank \matp{\phi(z)}{\psi(z)}
 \geq \rank \rk*{ \mat*{ \Theta_{21}(z), \Theta_{22}(z) } \matp{\phi(z)}{\psi(z)}} 
 =\rank \Iq 
 =q
\]
 and, hence, \(\rank \tmatp{\phi(z)}{\psi(z)}=q\) for each \(z \in \C \setminus (\rhl \cup \cD)\).
 In view of \eqref{SMG} and \eqref{B5-1}, we conclude that
\[
 \matp{\phi(z)}{\psi(z)}^\ad \rk*{ \frac{-\Jimq}{2\Im z}}\matp{\phi(z)}{\psi(z)}
 = \matp{f(z)}{\Iq}^\ad \frac{\Theta^\invad (z) (-\Jimq) \Theta^\inv (z)}{2\Im z} \matp{f(z)}{\Iq}
 \lgeq\NM
\]
 holds true for each \(z \in \C \setminus (\R \cup \cD)\).
 From \eqref{SMG} we obtain the equation \(\tmatp{(z-\alpha)\phi(z)}{\psi(z)} = \rk{\diag \ek{(z-\alpha)\Iq,\Iq}}\Theta^\inv (z) \tmatp{f(z)}{\Iq}\) for each \(z \in \C \setminus (\rhl \cup \cD)\), and, according to \eqref{B5-2}, consequently, \eqref{KD2} for each \(z \in \C \setminus (\R \cup \cD)\).
 Thus, we proved that \(\tmatp{\phi}{\psi}\) is a \tqSp{}.
 From \eqref{TF1-1} and \eqref{TF1-2} we get \eqref{GFZ} for each \(z \in \C \setminus (\rhl \cup \cD)\).
 
 \eqref{S51-1.c} Since \(\Theta\) belongs to \(\PJisf\), \rlem{NIN} shows that there is a discrete subset \(\cD_1\) of \(\Cs \) such that \(\Theta\) is holomorphic in \(\C \setminus (\rhl \cup \cD_1)\) and that \(\det \Theta (z) \neq 0\) is valid for each \(z \in \C \setminus (\rhl \cup \cD_1)\).
 Because of \(\tmatp{\phi}{\psi} \in\qSp \), there is a discrete subset \(\cD_2\) of \(\Cs \) such that \(\phi\) and \(\psi\) are holomorphic in \(\C \setminus (\rhl \cup \cD_2)\) and that \eqref{KD1} and \eqref{KD2} hold true for each \(z\in \C\setminus (\R \cup \cD_2)\).
 Since the meromorphic function \(\det (\Theta_{21} \phi+\Theta_{22} \psi)\) does not vanish identically, there is a discrete subset \(\cD_3\) of \(\Cs \) such that \(\det (\Theta_{21} \phi+\Theta_{22} \psi)\) is holomorphic in \(\C\setminus (\rhl \cup \cD_3)\) and that \(\det (\Theta_{21} \phi+\Theta_{22} \psi)(z) \neq 0\) holds true for all \(z \in \C\setminus (\rhl \cup \cD_3)\).
 Thus, the set \(\cD\defeq \cD_1 \cup \cD_2 \cup \cD_3\) is a discrete subset of \(\Cs \) and we see that \(\Theta\), \(\phi\), and \(\psi\) are holomorphic in \(\C\setminus (\rhl \cup \cD)\) and that the inequalities \eqref{KD1}, and \eqref{KD2} hold true for each \(z\in \C\setminus (\R \cup \cD)\).
 Furthermore, \(\det \Theta (z) \neq 0\) and \eqref{PKW} are valid for all \(z\in \C\setminus (\rhl \cup \cD)\).
 Consequently, \(f\) defined by \eqref{DF} is holomorphic in \(\C \setminus (\rhl \cup \cD)\) and \eqref{GFZ} is valid for each \(z\in \C\setminus (\rhl \cup \cD)\).
 Now we consider an arbitrary \(z\in \C\setminus (\R \cup \cD)\).
 Because of \rpart{S51-1.a}, \eqref{PKW}, \eqref{GFZ}, and \eqref{TBD}, we have
\beql{PM1}\begin{split}
 \Theta^\inv (z) \matp{f(z)}{\Iq}
 &= \Theta^\inv(z)\matp{ \Theta_{11} (z) \phi (z) + \Theta_{12} (z) \psi (z)}{\Theta_{21} (z) \phi (z) + \Theta_{22} (z) \psi (z)} \ek*{\Theta_{21}(z)\phi(z)+\Theta_{22}(z)\psi (z)}^\inv \\
 &= \matp{\phi(z)}{\psi(z)} \ek*{\Theta_{21}(z)\phi(z)+\Theta_{22}(z)\psi (z)}^\inv 
\end{split}\eeq
 and, consequently,
\begin{multline} \label{PLN} 
 \matp{f(z)}{\Iq}^\ad\frac{\Theta^\invad (z) (-\Jimq) \Theta^\inv (z)}{2\Im z} \matp{f(z)}{\Iq} \\
 = \ek{\Theta_{21}(z)\phi(z)+\Theta_{22}(z)\psi (z)}^\invad \matp{\phi(z)}{\psi(z)}^\ad \rk*{ \frac{-\Jimq}{2\Im z}}\matp{\phi(z)}{\psi(z)} \ek{\Theta_{21}(z)\phi(z)+\Theta_{22}(z)\psi (z)}^\inv.
\end{multline}
 In view of \eqref{KD1}, the matrix on the right-hand side of \eqref{PLN} is \tnnH{}.
 Thus, \eqref{B5-1} holds true.
 Using \eqref{PM1}, we get
\[
 \rk*{\diag \ek*{(z-\alpha)\Iq,\Iq}}\Theta^\inv (z) \matp{f(z)}{\Iq}
 = \matp{(z-\alpha)\phi(z)}{\psi(z)} \ek*{\Theta_{21}(z)\phi(z)+\Theta_{22}(z)\psi (z)}^\inv, 
\]
 which implies
\begin{multline} \label{PRL}
 \matp{f(z)}{\Iq}^\ad\Theta^\invad (z) \rk*{\diag\ek{(z-\alpha)\Iq,\Iq}}^\ad\rk*{\frac{-\Jimq}{2\Im z}}\rk*{\diag \ek*{(z-\alpha)\Iq,\Iq}}\Theta^\inv (z) \matp{f(z)}{\Iq}\\
 = \ek*{\Theta_{21}(z)\phi(z)+\Theta_{22}(z)\psi (z)}^\invad \matp{(z-\alpha)\phi(z)}{\psi(z)}^\ad\\
 \times\rk*{\frac{-\Jimq}{2\Im z}}\matp{(z-\alpha)\phi(z)}{\psi(z)} \ek*{\Theta_{21}(z)\phi(z)+\Theta_{22}(z)\psi (z)}^\inv.
\end{multline}
 Because of \eqref{KD2}, the matrix on the right-hand side of \eqref{PRL} is \tnnH{}.
 Consequently, \eqref{B5-2} is proved as well.

 \eqref{S51-1.d} In view of \rpart{S51-1.c} and \rlem{NIN}, the proof of \rpart{S51-1.d} is straightforward.
\end{proof}

\begin{lem} \label{160}
 Let \(\alpha \in \R\), let \(\kappa \in \Ninf \), and let \(\seqska  \in \Kggeqka \).
 Let \(\phi\) and \(\psi\) be \tqqa{matrix-valued} functions which are meromorphic in \(\Cs \).
 Let \(n\in\NO \) be such that \(2n+1 \leq \kappa\) and let \( \Theta _{n, \alpha}\colon\C \to \Coo{2q}{2q}\) be defined by \eqref{TD2-2}.
 Let \(\hat{\Theta}_{n,\alpha}\defeq \Rstr _{\Cs } \Theta _{n, \alpha}\) and let 
\beql{Qblock-2} 
 \hat{\Theta}_{n, \alpha}
 =\mat{\hat{\Theta}_{n, \alpha}^{(j,k)}}_{j,k=1}^2
\eeq
 be the \tqqa{block} representation of \(\hat{\Theta}_{n, \alpha}\).
 Let
\begin{align} 
 \tilde{\phi}&\defeq \hat{\Theta}_{n, \alpha}^{(1,1)}\phi + \hat{\Theta}_{n, \alpha}^{(1,2)} \psi&\label{sphi}
&\text{and}&
 \tilde{\psi}&\defeq \hat{\Theta}_{n, \alpha}^{(2,1)} \phi+ \hat{\Theta}_{n, \alpha}^{(2,2)} \psi.
\end{align} 
 Furthermore, let \(z \in (\hol{\phi} \cap \hol{\psi})\setminus \R\) be such that
\beql{KD1-v7}
 \matp{\phi (z)}{\psi(z)}^\ad \rk*{\frac{-\Jimq}{2 \Im z}} \matp{\phi (z)}{\psi(z)}
 \lgeq\NM
\eeq
 and
\beql{160-v8}
 \rk{\Iu{(n+1)q}-H_n^\mpi  H_n} R_{\Tqn}(\alpha) v_{q,n} \phi(z)
 =\NM
\eeq
 hold true.
 Then \(\nul{\tilde{\psi}(z)} \subseteq \cN(\tilde{\phi}(z))\).
 Moreover, if
\beql{160-v6}
 \rank\matp{\phi(z)}{\psi(z)}
 =q
\eeq
 holds true, then \(\det \tilde{\psi}(z) \neq 0\).
\end{lem}
\begin{proof}
 Because of \(\Kggeqka \subseteq\Kggqka \) and \rlem{311-1}, the equations in \eqref{N39-RI} are true.
 
 We consider an arbitrary \(y \in \nul{\tilde{\psi}(z)}\).
 Because of \rrem{J-1}, we have then 
\beql{WOW}
 y^\ad \matp{\tilde{\phi}(z)}{\tilde{\psi}(z)}^\ad \Jimq \matp{\tilde{\phi}(z)}{\tilde{\psi}(z)}y
 = \iu y^\ad\ek*{\tilde {\psi}^\ad (z) \tilde {\phi}(z) -\tilde {\phi}^\ad (z) \tilde {\psi}(z)}y
 =0. 
\eeq
 Obviously, \(\Theta _{n, \alpha}(z)=\hat{\Theta}_{n,\alpha}(z) \).
 From \eqref{Qblock-2} and \eqref{sphi} we conclude
\beql{160-5}
 \hat{\Theta}_{n,\alpha}(z) \matp{ \phi(z)}{\psi(z)}
 =  \matp{\tilde{\phi}(z)}{\tilde{\psi}(z)}. 
\eeq
 Using \eqref{160-5} and \eqref{WOW}, we conclude 
\beql{JQ-160} 
 y^\ad \matp{ \phi(z)}{\psi(z)} ^\ad \ek*{ \Jimq - \Theta_{n,\alpha} ^\ad (z) \Jimq \Theta_{n,\alpha} (z)} \matp{ \phi(z)}{\psi(z)} y 
 = - y^\ad \matp{ \phi(z)}{\psi(z)} ^\ad ( - \Jimq) \matp{ \phi(z)}{\psi(z)} y.
\eeq
 Because of \rlem{311-1}, \rrem{21112N}, \rlem{QB-BQ}, \eqref{JQ-160}, and \eqref{KD1-v7}, we obtain
\[\begin{split} 
 0
 &\leq\normE*{\sqrt{H_n}\ek*{R_{\Tqn^\ad}(\alpha)}^\inv R_{\Tqn^\ad }(z) H_n^\gi  R_{\Tqn}(\alpha) \mat{ \Iu{(n+1)q}, \Tqn H_n}(\Iu{2} \otimes v_{q,n} ) B_{n,\alpha} \matp{ \phi(z)}{\psi(z)} y}^2 \\
 &= y^\ad \matp{ \phi(z)}{\psi(z)} ^\ad B_{n,\alpha}^\ad (\Iu{2} \otimes v_{q,n} )^\ad \mat{ \Iu{(n+1)q}, \Tqn H_n}^\ad\ek*{R_{\Tqn}(\alpha)}^\ad \\
 &\qquad\times \rk{H_n^\gi}^\ad \ek*{R_{\Tqn^\ad}(z)}^\ad \ek*{R_{\Tqn^\ad}(\alpha)}^\invad \sqrt{H_n}^\ad \sqrt{H_n} \ek*{R_{\Tqn^\ad}(\alpha)}^\inv R_{\Tqn^\ad }(z) H_n^\gi  \\
 &\qquad\times R_{\Tqn}(\alpha) \mat{ \Iu{(n+1)q}, \Tqn H_n} (\Iu{2} \otimes v_{q,n} ) B_{n,\alpha} \matp{ \phi(z)}{\psi(z)} y \\
 &= y^\ad \matp{ \phi(z)}{\psi(z)} ^\ad B_{n,\alpha}^\ad (\Iu{2} \otimes v_{q,n} )^\ad \mat{ \Iu{(n+1)q}, \Tqn H_n}^\ad \ek*{R_{\Tqn}(\alpha)}^\ad \\
 &\qquad\times H_n^\gi  \ek*{R_{\Tqn^\ad}(z)}^\ad \ek*{R_{\Tqn}(\alpha)}^\inv H_n \ek*{R_{\Tqn^\ad}(\alpha)}^\inv R_{\Tqn^\ad }(z) H_n^\gi  \\
 &\qquad\times R_{\Tqn}(\alpha) \mat{ \Iu{(n+1)q}, \Tqn H_n} (\Iu{2} \otimes v_{q,n} ) B_{n,\alpha} \matp{ \phi(z)}{\psi(z)} y \\ 
 & = y^\ad \matp{ \phi(z)}{\psi(z)} ^\ad \frac{1}{\iu(\ko{z}-z)} \ek*{ \Jimq - \Theta_{n,\alpha} ^\ad (z) \Jimq \Theta_{n,\alpha} (z)} \matp{ \phi(z)}{\psi(z)} y\\
 &=-y^\ad \matp{ \phi(z)}{\psi(z)} ^\ad \rk*{\frac{-\Jimq}{2 \Im z}} \matp{ \phi(z)}{\psi(z)} y
 \leq0
\end{split}\]
 and, consequently, 
\beql{N292-1} 
 \sqrt{H_n} \ek*{R_{\Tqn^\ad}(\alpha)}^\inv R_{\Tqn^\ad }(z) H_n^\gi  R_{\Tqn}(\alpha) \mat{\Iu{(n+1)q},\Tqn H_n} (\Iu{2}\otimes v_{q,n} ) B_{n,\alpha} \matp{\phi(z)}{\psi(z)} y
 =\NM. 
\eeq
 Multiplying equation~\eqref{N292-1} from the left by \(\sqrt{H_n}\) and using \rrem{155}, we get 
\[
 H_n \ek*{R_{\Tqn^\ad}(\alpha)}^\inv R_{\Tqn^\ad }(z) \mat*{H_n^\gi R_{\Tqn}(\alpha),H_{\at n}^\gi H_n} (\Iu{2} \otimes v_{q,n} ) \matp{ \phi(z)}{\psi(z)} y
 =\NM
\]
 and, hence, 
\begin{multline} \label{160-6}
 \mat*{H_n \ek*{R_{\Tqn^\ad}(\alpha)}^\inv R_{\Tqn^\ad }(z) H_n^\gi  R_{\Tqn}(\alpha),H_n \ek*{R_{\Tqn^\ad}(\alpha)}^\inv R_{\Tqn^\ad }(z) H_{\at n}^\gi  H_n }\\
 \times(\Iu{2} \otimes v_{q,n} ) \matp{ \phi(z)}{\psi(z)} y
 =\NM. 
\end{multline} 
 Let \(X\defeq H_n \ek{R_{\Tqn^\ad }(\alpha)}^\inv R_{\Tqn^\ad }(z)H_{\at n}^\gi H_n\).
 Because of \eqref{160-5}, \(\Theta_{n,\alpha} (z) =\hat{\Theta}_{n,\alpha} (z)\), \rlem{RF1}, and \(v_{q,n}^\ad R_{\Tqn}(\alpha) v_{q,n}= \Iq\), we have
\beql{160-7}\begin{split}
 &\tilde{\phi}(z) y
 = \mat{\Iq,\Oqq} \matp{\tilde{\phi}(z)}{\tilde{\psi}(z)} y
 = \mat{\Iq,\Oqq} \Theta_{n,\alpha}(z) \matp{ \phi(z)}{\psi(z)} y \\
 &=\mat*{\Iq +(z-\alpha) v_{q,n}^\ad H_n \Tqn^\ad R_{\Tqn^\ad }(z)H_n^\gi  R_{\Tqn}(\alpha) v_{q,n},v_{q,n}^\ad  X  v_{q,n}}\matp{ \phi(z)}{\psi(z)} y \\
 &= v_{q,n}^\ad\mat*{R_{\Tqn}(\alpha) +(z-\alpha) H_n \Tqn^\ad R_{\Tqn^\ad }(z)H_n^\gi  R_{\Tqn}(\alpha), X }(\Iu{2} \otimes v_{q,n}) \matp{ \phi(z)}{\psi(z)} y.
\end{split}\eeq
 From \rrem{lemmart} we see that \eqref{215-1} holds true, which implies
\begin{multline} \label{S1-7}
 (z-\alpha)H_n \Tqn^\ad R_{\Tqn^\ad }(z)H_n^\gi  R_{\Tqn}(\alpha)
 =H_n \rk*{ \ek*{R_{\Tqn^\ad}(\alpha)}^\inv R_{\Tqn^\ad }(z)-\Iu{(n+1)q}}H_n^\gi  R_{\Tqn}(\alpha) \\
 =H_n \ek*{R_{\Tqn^\ad}(\alpha)}^\inv R_{\Tqn^\ad }(z)H_n^\gi  R_{\Tqn}(\alpha)-H_n H_n^\gi  R_{\Tqn}(\alpha).
\end{multline} 
 Taking \eqref{160-7}, \eqref{S1-7}, and \eqref{160-6} into account, we obtain  
\beql{N299-1}\begin{split} 
 &\tilde{\phi}(z) y\\
 &= v_{q,n}^\ad\mat*{R_{\Tqn}(\alpha) + H_n \ek*{R_{\Tqn^\ad}(\alpha)}^\inv R_{\Tqn^\ad }(z)H_n^\gi  R_{\Tqn}(\alpha)-H_n H_n^\gi  R_{\Tqn}(\alpha), X }\\
 &\qquad\times(\Iu{2} \otimes v_{q,n})\matp{ \phi(z)}{\psi(z)} y\\
 &= v_{q,n}^\ad\mat*{R_{\Tqn}(\alpha) -H_n H_n^\gi  R_{\Tqn}(\alpha), \Ouu{(n+1)q}{(n+1)q} }(\Iu{2} \otimes v_{q,n})\matp{ \phi(z)}{\psi(z)} y\\
 &\qquad+ v_{q,n}^\ad\mat*{H_n \ek*{R_{\Tqn^\ad}(\alpha)}^\inv R_{\Tqn^\ad }(z)H_n^\gi  R_{\Tqn}(\alpha), X }(\Iu{2} \otimes v_{q,n})\matp{ \phi(z)}{\psi(z)} y\\
 &= v_{q,n}^\ad\mat*{ R_{\Tqn}(\alpha) -H_n H_n^\gi  R_{\Tqn}(\alpha), \Ouu{(n+1)q}{(n+1)q}}(\Iu{2} \otimes v_{q,n})\matp{ \phi(z)}{\psi(z)} y. 
\end{split}\eeq
 Thus, using \eqref{N299-1}, \rrem{SBNNN}, and \eqref{160-v8}, we infer
\[\begin{split} 
 \tilde{\phi}(z) y 
 &= \mat*{v_{q,n}^\ad R_{\Tqn} (\alpha) v_{q,n} - v_{q,n}^\ad H_nH_n^\gi  R_{\Tqn}(\alpha)v_{q,n},\Oqq }\matp{\phi(z)y}{\psi(z)y} \\
 &= v_{q,n}^\ad\rk{ \Iu{(n+1)q} - H_nH_n^\gi } R_{\Tqn}(\alpha)v_{q,n} \phi(z) y \\
 &= v_{q,n}^\ad \rk{ \Iu{(n+1)q} - H_nH_n^\gi } \rk{ \Iu{(n+1)q} - H_n^\mpi H_n} R_{\Tqn}(\alpha)v_{q,n} \phi(z) y   = 0
\end{split}\]
 and, consequently, \(y \in \nul{\tilde{\phi}(z)}\).
 Hence \(\nul{\tilde{\psi}(z)} \subseteq \nul{\tilde{\phi}(z)}\) is proved.

 Now we suppose \eqref{160-v6}.
 We consider again an arbitrary \(y\in\nul{\tilde{\psi}(z)}\).
 Then we already know that \(y\in\nul{\tilde{\phi}(z)}\).
 In view of \rlemp{BLR4.3-1}{BLR4.3-1.b} and \eqref{160-5}, we get then 
\[
 \matp{\phi(z)}{\psi(z)} y
 = \ek*{\Theta_{n,\alpha}(z)}^\inv \hat{\Theta}_{n, \alpha}(z) \matp{\phi(z)}{\psi(z)} y 
 = \ek*{\Theta_{n,\alpha}(z)}^\inv \matp{\tilde{\phi}(z)y}{\tilde{\psi}(z)y}
 =\NM.
\]
 Because of \eqref{160-v6}, this implies \(y =  \Ouu{q}{1}\), and hence, \(\det \tilde{\psi}(z) \neq 0\).
\end{proof}

\begin{lem} \label{165}
 Let \(\alpha \in \R\), let \(\kappa \in \Ninf \), let \(\seqska  \in \Kggeqka \), and let \(\tmatp{\phi}{\psi} \in \qSp \) be such that \(\rk{\Iu{(n+1)q}-H_n^\mpi  H_n} R_{\Tqn}(\alpha) v_{q,n}\phi =\Ouu{(n+1)q}{q}\).
 Let \(n\in\NO \) be such that \(2n+1 \leq \kappa\), let \(\Theta_{n,\alpha}\colon\C\to \Coo{2q}{2q}\) be defined by \eqref{TD2-2},  and let \eqref{Qblock-2} be the \tqqa{block} partition of \(\hat{\Theta}_{n,\alpha}\defeq \Rstr _{\Cs } \Theta _{n, \alpha}\).
 Then there is a discrete subset \(\cD\) of \(\Cs \) such that \(\phi\) and \(\psi\) are holomorphic in \(\C \setminus (\rhl \cup \cD)\) and that
\beql{KTP}
 \det \ek*{\hat \Theta_{n, \alpha}^{(2,1)}(z) \phi(z) +\hat \Theta_{n, \alpha}^{(2,2)}(z) \psi(z)}
 \neq 0
\eeq
 holds true for each \(z \in \C \setminus (\R \cup \cD)\).
\end{lem}
\begin{proof}
 Use \rdefn{def-sp} and \rlem{160}.
\end{proof}

\section{A particular subclass of Stieltjes pairs}\label{S1321}
 In this section, we study a particular subclass of Stieltjes pairs.

\begin{nota} \label{bez-DT1}
 Let \(\alpha \in \R\), let \(\kappa \in \Ninf \) and let \(\seqska \) be a sequence of complex \tqqa{matrices}.
 For each \(n\in\NO \) with \(2n+1 \leq \kappa\), let \(\qSps{2n+1}\) be the set of all \(\tmatp{\phi}{\psi} \in \qSp \) such that
\begin{align} 
 \rk{\Iu{(n+1)q}-H_n^\mpi  H_n} R_{\Tqn}(\alpha) v_{q,n} \phi&=\NM\label{175-v7}
\intertext{and}
 \rk{\Iu{(n+1)q}-H_{\at n}^\mpi H_{\at n}}H_n v_{q,n} \psi& =\NM.\label{175-v8}
 \end{align} 
\end{nota}

\begin{rem}\label{R1604}
 Let \(\alpha \in \R\), let \(\seqs{2n+1}\) be a sequence of complex \tqqa{matrices}, let \(\tmatp{\phi}{\psi} \in \qSps{2n+1} \), and let \(g\) be a \tqqa{matrix-valued} function which is meromorphic in \(\Cs \) such that \(\det g\) does not vanish identically.
 Then it is readily checked that \(\tmatp{\phi g}{\psi g} \in \qSps{2n+1} \).
\end{rem}

\begin{lem} \label{9.36N} 
 Let \(\alpha \in \R\), let \(\kappa \in \Ninf \), let \(\seqska  \in \Kggeqka \), and let \(n\in\NO \) be such that \(2n+1 \leq \kappa\).
 Let \(\Theta _{n, \alpha}\colon\C \to \Coo{2q}{2q}\) be defined by \eqref{TD2-2}, let \(\hat{\Theta}_{n, \alpha}\defeq \Rstr_{\Cs } \Theta_{n, \alpha}\), let \eqref{Qblock-2} be the \tqqa{block} representation of \(\hat{\Theta}_{n, \alpha}\), and let \(\hat{R}_{\Tqn}\defeq \Rstr_{\Cs } R_{\Tqn}\).
 Let \(\tmatp{\phi}{\psi} \in \qSp \) be such that \(\det (\hat \Theta_{n, \alpha}^{(2,1)} \phi +\hat \Theta_{n, \alpha}^{(2,2)} \psi) \) does not vanish identically and let
\beql{N292}
 \hat{S}_{n, \alpha}
 \defeq \rk{\hat\Theta_{n, \alpha}^{(1,1)} \phi+\hat\Theta_{n, \alpha}^{(1,2)} \psi} \rk{\hat\Theta_{n, \alpha}^{(2,1)} \phi+\hat\Theta_{n, \alpha}^{(2,2)} \psi}^\inv.
\eeq
 Then the following statements~\eqref{9.36N.a} and~\eqref{9.36N.b} are equivalent:
\begin{enui}
 \item\label{9.36N.a} The equations
\beql{nN309-1}
 \rk{\Iu{(n+1)q}-H_n^\mpi  H_n}\hat R_{\Tqn} \mat{\Iu{(n+1)q},\Tqn H_n} (\Iu{2} \otimes v_{q,n} ) \matp{\hat{S}_{n, \alpha}}{\Iq}
 =\NM
\eeq
 and 
\begin{multline} \label{nN309-2}
 \rk{\Iu{(n+1)q}-H_{\at n}^\mpi H_{\at n}} \hat R_{\Tqn} \mat*{\Iu{(n+1)q},\ek*{R_{\Tqn}(\alpha)}^\inv H_n}\\
 \times(\Iu{2} \otimes v_{q,n} ) \matp{(\hat{\cE} -\alpha)\hat{S}_{n, \alpha}}{\Iq}
 =\NM
\end{multline} 
 hold true, where \(\hat{\cE}\colon\Cs \to \C\) is defined by \(\hat{\cE} (z)\defeq z\).
 \item\label{9.36N.b} \(\tmatp{\phi}{\psi} \in \qSps{2n+1}\).
 \end{enui} 
\end{lem}
\begin{proof}
 The proof is partitioned into twelve steps.
 
 (I) Since \(\det \rk{\hat \Theta_{n, \alpha}^{(2,1)} \phi +\hat \Theta_{n, \alpha}^{(2,2)} \psi} \) does not vanish identically, there is a discrete subset \(\cD\) of \(\Cs \) such that the conditions~\ref{def-sp.i},~\ref{def-sp.ii}, and~\ref{def-sp.iii} of \rdefn{def-sp} hold true and that \eqref{KTP} is fulfilled for each \(z \in \C \setminus (\rhl \cup \cD)\). 

 (II) In view of \rdefnp{def-sp}{def-sp.i}, \eqref{N292}, and \eqref{KTP}, the function \(\hat{S}_{n, \alpha}\) admits the representation
\beql{N293}
 \hat{S}_{n, \alpha}(z)
 =\ek*{\hat\Theta_{n, \alpha}^{(1,1)}(z) \phi(z)+\hat\Theta_{n, \alpha}^{(1,2)}(z) \psi(z)} \ek*{\hat\Theta_{n, \alpha}^{(2,1)}(z) \phi(z)+\hat\Theta_{n, \alpha}^{(2,2)}(z) \psi(z)}^\inv 
\eeq
 for each \(z\in \C \setminus (\rhl \cup \cD)\).
 Because of \rdefnp{def-sp}{def-sp.i}, \eqref {KTP}, \eqref{Qblock-2}, and \eqref{N293}, we see that 
\begin{multline} \label{150-1}
 \hat{\Theta}_{n,\alpha}(z) \matp{\phi(z)}{\psi(z)} \ek*{\hat{\Theta}_{n,\alpha}^{(2,1)}(z) \phi(z)+\hat{\Theta}_{n,\alpha}^{(2,2)}(z) \psi(z)}^\inv \\
 =\matp{\ek{\hat{\Theta}_{n,\alpha}^{(1,1)}(z) \phi(z)+\hat{\Theta}_{n,\alpha}^{(1,2)}(z) \psi(z)}\ek{\hat{\Theta}_{n,\alpha}^{(2,1)}(z) \phi(z)+\hat{\Theta}_{n,\alpha}^{(2,2)}(z) \psi(z)}^\inv}{\ek{\hat{\Theta}_{n,\alpha}^{(2,1)}(z) \phi(z)+\hat{\Theta}_{n,\alpha}^{(2,2)}(z) \psi(z)}\ek{\hat{\Theta}_{n,\alpha}^{(2,1)}(z) \phi(z)+\hat{\Theta}_{n,\alpha}^{(2,2)}(z) \psi(z)}^\inv} 
 = \matp{\hat{S}_{n, \alpha}(z)}{\Iq}
\end{multline}
 holds true for each \(z \in \C \setminus (\rhl \cup \cD )\).
 Let \(\tilde{\Theta}_{n,\alpha}\colon\C\to \Coo{2q}{2q}\) given by \eqref{TD2-1}.
 Taking into account \eqref{KTP}, \rlem{BL4.2-1}, and \eqref{150-1}, for each \(z \in \C \setminus (\rhl \cup \cD)\), this implies
\beql{150-2}\begin{split}
 &\ek*{\Rstr_{\Cs } \tilde{\Theta}_{n,\alpha}(z) } \matp{\ek{\hat{\cE} (z)-\alpha}\phi(z)}{\psi(z)} \ek*{\hat{\Theta}_{n,\alpha}^{(2,1)}(z) \phi(z)+\hat{\Theta}_{n,\alpha}^{(2,2)}(z) \psi(z)}^\inv \\
 & =\ek*{\diag\rk*{\ek*{\hat{\cE}(z)-\alpha}\Iq, \Iq }}\hat{\Theta}_{n,\alpha} (z)\ek*{\diag\rk*{(z - \alpha)^\inv \Iq, \Iq}}\\
 & \qquad\times \matp{(z-\alpha)\phi(z)}{\psi(z)} \ek*{\hat{\Theta}_{n,\alpha}^{(2,1)}(z) \phi(z)+\hat{\Theta}_{n,\alpha}^{(2,2)}(z) \psi(z)}^\inv \\
 & =\ek*{\diag \rk*{ \ek*{\hat{\cE}(z)-\alpha}\Iq, \Iq }}\hat{\Theta}_{n,\alpha} (z)\matp{\phi(z)}{\psi(z)} \ek*{\hat{\Theta}_{n,\alpha}^{(2,1)}(z) \phi(z)+\hat{\Theta}_{n,\alpha}^{(2,2)}(z) \psi(z)}^\inv \\
 & =\ek*{\diag \rk*{ \ek*{\hat{\cE}(z)-\alpha}\Iq, \Iq }}\matp{\hat{S}_{n, \alpha}(z)}{\Iq}
 = \matp{\ek{\hat{\cE}(z)-\alpha}\hat{S}_{n, \alpha}(z)}{\Iq}.
\end{split}\eeq

 (III) Since the functions \(\hat{\cE}\) and \(\hat{R}_{\Tqn}\) are holomorphic in \(\Cs \), statement~\eqref{9.36N.a} is equivalent to the following statement:
\begin{enui}\setcounter{enumi}{2}
 \item\label{9.36N.c} There exists a discrete subset \(\tilde\cD\) of \(\Cs \) such that \(\hat S_{n, \alpha}\) is holomorphic in \(\C \setminus (\rhl \cup \tilde\cD)\) and that
\beql{GL293}
 \rk{\Iu{(n+1)q}-H_n^\mpi  H_n} \hat{R}_{\Tqn}(z) \mat{\Iu{(n+1)q},\Tqn H_n} (\Iu{2} \otimes v_{q,n} ) \matp{\hat{S}_{n, \alpha}(z)}{\Iq}
 =\NM
\eeq
 and 
\begin{multline} \label{GL295}
 \rk{\Iu{(n+1)q}-H_{\at n}^\mpi H_{\at n}}\hat{R}_{\Tqn}(z) \mat*{\Iu{(n+1)q},\ek{R_{\Tqn}(\alpha)}^\inv H_n}\\
 \times (\Iu{2} \otimes v_{q,n} ) \matp{\ek{\hat{\cE}(z)-\alpha}\hat{S}_{n, \alpha}(z)}{\Iq}
 =\NM
\end{multline} 
 hold true for each \(z \in \C \setminus (\rhl \cup \tilde\cD)\). 
\end{enui} 

 (IV) In this step of the proof, we suppose~\eqref{9.36N.c}.
 We are going to prove that the following statement holds true:
\begin{enui}\setcounter{enumi}{3}
 \item\label{9.36N.d} There is a discrete subset \(\hat \cD\) of \(\Cs \) such that \(\phi\) and \(\psi\) are holomorphic in \(\C \setminus (\rhl \cup \hat \cD)\) and that 
\begin{multline} \label{FV1}
 (\Iu{(n+1)q}-H_n^\mpi  H_n) \hat{R}_{\Tqn} (z) \mat{\Iu{(n+1)q},\Tqn H_n}\\
 \times (\Iu{2} \otimes v_{q,n} )  \hat{\Theta}_{n,\alpha}(z) \matp{\phi(z)}{\psi(z)} \ek*{\hat{\Theta}_{n,\alpha}^{(2,1)}(z) \phi(z)+\hat{\Theta}_{n,\alpha}^{(2,2)}(z) \psi(z)}^\inv
 =\NM
\end{multline} 
 and 
\begin{multline} \label{FV2}
 (\Iu{(n+1)q}-H_{\at n}^\mpi  H_{\at n}) \hat{R}_{\Tqn} (z ) \mat*{\Iu{(n+1)q},\ek{R_{\Tqn}(\alpha)}^\inv H_n}\\
 \times (\Iu{2} \otimes v_{q,n} )  \ek*{\Rstr_{\Cs } \tilde{\Theta}_{n,\alpha}(z) } \matp{\ek{\hat{\cE} (z)-\alpha}\phi(z)}{\psi(z)}\\
 \times\ek*{\hat{\Theta}_{n,\alpha}^{(2,1)}(z) \phi(z)+\hat{\Theta}_{n,\alpha}^{(2,2)}(z) \psi(z)}^\inv
 =\NM
\end{multline} 
 are fulfilled for each \(z \in \C \setminus (\rhl \cup \hat \cD)\). 
\end{enui}
 First we observe that \(\cD_\#\defeq \cD \cup \tilde\cD\) is a discrete subset of \(\Cs \).
 Since \eqref{GL293} and \eqref{GL295} are valid for each \(z \in \C \setminus(\rhl \cup \cD_\#)\) and since~(II) shows that \eqref{150-1} and \eqref{150-2} are fulfilled for each \(z \in \C \setminus(\rhl \cup \cD_\#)\), we get that \eqref{FV1} and \eqref{FV2} hold true for each \(z \in \C \setminus(\rhl \cup \cD_\#)\).
 Setting \(\hat \cD =\cD_\#\), statement~\eqref{9.36N.d} is proved. 

 (V) In this step of the proof, we suppose~\eqref{9.36N.d}.
 We are going to prove that~\eqref{9.36N.c} holds true.
 Obviously, \(\cD_{\Box}\defeq \cD \cup \hat \cD\) is a discrete subset of \(\Cs \).
 According to~(I) and~(II), we get \eqref{KTP}, \eqref{150-1}, and \eqref{150-2} for each \(z \in \C \setminus(\rhl \cup \cD_\square)\).
 Using these arguments and \eqref{FV1} and \eqref{FV2}, we see that \eqref{GL293} and \eqref{GL295} are fulfilled for each \(z \in \C \setminus(\rhl \cup \cD_\square)\).
 Consequently, statement~\eqref{9.36N.c}  holds true with \(\tilde\cD = \cD_\square\).

 (VI) Now we verify that statement~\eqref{9.36N.d} implies the following statement: 
\begin{enui}\setcounter{enumi}{4}
 \item\label{9.36N.e} There is a discrete subset \(\tilde\cD_\#\) of \(\Cs\) such that the functions \(\phi\) and \(\psi\) are holomorphic in \(\C \setminus (\rhl \cup \tilde\cD_\#)\) and that 
 \beql{PKZ1-1}
  (\Iu{(n+1)q}-H_n^\mpi  H_n) \hat{R}_{\Tqn} (z) \mat{\Iu{(n+1)q},\Tqn H_n}(\Iu{2} \otimes v_{q,n} ) \hat{\Theta}_{n,\alpha}(z) \matp{\phi(z)}{\psi(z)}
  =\NM
\eeq
 and 
\begin{multline} \label{PKZ1-2}
 (\Iu{(n+1)q}-H_{\at n}^\mpi  H_{\at n}) \hat{R}_{\Tqn} (z ) \mat*{ \Iu{(n+1)q}, \ek{R_{\Tqn}(\alpha)}^\inv H_n}\\
 \times(\Iu{2} \otimes v_{q,n} )\ek*{\Rstr_{\Cs } \tilde{\Theta}_{n,\alpha}(z) } \matp{\ek{\hat{\cE} (z)-\alpha}\phi(z)}{\psi(z)}
 =\NM
\end{multline} 
 hold true for each \(z \in \C\setminus (\rhl \cup \tilde\cD_\#)\).
 \end{enui}
 Let us assume that~\eqref{9.36N.d} is fulfilled.
 Because of~(I), we know that \(\tilde\cD_\square\defeq \cD\cup \hat \cD\) is a discrete subset of \(\Cs \).
 From~(I) and~\eqref{9.36N.d} we see that \eqref{KTP}, \eqref{FV1}, and \eqref{FV2} are valid for each \(\C \setminus (\rhl \cup \tilde\cD_\square)\), which implies \eqref{PKZ1-1} and \eqref{PKZ1-2} for each \(z \in \C (\rhl \cup \tilde\cD_\square)\).
 Consequently,~\eqref{9.36N.e} holds true with \(\tilde\cD_\# = \tilde\cD_\square\). 

 (VII) Now we show that~\eqref{9.36N.e} implies~\eqref{9.36N.d}.
 Let~\eqref{9.36N.e} be fulfilled.
 Obviously, \(\hat \cD_\#\defeq \tilde\cD_\# \cup \cD\) is a discrete subset of \(\Cs \).
 Because of~(I) and~\eqref{9.36N.e}, we know that \eqref{KTP}, \eqref{PKZ1-1}, and \eqref{PKZ1-2} are valid for each \(z\in\C \setminus (\rhl \cup \hat \cD_\#)\).
 Consequently, \eqref{FV1} and \eqref{FV2} hold true for each \(z \in \C \setminus (\rhl \cup \hat \cD_\#)\).
 Hence,~\eqref{9.36N.d} is fulfilled with \(\hat \cD = \hat \cD_\#\).

 (VIII) Since \(\hat{R}_{T_{q,n}}\) is the restriction of \(R_{T_{q,n}}\) onto \(\Cs \), we see that~\eqref{9.36N.e} is equivalent to the following statement:
\begin{enui}\setcounter{enumi}{5}
 \item\label{9.36N.f} There is a discrete subset \(\cD'\) of \(\Cs\) such that \(\phi\) and \(\psi\) are holomorphic in \(\C \setminus (\rhl \cup \cD')\) and that
\beql{N313-1}
 (\Iu{(n+1)q}- H_n^\mpi  H_n) R_{\Tqn} (z) \mat{ \Iu{(n+1)q}, \Tqn H_n} (\Iu{2} \otimes v_{q,n} )\Theta_{n,\alpha}(z) \matp{\phi(z)}{\psi(z)}
 =\NM
\eeq
 and 
\begin{multline} \label{N313-2}
 (\Iu{(n+1)q}-H_{\at n}^\mpi  H_{\at n}) R_{\Tqn} (z ) \mat*{ \Iu{(n+1)q}, \ek{R_{\Tqn}(\alpha)}^\inv H_n}\\
 \times(\Iu{2} \otimes v_{q,n} )\tilde{\Theta}_{n,\alpha}(z)\ek*{\diag\rk*{(z-\alpha)\Iq, \Iq}}\matp{\phi(z)}{\psi(z)}
 =\NM
\end{multline} 
 hold true for each \(z \in \C \setminus (\rhl \cup \cD')\).
\end{enui}
 
 (IX) Let \(P_{n, \alpha}\), \(Q_{n, \alpha}\), and \(S_{n, \alpha}\) be the matrix-valued functions defined (on \(\C\)) by \eqref{V1}, \eqref{V2}, and \eqref{V3}.
 According to \rlemp{145}{145.a}, we see that \(\cN\defeq \zer{\det P_{n, \alpha}} \cup \zer{\det Q_{n, \alpha}} \cup \zer{\det S_{n, \alpha}} \) is a finite and, in particular, discrete subset of \(\C\). 

 (X) By virtue of~(IX), we know that \(\cN\) is a discrete subset of \(\C\).
 We suppose now~\eqref{9.36N.f}.
 Then \(\cN'\defeq \cN \cup \cD'\) is a discrete subset of \(\C\), too.
 From \rlemp{145}{145.b} we see then that the following statement holds true: 
\begin{enui}\setcounter{enumi}{6}
 \item\label{9.36N.g} There is a discrete subset \(\cD''\) of \(\Cs \) such that \(\phi\) and \(\psi\) are holomorphic in \(\C \setminus (\rhl \cup \cD'')\) and that 
\begin{align} 
 \rk{\Iu{(n+1)q}-H_n^\mpi  H_n} R_{\Tqn}(\alpha) v_{q,n} \phi(z)&=\NM\label{N314-1}
\intertext{and}
 \rk{\Iu{(n+1)q}-H_{\at n}^\mpi H_{\at n}} H_n v_{q,n}\psi (z)&=\NM\label{N292-4}
\end{align}
 are fulfilled for each \(z \in \C \setminus (\rhl \cup \cD'')\).
\end{enui}

 (XI) Conversely, now we suppose~\eqref{9.36N.g}.
 We are going to prove~\eqref{9.36N.f}.
 From~(IX) we see that \(\cN\) is a discrete subset of \(\C\).
 Hence, \(\tilde \cN \defeq \cN \cap\rk{\Cs}\) and \(\cD_\square '\defeq \cD'' \cup \tilde \cN\) are discrete subsets of \(\Cs \).
 Because of~\eqref{9.36N.g}, the functions \(\phi\) and \(\psi\) are holomorphic in \(\C \setminus(\rhl\cup \cD_\square ')\) and \eqref{N314-1} and \eqref{N292-4} are valid for each \(z\in\C (\rhl \cup \cD_\square ')\).
 Let us consider an arbitrary \(z \in \C \setminus(\rhl \cup \cD_\square ')\).
 From \eqref{N314-1} and \eqref{N292-4} we get then that \(x\defeq \phi(z)\) and \(y\defeq \psi (z)\) fulfill \eqref{L145.iia} and \eqref{L145.iib}.
 Consequently, \rlem{145} yields then that \eqref{L145.ia} and \eqref{L145.ib} hold true.
 Thus, we see that \eqref{N313-1} and \eqref{N313-2} are true.
 Hence,~\eqref{9.36N.f} is valid with \(\cD'=\cD_\square '\). 

 (XII) In view of \rnota{bez-DT1},~\eqref{9.36N.g} and~\eqref{9.36N.b} are equivalent.

 From~(III)--(VIII) and~(X)--(XII) we see that the statements~\eqref{9.36N.a} and~\eqref{9.36N.b} are equivalent.
\end{proof}

\begin{prop} \label{LB53L} 
 Let \(\alpha \in \R\), let \(\kappa \in \Ninf \), let \(\seqska  \in \Kggeqka \), and let \(n\in \NO \) be such that \(2n+1 \leq \kappa\).
 Let \eqref{Qblock-2} be the \tqqa{block} representation of \(\hat{\Theta}_{n, \alpha}\defeq \Rstr_{\Cs } \Theta_{n, \alpha}\).
 Then:
\begin{enui}
 \item\label{LB53L.a} For each \( \tmatp{\phi}{\psi} \in \qSps{2n+1}\), the function \(\det (\hat \Theta_{n, \alpha}^{(2,1)} \phi +\hat \Theta_{n, \alpha}^{(2,2)} \psi)\) does not vanish identically in \(\Cs \) and 
\beql{NT}
 \hat{S}_{n, \alpha}
 \defeq \rk{\hat\Theta_{n, \alpha}^{(1,1)} \phi+\hat\Theta_{n, \alpha}^{(1,2)} \psi}\rk{\hat\Theta_{n, \alpha}^{(2,1)} \phi+\hat\Theta_{n, \alpha}^{(2,2)} \psi}^\inv 
\eeq
 belongs to the class \(\SFOqskg{2n+1}\).
 \item\label{LB53L.b} For each \(S \in \SFOqskg{2n+1}\), there exists a pair \(\tmatp{\phi}{\psi}\in\qSps{2n+1}\) consisting of two in \(\Cs \) holomorphic \tqqa{matrix-valued} functions \(\phi\) and \(\psi\) such that 
\beql{RN322}
 \det \ek*{\hat \Theta_{n, \alpha}^{(2,1)}(z) \phi(z) +\hat \Theta_{n, \alpha}^{(2,2)}(z) \psi(z)}
 \neq 0
\eeq
 and 
\beql{KPL}
 S(z)
 =\ek*{\hat\Theta_{n, \alpha}^{(1,1)}(z) \phi(z)+\hat\Theta_{n, \alpha}^{(1,2)}(z) \psi(z)} \ek*{\hat\Theta_{n, \alpha}^{(2,1)} (z)\phi(z)+\hat\Theta_{n, \alpha}^{(2,2)}(z) \psi(z)}^\inv 
\eeq
 hold true for each \(z \in \Cs \).
 \item\label{LB53L.c} Let \(\tmatp{\phi_1}{\psi_1}, \tmatp{\phi_2}{\psi_2} \in \qSps{2n+1}\).
 Then \(\tmatpc{\phi_1}{\psi_1} =\tmatpc{\phi_2}{\psi_2}\) if and only if 
\begin{multline}\label{KL681} 
 \rk{\hat\Theta_{n, \alpha}^{(1,1)} \phi_1+\hat\Theta_{n, \alpha}^{(1,2)} \psi_1}\rk{\hat\Theta_{n, \alpha}^{(2,1)} \phi_1+\hat\Theta_{n, \alpha}^{(2,2)} \psi_1}^\inv \\
 =\rk{\hat\Theta_{n, \alpha}^{(1,1)} \phi_2+\hat\Theta_{n, \alpha}^{(1,2)} \psi_2}\rk{\hat\Theta_{n, \alpha}^{(2,1)} \phi_2+\hat\Theta_{n, \alpha}^{(2,2)} \psi_2}^\inv.
\end{multline} 
\end{enui}
\end{prop}
\begin{proof}
 Since \(\seqska \) belongs to \(\Kggeqka \), we have \(s_j^\ad =s_j \) for all \(j\in \mn{0}{ \kappa}\) and 
\beql{NB151}
 \set{H_n,H_{\at n}} 
 \subseteq \Cggo{(n+1)q}. 
\eeq
 \rrem{SBNNN} yields then \eqref{BL}.
 \rrem{L238-1} provides us \( \hat \Theta_{n,\alpha}\in \PJisf \).
 From \rlem{BLR4.3-1} we get \(\det \tilde \Theta_{n,\alpha} (z) \neq 0\) and \(\det \Theta_{n,\alpha} (z) \neq 0\) for each \(z\in\C \).

 \eqref{LB53L.a} Let \( \tmatp{\phi}{\psi} \in \qSps{2n+1}\).
 According to \rnota{bez-DT1}, the equations  \eqref{175-v7} and \eqref{175-v8} are fulfilled.
 Using \rlem{165}, we see that there is a discrete subset \(\tilde\cD\) of \(\Cs \) such that \(\phi\) and \(\psi\) are holomorphic in \(\C \setminus (\rhl \cup \tilde\cD)\) and that \eqref{KTP} holds true for each \(z \in \C \setminus (\R\cup \tilde\cD)\).
 Because of \(\hat \Theta_{n,\alpha} \in  \PJisf \) and \rpropp{S51-1}{S51-1.c}, the following three statements are valid:
\begin{Aeqi}{0}
 \item\label{LB53L.I} There is a discrete subset \(\cD\) of \(\Cs \) such that \(\hat \Theta_{n,\alpha}\), \(\phi\), and \(\psi\) are holomorphic in \(\C \setminus (\rhl \cup \cD)\).
 \item\label{LB53L.II} The matrix-valued function \(\hat{S}_{n, \alpha}\) is holomorphic in \(\C \setminus (\rhl \cup \cD)\), and \eqref{KTP} as well as the representation 
\[
 \hat{S}_{n, \alpha}(z)
 =\ek*{\hat\Theta_{n, \alpha}^{(1,1)}(z) \phi(z)+\hat\Theta_{n, \alpha}^{(1,2)}(z) \psi(z)} \ek*{\hat\Theta_{n, \alpha}^{(2,1)} (z)\phi(z)+\hat\Theta_{n, \alpha}^{(2,2)}(z) \psi(z)}^\inv 
\]
 of \(S_{n,\alpha}\) are fulfilled for each \(z\in\C \setminus (\rhl \cup \cD)\).
 \item\label{LB53L.III} For each \(z\in\C \setminus (\R \cup \cD)\), 
\beql{9N24}
 \matp{\hat{S}_{n, \alpha}(z)}{\Iq}^\ad \hat\Theta_{n, \alpha}^\invad (z) \rk*{\frac{-\Jimq}{2\Im z}} \hat\Theta_{n, \alpha}^\inv (z) \matp{\hat{S}_{n, \alpha}(z)}{\Iq} 
 \geq\NM
\eeq
 and
\begin{multline} \label{TB8}
 \matp{\hat{S}_{n, \alpha}(z)}{\Iq}^\ad \hat\Theta_{n, \alpha}^\invad (z)\ek*{\diag\rk*{(z-\alpha) \Iq,\Iq} }^\ad \rk*{\frac{-\Jimq}{2\Im z}} \\ 
 \times\ek*{\diag\rk*{(z-\alpha) \Iq,\Iq}}\hat\Theta_{n, \alpha}^\inv (z) \matp{\hat{S}_{n, \alpha}(z)}{\Iq}
 \lgeq\NM.
\end{multline}
\end{Aeqi}
 In view of \eqref{175-v7} and \eqref{175-v8}, \rlem{9.36N} provides us \eqref{nN309-1} and \eqref{nN309-2}, where \(\hat R_{\Tqn} \defeq \Rstr_{\Cs } R_{T_{q,n}}\) and where \(\hat{\cE}\colon\Cs \to \C\) is given by \(\hat \cE(z)\defeq z\).
 Using \eqref{nN309-1} and \eqref{nN309-2}, we obtain 
\[
 \hat R_{\Tqn}\mat{\Iu{(n+1)q},\Tqn H_n} (\Iu{2} \otimes v_{q,n}) \matp{\hat{S}_{n, \alpha}}{\Iq}
 = H_n^\mpi H_n \hat R_{\Tqn} \mat{\Iu{(n+1)q},\Tqn H_n} (\Iu{2} \otimes v_{q,n}) \matp{\hat{S}_{n, \alpha}}{\Iq}
\]
 and 
\begin{multline*}
 \hat R_{\Tqn}  \mat*{\Iu{(n+1)q},\ek{R_{\Tqn}(\alpha)}^\inv H_n} (\Iu{2} \otimes v_{q,n} ) \matp{(\hat{\cE} -\alpha)\hat{S}_{n, \alpha}}{\Iq} \\ 
 =H_{\at n}^\mpi  H_{\at n} \hat R_{\Tqn} \mat*{\Iu{(n+1)q},\ek{R_{\Tqn}(\alpha)}^\inv H_n} (\Iu{2} \otimes v_{q,n} ) \matp{(\hat{\cE} -\alpha)\hat{S}_{n, \alpha}}{\Iq}.
\end{multline*}
 Consequently, from \eqref{BL} and~\ref{LB53L.II} we get then
\beql{N324-1}
 \Ran{\hat R_{\Tqn} (z)\mat{\Iu{(n+1)q}, \Tqn H_n} (\Iu{2} \otimes v_{q,n} ) \matp{\hat{S}_{n, \alpha}(z)}{\Iq}}
 \subseteq \ran{H_n}
\eeq
 and 
\beql{N324-2}
 \Ran{ \hat R_{\Tqn} (z)\mat*{\Iu{(n+1)q}, \ek{R_{\Tqn}(\alpha)}^\inv H_n}(\Iu{2} \otimes v_{q,n} ) \matp{(z -\alpha)\hat{S}_{n, \alpha}(z)}{\Iq} } 
 \subseteq \ran{H_{\at n}}
\eeq
 for each \(z \in \C \setminus (\R \cup \cD)\).
 From \rrem{lemC21-1} we know that \(T_{q,n} H_n v_{q,n} = -u_n\).
 Therefore, for all \(z \in \C \setminus (\R \cup \cD)\), we get
\beql{NMK}
 \hat R_{\Tqn} (z) \mat{\Iu{(n+1)q},\Tqn H_n} (\Iu{2} \otimes v_{q,n}) \matp{\hat{S}_{n, \alpha}(z)}{\Iq}
 = \hat R_{\Tqn} (z) \ek*{v_{q,n}\hat{S}_{n, \alpha}(z)-u_n}
\eeq
 and 
\[
 \ek*{R_{\Tqn}(\alpha)}^\inv H_nv_{q,n}
 =(\Iu{(n+1)q}-\alpha \Tqn) H_n v_{q,n}
 = H_n v_{q,n}-\alpha \Tqn H_n v_{q,n}
 = y_{0,n}+\alpha u_n
\]
 and, hence,
\beql{NML}\begin{split} 
 &\hat R_{\Tqn}(z)\mat*{\Iu{(n+1)q},\ek{R_{\Tqn}(\alpha)}^\inv H_n} (\Iu{2} \otimes v_{q,n} ) \matp{(z -\alpha)\hat{S}_{n, \alpha}(z)}{\Iq} \\
 &= \hat R_{\Tqn} (z) \rk*{v_{q,n}\ek*{(z -\alpha)\hat{S}_{n, \alpha}(z)}+\ek*{R_{\Tqn}(\alpha)}^\inv H_n v_{q,n}} \\ 
 &= \hat R_{\Tqn} (z) \rk*{v_{q,n}\ek*{(z-\alpha)\hat{S}_{n, \alpha}(z)} -(-\alpha u_n -y_{0,n} )}.
\end{split}\eeq
 For all \(z \in \C \setminus (\R \cup \cD)\), the equations \eqref{NMK} and \eqref{N324-1} imply 
\beql{TX}
 \Ran{R_{\Tqn} (z) \ek*{v_{q,n}\hat{S}_{n, \alpha}(z)-u_n} }
 = \Ran{\hat R_{\Tqn} (z) \ek*{v_{q,n}\hat{S}_{n, \alpha}(z)-u_n} }
 \subseteq \ran{H_n} 
\eeq
 and, in view of \eqref{NML} and \eqref{N324-2}, furthermore, 
\begin{multline} \label{TQ} 
 \Ran{ R_{\Tqn} (z) \rk*{v_{q,n}\ek*{(z-\alpha)\hat{S}_{n, \alpha}(z)} -(-\alpha u_n -y_{0,n} )} } \\
 = \Ran{ \hat R_{\Tqn} (z) \rk*{v_{q,n}\ek*{(z-\alpha)\hat{S}_{n, \alpha}(z)} -(-\alpha u_n -y_{0,n} )} }
 \subseteq \ran{H_{\at n}}. 
\end{multline} 
 \rlem{OQ-1} shows that, for each \(z \in \C \setminus (\R \cup \cD)\), the matrix \(\Sigma_{2n}^{[\hat{S}_{n, \alpha}]}(z)\) given by \eqref{N54} admits the representation
\beql{11.35A}
 \Sigma_{2n}^{[\hat{S}_{n, \alpha}]}(z)
 = \matp{\hat{S}_{n, \alpha}(z)}{\Iq}^\ad\hat\Theta_{n, \alpha}^\invad (z) \rk*{\frac{-\Jimq}{2\Im z}} \hat\Theta_{n, \alpha}^\inv (z) \matp{\hat{S}_{n, \alpha}(z)}{\Iq}.
\eeq
 In view of \(\det \Theta_{n,\alpha} (z) \neq 0\), for each \(z \in \Cs \), we have 
\begin{multline} \label{N325N}
 \rk*{\ek*{\diag \rk*{(z - \alpha)\Iq,\Iq}}\Theta_{n,\alpha} (z)\ek*{\diag \rk*{(z - \alpha)^\inv\Iq,\Iq} }}^\inv \\
 =\ek*{\diag \rk*{(z - \alpha)\Iq,\Iq}} \hat \Theta_{n,\alpha}^\inv (z)\ek*{\diag \rk*{(z - \alpha)^\inv\Iq,\Iq}}. 
\end{multline}
 Taking into account  \rlem{OQ-1}, we see that, for each \(z \in \C \setminus (\R \cup \cD)\), the matrix \(\Sigma_{2n+1}^{[\hat{S}_{n, \alpha}]}(z)\) given by \eqref{N55} can be represented by 
\beql{TKZW1}
 \Sigma_{2n+1}^{[\hat{S}_{n, \alpha}]}(z)
 =\matp{(z-\alpha)\hat{S}_{n, \alpha}(z)}{\Iq}^\ad  \tilde\Theta_{n, \alpha}^\invad (z) \rk*{\frac{-\Jimq}{2\Im z}} \tilde\Theta_{n, \alpha}^\inv (z) \matp{(z-\alpha)\hat{S}_{n, \alpha}(z)}{\Iq}.
\eeq
 For each \(z \in \Cs\), \rlem{BL4.2-1} and \eqref{N325N} yield 
\[
 \tilde\Theta_{n, \alpha}^\inv (z) \matp{(z-\alpha)\hat{S}_{n, \alpha}(z)}{\Iq}
 =\ek*{\diag \rk*{(z - \alpha)\Iq,\Iq}}\hat \Theta_{n,\alpha}^\inv (z) \matp{\hat{S}_{n, \alpha}(z)}{\Iq}, 
\]
 which, because of \eqref{TKZW1}, implies 
\begin{multline} \label{TKZW2}
 \Sigma_{2n+1}^{[\hat{S}_{n, \alpha}]}(z)
 =\matp{\hat{S}_{n, \alpha}(z)}{\Iq}^\ad \hat\Theta_{n, \alpha}^\invad (z) \ek*{\diag\rk*{(z - \alpha)\Iq,\Iq} }^\ad\rk*{\frac{-\Jimq}{2\Im z}} \\ 
 \times\ek*{\diag \rk*{(z - \alpha)\Iq,\Iq}}\hat\Theta_{n, \alpha}^\inv (z) \matp{\hat{S}_{n, \alpha}(z)}{\Iq} 
\end{multline}
 for each \(z \in \C \setminus (\R \cup \cD)\).
 From \eqref{11.35A}, \eqref{9N24}, \eqref{TKZW2}, and \eqref{TB8} it follows 
\beql{S71}
 \set*{ \Sigma_{2n}^{[\hat{S}_{n, \alpha}]}(z), \Sigma_{2n+1}^{[\hat{S}_{n, \alpha}]}(z)}
 \subseteq \Cggq 
\eeq
 for each \(z \in \C \setminus (\R \cup \cD)\).
 Thus, for all \(z \in \C \setminus (\R \cup \cD)\), by virtue of \eqref{N52}, \eqref{NB151}, \eqref{TX}, \eqref{TQ}, \eqref{S71}, and \rrem{remark4.5}, we conclude then \(\set{P_{2n}^{[\hat{S}_{n, \alpha}]}(z), P_{2n+1}^{[\hat{S}_{n, \alpha}]}(z)} \subseteq \Cggo{(n+2)q}\).
 Obviously, \(\tilde\cD\defeq \cD \cap \uhp\) is a discrete subset of \(\uhp\) and \(f_{n, \alpha}\defeq \Rstr_{\uhp\setminus \tilde\cD} \hat{S}_{n, \alpha}\) is holomorphic in \(\uhp\setminus \tilde\cD\).
 For each \(z \in \uhp\setminus \tilde\cD\), then \(\set{P_{2n}^{[f_{n,\alpha}]} (z), P_{2n+1}^{[f_{n,\alpha}]} (z)} \subseteq \Cggo{(n+2)q}\).
 Thus, \rthm{T1747F} provides us that there is a unique \(S \in \SFOqskg{2n+1}\) such that \(\Rstr_{\uhp\setminus \tilde\cD} S= f_{n, \alpha}\).
 Consequently, for each \(z \in \uhp\setminus \tilde\cD\), we have \(S(z)= f_{n, \alpha}(z)= \hat{S}_{n, \alpha}(z)\).
 Since \(S\) is holomorphic in \(\Cs \), we get \(S=\hat{S}_{n, \alpha}\).

 \eqref{LB53L.b} Now we consider an arbitrary \(S \in \SFOqskg{2n+1}\).
 Then \rprop{lemM4213-2} yields \(\set{ P_{2n}^{[S]}(z), P_{2n+1}^{[S]}(z)} \subseteq \Cggo{(n+2)q}\) for each \(z \in \C \setminus \R\).
 Consequently, \rrem{remark4.5} shows that \eqref{NB151} is valid and that, for each \(z\in\C\setminus\R\), the following statements hold true: 
\begin{aeqi}{0}
 \item\label{LB53L.i} The inclusions \(\ran{R_{\Tqn} (z) \ek{v_{q,n}S(z)-u_n}} \subseteq \ran{H_n}\) and
\[
 \Ran{R_{\Tqn} (z) \rk*{v_{q,n}\ek*{(z-\alpha)S(z)}-(-\alpha u_n -y_{0,n})}}
 \subseteq \ran{H_{\at n}}
\]
 are valid.
 \item\label{LB53L.ii} The matrices \( \Sigma_{2n}^{[S]}(z) \) and \( \Sigma_{2n+1}^{[S]}(z)\) are both \tnnH{}.
\end{aeqi}
 For each \(z \in \Cs \), from \rrem{21112N} we get
\beql{TVR1}\begin{split}
 R_{\Tqn} (z) \ek*{v_{q,n} S(z) -u_n}
 &=\hat R_{\Tqn} (z) \mat{v_{q,n},-u_n} \matp{S(z)}{\Iq}\\
 &=\hat R_{\Tqn} (z) \mat{\Iu{(n+1)q},\Tqn H_n} (\Iu{2} \otimes v_{q,n} ) \matp{S(z)}{\Iq} 
\end{split}\eeq
 and, because of the second equation in \eqref{RHV-1}, furthermore
\beql{TVR2}\begin{split}
 &R_{\Tqn}(z) \rk*{v_{q,n}\ek*{(z -\alpha)S(z)}-(-\alpha u_n-y_{0,n})}\\
 & = R_{\Tqn} (z)\rk*{v_{q,n}\ek*{(z -\alpha)S(z)}+\ek*{R_{\Tqn}(\alpha)}^\inv H_n v_{q,n}}\\ 
 &= \hat R_{\Tqn} (z)\mat*{\Iu{(n+1)q}, \ek*{R_{\Tqn}(\alpha)}^\inv H_n}(\Iu{2} \otimes v_{q,n} ) \matp{(z -\alpha)S(z)}{\Iq}.
\end{split}\eeq
 Using~\ref{LB53L.ii} and \rlemss{BLR4.3-1}{OQ-1}, we see that
\begin{align}
 \matp{S(z)}{\Iq}^\ad \hat\Theta_{n, \alpha}^\invad (z) \rk*{\frac{-\Jimq}{2\Im z}} \hat\Theta_{n, \alpha}^\inv (z) \matp{S(z)}{\Iq} 
 &\lgeq\NM\label{TSW1}
\intertext{and}
 \matp{(z-\alpha)S(z)}{\Iq}^\ad \tilde\Theta_{n, \alpha}^\invad (z) \rk*{\frac{-\Jimq}{2\Im z}} \tilde\Theta_{n, \alpha}^\inv (z) \matp{(z-\alpha)S(z)}{\Iq} 
 &\lgeq\NM\label{TSW2}
\end{align}
 hold true for each \(z \in \C \setminus \R\).
 In view of \rlem{BL4.2-1}, we have
\[
 \tilde\Theta_{n, \alpha}^\inv (z) \matp{(z-\alpha)S(z)}{\Iq}
 = \ek*{\diag \rk*{(z - \alpha)\Iq,\Iq}}\hat \Theta_{n,\alpha}^\inv (z) \matp{S(z)}{\Iq} 
\]
 for each \(z \in \C \setminus \R\).
 Consequently, from \eqref{TSW2} it follows 
\beql{TSW3}
 \matp{S(z)}{\Iq}^\ad \hat\Theta_{n, \alpha}^\invad (z)\ek*{\diag\rk*{(z - \alpha)\Iq,\Iq} }^\ad  \rk*{\frac{-\Jimq}{2\Im z}}\ek*{\diag\rk*{(z - \alpha)\Iq,\Iq}}\hat\Theta_{n, \alpha}^\inv (z) \matp{S(z)}{\Iq}
 \lgeq\NM
\eeq
 for all \(z \in \C \setminus \R\).
 Since \rlem{BLR4.3-1} shows that \(\det \hat\Theta_{n, \alpha}(z) \neq 0\) holds true for each \(z \in \Cs \), we get from  \(\hat\Theta_{n, \alpha} \in \PJisf \), \eqref{Qblock-2}, \rlem{RF1}, \eqref{TSW1}, \eqref{TSW3}, and \rpropp{S51-1}{S51-1.b} (with \(\cD =\emptyset\)) that there is a pair \(\tmatp{\phi}{\psi} \in \qSp \) of in \(\Cs \) holomorphic matrix-valued functions \(\phi\) and \(\psi\) such that \eqref{RN322} and \eqref{KPL} hold true for each \(z \in \Cs \).
 Because of \eqref {TVR1} and~\ref{LB53L.i}, we have
\[
 \Ran{ \hat R_{\Tqn} (z) \mat{\Iu{(n+1)q}, \Tqn H_n } (\Iu{2} \otimes v_{q,n} ) \matp{S(z)}{\Iq} } 
 \subseteq \ran{H_n}
\]
 for each \(z \in \C \setminus \R\).
 Consequently, from \rlem{Gl,3.1.10} and \eqref{BL} it follows
\[
 (\Iu{(n+1)q} -H_n^\mpi  H_n)\hat R_{\Tqn} (z) \mat{\Iu{(n+1)q},\Tqn H_n} (\Iu{2} \otimes v_{q,n} ) \matp{S(z)}{\Iq}
 =\NM
\]
 for each \(z \in \C \setminus \R\).
 Hence, the identity theorem for holomorphic functions yields
\beql{11.50A}
 (\Iu{(n+1)q} -H_n^\mpi  H_n)\hat R_{\Tqn} \mat{\Iu{(n+1)q},\Tqn H_n} (\Iu{2} \otimes v_{q,n} ) \matp{S}{\Iq} 
 =\NM.
\eeq
 Because of \eqref{TVR2} and~\ref{LB53L.i}, we obtain
\[
 \Ran{ \hat R_{\Tqn} (z) \mat*{ \Iu{(n+1)q}, \ek{R_{\Tqn}(\alpha)}^\inv H_n }(\Iu{2} \otimes v_{q,n} ) \matp{(z-\alpha)S(z)}{\Iq} }
 \subseteq \ran{H_{\at n}}
\]
 for each \(z \in \C \setminus \R\).
 Thus, for each \(z \in \C \setminus \R\), \rlem{Gl,3.1.10} and \eqref{BL} imply
\[
 (\Iu{(n+1)q} -H_{\at n}^\mpi  H_{\at n}) \hat R_{\Tqn} (z) \mat*{\Iu{(n+1)q},\ek{R_{\Tqn}(\alpha)}^\inv H_n} (\Iu{2} \otimes v_{q,n} ) \matp{(z -\alpha)S(z)}{\Iq} 
 =\NM.
\]
 Applying again the identity theorem for holomorphic functions, it follows 
\beql{11.51A}
 (\Iu{(n+1)q} -H_{\at n}^\mpi  H_{\at n}) \hat R_{\Tqn}(z) \mat*{\Iu{(n+1)q},\ek{R_{\Tqn}(\alpha)}^\inv H_n} (\Iu{2} \otimes v_{q,n} ) \matp{(z -\alpha)S(z)}{\Iq} =0
\eeq
 for all \(z\in\C\setminus\R\).
 In view of \eqref{Qblock-2}, \eqref{RN322}, \eqref{KPL}, \eqref{11.50A}, and \eqref{11.51A}, then \rlem{9.36N} shows that \(\tmatp{\phi}{\psi}\) belongs to \(\qSps{2n+1}\). 

 \eqref{LB53L.c} In view of \rpart{LB53L.a}, we know that, for each \(k \in \set{1,2}\), the function \(\det (\hat \Theta_{n, \alpha}^{(2,1)} \phi_k +\hat \Theta_{n, \alpha}^{(2,2)} \psi_k)\) does not vanish identically in \(\Cs \).
 Because of \(\hat\Theta_{n, \alpha} \in \PJisf \), the application of \rpropp{S51-1}{S51-1.d} provides us the asserted equivalence.
\end{proof}

\section{Parametrization of the solution set of the truncated matricial Stieltjes moment problem in the non-degenerate and degenerate cases}\label{S1313}
 In this section, we state a parametrization of the solution set of the matricial truncated Stieltjes moment problem~\iproblem{\rhl}{2n+1}{\lleq} in the non-degenerate and degenerate cases.
 First we recall that, in view of \rthmss{T1122}{NT1}, one can suppose that the given sequence \(\seqs{2n+1}\) of  complex \tqqa{matrices} belongs to the set \(\Kggeq{2n+1}\).
 
 Let \(\alpha \in \R\), let \(\kappa \in \Ninf \), let \(\seqska  \in \Kggeqka \), and let \(n\) be a \tnn{} integer with \(2n+1 \leq \kappa\).
 According to \rlem{170}, the \tnn{} integers
\begin{align} 
  m & \defeq \rank \ek*{\rk{\Iu{(n+1)q}-H_n^\mpi  H_n} R_{\Tqn}(\alpha) v_{q,n}} \label{175-v1} 
\intertext{and} 
  \ell&\defeq \rank \ek*{\rk{\Iu{(n+1)q}-H_{\at n}^\mpi H_{\at n}}H_nv_{q,n} } \label{175-v2}
\end{align}
 fulfill \(m+\ell\leq q\).
 In particular, \(0 \leq m \leq q\) and \(0 \leq\ell\leq q\).
 We consider separately the following three cases:
\begin{Aeqi}{0}
 \item\label{S1313.I} \(m+\ell=0\), \tie{}, \(m=0\) and \(\ell=0\).
  \item\label{S1313.II} \(1 \leq m+\ell \leq q-1\). 
  \item\label{S1313.III} \(m+\ell = q\).
 \end{Aeqi} 

 Throughout this section, let \(\Theta _{n, \alpha}\colon\C \to \Coo{2q}{2q}\) be defined by \eqref{TD2-2}, let \(\hat{\Theta}_{n, \alpha}\defeq \Rstr_{\Cs } \Theta_{n, \alpha}\), and let \eqref{Qblock-2} be the \tqqa{block} partition of \( \hat{\Theta}_{n, \alpha}\).

\subsection{The non-degenerate case}\label{sub1258} 
 First we study the so-called non-degenerate case~\ref{S1313.I}, \tie{}, we consider the situation that \( \rk{\Iu{(n+1)q}-H_n^\mpi  H_n} R_{\Tqn}(\alpha) v_{q,n}=\NM\) and \( \rk{\Iu{(n+1)q}-H_{\at n}^\mpi H_{\at n}}H_nv_{q,n}=\NM\).
 
\begin{rem} \label{r-q-175}%
 Let \(\alpha \in \R\), let \(\kappa \in \Ninf \), and let \(\seqska  \in \Kggeqka \).
 Let \(n\in\NO \) be such that \(2n+1 \leq \kappa\) and let \(m\) and \(\ell\) be given by \eqref{175-v1} and \eqref{175-v2}, respectively.
 If \(m=0\) and \(\ell=0\), then \rlem{170} and \rnota{bez-DT1} show that \(\cU_{n, \alpha} =\set{\Ouu{q}{1}}\), \(\cV_{n, \alpha}=\set{\Ouu{q}{1}}\), and \(\qSps{2n+1}=\qSp \) hold true.
\end{rem}

 Thus, in the non-degenerate case~\ref{S1313.I}, we get immediately a parametrization of the set \(\SFOqskg{2n+1}\):

\begin{thm} \label{T1} 
 Let \(\alpha \in \R\), let \(\kappa \in \Ninf \), let \(\seqska  \in \Kggeqka \), and let \(n \in \NO \) be such that \(2n+1 \leq \kappa\).
 Suppose that \(m\) and \(\ell\) given by \eqref{175-v1} and \eqref{175-v2} fulfill \(m=0\) and \(\ell=0\).
 Then the following statements hold true:
\begin{enui}
 \item\label{T1.a} For each pair \( \tmatp{\phi}{\psi} \in \qSp \), the meromorphic function \(\det (\hat \Theta_{n, \alpha}^{(2,1)} \phi +\hat \Theta_{n, \alpha}^{(2,2)} \psi)\) does not vanish identically in \(\Cs \) and the matrix-valued function 
 \( \hat{S}_{n, \alpha}\defeq \rk{\hat\Theta_{n, \alpha}^{(1,1)} \phi+\hat\Theta_{n, \alpha}^{(1,2)} \psi} \rk{\hat\Theta_{n, \alpha}^{(2,1)} \phi+\hat\Theta_{n, \alpha}^{(2,2)} \psi}^\inv \) belongs to \(\SFOqskg{2n+1}\).
 \item\label{T1.b} For each \(S \in \SFOqskg{2n+1}\), there is a pair \(\tmatp{\phi}{\psi} \in \qSp \) of \tqqa{matrix-valued} functions \(\phi\) and \(\psi\) which are holomorphic in \(\Cs \) such that \eqref{RN322} and \eqref{KPL} hold true for each \(z \in \Cs \).
 \item\label{T1.c} Let \( \tmatp{\phi_1}{\psi_1}, \tmatp{\phi_2}{\psi_2} \in \qSp \).
 Then \(\tmatpc{\phi_1}{\psi_1} =\tmatpc{\phi_2}{\psi_2}\) if and only if \eqref{KL681} holds true.
\end{enui}
\end{thm}
\begin{proof}
 Apply \rprop{LB53L} and \rrem{r-q-175}.
\end{proof}

\subsection{The degenerate, but not completely degenerate case}\label{sub1310}
 Now we turn our attention to case~\ref{S1313.II}.

\begin{rem} \label{WN2}
 Let \(\alpha \in \R\).
 Let \(V\) and \(W\) be complex \tqqa{matrices} with \(W^\ad V=\Iq\).
 Then it is readily checked that the following statements hold true:
\begin{enui}
 \item\label{WN2.a} If \(\tmatp{\phi}{\psi} \in \qSp \), then the pair \(\tmatp{\phi^\square }{\psi^\square } \) given by \(\phi^\square \defeq V \phi\) and \( \psi^\square \defeq W \psi\) belongs to \( \qSp \).
 \item\label{WN2.b} Let \(\tmatp{\phi_1}{\psi_1}, \tmatp{\phi_2}{\psi_2} \in \qSp \) and let \(\phi_1^\square  \defeq V\phi_1\), \(\psi_1^\square  \defeq W\psi_1\), \(\phi_2^\square  \defeq V \phi_2\), and \(\psi_2^\square  \defeq W\psi_2\).
 Then \(\tmatpc{\phi_1^\square }{\psi_1^\square } =\tmatpc{\phi_2^\square }{\psi_2^\square }\) if and if \(\tmatpc{\phi_1}{\psi_1} =\tmatpc{\phi_2}{\psi_2}\).
\end{enui}
\end{rem} 

\begin{lem} \label{WN1}
 Let \(\alpha \in \R\) and let \(r \in \N\) be such that \(r<q\).
 Let \(U\) and \(V\) be complex \taaa{(q-r)}{(q-r)}{matrices} with \(\rank \tmatp{U}{V}=q-r\) and \(V^\ad U= \Ouu{(q-r)}{(q-r)}\).
 Let \(\cU\) (\tresp{}\ \(\cV\)) be the constant matrix-valued function (defined on \(\Cs \)) with value \(U\) (\tresp{}\ \(V\)).
 Then: 
\begin{enui}
 \item\label{WN1.a} If \(\tmatp{\phi}{\psi} \in \aSp{r}\), then the pair \(\tmatp{\phi^\square }{\psi^\square }\) given by
 \(\phi^\square \defeq \diag (\phi, \cU) \) and \(\psi^\square \defeq \diag (\psi, \cV)\) belongs to \( \qSp \).
 \item\label{WN1.b} Let \(\tmatp{\phi_1}{\psi_1}, \tmatp{\phi_2}{\psi_2} \in \aSp{r}\).
 For each \(k \in \set{1,2}\), let \(\phi_k^\square \defeq \diag (\phi_k, \cU) \) and \(\psi_k^\square \defeq \diag (\psi_k, \cV) \).
 Then \(\tmatpc{\phi_1^\square }{\psi_1^\square } =\tmatpc{\phi_2^\square }{\psi_2^\square }\) if and only if \(\tmatpc{\phi_1}{\psi_1} =\tmatpc{\phi_2}{\psi_2}\).
\end{enui}
\end{lem} 

 The proof of \rlem{WN1} is straightforward.
 We omit the details.
 
 In  the following, we will use again \(\OPu{\cU}\) to denote the complex \tqqa{matrix} which represents the orthogonal projection onto a given subspace \(\cU\) of \(\Cq \) with respect to the standard basis of \(\Cq \), \tie{}, for each subspace \(\cU\) of \(\Cq \), the matrix \(\OPu{\cU}\) is the unique complex \tqqa{matrix} \(P\) which fulfills the three conditions \(P^2 = P\), \(P^\ad = P\), and \(\cR (P) = \cU\).
 
\begin{lem}\label{N11.7}
 Let \(m\) and \(\ell\) be \tnn{} integers such that 
\beql{175-v3}
 r
 \defeq q - (m+\ell)
\eeq
 fulfills \(1\le r\le q - 1\).
 Let \(\cU\) and \(\cV\) be orthogonal subspaces of \(\Cq \) with \(\dim \cU = m\) and \(\dim \cV =\ell\).
 Then:
\begin{enui}
 \item\label{N11.7.a} There exists a unitary complex \tqqa{matrix} \(W\) such that
\beql{N11.16}
 W^\ad \OPu{\cU} W
 =
 \begin{cases}
  \diag (\Ouu{r}{r}, \Iu{m}, \Ouu{\ell}{\ell} ),\ifa{m\ge 1\text{ and }\ell\ge 1}\\
  \diag (\Ouu{r}{r}, \Iu{m}),\ifa{m\ge 1\text{ and }\ell = 0}
 \end{cases}
\eeq 
 and
\beql{N11.17}
 W^\ad \OPu{\cV} W
 =
 \begin{cases}
  \diag (\Ouu{r}{r}, \Ouu{m}{m}, \Iu{\ell}),\ifa{m\ge 1\text{ and }\ell\ge 1}\\
  \diag (\Ouu{r}{r}, \Iu{\ell}),\ifa{m= 0\text{ and }\ell\ge 1}
 \end{cases}.
\eeq 
 \item\label{N11.7.b} Let \(\alpha\in\R\) and let \(W\) be a unitary complex \tqqa{matrix} such that \eqref{N11.16} and \eqref{N11.17} are fulfilled.
 \begin{enuii}
  \item\label{N11.7.b1} If \(\tmatp{\tilde{\phi}}{\tilde{\psi}} \in \qSp \) is such that
\begin{align}\label{Puv}
 \OPu{\cU} \tilde{\phi}&= \Oqq&%
&\text{and}&
 \OPu{\cV} \tilde{\psi}&= \Oqq,%
\end{align}
 then there exists a pair \(\tmatp{\phi}{\psi} \in \rSp \) such that \(\phi\) and \(\psi\) and the functions
\begin{align}
 \phi^\square
 &\defeq
 \begin{cases}
   W \cdot \diag (\phi,\Ouu{m}{m}, \Iu{\ell}),\ifa{m\ge 1\text{ and }\ell\ge 1}\\
   W \cdot \diag (\phi, \Ouu{m}{m}),\ifa{m\ge 1 \text{ and }\ell=0}\\
   W \cdot \diag (\phi, \Iu{\ell}),\ifa{m=0 \text{ and }\ell\ge 1}\label{NN232}
 \end{cases}
\intertext{and}
 \psi^\square
 &\defeq
 \begin{cases}
   W \cdot \diag (\psi, \Iu{m}, \Ouu{\ell}{\ell} ),\ifa{m\ge 1\text{ and }\ell\ge 1}\\
   W \cdot \diag (\psi, \Iu{m}),\ifa{m\ge 1 \text{ and }\ell=0}\\
   W \cdot \diag (\psi, \Ouu{\ell}{\ell} ),\ifa{m=0 \text{ and }\ell\ge 1}\label{NM232}
 \end{cases}
\end{align}
 fulfill the following three conditions:
\begin{aeqiii}{0}
 \item\label{N11.7.i} \(\phi\), \(\psi\), \(\phi^\square\), and \(\psi^\square\) are holomorphic in \(\uhp \).
 \item\label{N11.7.ii} \(\tmatp{\phi^\square}{\psi^\square} \in \qSp \).
 \item\label{N11.7.iii} \(\tmatpc{ \tilde{\phi}}{\tilde{\psi}}=\tmatpc{\phi^\square}{\psi^\square}\).
\end{aeqiii}
 \item\label{N11.7.b2} For each pair \(\tmatp{\phi}{\psi} \in \rSp \), the functions \(\phi^\square\) and \(\psi^\square\) given by \eqref{NN232} and \eqref{NM232} fulfill~\ref{N11.7.ii}.
 \item\label{N11.7.b3} Let \(\tmatp{\phi}{\psi} \in \rSp \).
 Let \(\phi^\square\) and \(\psi^\square\) be defined by \eqref{NN232} and \eqref{NM232}.
 Then every pair \(\tmatp{\tilde{\phi}}{\tilde{\psi}}\in \qSp \) for which~\ref{N11.7.iii} holds true fulfills necessarily \eqref{Puv}.
\end{enuii}
\end{enui}
\end{lem}

 \rlem{N11.8} is substantially proved in~\cite[\clem{5.2}, \cpage{459/460}]{MR1362524}.
 (A detailed proof for the case that \(m\ge 1\) and \(\ell\ge 1\) is also given in~\cite[\clem{11.7}]{Mak14}.)

\begin{lem}\label{N11.8}
 Let \(\alpha\in\R\), let \(\kappa\in\Ninf \), let \(\seqska \in\Kggeqka \), and let \(n\in\NO \) be such that \(2n+1 \le \kappa\).
 Let \(m\), \(\ell\), and \(r\) be given by \eqref{175-v1}, \eqref{175-v2}, and \eqref{175-v3}.
 Suppose \(r\ge 1\).
 Let \(\cU_{n,\alpha}\) and \(\cV_{n,\alpha}\) be given by \eqref{170-v1} and \eqref{170-v2}.
 Then:
\begin{enui}
 \item\label{N11.8.a} There exists a unitary complex \tqqa{matrix} \(W\) such that
\begin{align}
 W^\ad \OPu{\cU_{n,\alpha}} W
 &=
 \begin{cases}
  \diag (\Ouu{r}{r}, \Iu{m}, \Ouu{\ell}{\ell} ),\ifa{ m\ge 1\text{ and }\ell\ge 1}\\
  \diag (\Ouu{r}{r}, \Iu{m}),\ifa{m \ge 1\text{ and }\ell=0}
 \end{cases} \label{N11.53}
\intertext{and}
 W^\ad \OPu{\cV_{n,\alpha}} W
 &=
 \begin{cases}
  \diag (\Ouu{r}{r}, \Ouu{m}{m}, \Iu{\ell}),\ifa{ m\ge 1\text{ and }\ell\ge 1}\\
  \diag (\Ouu{r}{r}, \Iu{\ell}),\ifa{ m=0\text{ and }\ell\ge 1}.
 \end{cases} \label{N11.54}
\end{align}
 \item\label{N11.8.b} Let \(W\) be a unitary complex \tqqa{matrix} such that \eqref{N11.53} and \eqref{N11.54} are valid.
\begin{enuii}
 \item\label{N11.8.b1} Let \(\tmatp{\tilde{\phi}}{\tilde{\psi}} \in \qSps{2n+1}\).
 Then there exists a pair \(\tmatp{\phi}{\psi} \in \rSp \) such that the conditions~\ref{N11.7.i}--\ref{N11.7.iii} of \rlem{N11.7} hold true with \(\phi^\square\) and \(\psi^\square\) given by \eqref{NN232} and \eqref{NM232}.
 \item\label{N11.8.b2} If \(\tmatp{\phi}{\psi} \in \rSp \), then \(\phi^\square\) and \(\psi^\square\) be given by \eqref{NN232} and \eqref{NM232} fulfill condition~\ref{N11.7.ii} of  \rlem{N11.7}.
 \item\label{N11.8.b3} Let \(\tmatp{\phi}{\psi}\in\rSp\) and let \(\phi^\square\) and \(\psi^\square\) be given by \eqref{NN232} and \eqref{NM232}.
 If \(\tmatp{\tilde{\phi}}{\tilde{\psi}}\in \qSp \) fulfills condition~\ref{N11.7.iii} of \rlem{N11.7}, then \(\tmatp{\tilde{\phi}}{\tilde{\psi}}\) belongs to \(\qSps{2n+1}\).
\end{enuii}
\end{enui}
\end{lem}
\begin{proof}
 Let us consider the case that \(\ell\ge 1\) and \(m\ge 1\) hold true.
 (If \(m=0\) or if \(\ell=0\), then the assertions can be checked analogously, so that we omit the details of the proof in these cases.)
 In view of \eqref{170-v1}, \eqref{170-v2}, \rlem{170}, \eqref{175-v1}, and \eqref{175-v2}, we obtain that \(\cU_{n, \alpha}\) and \(\cV_{n, \alpha}\) are orthogonal subspaces of \(\Cq\) with \(\dim \cU_{n, \alpha}=m \geq 1\) and \(\dim \cV_{n, \alpha}=\ell\geq 1\).
 Taking into account that \(r=q-(m+\ell)\leq q-2\), we see that we can apply \rlem{N11.7} with \(\cU=\cU_{n,\alpha}\) and \(\cV=\cV_{n,\alpha}\).

 \eqref{N11.8.a} Use \rlemp{N11.7}{N11.7.a}.

 \eqref{N11.8.b1} Let \(\tmatp{\tilde{\phi}}{\tilde{\psi}} \in \qSp \) be such that \eqref{175-v7} and \eqref{175-v8} hold true.
 Because of \eqref{175-v7}, \eqref{170-v1}, and \rlem{170}, we have \(\OPu{\cU_{n, \alpha}} \phi=\NM\).
 Using \eqref{175-v8}, \eqref{170-v2}, and \rlem{170}, we get \(\OPu{\cV_{n, \alpha}} \psi=\NM\).
 In view of \rlem{N11.7}, \rpart{N11.8.b1} is proved.

 \eqref{N11.8.b2} Apply \rlemp{N11.7}{N11.7.b2}.

 \eqref{N11.8.b3} Use \rlemp{N11.7}{N11.7.b3} and \rlem{170}.
\end{proof}

 Now we obtain a parametrization of the solution set of the matricial truncated Stieltjes moment problem in the so-called degenerate, but not completely degenerate case.

\begin{thm}\label{N11.9}
 Let \(\alpha\in\R\), let \(\kappa\in\Ninf \), let \(\seqska \in\Kggeqka \), and let \(n\in\NO \) be such that \(2n+1\le\kappa\).
 Let the integers \(m\), \(\ell\), and \(r\) be given by \eqref{175-v1}, \eqref{175-v2}, and \eqref{175-v3}.
 Suppose \(r\ge 1\).
 Let \(\cU_{n,\alpha}\) and \(\cV_{n,\alpha}\) be the subspaces of \(\Cq \) which are defined in \eqref{170-v1} and \eqref{170-v2}.
 Let \(W\) be a unitary complex \tqqa{matrix} such that \eqref{N11.53} and \eqref{N11.54} hold true.
 Then:
\begin{enui}
 \item\label{N11.9.a} Let \(\tmatp{\phi }{\psi} \in \rSp \) and let \(\phi^\square\) and \(\psi^\square\) be defined by \eqref{NN232} and \eqref{NM232}.
 Then the function \( \det \rk{\hat{\Theta}_{n,\alpha}^{(2,1)} \phi^\square + \hat{\Theta}_{n,\alpha}^{(2,2)} \psi^\square}\) does not vanish identically and
\[
 S
 \defeq \rk{\hat{\Theta}_{n,\alpha}^{(1,1)} \phi^\square + \hat{\Theta}_{n,\alpha}^{(1,2)} \psi^\square} \rk{\hat{\Theta}_{n,\alpha}^{(2,1)} \phi^\square + \hat{\Theta}_{n,\alpha}^{(2,2)} \psi^\square}^\inv 
\]
 belongs to the class \(\SFOqskg{2n+1}\).
 \item\label{N11.9.b} For each \(S\in\SFOqskg{2n+1}\), there exists a pair \(\tmatp{\phi}{\psi}\in \rSp \) such that the function \(\det\rk{\hat{\Theta}_{n,\alpha}^{(2,1)} \phi^\square + \hat{\Theta}_{n,\alpha}^{(2,2)} \psi^\square}\) does not vanish identically and that \(S\) admits the representation
\beql{Nr.11.55}
 S
 = \rk{\hat{\Theta}_{n,\alpha}^{(1,1)} \phi^\square + \hat{\Theta}_{n,\alpha}^{(1,2)} \psi^\square} \rk{\hat{\Theta}_{n,\alpha}^{(2,1)} \phi^\square + \hat{\Theta}_{n,\alpha}^{(2,2)} \psi^\square}^\inv 
\eeq
 where \(\phi^\square\) and \(\psi^\square\) are given by \eqref{NN232} and \eqref{NM232}.
 \item\label{N11.9.c} Let \(\tmatp{\phi_1}{\psi_1}, \tmatp{\phi_2}{\psi_2} \in \rSp \).
 For each \(k\in\set{1,2}\), let \(\phi_k^\square\) be defined as in \eqref{NN232} where \({\phi}\) is replaced by \(\phi_k\) and let \(\psi_k^\square\) be defined as in \eqref{NM232} where \(\psi\) is replaced by \(\psi_k\).
 Then the following statements are equivalent:
\begin{aeqii}{0}
 \item\label{N11.9.i} \(\tmatpc{\phi_1}{\psi_1}=\tmatpc{\phi_2}{\psi_2}\).
 \item\label{N11.9.ii} \(\rk{\hat{\Theta}_{n,\alpha}^{(1,1)} \phi_1^\square + \hat{\Theta}_{n,\alpha}^{(1,2)} \psi_1^\square} \rk{\hat{\Theta}_{n,\alpha}^{(2,1)} \phi_1^\square + \hat{\Theta}_{n,\alpha}^{(2,2)} \psi_1^\square}^\inv\\=\rk{\hat{\Theta}_{n,\alpha}^{(1,1)} \phi_2^\square + \hat{\Theta}_{n,\alpha}^{(1,2)} \psi_2^\square} \rk{\hat{\Theta}_{n,\alpha}^{(2,1)} \phi_2^\square + \hat{\Theta}_{n,\alpha}^{(2,2)} \psi_2^\square}^\inv\).
\end{aeqii}
\end{enui}
\end{thm}
\begin{proof}
 Let us consider the case that \(m\ge 1\) and \(\ell\ge 1\) hold true.
 (If \(m=0\) or if \(\ell=0\), then the assertions can be proved analogously.)

 \eqref{N11.9.a} Let \( \tmatp{\phi}{\psi} \in \aSp{r}\).
 \rPartss{N11.8.b3}{N11.8.b2} of \rlem{N11.8} and \rnota{bez-DT1} show that \(\tmatp{ \phi^\square}{ \psi^\square }\) belongs to \(\qSps{2n+1}\).
 Applying \rpropp{LB53L}{LB53L.a} completes the proof of \rpart{N11.9.a}. 

 \eqref{N11.9.b} Let \(S \in \SFOqskg{2n+1}\).
 According to \rpropp{LB53L}{LB53L.b}, then there is a pair \(\tmatp{\phi_\#}{\psi_\#} \in \qSps{2n+1}\), where \(\phi_\#\) and \(\psi_\#\) are matrix-valued functions which are holomorphic in \(\Cs \) and which fulfill 
\beql{TRB}
 \det \ek*{\hat \Theta_{n, \alpha}^{(2,1)} (z) \phi_\#(z) +\hat \Theta_{n, \alpha}^{(2,2)}(z) \psi_\#(z)}
 \neq 0
\eeq
 and 
\beql{QX}
 S(z)
 =\ek*{\hat\Theta_{n, \alpha}^{(1,1)}(z) \phi_\#(z)+\hat\Theta_{n, \alpha}^{(1,2)}(z) \psi_\#(z)}\ek*{\hat\Theta_{n, \alpha}^{(2,1)} (z)\phi_\#(z)+\hat\Theta_{n, \alpha}^{(2,2)}(z) \psi_\#(z)}^\inv 
\eeq
 for all \(z\in\Cs \).
 In view of \rnota{bez-DT1} and \rlemp{N11.8}{N11.8.b1}, there is a pair \( \tmatp{\phi}{\psi} \in\aSp{r}\) such that \(\tmatp{ \phi^\square}{\psi^\square} \in \qSp \) and \(\tmatpc{ \phi_\#}{\psi_\# }=\tmatpc{\phi^\square}{\psi^\square}\) hold true.
 Consequently, there are a discrete subset \(\cD\) of \(\Cs \) and a \tqqa{matrix-valued} function \(g\) which is meromorphic in \(\Cs \) such that \(\phi_\#\), \(\psi_\#\), \(\phi\), \(\psi\), and \(g\) are holomorphic in \(\C \setminus (\rhl \cup \cD)\) and that \(\det g(z) \neq 0\) and
\beql{VZ1}
 \matp{\phi_\# (z)}{\psi_\#(z)}
 =\matp{\phi^\square(z)g(z)}{\psi^\square(z)g(z)}
\eeq
 hold true for each \(z \in \C \setminus (\rhl \cup \cD)\).
 Therefore, for each \(z \in \C \setminus (\rhl \cup \cD)\), it follows from \eqref{TRB} that \( 0\neq \det \ek{\hat \Theta_{n, \alpha}^{(2,1)} (z) \psi^\square (z) +\hat \Theta_{n, \alpha}^{(2,2)} (z) \psi^\square (z) } \cdot \det g(z)\).
 In particular, the function \( \det\rk{\hat \Theta_{n, \alpha}^{(2,1)} \phi^\square +\hat \Theta_{n, \alpha}^{(2,2)} \psi^\square } \) does not vanish identically in \(\Cs \).
 Because of \eqref{QX} and \eqref{VZ1}, for all \(z \in \C \setminus (\rhl \cup \cD)\), we get furthermore
\[\begin{split}
 S(z)
 & = \ek*{ \hat \Theta_{n, \alpha}^{(1,1)} (z) \phi^\square (z)g(z) +\hat \Theta_{n, \alpha}^{(1,2)} (z) \psi^\square (z)g(z) }\ek*{ \hat \Theta_{n, \alpha}^{(2,1)} (z) \phi^\square (z)g(z) +\hat \Theta_{n, \alpha}^{(2,2)} (z) \psi^\square (z)g (z) }^\inv \\
 & = \ek*{ \hat \Theta_{n, \alpha}^{(1,1)} (z) \phi^\square (z) +\hat \Theta_{n, \alpha}^{(1,2)} (z) \psi^\square (z) }\ek*{ \hat \Theta_{n, \alpha}^{(2,1)} (z) \phi^\square (z) +\hat \Theta_{n, \alpha}^{(2,2)} (z) \psi^\square (z) }^\inv.
\end{split}\]
 In particular, \eqref{Nr.11.55} holds true.

 \eqref{N11.9.c} The matrices \(U\defeq \diag (\Ouu{m}{m}, \Iu{\ell} )\) and \(V\defeq \diag (\Iu{m}, \Ouu{\ell}{\ell}) \) fulfill \(\rank \tmatp{U}{V} = m +\ell = q - r\) and \( V^\ad U = \Ouu{(q-r)}{(q-r)}\).
 For each \(k\in\set{1,2}\), let \(\phi_k^\#= \diag (\phi_k, U) \) and \(\psi_k^\#= \diag (\psi_k, V) \).
 From \rlem{WN1} we see that \(\tmatp{\phi_1^\# }{\psi_1^\# } \) and \(\tmatp{\phi_2^\# }{\psi_2^\# } \) belong to \( \qSp \) and that~\ref{N11.9.i} is equivalent to \(\tmatpc{\phi_1^\# }{\psi_1^\# } =\tmatpc{\phi_2^\# }{\psi_2^\# }\).
 Obviously, \(\phi_k^\square\defeq W \phi_k^\# \) and \(\psi_k^\square\defeq W \psi_k^\# \) for each \(k\in\set{1,2}\).
 Taking into account \(W^\ad W =\Iq\) and \rrem{WN2}, we get that \(\tmatp{\phi_1^\square}{\psi_1^\square} \)  and \(\tmatp{\phi_2^\square}{\psi_2^\square} \) belong to \( \qSp \) and that \(\tmatpc{\phi_1^\# }{\psi_1^\# } =\tmatpc{\phi_2^\# }{\psi_2^\# }\) is equivalent to \(\tmatpc{\phi_1^\square}{\psi_1^\square}=\tmatpc{\phi_2^\square}{\psi_2^\square}\).
 Because of \rlemp{N11.7}{N11.7.b3}, we obtain \(\tmatp{\phi_k^\square}{\psi_k^\square}\in\qSps{2n+1}\) for each \(k \in \set{1,2}\).
 Using \rpropp{LB53L}{LB53L.c}, we see that \(\tmatpc{\phi_1^\square}{\psi_1^\square}=\tmatpc{\phi_2^\square}{\psi_2^\square}\) and~\ref{N11.9.ii} are equivalent.
 Consequently,~\ref{N11.9.i} holds true if and only~\ref{N11.9.ii} is fulfilled.
\end{proof}

\subsection{The completely degenerate case}\label{sub1431}
 Now we consider the so-called completely degenerate case~\ref{S1313.III}.
 We will see that, in this situation, the problem in question has a unique solution.

\begin{lem} \label{r-0-B55}
 Let \(m,\ell \in \N\) be such that \(m+\ell= q\).
 Let \(\cU\) and \(\cV\) be orthogonal subspaces of \(\Cq\) with \(\dim \cU=m\) and \(\dim \cV=\ell\).
 Then:
\begin{enui}
 \item\label{r-0-B55.a} There exists a unitary complex \tqqa{matrix} \(W\) such that
\begin{align} \label{r-0-BL55-2}
 W^\ad \OPu{\cU} W&=\diag ( \Iu{m}, \Ouu{\ell}{\ell} )&
&\text{and}&
 W^\ad \OPu{\cV} W&=\diag ( \Ouu{m}{m}, \Iu{\ell} ).
\end{align}
 \item\label{r-0-B55.b} Let \(W\) be a unitary complex \tqqa{matrix} such that \eqref{r-0-BL55-2} is fulfilled.
 Let \(\phi_\#\) and \(\psi_\#\) be the constant matrix-valued functions defined on \(\Cs \) given by 
\begin{align} \label{r-0-DC}
 \phi_\#(z)&\defeq W \cdot \diag (\Ouu{m}{m},\Iu{\ell})&
&\text{and}&
 \psi_\#(z)&\defeq W \cdot \diag (\Iu{m},\Ouu{\ell}{\ell})
\end{align}
 for all \(z\in\Cs\).
 Then:
\begin{enuii}
 \item\label{r-0-B55.b1} The pair \( \tmatp{\phi_\#}{\psi_\#}\) belongs to \(\qSp \).
 Furthermore, \(\OPu{\cU}\phi_\#=\NM\) and \(\OPu{\cV}\psi_\#=\NM\).
 \item\label{r-0-B55.b2} If \( \tmatp{\phi}{\psi} \in \qSp \) fulfills
\beql{N11-112-1}
 \matpc{\phi}{\psi}
 =\matpc{\phi_\#}{\psi_\#},
\eeq
 then 
\begin{align} \label{r-0-BL55-1}
 \OPu{\cU} \phi&=\NM&
&\text{and}&
 \OPu{\cV} \psi&=\NM.
\end{align}
 \item\label{r-0-B55.b3} If \(\tmatp{\phi}{\psi}\) belongs to \(\qSp \) and fulfills \eqref{r-0-BL55-1}, then \eqref{N11-112-1} is valid.
\end{enuii} 
\end{enui} 
\end{lem}
\begin{proof}
 \eqref{r-0-B55.a} Let \(\set{u_1,u_1,\dotsc,u_m}\) be an orthonormal basis of \(\cU\) and let \(\set{v_1,v_2,\dotsc,v_{\ell}}\) be an orthonormal basis of \(\cV\).
 Let \(U\defeq\mat{u_1, u_2,\dotsc,u_m}\), let \(V\defeq\mat{v_1, v_2,\dotsc,v_{\ell}}\), and let \(W \defeq\mat{U,V}\).
 Because of \(m+\ell=q\) and since \(\cU\) and \(\cV\) are orthogonal subspaces, the matrix \(W\) is unitary.
 Obviously, we have \(\OPu{\cU} U=U\), \(\OPu{\cU} V =\NM\), \(U^\ad U = \Iu{m} \), and \(V^\ad U = \NM\).
 Consequently, \( W^\ad \OPu{\cU} W = \diag (\Iu{m}, \Ouu{\ell}{\ell} )\).
 Analogously, \(\OPu{\cV} U =\NM\), \(\OPu{\cV} V = V\), \(U^\ad V =\NM\), and \(V^\ad V = \Iu{\ell}\) imply the second equation in \eqref{r-0-BL55-2}.
 
 \eqref{r-0-B55.b1} Clearly, the constant matrix-valued functions \(\phi_\#\) and \(\psi_\#\) are holomorphic in \(\Cs \).
 Since the matrix \(W\) is non-singular, we have 
\[
 \rank \matp{ \phi_\#(z)}{ \psi_\#(z) }
 = \rank \matp{\diag ( \Ouu{m}{m}, \Iu{\ell} ) } {\diag ( \Iu{m}, \Ouu{\ell}{\ell} )}
 =m+\ell
 =q
\]
 for each \(z \in \Cs \).
 For every choice of \(k \in \set{0,1}\) and \(z\in \C \setminus\R \), from \rrem{J-1} and \(W^\ad W = \Iq \), we conclude 
\[\begin{split} 
 &\matp{(z-\alpha)^k \phi_\#(z)}{ \psi_\#(z)}^\ad\rk*{\frac{-\Jimq}{2 \Im z}}\matp{(z-\alpha)^k \phi_\#(z)}{ \psi_\#(z)}\\
 &= \frac{-\iu}{2 \Im z}\rk*{\psi_\#^\ad (z)\ek*{(z- \alpha)^k \phi_\#(z)}-\ek*{(\ko z -\alpha)^k\phi_\#^\ad (z)}\psi_\#(z)}\\
 &=\frac{-\iu}{2 \Im z}\bigl\{ ( z-\alpha)^k\cdot \diag (\Iu{m}, \Ouu{\ell}{\ell}) \cdot W^\ad W \cdot \diag (\Ouu{m}{m}, \Iu{\ell})\\
 &\qquad- (\ko z-\alpha)^k\cdot \diag ( \Ouu{m}{m}, \Iu{\ell} ) \cdot W^\ad W \cdot \diag ( \Iu{m}, \Ouu{\ell}{\ell} ) \bigr\}\\
 &=\Oqq.
\end{split}\]
 In view of \rdefn{def-sp}, then \(\tmatp{ \phi_\#}{ \psi_\#}\) belongs to \(\qSp \).
 Further, from \eqref{r-0-BL55-2} we obtain \( \OPu{\cU} \phi_\# = \Iq \OPu{\cU} W \cdot \diag (\Ouu{m}{m},\Iu{\ell}) = WW^\ad \OPu{\cU} W \cdot \diag (\Ouu{m}{m},\Iu{\ell}) = W \cdot \diag (\Iu{m}, \Ouu{\ell}{\ell} ) \cdot \diag (\Ouu{m}{m}, \Iu{\ell}) =\Oqq\) and, analogously, \(\OPu{\cV} \psi_\# =\Oqq\).

 \eqref{r-0-B55.b2} Let \( \tmatp {\phi}{\psi} \in \qSp \) be such that \eqref{N11-112-1} holds true.
 According to \rrem{def-aq}, there are a discrete subset \(\cD\) of \(\Cs \) and a matrix-valued function \(g\) meromorphic in \(\Cs \) such that \(\phi\), \(\psi\), and \(g\) are holomorphic in \(\C \setminus (\rhl \cup \cD)\) and that \(\det g(z) \neq 0\) as well as \(\phi(z)=W \cdot \diag ( \Ouu{m}{m}, \Iu{\ell} ) \cdot g(z)\) and \(\psi(z)=W \cdot \diag ( \Iu{m}, \Ouu{\ell}{\ell} ) \cdot g(z)\) hold true for each \(z \in \C \setminus (\rhl \cup \cD)\).
 Taking into account \eqref{r-0-BL55-2} and \(WW^\ad =\Iq\), for each \(z \in \C \setminus (\rhl \cup \cD)\), we get 
\begin{multline*} 
 \OPu{\cU} \phi(z) 
 = WW^\ad \OPu{\cU} W \cdot\diag\rk{\Ouu{m}{m}, \Iu{\ell}}\cdot g(z) \\
 =W \cdot\diag\rk{\Iu{m}, \Ouu{\ell}{\ell}}\cdot \diag\rk{\Ouu{m}{m}, \Iu{\ell}}\cdot g(z) 
 =\Oqq
\end{multline*}
 and, analogously \(\OPu{\cV} \psi(z) =\Oqq\).
 This implies \eqref{r-0-BL55-1}.
 
 \eqref{r-0-B55.b3} Let \(\tmatp{\phi}{\psi}\in\qSp \) be such that \eqref{r-0-BL55-1} holds true.
 According to \rlem{T01}, we see that the function \(\det (\psi-\iu\phi)\) does not vanish identically.
 Let \(F\defeq (\psi+\iu\phi)(\psi-\iu\phi)^\inv \).
 \rlem{T01} shows that there is a discrete subset \(\cD\) of \(\Cs \) such that the following three conditions are fulfilled: 
\begin{aeqi}{0}
 \item\label{r-0-B55.i} \(F\) is holomorphic in \(\uhp \cup [\C \setminus (\rhl \cup \cD)]\).
 \item\label{r-0-B55.ii} The matrix-valued functions \(\phi\), \(\psi\) and \((\psi -\iu\phi)^\inv \) are holomorphic in \(\uhp\cup\ek{\C \setminus (\rhl \cup \cD)}\).
 \item\label{r-0-B55.iii} For each \(z \in \C \setminus (\rhl \cup \cD)\), the inequality \(\det [\psi (z) -\iu\phi (z)] \ne 0\) and the equations in \eqref{T-N257-1} and \eqref{T-B55-252} hold true.
\end{aeqi} 
 Obviously, because of~\ref{r-0-B55.i}, the  functions \( \tilde{\phi}\defeq \frac{\iu}{2} (\Iq-F) W\) and \( \tilde{\psi}\defeq \frac{1}{2} (\Iq+F) W\) are meromorphic in \(\Cs \) and holomorphic in \(\uhp \cup [\C \setminus (\rhl \cup \cD)]\).
 In view of~\ref{r-0-B55.ii}, the functions \(\phi\), \(\psi\), \(\tilde{\phi}\), \(\tilde{\psi}\) and \((\psi -\iu\phi)^\inv W\) are holomorphic in \(\uhp\cup\ek{\C \setminus (\rhl \cup \cD)}\).
 From~\ref{r-0-B55.iii} we see that 
\begin{align} \label{D3}
 \tilde{\phi}(z)&= \phi(z) \ek*{\psi(z)-\iu\phi(z)}^\inv W&
&\text{and}&
 \tilde{\psi} (z)&= \psi(z) \ek*{\psi(z)-\iu\phi(z)}^\inv W
\end{align}
 hold true for each \(z \in \C \setminus (\rhl \cup \cD)\).
 In view of~\ref{r-0-B55.ii}, the matrix-valued functions \((\psi+\iu\phi)\) and \((\psi-\iu\phi)^\inv W\) are meromorphic in \(\Cs \).
 Since the matrix \(W\) is unitary, for each \(z\in \C \setminus (\rhl \cup \cD)\), we have \(\det ([\psi(z)-\iu\phi(z)]^\inv W) \neq 0\) by~\ref{r-0-B55.iii}.
 Consequently, \eqref{D3} and \rrem{SP.1} imply that \(\tmatp{\tilde{\phi}}{\tilde{\psi}}\) belongs to \(\qSp \).
 Furthermore, from \rrem{def-aq} we get 
\beql{N198N}
 \matpc{\tilde{\phi}}{\tilde{\psi}}
 =\matpc{\phi}{\psi}.
\eeq
 By virtue of \(W^\ad W = \Iq \), \eqref{D3}, \eqref{r-0-BL55-2}, and \eqref{r-0-BL55-1}, we conclude
\beql{r-0-MV1}\begin{split}
 (\Iu{m},\Ouu{m}{\ell}) (\Iq-W^\ad F W)
 &= (\Iu{m},\Ouu{m}{\ell}) W^\ad (\Iq- F )W
 = - 2\iu (\Iu{m},\Ouu{m}{\ell}) W^\ad \tilde{\phi}\\
 &= - 2\iu (\Iu{m},\Ouu{m}{\ell}) W^\ad \phi(\psi-\iu\phi)^\inv W\\
 &= - 2\iu (\Iu{m},\Ouu{m}{\ell}) \cdot \diag ( \Iu{m}, \Ouu{\ell}{\ell} ) \cdot W^\ad \phi(\psi-\iu\phi)^\inv W \\
 &= - 2\iu (\Iu{m},\Ouu{m}{\ell}) W^\ad \OPu{\cU} \phi(\psi-\iu\phi)^\inv W\\
 &= - 2\iu (\Iu{m},\Ouu{m}{\ell}) W^\ad \Oqq (\psi-\iu\phi)^\inv W
 =\Ouu{m}{q} 
\end{split}\eeq
 and, analogously, 
\beql{r-0-MV2} 
 (\Ouu{\ell}{m}, \Iu{\ell}) (\Iq+W^\ad F W)
 =  \Ouu{\ell}{q}. 
\eeq
 Because of~\ref{r-0-B55.i}, we see that \(G\defeq W^\ad F W\) is a matrix-valued function which is meromorphic in \(\Cs \) and holomorphic in \(\uhp \cup [\C \setminus (\rhl \cup \cD)]\).
 From \eqref{r-0-MV1} and \eqref{r-0-MV2} we obtain \(G(w)=\diag ( \Iu{m}, -\Iu{\ell} )\) for each \(w\in \uhp\).
 Hence, \(G=\diag(\Iu{m},-\Iu{\ell})\) by the identity theorem for holomorphic functions.
 Thus, since the matrix \(W\) is unitary, this implies \( F= W \cdot \diag ( \Iu{m}, - \Iu{\ell} ) \cdot W^\ad \).
 Then
\beql{r-0-D33} \begin{split}
 \tilde{\phi} 
 =\frac{\iu}{2} (\Iq-F) W
 &=\frac{\iu}{2} \ek*{\Iq-W \cdot \diag ( \Iu{m}, - \Iu{\ell} ) \cdot W^\ad } W \\
 &=\frac{\iu}{2} W \ek*{ \diag( \Iu{m}, \Iu{\ell} )-\diag ( \Iu{m}, - \Iu{\ell} )}
 =W \cdot \diag ( \Ouu{m}{m},\iu\Iu{\ell} ) 
\end{split}\eeq
 and, analogously,
\beql{r-0-D34} 
 \tilde{\psi}
 = W \cdot \diag ( \Iu{m}, \Ouu{\ell}{\ell} ). 
\eeq
 Since \(\tilde{\phi}\) and \(\tilde{\psi}\) are holomorphic in \(\uhp\cup\ek{\C\setminus(\rhl \cup \cD)}\), the matrix-valued functions \( \phi_\square \defeq \tilde{\phi} \cdot \diag ( \Iu{m}, -\iu\Iu{\ell} )\) and \( \psi_\square \defeq \tilde{\psi} \cdot \diag ( \Iu{m}, -\iu\Iu{\ell} ) \) are holomorphic in \(\uhp\cup\ek{\C\setminus(\rhl \cup \cD)}\).
 From \(\det (\Iu{m}, -\iu\Iu{\ell}) \neq 0\), \rrem{SP.1}, \rrem{def-aq}, and \eqref{N198N} we get
\begin{align} \label{r-0-B55-N256}
 \matp{\phi_\square}{\psi_\square}&\in \qSp&
&\text{and}&
 \matpc{\phi_\square}{\psi_\square}&=\matpc{\tilde{\phi}}{\tilde{\psi}}=\matpc{\phi}{\psi}.
\end{align}
 Because of \eqref{r-0-D33} and \eqref{r-0-DC}, we have 
\[
 \phi_\square
 =\tilde{\phi} \cdot \diag ( \Iu{m}, -\iu\Iu{\ell})
 = W \cdot \diag ( \Ouu{m}{m}, \iu \Iu{\ell}) \cdot \diag ( \Iu{m}, -\iu\Iu{\ell})
 =\phi_\#.
\]
 Analogously, \eqref{r-0-D34} and \eqref{r-0-DC} imply \(\psi_\square =\psi_\#\).
 Thus, \eqref{N11-112-1} follows from \eqref{r-0-B55-N256}. 
\end{proof}

\begin{lem} \label{r-0-175} %
 Let \(\alpha \in \R\), let \(\kappa \in \Ninf \), and let \(\seqska  \in \Kggeqka \).
 Let \(n\in\NO \) be such that \(2n+1 \leq \kappa\).
 Suppose that the integers \(m\) and \(\ell\) given by \eqref{175-v1} and \eqref{175-v2} fulfill \( m+\ell=q\), \(m \geq 1\), and \(\ell \geq 1\).
 Let \(\cU_{n, \alpha}\) and \( \cV_{n, \alpha}\) be given by \eqref{170-v1} and \eqref{170-v2}.
 Then:
\begin{enui}
 \item\label{r-0-175.a} There exists a unitary complex \tqqa{matrix} \(W\) such that 
\begin{align} \label{r-0-c175-1}
 W^\ad \OPu{\cU_{n, \alpha}} W&=\diag ( \Iu{m}, \Ouu{\ell}{\ell} )&
&\text{and}&
 W^\ad \OPu{\cV_{n, \alpha}} W&=\diag ( \Ouu{m}{m}, \Iu{\ell} ).
\end{align}
 \item\label{r-0-175.ba} Let \(W\) be a unitary complex \tqqa{matrix} such that \eqref{r-0-c175-1} holds true.
 Furthermore, let \(\phi_\square\) and \(\psi_\square\) be the constant matrix-valued functions defined on \(\Cs \) given by \(\phi_\square(z)\defeq W \cdot \diag (\Ouu{m}{m},\Iu{\ell}) \) and \({\psi_\square}(z)\defeq W \cdot \diag (\Iu{m},\Ouu{\ell}{\ell})\) for all \(z\in\Cs\).
 Then:
\begin{enuii}
 \item\label{r-0-175.b1} \( \tmatp{\phi_\square}{\psi_\square} \in \qSp \).
 \item\label{r-0-175.b2} Each pair \(\tmatp{{\phi}}{{\psi}} \in \qSp \) with
\beql{R83} 
 \matpc{\phi}{\psi }
 =\matpc{ \phi_\square}{\psi_\square }
\eeq
 belongs to \( \qSps{2n+1}\).
 \item\label{r-0-175.b3} Each \(\tmatp{\phi}{\psi} \in \qSps{2n+1}\) fulfills \eqref{R83}.
\end{enuii}
\end{enui}
\end{lem}
\begin{proof}
 From \eqref{170-v1}, \eqref{170-v2}, \rlem{170}, \eqref{175-v1} and \eqref{175-v2} we see that \(\cU_{n, \alpha}\) and \(\cV_{n, \alpha}\)  are orthogonal subspaces of \(\Cq\) with \(\dim \cU_{n, \alpha}=m \) and \(\dim \cV_{n, \alpha}=\ell\).
 Since \(m\) and \(\ell\) are positive integers with \(m+\ell=q\), we can apply \rlem{r-0-B55} with \(\cU=\cU_{n,\alpha}\) and \(\cV=\cV_{n,\alpha}\).

 \eqref{r-0-175.a} Use \rlemp{r-0-B55}{r-0-B55.a}.
 
 \eqref{r-0-175.b1} Apply \rlemp{r-0-B55}{r-0-B55.b1}.

 \eqref{r-0-175.b2} Suppose that \(\tmatp{\phi}{\psi} \in \qSp \) is such that \eqref{R83} holds true.
 \rlemp{r-0-B55}{r-0-B55.b2} shows then that \(\OPu{\cU_{n, \alpha}} \phi=\NM\) and \(\OPu{\cV_{n, \alpha}} \psi=\NM\).
 Thus, \rlem{170} implies \eqref{175-v7} and \eqref{175-v8}.
 In view of \rnota{bez-DT1}, \rpart{r-0-175.b2} is proved.

 \eqref{r-0-175.b3} Let \(\tmatp{\phi}{\psi} \in \qSps{2n+1}\).
 Then \eqref{175-v7} and \eqref{175-v8} hold true.
 Because of \eqref{170-v1}, \eqref{175-v7} and \rlem{170}, we have \(\OPu{\cU_{n, \alpha}} \phi=\NM\), whereas \eqref{170-v2}, \eqref{175-v8}, and \rlem{170} yield \(\OPu{\cV_{n, \alpha}} \psi=\NM\).
 Since \(W\) is a unitary matrix which fulfills \eqref{r-0-c175-1}, \rlemp{r-0-B55}{r-0-B55.b3} implies \eqref{R83}.
\end{proof}

\begin{rem} \label{rm-0-B55}
 Let \(W\) be a non-singular complex \tqqa{matrix} and let \(\cW\) be the constant function with value \(W\) defined on \(\Cs \).
 Then it is readily checked that the following statements hold true:
\begin{enui}
 \item\label{rm-0-B55.a} The pairs \(\tmatp{\Oqq }{\cW}\) and \(\tmatp{\cW}{\Oqq }\) belong to \(\qSp \).
 \item\label{rm-0-B55.b} Each pair \(\tmatp{\phi}{\psi} \in \qSp \) with 
\beql{N201M}
 \matpc{\phi}{\psi}
 =\matpc{\Oqq}{\cW }
\eeq
 fulfills \(\phi =\Oqq\).
 Conversely, if \(\tmatp{\phi}{\psi} \in \qSp \) is such that \(\phi =\Oqq \) holds true, then \(\det \psi\) does not vanish identically and \eqref{N201M} is valid.
 \item\label{rm-0-B55.c} Each pair \(\tmatp{\phi}{\psi} \in \qSp \) with
\beql{N201K}
 \matpc{\phi}{\psi }
 =\matpc{ \cW}{\Oqq }
\eeq  
 fulfills \(\psi =\Oqq \).
 Conversely, if \(\tmatp{\phi}{\psi} \in \qSp \) is such that \(\psi =\Oqq \) holds true, then \(\det \phi\) does not vanish identically and \eqref{N201K} is valid.
\end{enui} 
\end{rem}

\begin{lem} \label{N11.20}
 Let \(\alpha \in \R\), let \(\kappa \in \Ninf \), and let \(\seqska \in \Kggeqka \).
 Let \(n\in\NO \) be such that \(2n+1 \leq \kappa\).
 Suppose that \(\ell\) given by \eqref{175-v2} fulfills \(\ell=q\).
 Then \(m\) given by \eqref{175-v1} fulfills \(m=0\) and:
\begin{enui}
 \item\label{N11.20.a} \(\cV_{n,\alpha}\) defined by \eqref{170-v2} fulfills \(\cV_{n, \alpha} =\Cq \) and, in particular, \(\OPu{\cV_{n,\alpha}} = \Iq \).
 \item\label{N11.20.b} Let \(W\) be a non-singular complex \tqqa{matrix} and let \(\cW\) be the constant function with value \(W\) defined on \(\Cs \).
 Then \( \tmatp{ \cW}{\Oqq }\) belongs to \(\qSp \) and each pair \(\tmatp{\phi}{\psi} \in \qSp \) with \eqref{N201K} belongs to the class \(\qSps{2n+1}\).
 \item\label{N11.20.c} If \(\tmatp{\phi}{\psi} \in \qSp \) fulfills \eqref{175-v8}, then \eqref{N201K} holds true.
\end{enui}
\end{lem}
\begin{proof}
 Because of \(0\leq m\leq m+\ell\leq q=\ell\), we have \(m=0\).
 
 \eqref{N11.20.a} From \eqref{170-v2}, \eqref{175-v2}, \(\ell=q\), and \rlem{170} we see that~\eqref{N11.20.a} is valid.

 \eqref{N11.20.b} \rrem{rm-0-B55} shows that \(\tmatp{\cW }{ \Oqq }\) belongs to \( \qSp \).
 Let \(\tmatp{\phi}{\psi} \in \qSp \) fulfill \eqref{N201K}.
 Then \rpart{N11.20.a} and \rremp{rm-0-B55}{rm-0-B55.c} yield \(\OPu{\cV_{n,\alpha}} \psi =\NM\) and, in view of \rlem{170}, consequently \eqref{175-v8}.
 Since \(m=0\) holds, we get from \eqref{175-v1} that \eqref{175-v7} is true.
 In view of \rnota{bez-DT1}, \rpart{N11.20.b} is proved.

 \eqref{N11.20.c} Let \(\tmatp{\phi}{\psi} \in \qSp \) be such that \eqref{175-v8} holds true.
 Because of \eqref{175-v8} and \rlem{170}, we have \(\OPu{\cV_{n,\alpha}} \psi =\NM\).
 Thus, \rpart{N11.20.a} implies \(\psi =\NM\).
 From \rremp{rm-0-B55}{rm-0-B55.c} it follows \eqref{N201K}.
\end{proof}

\begin{lem} \label{N11.22}
 Let \(\alpha \in \R\), let \(\kappa \in \Ninf \), and let \(\seqska\in \Kggeqka \).
 Let \(n\in\NO \) be such that \(2n+1 \leq \kappa\).
 Suppose that \(m\) given by \eqref{175-v1} fulfills \(m=q\).
 Let \(\cU_{n,\alpha}\) be defined by \eqref{170-v1}.
 Then \(\ell\) given by \eqref{175-v2} fulfills \(\ell=0\) and the following statements hold true:
\begin{enui}
 \item\label{N11.22.a} \(\cU_{n, \alpha} =\Cq \) and, in particular, \(\OPu{\cU_{n,\alpha}} = \Iq \).
 \item\label{N11.22.b} Let \(W\) be a non-singular complex \tqqa{matrix} and let \(\cW\) be the constant function with value \(W\) defined on \(\Cs \).
 Then \( \tmatp{ \Oqq }{\cW} \in \qSp \) and each pair \(\tmatp{\phi}{\psi} \in \qSp \) with \eqref{N201M} belongs to the class \(\qSps{2n+1}\).
 \item\label{N11.22.c} If \(\tmatp{\phi}{\psi} \in \qSp \) fulfills \eqref{175-v7}, then \eqref{N201M} holds true.
\end{enui}
\end{lem}
\begin{proof} 
 Using \rlem{170} and \rrem{rm-0-B55}, \rlem{N11.22} can be proved analogous to \rlem{N11.20}.
 We omit the details
\end{proof}

\begin{thm} \label{B53-r-0} 
 Let \(\alpha \in \R\), let \(\kappa \in \Ninf \), let \(\seqska  \in \Kggeqka \), and let \(n \in \NO \) be such that \(2n+1 \leq \kappa\).
 Suppose that the integers \(m\) and \(\ell\) given by \eqref{175-v1} and \eqref{175-v2} fulfill \(m+\ell=q\).
 If \(m \geq 1\) and \(\ell \geq 1\), then let \(W\) be a unitary complex \tqqa{matrix} such that the equations in \eqref{r-0-c175-1} hold true where \(\cU_{n,\alpha}\) and \(\cV_{n,\alpha}\) are given by \eqref{170-v1} and \eqref{170-v2}.
\begin{enui}
 \item\label{B53-r-0.a} If \(\phi\) and \(\psi\) are the matrix-valued functions defined on \(\Cs \) by
\begin{align} 
 \phi(z)
 &\defeq
 \begin{cases}
  W \cdot \diag (\Ouu{m}{m}, \Iu{\ell}),\ifa{m \geq 1\text{ and }\ell \geq 1}\\
  \Iq,\ifa{m =0}\\
  \Oqq,\ifa{\ell=0} 
 \end{cases}\label{TL1}
\intertext{and}
 \psi(z)   &\defeq
 \begin{cases}
  W \cdot \diag (\Iu{m}, \Ouu{\ell}{\ell}),\ifa{ m \geq 1\text{ and }\ell \geq 1} \\
  \Oqq,\ifa{m =0}\\
  \Iu{q},\ifa{\ell=0}
 \end{cases},\label{TL2}
\end{align}
 then the function \(\det \rk{\hat \Theta_{n, \alpha}^{(2,1)} \phi +\hat \Theta_{n, \alpha}^{(2,2)} \psi }\) does not vanish identically.
 \item\label{B53-r-0.b} The set \(\SFOqskg{2n+1}\) consists of exactly one element, namely the matrix-valued function 
\beql{FDG}
 S
 \defeq
 \begin{cases}
  \rk{\hat \Theta_{n, \alpha}^{(1,1)}\mathbb{V}+\hat \Theta_{n, \alpha}^{(1,2)}\mathbb{U}}\rk{\hat \Theta_{n, \alpha}^{(2,1)}\mathbb{V}+\hat \Theta_{n, \alpha}^{(2,2)}\mathbb{U}}^\inv,\ifa{m \geq 1 \text{ and }\ell \geq 1}\\
  \hat \Theta_{n, \alpha}^{(1,1)} \rk{\hat \Theta_{n, \alpha}^{(2,1)} }^\inv,\ifa{m =0}\\
  \hat \Theta_{n, \alpha}^{(1,2)} \rk{\hat \Theta_{n, \alpha}^{(2,2)} } ^\inv,\ifa{\ell=0}
 \end{cases},
\eeq
 where, in the case \(m\geq1\) and \(\ell\geq1\), the matrices \(\mathbb{U}\defeq\mat{U,\Ouu{q}{\ell}}\) and \(\mathbb{V}\defeq\mat{\Ouu{q}{m},V}\) are built with the \taaa{q}{m}{block} \(U\) and the \taaa{q}{\ell}{block} \(V\) from the  block partition \(W=\mat{U,V}\) of \(W\).
\end{enui}
\end{thm}
\begin{proof}
 Let \(\phi\) and \(\psi\) be given by \eqref{TL1} and \eqref{TL2}.
 Then \rlemsss{r-0-175}{N11.20}{N11.22} yield \( \tmatp{\phi}{\psi} \in \qSps{2n+1}\).
 Thus, \rpropp{LB53L}{LB53L.a} shows that~\eqref{B53-r-0.a} is valid and that \(\hat{S}_{n, \alpha}\) given by \eqref{NT} belongs to \(\SFOqskg{2n+1}\).
 From \eqref{NT}, \eqref{TL1}, \eqref{TL2}, and \eqref{FDG} we get \(S= \hat{S}_{n, \alpha}\) and, consequently, \(\set{S} \subseteq \SFOqskg{2n+1}\).
 Now we consider an arbitrary \(S_\#\) belonging to \(\SFOqskg{2n+1}\).
 By virtue of \rpropp{LB53L}{LB53L.b}, there exists a pair \( \tmatp{\phi_\#}{\psi_\#} \in \qSps{2n+1}\) of in \(\Cs \) holomorphic \tqqa{matrix-valued} functions \(\phi_\#\) and \(\psi_\#\) which fulfill \eqref{TRB} and 
\beql{LS2}
 S_\#(z)
 =\ek*{\hat\Theta_{n, \alpha}^{(1,1)}(z) \phi_\#(z)+\hat\Theta_{n, \alpha}^{(1,2)}(z) \psi_\#(z)} \ek*{\hat\Theta_{n, \alpha}^{(2,1)} (z)\phi_\#(z)+\hat\Theta_{n, \alpha}^{(2,2)}(z) \psi_\#(z)}^\inv 
\eeq
 for each \(z\in \Cs \).
 If \(m\ge 1\) and \(\ell\ge 1\) hold true, then \eqref{N11-112-1} follows from \rlemp{r-0-175}{r-0-175.b3}.
 If \(m=0\), then \(\ell=q-m=q\) and \eqref{N11-112-1} is a consequence of \rlemp{N11.20}{N11.20.c}.
 If \(\ell=0\), then \(m=q-\ell=q\), and \eqref{N11-112-1} follows from \rlemp{N11.22}{N11.22.c}.
 Thus, \eqref{N11-112-1} is fulfilled in each case.
 Because of \( \set{ \tmatp{\phi}{\psi},\tmatp{\phi_\#}{\psi_\#}} \subseteq \qSps{2n+1} \), \eqref{N11-112-1}, \rpropp{LB53L}{LB53L.c}, \eqref{TRB},~\eqref{B53-r-0.a}, \eqref{LS2}, \eqref{NT}, and \(S = \hat{S}_{n,\alpha}\), then \(S_\#=\hat{S}_{n, \alpha}=S\) follows.
 Consequently, \(\SFOqskg{2n+1} \subseteq \set{S}\).
 Thus, \(\SFOqskg{2n+1} = \set{S}\).
\end{proof}

\begin{rem} \label{B12L}
 Under the assumptions of \rthm{B53-r-0}, we see from \rthm{B53-r-0},~\cite[\cthmss{6.5}{6.4}]{MR2735313},~\cite[\cdefn{4.10}]{MR2570113}, and~\cite[\cthm{5.1}]{MR3644521} that \(S\) given by \eqref{FDG} is exactly the \taSt{} of the restriction onto \(\BorK \) of the completely degenerate \tnnH{} measure corresponding to \(\seqs{2n+1}\).
\end{rem}

 Observe that using~\rprop{lemM4213-2},~\cite[\cthm{5.2}]{MR2735313}, \rexam{SP.2}, \rthmsss{TRB}{T1}{N11.9}, and~\cite[\cthm{5.1}]{MR3380267}, one can obtain a self-contained proof of \rthm{T1122} in the case of a positive odd integer \(m\).
 We omit the details.

\appendix
\section{Particular results of matrix theory}\label{A1239}

\begin{rem}\label{A.1.}
 If \(A\in\Cpq \), then \(\nul{A} =\ran{A^\ad}^\orth\) and \(\ran{A} =\nul{A^\ad}^\orth\).
\end{rem}

\begin{lem}[{\cite{MR0203464}}]\label{Gl,3.1.10} 
 Let \(A\in\Cpq\), let \(B\in\Coo{p}{r}\), and let \(B^{(1)} \in B\set{1}\).
 Then the following statements are equivalent:
\begin{aeqi}{0}
 \item\label{Gl,3.1.10.i} \(\ran{A}\subseteq\ran{B}\).
 \item\label{Gl,3.1.10.ii} There is a matrix \(X\in\Coo{r}{q}\) such that \(A=BX\).
 \item\label{Gl,3.1.10.iii} \(BB^{(1)}A=A\).
 \item\label{Gl,3.1.10.iv} There is a positive real number \(\beta\) such that \(AA^\ad \lleq\beta BB^\ad \).
 \item\label{Gl,3.1.10.v} \(\nul{B^\ad}\subseteq\nul{A^\ad}\).
\end{aeqi}
\end{lem}

 It seems to be useful to state the following dual reformulation of \rlem{Gl,3.1.10}:

\begin{lem} \label{GL3.1.11}
 Let \(A\in\Cpq \), let \(C\in\Coo{r}{q}\), and let \(C^{(1)} \in C\set{1}\).
 Then the following statements are equivalent:
\begin{aeqi}{0}
 \item\label{GL3.1.11.i} \(\nul{C} \subseteq \nul{A}\).
 \item\label{GL3.1.11.ii} There is a matrix \(Y\in\Coo{p}{r}\) such that \(A=YC\).
 \item\label{GL3.1.11.iii} \(AC^{(1)}C=A\).
 \item\label{GL3.1.11.iv} There is a positive real number \(\gamma\) such that, \(A^\ad A \lleq \gamma C^\ad C\).
 \item\label{GL3.1.11.v} \(\ran{A^\ad} \subseteq\ran{C^\ad}\).
\end{aeqi}
\end{lem}

\begin{lem}[\cite{MR0245582,MR0394287}] \label{DFK1} 
 Let \(E\in \Coo{(p+q)}{ (p+q)}\) and let \( E=\tmat{ A & B \\ C & D }\) be the block partition of \(E\) with \tppa{block} \(A\).
 Let \(A^{(1)}\in A \set{1}\) and let \(D^{(1)} \in D \set{1}\).
 Furthermore, let \(L \defeq D - CA^{(1)} B\) and let \(R\defeq A-BD^{(1)} C\).
 Then: 
\begin{enui} 
 \item\label{DFK1.a} The following statements are equivalent:
\begin{aeqii}{0} 
 \item\label{DFK1.i} \(E \in \Cggo{(p+q)}\).
 \item\label{DFK1.ii} \(A\in\Cggp\), \(\ran{B}\subseteq \ran{A}\), \(C=B^\ad \), and \(L\in\Cggq\).
 \item\label{DFK1.iii} \(D\in\Cggq\), \(\ran{C}\subseteq \ran{D}\), \(B=C^\ad \), and \(R\in\Cggp\).
\end{aeqii}
 \item\label{DFK1.b} If~\ref{DFK1.i} is fulfilled, then \(L = D-CA^\mpi D\) and \(R=A-BD^\mpi C\).
 \item\label{DFK1.c} If~\ref{DFK1.i} holds true, then \(\rank E = \rank A + \rank L\) and \(\rank E = \rank D + \rank R\).
\end{enui}
\end{lem}

 A detailed proof of \rpartss{DFK1.a}{DFK1.c} of \rlem{DFK1} is given, \teg{}, in~\cite[\clemss{1.1.9}{1.1.7}]{MR1152328}.
 \rPart{DFK1.b} of \rlem{DFK1} is a consequence of the \rlemss{Gl,3.1.10}{GL3.1.11}.

\begin{rem} \label{CDFK06}
 Let \(n\in\N\) and let \((d_j)_{j=0}^{2n}\) be a sequence of complex \tqqa{matrices}.
 If \(d_0 = \Oqq\) and if the block Hankel matrix \(\mat{d_{j+k}}_{j,k=0}^n\) is \tnnH{}, then, in view of \rlem{DFK1}, it is readily proved by induction that \(d_j = \Oqq\) for all \(j\in\mn{0}{2n-1}\).
\end{rem}

\begin{rem} \label{NR41N}
 It is readily checked that if \(E\) is \tnnH{}, then \(\normS{B}^2\le \normS{A}\cdot \normS{D}\) (see, \teg{}~\cite[proof of Lemma~1.1.10]{MR1152328}).
\end{rem}

\begin{rem} \label{PBT}
 Let \(C\) be a non-singular complex \tqqa{matrix}, let \(B \in \Cqq\) and let \(A\defeq BC\).
 Then \(\ran{B}=\ran{A}\) and \(\nul{B}=C\nul{A}\).
\end{rem}

\begin{rem} \label{ZLM}
 Let \(A,B\in \Cqq\) and let \(\alpha,\beta \in \C\).
 Then \(( \alpha A +\beta B) \nul{B} \subseteq A \nul{B}\).
 Furthermore, if \(\alpha \neq 0\), then \(( \alpha A +\beta B) \nul{B} = A \nul{B}\).
\end{rem}

\begin{rem}\label{A11}
 Let \(A \in \Cqq\).
 For all \(z\in\C\), then \(\Re(zA) = \Re(z)\Re(A) - \Im(z)\Im(A)\) and \(\Im(zA)=\Re(z)\Im(A)+\Im(z)\Re(A)\).
\end{rem}

\section{A particular generalized inverse of a complex matrix}\label{A1556}
 In this section, we state some useful identities for the particular generalized inverse of a \tH{} complex matrix, which is introduced in \rrem{8B}.

\begin{lem}\label{A.40}
 Let \(A\) be a \tH{} complex \tqqa{matrix} and let \(\cU\) be a subspace of \(\Cq \) such that \(\nul{A} \msum \cU =\Cq \).
 Then \((A_\cU^\gi )^\ad = A_\cU^\gi\), \(\cR(A_\cU^\gi ) = \cU\), \(\cN (A_\cU^\gi ) = \cU^\orth\),
\begin{align}
 AA_\cU^\gi  A=A, & & A_\cU^\gi  AA_\cU^\gi = A_\cU^\gi,\label{D1N}\\
 \dim \cR (A_\cU^\gi )= \rank A, & \qquad \text{and} & \dim\cN (A_\cU^\gi ) = q - \rank A.\label{D3N}
\end{align}
 In particular, if \(A\) is \tnnH{}, then \(A_\cU^\gi\) is \tnnH{}, too.
\end{lem}
\begin{proof}
 By definition of \(A_\cU^\gi \), we have \eqref{D1N} as well as \(\cR (A_\cU^\gi ) = \cU\) and \(\cN (A_\cU^\gi ) =\cU^\orth\).
 In view of \(\nul{A}\msum \cU =\Cq \), then \eqref{D3N} follows.
 Since \(A^\ad = A\) is supposed, \eqref{D1N} implies \(A (A_\cU^\gi )^\ad A=A\) and \((A_\cU^\gi )^\ad A (A_\cU^\gi )^\ad =(A_\cU^\gi )^\ad\).
 Moreover, \(\cN ((A_\cU^\gi )^\ad) =\cR (A_\cU^\gi )^\orth =\cU^\orth\) and \(\cR ((A_\cU^\gi )^\ad) = \cN (A_\cU^\gi )^\orth = (\cU^\orth)^\orth = \cU\).
 Consequently, \((A_\cU^\gi )^\ad = A_{\cU, \cU^\orth}^{(1,2)} = A_\cU^\gi \).
\end{proof}

\begin{rem}\label{A.41}
 Let \(A\in\CHq \) and let \(\cU\) be a subspace of \(\Cq \) such that \(\nul{A} \msum \cU =\Cq \).
 Then \(AA^\mpi  = A^\mpi A\) and, in view of \rlem{A.40}, one can easily check that
\begin{gather*}
 (A_\cU^\gi  A)^\ad= AA_\cU^\gi,\quad
 (AA_\cU^\gi )^\ad=A_\cU^\gi A,\quad
 (A_\cU^\gi  A)^2= A_\cU^\gi  A,\quad
 (AA_\cU^\gi )^2= AA_\cU^\gi,\\
 \cR(A_\cU^\gi  A)= \cU,\quad
 \cR (AA_\cU^\gi )= \ran{A},\quad
 \cN (AA_\cU^\gi )= \cU^\orth,\quad
 \cN (A_\cU^\gi A)= \nul{A},\\
 \dim\cR (A_\cU^\gi  A)= \dim \cR (AA_\cU^\gi ) = \rank A,\quad
 \dim \cN (AA_\cU^\gi )= \dim \cN (A_\cU^\gi  A) = q - \rank A,\\
 A^\mpi  AA_\cU^\gi  A= A^\mpi  A = AA^\mpi  = AA^\mpi  A_\cU^\gi  A,\quad\text{and}\quad
 AA_\cU^\gi  AA^\mpi= AA^\mpi  = A^\mpi A = AA_\cU^\gi  A^\mpi A.
\end{gather*}
\end{rem}

 Parts of the following result are already contained in~\cite[\clem{2.3}]{MR1362524}.

\begin{lem}\label{A.42}
 Let \(A\) be a \tH{} complex \tqqa{matrix} and let \(\cU\) be a subspace of \(\Cq \) such that \(\nul{A} \msum \cU =\Cq \).
 Then
\begin{align*}
 (\Iq - AA_\cU^\gi )^\ad&= \Iq - A_\cU^\gi A,&&&
 (\Iq - AA_\cU^\gi )^2&= \Iq -AA_\cU^\gi,\\
 \cN (\Iq - AA_\cU^\gi  )&= \ran{A},&&&
 \cR (\Iq - AA_\cU^\gi )&= \cU^\orth,\\ 
 (\Iq -AA_\cU^\gi ) AA^\mpi&=\NM&&&
 (\Iq - AA_\cU^\gi ) A^\mpi A&=\NM,\\
 (\Iq - AA_\cU^\gi ) (\Iq - AA^\mpi )&= \Iq -AA_\cU^\gi,&
&\text{and}&
 (\Iq -AA_\cU^\gi )(\Iq - A^\mpi A)&=\Iq -AA_\cU^\gi.
\end{align*}
\end{lem}
\begin{proof} 
 The first two equations follow from \rrem{A.41}.
 From \rlem{A.40} we know that \eqref{D1N} is true.
 Using \eqref{D1N}, the equation \(\cN (\Iq - AA_\cU^\gi ) = \ran{A}\) can be easily checked by straightforward calculations.
 In order to prove \(\cR (\Iq - AA_\cU^\gi ) = \cU^\orth\), one shows that \eqref{D1N} implies \(\cR (\Iq - AA_\cU^\gi ) = \cN (AA_\cU^\gi )\) and one applies the equation \(\cN (AA_\cU^\gi ) = \cU^\orth\) stated in \rrem{A.41}.
 Because of \(A^\ad = A\), we have \(AA^\mpi  = A^\mpi A\).
 Thus, from \eqref{D1N} we easily see that the remaining equations hold true.
\end{proof}

\begin{lem}\label{A.43}
 Let \(A\) be a \tH{} complex \tqqa{matrix} and let \(\cU\) be a subspace of \(\Cq \) such that \(\nul{A} \msum \cU =\Cq \).
 Then
\begin{gather*}
 \cR(\Iq - A_\cU^\gi  A)= \cN (A_\cU^\gi  A) = \nul{A},\quad  \cN (\Iq - A_\cU^\gi  A)= \cR (A_\cU^\gi  A) =\cU,\\
 \dim \cR (A_\cU^\gi  A)=\rank A,\quad A^\mpi  AA_\cU^\gi  A=A^\mpi  A=AA^\mpi =AA^\mpi A_\cU^\gi  A,\\
\text{and }(\Iq-A^\mpi A)A_\cU^\gi  A=A_\cU^\gi  A-A^\mpi  A=A_\cU^\gi  A-AA^\mpi =(\Iq-AA^\mpi )A_\cU^\gi  A.
\end{gather*}
\end{lem}
\begin{proof}
 Because of \rlem{A.40}, we get \eqref{D1N}.
 From \eqref{D1N} we obtain \(\cN (A_\cU^\gi  A) = \nul{A}\) and \(\cR (A_\cU^\gi  A) = \cR (A_\cU^\gi ) = \cU\).
 In particular, \(\dim \cR (A_\cU^\gi  A) = \dim \cU = \dim \Cq  - \dim \nul{A} = \rank A\).
 Because of \(A^\ad = A\), we have \(AA^\mpi  = A^\mpi  A\).
 Therefore, \eqref{D1N} shows that \(A^\mpi  AA_\cU^\gi  A = A^\mpi A = AA^\mpi \) and \(A^\mpi  AA_\cU^\gi  A = AA^\mpi \).
 Thus, the remaining equations immediately follow.
\end{proof}

\begin{rem}\label{A.44}
 Let \(A\) be a \tH{} complex \tqqa{matrix} and let \(\cU\) be a subspace of \(\Cq \) such that \(\nul{A} \msum \cU =\Cq \).
 In view of the \rlemss{A.43}{A.40}, it is readily checked that 
\begin{align*}
 (\Iq - A_\cU^\gi  A)^\ad&= \Iq - AA_\cU^\gi,&&&
 (\Iq - A_\cU^\gi  A)^2&= \Iq - A_\cU^\gi  A,\\
 \cR (\Iq - A_\cU^\gi  A)&=\nul{A},&&&
 \cN (\Iq - A_\cU^\gi  A)&=\cU,\\
 A^\mpi A (\Iq - A_\cU^\gi  A)&=\NM,&&&
 AA^\mpi  (\Iq - A_\cU^\gi  A)&=\NM,\\
 (\Iq - A^\mpi A) (\Iq - A_\cU^\gi  A)&= \Iq - A_\cU^\gi  A,&
&\text{and}&
 (\Iq - AA^\mpi ) (\Iq - A_\cU^\gi  A)&= \Iq- A_\cU^\gi  A.
\end{align*}
\end{rem}

\begin{rem}\label{A.45}
 Let \(T\in\Cqq \) and let \(\cU\) and \(\cV\) be subspaces of \(\Cq \) with \(T^\ad (\cU) \subseteq \cV\subseteq\cU\).
 Then it is readily checked that \((T^\ad)^k (\cU) \subseteq \cV\subseteq \cU\) and \(T^k (\cV^\orth) \subseteq \cU^\orth \subseteq \cV^\orth\) is valid for each \(k\in\N \) and that \((T^\ad)^\ell (\cV) \subseteq \cU\) and \(T^\ell (\cU^\orth) \subseteq \cU^\orth \subseteq\cV^\orth\) for each \(\ell\in\NO \) hold true (see also~\cite[\ccor{3.3}]{MR1362524}, where a special case is discussed).
\end{rem}

 The following lemma is a generalization of~\cite[\clem{4.1}]{MR1362524}, where special pairs of block Hankel matrices are considered.

\begin{lem}\label{A.46}
 Let \(A\) and \(B\) be \tH{} complex \tqqa{matrices} and let \(T\in\Cqq \).
 Suppose that \(\cU\) and \(\cV\) are subspaces of \(\Cq \) such that \(\nul{A}\msum \cU =\Cq \) and \(\nul{B}\msum\cV =\Cq \) and \(T^\ad (\cU) \subseteq\cV\subseteq \cU\) hold true.
 Then 
\begin{align}\label{MU1}
 A_\cU^\gi  T^\ell (\Iq - AA_\cU^\gi )&=\NM&
&\text{and}&
 B_\cV^\gi  T^\ell (\Iq - AA_\cU^\gi )&=\NM
\end{align}
 for each \(\ell\in\NO \) and, for each \(k\in\N \), furthermore
\begin{align}\label{MU2}
 A_\cU^\gi  T^k (\Iq - BB_\cV^\gi )&=\NM&
&\text{and}&
 B_\cV^\gi  T^k (\Iq - BB_\cV^\gi )&=\NM.
\end{align}
\end{lem}
\begin{proof}
 \rlem{A.42} yields \(\cR (\Iq - AA_\cU^\gi ) = \cU^\orth\).
 Hence, from \rrem{A.45} we conclude
\beql{MU3}
 T^\ell\rk*{\cR (\Iq - AA_\cU^\gi )}
 = T^\ell(\cU^\orth)
 \subseteq \cU^\orth
 \subseteq \cV^\orth
\eeq
 for each \(\ell\in\NO \).
 Since we know from \rlem{A.40} that \(\cN (A_\cU^\gi ) = U^\orth\) is valid, it follows \(T^\ell(\Iq - AA_\cU^\gi ) x\in\cN (A_\cU^\gi )\) for each \(\ell\in\NO \) and each \(x\in\Cq \).
 Consequently, the first equation in \eqref{MU1} is fulfilled for each \(\ell\in\NO \).
 \rlem{A.40} yields \(\cN (B_\cV^\gi ) =\cV^\orth\).
 Thus, we obtain from \eqref{MU3} that \(T^\ell(\Iq - AA_\cU^\gi ) x\in\cN (B_\cV^\gi )\) is fulfilled for every choice of \(\ell\) in \(\NO \) and \(x\) in \(\Cq \).
 Therefore, the second equation in \eqref{MU1} is proved for each \(\ell\in\NO \).
 \rlem{A.42} provides us \(\cR (\Iq - BB_\cV^\gi ) = \cV^\orth\).
 Hence, \rrem{A.45} yields
\beql{MU}
 T^k\rk*{\cR (\Iq - BB_\cV^\gi )}
 = T^k (\cV^\orth)
 \subseteq 
 \cU^\orth 
 \subseteq \cV^\orth
\eeq
 for each \(k\in\N \).
 Since \rlem{A.40} shows that \(\cN (A_\cU^\gi ) = \cU^\orth\) is true, we obtain then \(T^k (\Iq - BB_\cV^\gi ) x\in\cN (A_\cU^\gi )\) for every choice of \(k\in\N \) and \(x \) in \(\Cq \).
 Consequently, the first equation in \eqref{MU2} is true for each \(k\in\N \).
 Using \eqref{MU} and the equation \(\cN (B_\cV^\gi ) = \cV^\orth\), which is proved in \rlem{A.40}, we get \(T^k (\Iq - BB_\cV^\gi ) x \in\cN (B_\cV^\gi )\) for each \(k\in\N\) and each \(x\in \Cq \).
 Thus, the second equation in \eqref{MU2} is verified for each \(k\in\N \) as well.
\end{proof}

 Now we state some more or less known identities for the matrix-valued functions defined in \rrem{21112N}.
 
\begin{rem} \label{lemmart}
 Let \(n\in \NO \) and let \(w, z\in\C\).
 Then one can easily see that the equations 
\begin{align} 
 R_{\Tqn}(z)(\Iu{(n+1)q}-w\Tqn)&=(\Iu{(n+1)q}-w\Tqn)R_{\Tqn}(z),\notag \\ 
 R_{\Tqn^\ad }(z)(\Iu{(n+1)q}-w\Tqn^\ad )&=(\Iu{(n+1)q}-w\Tqn^\ad )R_{\Tqn^\ad }(z),\notag \\
 R_{\Tqn}(z)-R_{\Tqn}(w)&=(z-w)R_{\Tqn}(w)\Tqn R_{\Tqn}(z),\notag \\ 
 R_{\Tqn^\ad }(z)-R_{\Tqn^\ad }(w)&=(z-w)R_{\Tqn^\ad }(z)\Tqn^\ad R_{\Tqn^\ad }(w),\notag \\
 \ek*{R_{\Tqn}(w)}^\inv - \ek*{R_{\Tqn}(z)}^\inv &= (z-w) \Tqn,\notag \\
 R_{\Tqn}(z)+(w-z)R_{\Tqn}(z) \Tqn R_{\Tqn}(w)&=R_{\Tqn}(w),\notag \\
 zR_{\Tqn}(z)+(w-z)R_{\Tqn}(z) R_{\Tqn}(w)&=wR_{\Tqn}(w), \notag\\
 (z-\ko{w}) \ek*{R_{T^\ad _{q,n}}(w)}^\ad \Tqn R_{\Tqn}(z) &= R_{\Tqn}(z) - \ek*{R_{T^\ad _{q,n}}(w)}^\ad,\notag \\
 (z-w) \Tqn R_{\Tqn}(z) &= R_{\Tqn}(z)\ek*{R_{\Tqn}(w)}^\inv -\Iu{(n+1)q}, \label{N52} \\
 (z-w) \Tqn^\ad R_{\Tqn^\ad }(z) &=\ek*{R_{\Tqn^\ad }(w)}^\inv R_{\Tqn^\ad }(z)-\Iu{(n+1)q}, \label{S1-1}
\intertext{and}
 (z- w) R_{\Tqn}(z) \Tqn &= R_{\Tqn}(z)\ek*{R_{\Tqn}(w)}^\inv -\Iu{(n+1)q} \label{S1-2} 
\end{align} 
 hold true.
 Furthermore, for each \(\ell\in \NO \), it is readily checked that 
\begin{align*}
 \Tqn^\ell R_{\Tqn}(z)& = R_{\Tqn}(z)\Tqn^\ell,&
 R_{\Tqn}(z)\Tqn^\ell R_{\Tqn}(w)&=R_{\Tqn}(w)\Tqn^\ell R_{\Tqn}(z),
\intertext{and}
 (\Tqn^\ad )^\ell  R_{\Tqn^\ad }(z)&=R_{\Tqn^\ad }(z)(\Tqn^\ad )^\ell,&
 R_{\Tqn^\ad }(z)(\Tqn^\ad )^\ell  R_{\Tqn^\ad }(w)&=R_{\Tqn^\ad }(w)(\Tqn^\ad )^\ell  R_{\Tqn^\ad }(z).
\end{align*}
\end{rem}

\section{Some considerations on \tnnH{} measures}\label{A1244}
 In this appendix, we summarize some facts of the integration theory of \tnnH{} measures.
 We consider a measurable space \((\Omega,\gA)\) and use the notation \(\Mggqa{\Omega,\gA}\) to denote the set of all \tnnH{} \tqqa{measures} on \((\Omega,\gA)\).

\begin{rem} \label{B3}
 Let \(\mu \colon\gA\to \Cpp\) be a mapping.
 Then \(\mu\in\Mggqa{\Omega,\gA}\) if and only if \(B^\ad\mu B\colon\gA\to\Cpp\) defined by  \((B^\ad\mu B) (A)\defeq B^\ad\mu (A) B\) belongs to \(\Mggoa{p}{\Omega,\gA}\) for all \(B\in\Cqp \).
\end{rem}

\begin{rem} \label{B8}
 Let \(\mu\in\Mggqa{\Omega,\gA}\) and let \(f\colon\Omega\to \C \) be a function.
 Then it is readily checked by standard arguments of measure and integration theory that the following statements are equivalent:
\begin{aeqi}{0}
 \item\label{B8.i} \(f\in \LoaaaC{1}{\Omega}{\gA}{\mu}\).
 \item\label{B8.ii} \(f\in \LoaaaC{1}{\Omega}{\gA}{B^\ad\mu B}\) for all \(B\in\Cqp\).
 \item\label{B8.iii} \(f\in \LoaaaC{1}{\Omega}{\gA}{\tau}\) where \(\tau \defeq \tr \mu\) is the trace measure of \(\mu\).
\end{aeqi}
 If~\ref{B8.i} holds true, then \(\int_{A} f \dif(B^\ad\mu B) = B^\ad \rk{\int_A f\dif\mu} B\) for all \(A\in \gA\) and all \(B\in\Cqp\).
\end{rem}

\begin{lem} \label{B26}
 Let \(\mu\in\Mggqa{\Omega,\gA}\) and let \(\mu'_\tau\) be a version of the Radon--Nikodym derivative of \(\mu\) with respect to the trace measure \(\tau \defeq  \tr \mu\) of \(\mu\).
 Let \(f \colon\Omega \to \C\) and \(g\colon\Omega\to\C\) be \(\gA\)\nobreakdash-\(\BorC\)\nobreakdash-measurable functions.
 Then the following statements are equivalent:
\begin{aeqi}{0}
 \item\label{B26.i} \(f\ko{g} \in  \LoaaaC{1}{\Omega}{\gA}{\mu}\).
 \item\label{B26.ii} The pair \([f\Iq, g\Iq]\) is left-integrable with respect to \(\mu\).
\end{aeqi}
 If~\ref{B26.i} is fulfilled, then \(\int_\Omega f\ko{g} \dif\mu = \int_\Omega (f\Iq) \dif\mu (g\Iq)^\ad\).
\end{lem}

 \rlem{B26} can be proved by standard methods of measure and integration theory.
 
\begin{rem} \label{M22}
 Let \(\mu\in\Mggqa{\Omega,\gA}\) and let \(m,n\in\N\).
 For each \(j\in \mn{1}{m}\), let \(p_j\in\N\) and let \(\Phi_j \colon\Omega\to \Coo{p_j}{q}\)  be an \(\gA\)\nobreakdash-\(\Boruu{p_j}{q}\)\nobreakdash-measurable matrix-valued function.
 For each \(k\in\mn{1}{n}\), let \(r_k\in\N\) and let \(\Psi_k \colon\Omega\to\Coo{r_k}{q}\) be an \(\gA\)\nobreakdash-\(\Boruu{r_k}{q}\)\nobreakdash-measurable matrix-valued function.
 Suppose that, for every choice of \(j\in\mn{1}{m}\) and \(k\in\mn{1}{n}\) the pair \([\Phi_j, \Psi_k]\) is left-integrable with respect to \(\mu\).
 Let \(s,t\in\N\).
 For each \(j\in\mn{1}{m}\), let \(A_j\in\Coo{s}{p_j}\), and, for each \(k\in\mn{1}{n}\), let \(B_k\in\Coo{t}{r_k}\).
 Then it is readily checked that the pair \(\ek{ \sum_{j=1}^m A_j\Phi_j, \sum_{k=1}^n B_k \Psi_k}\) is left-integrable with respect to \(\mu\) and that
\[
 \int_\Omega \rk*{\sum_{j=1}^m A_j\Phi_j} \dif\mu   \rk*{\sum_{k=1}^m B_k\Psi_k}^\ad
 = \sum_{j=1}^m \sum_{k=1}^n A_j \rk*{ \int_\Omega\Phi \dif\mu\Psi^\ad }B_k^\ad.
\]
\end{rem}

\begin{prop} \label{M24}
 Let \(\mu\in\Mggqa{\Omega,\gA}\), let \(\tau \defeq  \tr\mu\) be the trace measure of \(\mu\), and let \(\mu'_\tau\) be a  version of the Radon--Nikodym derivative of \(\mu\) with respect to \(\tau\).
 Furthermore, let \(\Theta \in \pqLsaaaC{\Omega}{\gA}{\mu}\).
 Then:
\begin{enui}
 \item \(\mu_\Theta \colon\gA\to\Cpp\) defined by \(\mu_\Theta (A) \defeq  \int_A \Theta \dif\mu \Theta^\ad\) belongs to \(\Mggoa{p}{\Omega,\gA}\).
 \item The \tnnH{} measure \(\mu_\Theta\) is absolutely continuous with respect to \(\tau\) and \(\Theta \mu'_\tau \Theta^\ad\) is a version of the Radon--Nikodym derivative of \(\mu_\Theta\) with respect to \(\tau\).
 \item Let \(r,s\in\N\), let \(\Phi \colon\Omega\to \Coo{r}{p}\) be an \(\gA\)\nobreakdash-\(\Boruu{r}{p}\)\nobreakdash-measurable function and let \(\Psi \colon\Omega\to\Coo{s}{p}\) be an \(\gA\)\nobreakdash-\(\Boruu{s}{p}\)\nobreakdash-measurable function.
 Then the pair \([\Phi, \Psi]\) is left-integrable with respect to \(\mu_\Theta\) if and only if the pair \([\Phi\Theta, \Psi\Theta]\) is left-integrable with respect to \(\mu\).
 In this case, 
\(
 \int_\Omega \Phi \dif\mu_\Theta \Psi^\ad
 = \int_\Omega (\Phi\Theta) \dif\mu (\Psi\Theta)^\ad
\).
\end{enui}
\end{prop}

 \rprop{M24} can be proved by standard arguments of measure and integration theory.

\bibliography{170arxiv_ams,170arxiv_unpub}

\def\cprime{$'$}
\begin{thebibliography}{10}
  \providebibliographyfont{name}{}%
  \providebibliographyfont{lastname}{}%
  \providebibliographyfont{title}{\emph}%
  \providebibliographyfont{jtitle}{\btxtitlefont}%
  \providebibliographyfont{etal}{\emph}%
  \providebibliographyfont{journal}{}%
  \providebibliographyfont{volume}{}%
  \providebibliographyfont{ISBN}{\MakeUppercase}%
  \providebibliographyfont{ISSN}{\MakeUppercase}%
  \providebibliographyfont{url}{\url}%
  \providebibliographyfont{numeral}{}%
  \expandafter\btxselectlanguage\expandafter {\btxfallbacklanguage}

\expandafter\btxselectlanguage\expandafter {\btxfallbacklanguage}
\bibitem {MR2155645}
\btxnamefont {\btxlastnamefont {Adamyan},~V.\btxfnamespaceshort M.}
  \btxandshort {.}\ \btxnamefont {I.\btxfnamespaceshort M. \btxlastnamefont
  {Tkachenko}}\btxauthorcolon\ \btxjtitlefont {\btxifchangecase {Solution of
  the {S}tieltjes truncated matrix moment problem}{Solution of the {S}tieltjes
  truncated matrix moment problem}}.
\newblock \btxjournalfont {Opuscula Math.}, 25(1):5--24, 2005\ifbtxprintISSN {,
  \mbox{\btxISSN~\btxISSNfont {1232-9274}}}.

\bibitem {MR2215856}
\btxnamefont {\btxlastnamefont {Adamyan},~V.\btxfnamespaceshort M.}
  \btxandshort {.}\ \btxnamefont {I.\btxfnamespaceshort M. \btxlastnamefont
  {Tkachenko}}\btxauthorcolon\ \btxtitlefont {\btxifchangecase {General
  solution of the {S}tieltjes truncated matrix moment problem}{General solution
  of the {S}tieltjes truncated matrix moment problem}}.
\newblock \Btxinshort {.}\ \btxtitlefont {Operator theory and indefinite inner
  product spaces}, \btxvolumeshort {.}\ \btxvolumefont {163} \btxofseriesshort
  {.}\ \btxtitlefont {Oper. Theory Adv. Appl.}, \btxpagesshort {.}\ 1--22.
  \btxpublisherfont {Birkh\"auser}, Basel, 2006.

\bibitem {MR0184042}
\btxnamefont {\btxlastnamefont {Akhiezer},~N.\btxfnamespaceshort
  I.}\btxauthorcolon\ \btxtitlefont {The classical moment problem and some
  related questions in analysis}.
\newblock Translated by N. Kemmer. \btxpublisherfont {Hafner Publishing Co.,
  New York}, 1965.

\bibitem {MR0245582}
\btxnamefont {\btxlastnamefont {Albert},~A.}\btxauthorcolon\ \btxjtitlefont
  {\btxifchangecase {Conditions for positive and nonnegative definiteness in
  terms of pseudoinverses}{Conditions for positive and nonnegative definiteness
  in terms of pseudoinverses}}.
\newblock \btxjournalfont {SIAM J. Appl. Math.}, 17:434--440,
  1969\ifbtxprintISSN {, \mbox{\btxISSN~\btxISSNfont {0036-1399}}}.

\bibitem {MR0290157}
\btxnamefont {\btxlastnamefont {And{\^o}},~T.}\btxauthorcolon\ \btxjtitlefont
  {\btxifchangecase {Truncated moment problems for operators}{Truncated moment
  problems for operators}}.
\newblock \btxjournalfont {Acta Sci. Math. (Szeged)}, 31:319--334,
  1970\ifbtxprintISSN {, \mbox{\btxISSN~\btxISSNfont {0001-6969}}}.

\bibitem {MR587113}
\btxnamefont {\btxlastnamefont {Ben-Israel},~A.} \btxandshort {.}\ \btxnamefont
  {T.\btxfnamespaceshort N.\btxfnamespaceshort E. \btxlastnamefont
  {Greville}}\btxauthorcolon\ \btxtitlefont {Generalized inverses: theory and
  applications}.
\newblock \btxpublisherfont {Robert E. Krieger Publishing Co., Inc.,
  Huntington, N.Y.}, 1980\ifbtxprintISBN {, \mbox{\btxISBN~\btxISBNfont
  {0-88275-991-4}}}.
\newblock Corrected reprint of the 1974 original.

\bibitem {MR975671}
\btxnamefont {\btxlastnamefont {Bolotnikov},~V.\btxfnamespaceshort
  A.}\btxauthorcolon\ \btxjtitlefont {\btxifchangecase {Descriptions of
  solutions of a degenerate moment problem on the axis and the
  halfaxis}{Descriptions of solutions of a degenerate moment problem on the
  axis and the halfaxis}}.
\newblock \btxjournalfont {Teor. Funktsi\u\i\ Funktsional. Anal. i Prilozhen.},
  (50):25--31, i, 1988\ifbtxprintISSN {, \mbox{\btxISSN~\btxISSNfont
  {0321-4427}}}.

\bibitem {MR1362524}
\btxnamefont {\btxlastnamefont {Bolotnikov},~V.\btxfnamespaceshort
  A.}\btxauthorcolon\ \btxjtitlefont {\btxifchangecase {Degenerate {S}tieltjes
  moment problem and associated {$J$}-inner polynomials}{Degenerate {S}tieltjes
  moment problem and associated {$J$}-inner polynomials}}.
\newblock \btxjournalfont {Z. Anal. Anwendungen}, 14(3):441--468,
  1995\ifbtxprintISSN {, \mbox{\btxISSN~\btxISSNfont {0232-2064}}}.

\bibitem {MR1395706}
\btxnamefont {\btxlastnamefont {Bolotnikov},~V.\btxfnamespaceshort
  A.}\btxauthorcolon\ \btxjtitlefont {\btxifchangecase {On degenerate
  {H}amburger moment problem and extensions of nonnegative {H}ankel block
  matrices}{On degenerate {H}amburger moment problem and extensions of
  nonnegative {H}ankel block matrices}}.
\newblock \btxjournalfont {Integral Equations Operator Theory}, 25(3):253--276,
  1996\ifbtxprintISSN {, \mbox{\btxISSN~\btxISSNfont {0378-620X}}}.

\bibitem {MR1433234}
\btxnamefont {\btxlastnamefont {Bolotnikov},~V.\btxfnamespaceshort
  A.}\btxauthorcolon\ \btxjtitlefont {\btxifchangecase {On a general moment
  problem on the half axis}{On a general moment problem on the half axis}}.
\newblock \btxjournalfont {Linear Algebra Appl.}, 255:57--112,
  1997\ifbtxprintISSN {, \mbox{\btxISSN~\btxISSNfont {0024-3795}}}.

\bibitem {MR1722780}
\btxnamefont {\btxlastnamefont {Bolotnikov},~V.\btxfnamespaceshort A.}
  \btxandshort {.}\ \btxnamefont {L.\btxfnamespaceshort A. \btxlastnamefont
  {Sakhnovich}}\btxauthorcolon\ \btxjtitlefont {\btxifchangecase {On an
  operator approach to interpolation problems for {S}tieltjes functions}{On an
  operator approach to interpolation problems for {S}tieltjes functions}}.
\newblock \btxjournalfont {Integral Equations Operator Theory}, 35(4):423--470,
  1999\ifbtxprintISSN {, \mbox{\btxISSN~\btxISSNfont {0378-620X}}}.

\bibitem {MR555733}
\btxnamefont {\btxlastnamefont {Burckel},~R.\btxfnamespaceshort
  B.}\btxauthorcolon\ \btxtitlefont {An introduction to classical complex
  analysis. {V}ol. 1}, \btxvolumeshort {.}~\btxvolumefont {82}
  \btxofseriesshort {.}\ \btxtitlefont {Pure and Applied Mathematics}.
\newblock \btxpublisherfont {Academic Press, Inc. [Harcourt Brace Jovanovich,
  Publishers], New York-London}, 1979\ifbtxprintISBN {,
  \mbox{\btxISBN~\btxISBNfont {0-12-141701-8}}}.

\bibitem {MR1624548}
\btxnamefont {\btxlastnamefont {Chen},~G.\btxfnamespaceshort N.} \btxandshort
  {.}\ \btxnamefont {Y.\btxfnamespaceshort J. \btxlastnamefont
  {Hu}}\btxauthorcolon\ \btxjtitlefont {\btxifchangecase {The truncated
  {H}amburger matrix moment problems in the nondegenerate and degenerate cases,
  and matrix continued fractions}{The truncated {H}amburger matrix moment
  problems in the nondegenerate and degenerate cases, and matrix continued
  fractions}}.
\newblock \btxjournalfont {Linear Algebra Appl.}, 277(1-3):199--236,
  1998\ifbtxprintISSN {, \mbox{\btxISSN~\btxISSNfont {0024-3795}}}.

\bibitem {MR1807884}
\btxnamefont {\btxlastnamefont {Chen},~G.\btxfnamespaceshort N.} \btxandshort
  {.}\ \btxnamefont {Y.\btxfnamespaceshort J. \btxlastnamefont
  {Hu}}\btxauthorcolon\ \btxjtitlefont {\btxifchangecase {A unified treatment
  for the matrix {S}tieltjes moment problem in both nondegenerate and
  degenerate cases}{A unified treatment for the matrix {S}tieltjes moment
  problem in both nondegenerate and degenerate cases}}.
\newblock \btxjournalfont {J. Math. Anal. Appl.}, 254(1):23--34,
  2001\ifbtxprintISSN {, \mbox{\btxISSN~\btxISSNfont {0022-247X}}}.

\bibitem {MR1670527}
\btxnamefont {\btxlastnamefont {Chen},~G.\btxfnamespaceshort N.} \btxandshort
  {.}\ \btxnamefont {X.\btxfnamespaceshort Q. \btxlastnamefont
  {Li}}\btxauthorcolon\ \btxjtitlefont {\btxifchangecase {The
  {N}evanlinna-{P}ick interpolation problems and power moment problems for
  matrix-valued functions}{The {N}evanlinna-{P}ick interpolation problems and
  power moment problems for matrix-valued functions}}.
\newblock \btxjournalfont {Linear Algebra Appl.}, 288(1-3):123--148,
  1999\ifbtxprintISSN {, \mbox{\btxISSN~\btxISSNfont {0024-3795}}}.

\bibitem {CR01}
\btxnamefont {\btxlastnamefont {Choque~Rivero},~A.\btxfnamespaceshort
  E.}\btxauthorcolon\ \btxtitlefont {Ein finites {M}atrixmomentenproblem auf
  einem endlichen {I}ntervall}.
\newblock \btxifchangecase {Dissertation}{Dissertation}, Universit{\"a}t
  Leipzig, Leipzig, 2001.

\bibitem {MR2222521}
\btxnamefont {\btxlastnamefont {Choque~Rivero},~A.\btxfnamespaceshort E.},
  \btxnamefont {{\relax Yu}.\btxfnamespaceshort M. \btxlastnamefont
  {Dyukarev}}, \btxnamefont {B.~\btxlastnamefont {Fritzsche}}\btxandcomma {}
  \btxandshort {.}\ \btxnamefont {B.~\btxlastnamefont
  {Kirstein}}\btxauthorcolon\ \btxtitlefont {\btxifchangecase {A truncated
  matricial moment problem on a finite interval}{A truncated matricial moment
  problem on a finite interval}}.
\newblock \Btxinshort {.}\ \btxtitlefont {Interpolation, {S}chur functions and
  moment problems}, \btxvolumeshort {.}\ \btxvolumefont {165} \btxofseriesshort
  {.}\ \btxtitlefont {Oper. Theory Adv. Appl.}, \btxpagesshort {.}\ 121--173.
  \btxpublisherfont {Birkh\"auser, Basel}, 2006.

\bibitem {MR0203464}
\btxnamefont {\btxlastnamefont {Douglas},~R.\btxfnamespaceshort
  G.}\btxauthorcolon\ \btxjtitlefont {\btxifchangecase {On majorization,
  factorization, and range inclusion of operators on {H}ilbert space}{On
  majorization, factorization, and range inclusion of operators on {H}ilbert
  space}}.
\newblock \btxjournalfont {Proc. Amer. Math. Soc.}, 17:413--415,
  1966\ifbtxprintISSN {, \mbox{\btxISSN~\btxISSNfont {0002-9939}}}.

\bibitem {Dub}
\btxnamefont {\btxlastnamefont {Dubovoj},~V.\btxfnamespaceshort
  K.}\btxauthorcolon\ \btxtitlefont {\btxifchangecase {Indefinite metric in
  {S}chur's interpolation problem for analytic functions}{Indefinite metric in
  {S}chur's interpolation problem for analytic functions}}.
\newblock \textit{Teor. Funktsi\u\i\ Funktsional. Anal. i Prilozhen.},
  I~(37):14--26, 1982; II~(38):32--39, 127, 1982; III~(41):55--64, 1984;
  IV~(42):46--57, 1984; V~(45):16--26, i, 1986; VI~(47):112--119, 1987.

\bibitem {MR1152328}
\btxnamefont {\btxlastnamefont {Dubovoj},~V.\btxfnamespaceshort K.},
  \btxnamefont {B.~\btxlastnamefont {Fritzsche}}\btxandcomma {} \btxandshort
  {.}\ \btxnamefont {B.~\btxlastnamefont {Kirstein}}\btxauthorcolon\
  \btxtitlefont {Matricial version of the classical {S}chur problem},
  \btxvolumeshort {.}\ \btxvolumefont {129} \btxofseriesshort {.}\
  \btxtitlefont {Teubner-Texte zur Mathematik [Teubner Texts in Mathematics]}.
\newblock \btxpublisherfont {B. G. Teubner Verlagsgesellschaft mbH}, Stuttgart,
  1992\ifbtxprintISBN {, \mbox{\btxISBN~\btxISBNfont {3-8154-2015-6}}}.
\newblock With German, French and Russian summaries.

\bibitem {Dyu81}
\btxnamefont {\btxlastnamefont {Dyukarev},~{\relax Yu}.\btxfnamespaceshort
  M.}\btxauthorcolon\ \btxtitlefont {\btxifchangecase {The {S}tieltjes matrix
  moment problem}{The {S}tieltjes matrix moment problem}}.
\newblock Deposited in VINITI (Moscow) at~22.03.81, No.~2628-81, 1981.
\newblock Manuscript,~37 pp.

\bibitem {MR686076}
\btxnamefont {\btxlastnamefont {Dyukarev},~{\relax Yu}.\btxfnamespaceshort
  M.}\btxauthorcolon\ \btxjtitlefont {\btxifchangecase {Multiplicative and
  additive {S}tieltjes classes of analytic matrix-valued functions and
  interpolation problems connected with them. {II}}{Multiplicative and additive
  {S}tieltjes classes of analytic matrix-valued functions and interpolation
  problems connected with them. {II}}}.
\newblock \btxjournalfont {Teor. Funktsi\u\i\ Funktsional. Anal. i Prilozhen.},
  (38):40--48, 127, 1982\ifbtxprintISSN {, \mbox{\btxISSN~\btxISSNfont
  {0321-4427}}}.

\bibitem {Dyu01}
\btxnamefont {\btxlastnamefont {Dyukarev},~{\relax Yu}.\btxfnamespaceshort M.},
  2001.
\newblock handwritten manuscript.

\bibitem {MR2735313}
\btxnamefont {\btxlastnamefont {Dyukarev},~{\relax Yu}.\btxfnamespaceshort M.},
  \btxnamefont {B.~\btxlastnamefont {Fritzsche}}, \btxnamefont
  {B.~\btxlastnamefont {Kirstein}}\btxandcomma {} \btxandshort {.}\
  \btxnamefont {C.~\btxlastnamefont {M{\"a}dler}}\btxauthorcolon\
  \btxjtitlefont {\btxifchangecase {On truncated matricial {S}tieltjes type
  moment problems}{On truncated matricial {S}tieltjes type moment problems}}.
\newblock \btxjournalfont {Complex Anal. Oper. Theory}, 4(4):905--951,
  2010\ifbtxprintISSN {, \mbox{\btxISSN~\btxISSNfont {1661-8254}}}.

\bibitem {MR2570113}
\btxnamefont {\btxlastnamefont {Dyukarev},~{\relax Yu}.\btxfnamespaceshort M.},
  \btxnamefont {B.~\btxlastnamefont {Fritzsche}}, \btxnamefont
  {B.~\btxlastnamefont {Kirstein}}, \btxnamefont {C.~\btxlastnamefont
  {M{\"a}dler}}\btxandcomma {} \btxandshort {.}\ \btxnamefont
  {H.\btxfnamespaceshort C. \btxlastnamefont {Thiele}}\btxauthorcolon\
  \btxjtitlefont {\btxifchangecase {On distinguished solutions of truncated
  matricial {H}amburger moment problems}{On distinguished solutions of
  truncated matricial {H}amburger moment problems}}.
\newblock \btxjournalfont {Complex Anal. Oper. Theory}, 3(4):759--834,
  2009\ifbtxprintISSN {, \mbox{\btxISSN~\btxISSNfont {1661-8254}}}.

\bibitem {MR645305}
\btxnamefont {\btxlastnamefont {Dyukarev},~{\relax Yu}.\btxfnamespaceshort M.}
  \btxandshort {.}\ \btxnamefont {V.\btxfnamespaceshort E. \btxlastnamefont
  {Katsnelson}}\btxauthorcolon\ \btxjtitlefont {\btxifchangecase
  {Multiplicative and additive {S}tieltjes classes of analytic matrix-valued
  functions and interpolation problems connected with them. {I}}{Multiplicative
  and additive {S}tieltjes classes of analytic matrix-valued functions and
  interpolation problems connected with them. {I}}}.
\newblock \btxjournalfont {Teor. Funktsi\u\i \ Funktsional. Anal. i
  Prilozhen.}, (36):13--27, 126, 1981\ifbtxprintISSN {,
  \mbox{\btxISSN~\btxISSNfont {0321-4427}}}.

\bibitem {MR752057}
\btxnamefont {\btxlastnamefont {Dyukarev},~{\relax Yu}.\btxfnamespaceshort M.}
  \btxandshort {.}\ \btxnamefont {V.\btxfnamespaceshort E. \btxlastnamefont
  {Katsnelson}}\btxauthorcolon\ \btxjtitlefont {\btxifchangecase
  {Multiplicative and additive {S}tieltjes classes of analytic matrix-valued
  functions, and interpolation problems connected with them.
  {III}}{Multiplicative and additive {S}tieltjes classes of analytic
  matrix-valued functions, and interpolation problems connected with them.
  {III}}}.
\newblock \btxjournalfont {Teor. Funktsi\u\i\ Funktsional. Anal. i Prilozhen.},
  (41):64--70, 1984\ifbtxprintISSN {, \mbox{\btxISSN~\btxISSNfont
  {0321-4427}}}.

\bibitem {MR0394287}
\btxnamefont {\btxlastnamefont {Efimov},~A.\btxfnamespaceshort V.} \btxandshort
  {.}\ \btxnamefont {V.\btxfnamespaceshort P. \btxlastnamefont
  {Potapov}}\btxauthorcolon\ \btxjtitlefont {\btxifchangecase {{$J$}-expanding
  matrix-valued functions, and their role in the analytic theory of electrical
  circuits}{{$J$}-expanding matrix-valued functions, and their role in the
  analytic theory of electrical circuits}}.
\newblock \btxjournalfont {Uspehi Mat. Nauk}, 28(1(169)):65--130,
  1973\ifbtxprintISSN {, \mbox{\btxISSN~\btxISSNfont {0042-1316}}}.

\bibitem {MR2257838}
\btxnamefont {\btxlastnamefont {Elstrodt},~J.}\btxauthorcolon\ \btxtitlefont
  {Ma\ss- und {I}ntegrationstheorie}.
\newblock Springer-Lehrbuch. [Springer Textbook]. \btxpublisherfont
  {Springer-Verlag, Berlin}, \btxeditionnumshort {fourth}{.},
  2005\ifbtxprintISBN {, \mbox{\btxISBN~\btxISBNfont {3-540-21390-2}}}.
\newblock Grundwissen Mathematik. [Basic Knowledge in Mathematics].

\bibitem {MR975253}
\btxnamefont {\btxlastnamefont {Fritzsche},~B.} \btxandshort {.}\ \btxnamefont
  {B.~\btxlastnamefont {Kirstein}}\btxauthorcolon\ \btxjtitlefont
  {\btxifchangecase {Schwache {K}onvergenz nichtnegativ hermitescher
  {B}orelma\ss e}{Schwache {K}onvergenz nichtnegativ hermitescher {B}orelma\ss
  e}}.
\newblock \btxjournalfont {Wiss. Z. Karl-Marx-Univ. Leipzig Math.-Natur.
  Reihe}, 37(4):375--398, 1988\ifbtxprintISSN {, \mbox{\btxISSN~\btxISSNfont
  {0300-0540}}}.

\bibitem {MR2805417}
\btxnamefont {\btxlastnamefont {Fritzsche},~B.}, \btxnamefont
  {B.~\btxlastnamefont {Kirstein}}\btxandcomma {} \btxandshort {.}\
  \btxnamefont {C.~\btxlastnamefont {M{\"a}dler}}\btxauthorcolon\
  \btxjtitlefont {\btxifchangecase {On {H}ankel nonnegative definite sequences,
  the canonical {H}ankel parametrization, and orthogonal matrix polynomials}{On
  {H}ankel nonnegative definite sequences, the canonical {H}ankel
  parametrization, and orthogonal matrix polynomials}}.
\newblock \btxjournalfont {Complex Anal. Oper. Theory}, 5(2):447--511,
  2011\ifbtxprintISSN {, \mbox{\btxISSN~\btxISSNfont {1661-8254}}}.

\bibitem {MR3014201}
\btxnamefont {\btxlastnamefont {Fritzsche},~B.}, \btxnamefont
  {B.~\btxlastnamefont {Kirstein}}\btxandcomma {} \btxandshort {.}\
  \btxnamefont {C.~\btxlastnamefont {M{\"a}dler}}\btxauthorcolon\ \btxtitlefont
  {\btxifchangecase {On a special parametrization of matricial
  {$\alpha$-S}tieltjes one-sided non-negative definite sequences}{On a special
  parametrization of matricial {$\alpha$-S}tieltjes one-sided non-negative
  definite sequences}}.
\newblock \Btxinshort {.}\ \btxtitlefont {Interpolation, {S}chur functions and
  moment problems. {II}}, \btxvolumeshort {.}\ \btxvolumefont {226}
  \btxofseriesshort {.}\ \btxtitlefont {Oper. Theory Adv. Appl.},
  \btxpagesshort {.}\ 211--250. \btxpublisherfont {Birkh\"auser/Springer Basel
  AG, Basel}, 2012.

\bibitem {MR2988005}
\btxnamefont {\btxlastnamefont {Fritzsche},~B.}, \btxnamefont
  {B.~\btxlastnamefont {Kirstein}}\btxandcomma {} \btxandshort {.}\
  \btxnamefont {C.~\btxlastnamefont {M{\"a}dler}}\btxauthorcolon\
  \btxjtitlefont {\btxifchangecase {On matrix-valued {H}erglotz-{N}evanlinna
  functions with an emphasis on particular subclasses}{On matrix-valued
  {H}erglotz-{N}evanlinna functions with an emphasis on particular
  subclasses}}.
\newblock \btxjournalfont {Math. Nachr.}, 285(14-15):1770--1790,
  2012\ifbtxprintISSN {, \mbox{\btxISSN~\btxISSNfont {0025-584X}}}.

\bibitem {MR3380267}
\btxnamefont {\btxlastnamefont {Fritzsche},~B.}, \btxnamefont
  {B.~\btxlastnamefont {Kirstein}}\btxandcomma {} \btxandshort {.}\
  \btxnamefont {C.~\btxlastnamefont {M{\"a}dler}}\btxauthorcolon\ \btxtitlefont
  {\btxifchangecase {On a simultaneous approach to the even and odd truncated
  matricial {H}amburger moment problems}{On a simultaneous approach to the even
  and odd truncated matricial {H}amburger moment problems}}.
\newblock \Btxinshort {.}\ \btxtitlefont {Recent advances in inverse
  scattering, {S}chur analysis and stochastic processes}, \btxvolumeshort {.}\
  \btxvolumefont {244} \btxofseriesshort {.}\ \btxtitlefont {Oper. Theory Adv.
  Appl.}, \btxpagesshort {.}\ 181--285. \btxpublisherfont
  {Birkh\"auser/Springer, Cham}, 2015.

\bibitem {MR3644521}
\btxnamefont {\btxlastnamefont {Fritzsche},~B.}, \btxnamefont
  {B.~\btxlastnamefont {Kirstein}}\btxandcomma {} \btxandshort {.}\
  \btxnamefont {C.~\btxlastnamefont {M\"adler}}\btxauthorcolon\ \btxtitlefont
  {\btxifchangecase {On matrix-valued {S}tieltjes functions with an emphasis on
  particular subclasses}{On matrix-valued {S}tieltjes functions with an
  emphasis on particular subclasses}}.
\newblock \Btxinshort {.}\ \btxtitlefont {Large truncated {T}oeplitz matrices,
  {T}oeplitz operators, and related topics}, \btxvolumeshort {.}\
  \btxvolumefont {259} \btxofseriesshort {.}\ \btxtitlefont {Oper. Theory Adv.
  Appl.}, \btxpagesshort {.}\ 301--352. \btxpublisherfont
  {Birkh\"auser/Springer, Cham}, 2017.

\bibitem {MR1784638}
\btxnamefont {\btxlastnamefont {Gesztesy},~F.} \btxandshort {.}\ \btxnamefont
  {E.\btxfnamespaceshort R. \btxlastnamefont {Tsekanovskii}}\btxauthorcolon\
  \btxjtitlefont {\btxifchangecase {On matrix-valued {H}erglotz functions}{On
  matrix-valued {H}erglotz functions}}.
\newblock \btxjournalfont {Math. Nachr.}, 218:61--138, 2000\ifbtxprintISSN {,
  \mbox{\btxISSN~\btxISSNfont {0025-584X}}}.

\bibitem {MR722914}
\btxnamefont {\btxlastnamefont {Golinski{\u\i}},~L.\btxfnamespaceshort
  B.}\btxauthorcolon\ \btxjtitlefont {\btxifchangecase {A generalization of the
  matrix {N}evanlinna-{P}ick problem}{A generalization of the matrix
  {N}evanlinna-{P}ick problem}}.
\newblock \btxjournalfont {Izv. Akad. Nauk Armyan. SSR Ser. Mat.},
  18(3):187--205, 1983\ifbtxprintISSN {, \mbox{\btxISSN~\btxISSNfont
  {0002-3043}}}.

\bibitem {G1}
\btxnamefont {\btxlastnamefont {Golinski{\u\i}},~L.\btxfnamespaceshort
  B.}\btxauthorcolon\ \btxtitlefont {\btxifchangecase {On the
  {N}evanlinna-{P}ick problem in the generalized {S}chur class of analytic
  matrix functions}{On the {N}evanlinna-{P}ick problem in the generalized
  {S}chur class of analytic matrix functions}}.
\newblock \Btxinshort {.}\ \btxnamefont {\btxlastnamefont
  {Marchenko},~V.\btxfnamespaceshort A.}\ (\btxeditorshort {.}): \btxtitlefont
  {Analysis in Indefinite-Dimensional Spaces and Operator Theory},
  \btxpagesshort {.}\ 23--33. \btxpublisherfont {Naukova Dumka}, Kiev, 1983.

\bibitem {MR2038751}
\btxnamefont {\btxlastnamefont {Hu},~Y.\btxfnamespaceshort J.} \btxandshort
  {.}\ \btxnamefont {G.\btxfnamespaceshort N. \btxlastnamefont
  {Chen}}\btxauthorcolon\ \btxjtitlefont {\btxifchangecase {A unified treatment
  for the matrix {S}tieltjes moment problem}{A unified treatment for the matrix
  {S}tieltjes moment problem}}.
\newblock \btxjournalfont {Linear Algebra Appl.}, 380:227--239,
  2004\ifbtxprintISSN {, \mbox{\btxISSN~\btxISSNfont {0024-3795}}}.

\bibitem {MR1310360}
\btxnamefont {\btxlastnamefont {Ivanchenko},~T.\btxfnamespaceshort S.}
  \btxandshort {.}\ \btxnamefont {L.\btxfnamespaceshort A. \btxlastnamefont
  {Sakhnovich}}\btxauthorcolon\ \btxtitlefont {\btxifchangecase {An operator
  approach to the {P}otapov scheme for the solution of interpolation
  problems}{An operator approach to the {P}otapov scheme for the solution of
  interpolation problems}}.
\newblock \Btxinshort {.}\ \btxtitlefont {Matrix and operator valued
  functions}, \btxvolumeshort {.}~\btxvolumefont {72} \btxofseriesshort {.}\
  \btxtitlefont {Oper. Theory Adv. Appl.}, \btxpagesshort {.}\ 48--86.
  \btxpublisherfont {Birkh\"auser, Basel}, 1994.

\bibitem {MR0080280}
\btxnamefont {\btxlastnamefont {Kats},~I.\btxfnamespaceshort
  S.}\btxauthorcolon\ \btxjtitlefont {\btxifchangecase {On {H}ilbert spaces
  generated by monotone {H}ermitian matrix-functions}{On {H}ilbert spaces
  generated by monotone {H}ermitian matrix-functions}}.
\newblock \btxjournalfont {Har\cprime kov Gos. Univ. U\v c. Zap. 34 = Zap. Mat.
  Otd. Fiz.-Mat. Fak. i Har\cprime kov. Mat. Ob\v s\v c. (4)}, 22:95--113
  (1951), 1950.

\bibitem {MR645308}
\btxnamefont {\btxlastnamefont {Katsnelson},~V.\btxfnamespaceshort
  E.}\btxauthorcolon\ \btxjtitlefont {\btxifchangecase {Continual analogues of
  the {H}amburger-{N}evanlinna theorem and fundamental matrix inequalities of
  classical problems. {I}}{Continual analogues of the {H}amburger-{N}evanlinna
  theorem and fundamental matrix inequalities of classical problems. {I}}}.
\newblock \btxjournalfont {Teor. Funktsi\u\i\ Funktsional. Anal. i Prilozhen.},
  (36):31--48, 127, 1981\ifbtxprintISSN {, \mbox{\btxISSN~\btxISSNfont
  {0321-4427}}}.

\bibitem {MR701996}
\btxnamefont {\btxlastnamefont {Katsnelson},~V.\btxfnamespaceshort
  E.}\btxauthorcolon\ \btxjtitlefont {\btxifchangecase {Continual analogues of
  the {H}amburger-{N}evanlinna theorem and fundamental matrix inequalities of
  classical problems. {II}}{Continual analogues of the {H}amburger-{N}evanlinna
  theorem and fundamental matrix inequalities of classical problems. {II}}}.
\newblock \btxjournalfont {Teor. Funktsi\u\i\ Funktsional. Anal. i Prilozhen.},
  (37):31--48, 1982\ifbtxprintISSN {, \mbox{\btxISSN~\btxISSNfont
  {0321-4427}}}.

\bibitem {MR734686}
\btxnamefont {\btxlastnamefont {Katsnelson},~V.\btxfnamespaceshort
  E.}\btxauthorcolon\ \btxjtitlefont {\btxifchangecase {Continual analogues of
  the {H}amburger-{N}evanlinna theorem and fundamental matrix inequalities of
  classical problems. {III}}{Continual analogues of the
  {H}amburger-{N}evanlinna theorem and fundamental matrix inequalities of
  classical problems. {III}}}.
\newblock \btxjournalfont {Teor. Funktsi\u\i\ Funktsional. Anal. i Prilozhen.},
  (39):61--73, 1983\ifbtxprintISSN {, \mbox{\btxISSN~\btxISSNfont
  {0321-4427}}}.

\bibitem {MR738449}
\btxnamefont {\btxlastnamefont {Katsnelson},~V.\btxfnamespaceshort
  E.}\btxauthorcolon\ \btxjtitlefont {\btxifchangecase {Continual analogues of
  the {H}amburger-{N}evanlinna theorem, and fundamental matrix inequalities of
  classical problems. {IV}}{Continual analogues of the {H}amburger-{N}evanlinna
  theorem, and fundamental matrix inequalities of classical problems. {IV}}}.
\newblock \btxjournalfont {Teor. Funktsi\u\i\ Funktsional. Anal. i Prilozhen.},
  (40):79--90, 1983\ifbtxprintISSN {, \mbox{\btxISSN~\btxISSNfont
  {0321-4427}}}.

\bibitem {MR777324}
\btxnamefont {\btxlastnamefont {Katsnelson},~V.\btxfnamespaceshort
  E.}\btxauthorcolon\ \btxtitlefont {Methods of {$J$}-theory in continuous
  interpolation problems of analysis. {P}art {I}}.
\newblock \btxpublisherfont {T. Ando, Hokkaido University, Sapporo}, 1985.
\newblock Translated from the Russian and with a foreword by T. Ando.

\bibitem {MR1473259}
\btxnamefont {\btxlastnamefont {Katsnelson},~V.\btxfnamespaceshort
  E.}\btxauthorcolon\ \btxtitlefont {\btxifchangecase {On transformations of
  {P}otapov's fundamental matrix inequality}{On transformations of {P}otapov's
  fundamental matrix inequality}}.
\newblock \Btxinshort {.}\ \btxtitlefont {Topics in interpolation theory
  ({L}eipzig, 1994)}, \btxvolumeshort {.}~\btxvolumefont {95} \btxofseriesshort
  {.}\ \btxtitlefont {Oper. Theory Adv. Appl.}, \btxpagesshort {.}\ 253--281.
  \btxpublisherfont {Birkh\"auser, Basel}, 1997.

\bibitem {MR703593}
\btxnamefont {\btxlastnamefont {Kovalishina},~I.\btxfnamespaceshort
  V.}\btxauthorcolon\ \btxjtitlefont {\btxifchangecase {Analytic theory of a
  class of interpolation problems}{Analytic theory of a class of interpolation
  problems}}.
\newblock \btxjournalfont {Izv. Akad. Nauk SSSR Ser. Mat.}, 47(3):455--497,
  1983\ifbtxprintISSN {, \mbox{\btxISSN~\btxISSNfont {0373-2436}}}.

\bibitem {MR1009144}
\btxnamefont {\btxlastnamefont {Kovalishina},~I.\btxfnamespaceshort
  V.}\btxauthorcolon\ \btxjtitlefont {\btxifchangecase {A multiple boundary
  value interpolation problem for contracting matrix functions in the unit
  disk}{A multiple boundary value interpolation problem for contracting matrix
  functions in the unit disk}}.
\newblock \btxjournalfont {Teor. Funktsi\u\i\ Funktsional. Anal. i Prilozhen.},
  (51):38--55, 1989\ifbtxprintISSN {, \mbox{\btxISSN~\btxISSNfont
  {0321-4427}}}.

\bibitem {MR0044591}
\btxnamefont {\btxlastnamefont {Kre{\u\i}n},~M.\btxfnamespaceshort
  G.}\btxauthorcolon\ \btxjtitlefont {\btxifchangecase {The ideas of {P}. {L}.
  \v {C}eby\v sev and {A}. {A}. {M}arkov in the theory of limiting values of
  integrals and their further development}{The ideas of {P}. {L}. \v {C}eby\v
  sev and {A}. {A}. {M}arkov in the theory of limiting values of integrals and
  their further development}}.
\newblock \btxjournalfont {Uspehi Matem. Nauk (N.S.)}, 6(4 (44)):3--120,
  1951\ifbtxprintISSN {, \mbox{\btxISSN~\btxISSNfont {0042-1316}}}.

\bibitem {MR0233157}
\btxnamefont {\btxlastnamefont {Kre{\u\i}n},~M.\btxfnamespaceshort
  G.}\btxauthorcolon\ \btxjtitlefont {\btxifchangecase {The description of all
  solutions of the truncated power moment problem and some problems of operator
  theory}{The description of all solutions of the truncated power moment
  problem and some problems of operator theory}}.
\newblock \btxjournalfont {Mat. Issled.}, 2(vyp. 2):114--132,
  1967\ifbtxprintISSN {, \mbox{\btxISSN~\btxISSNfont {0542-9994}}}.

\bibitem {MR0458081}
\btxnamefont {\btxlastnamefont {Kre{\u\i}n},~M.\btxfnamespaceshort G.}
  \btxandshort {.}\ \btxnamefont {A.\btxfnamespaceshort A. \btxlastnamefont
  {Nudel{\cprime}man}}\btxauthorcolon\ \btxtitlefont {The {M}arkov moment
  problem and extremal problems}.
\newblock \btxpublisherfont {American Mathematical Society}, Providence, R.I.,
  1977\ifbtxprintISBN {, \mbox{\btxISBN~\btxISBNfont {0-8218-4500-4}}}.
\newblock Ideas and problems of P. L. {\v{C}}eby{\v{s}}ev and A. A. Markov and
  their further development, Translated from the Russian by D. Louvish,
  Translations of Mathematical Monographs, Vol. 50.

\bibitem {Mak14}
\btxnamefont {\btxlastnamefont {Makarevich},~T.}\btxauthorcolon\ \btxtitlefont
  {Ein matrizielles {M}omentenproblem vom {S}tieljes-{T}yp}.
\newblock \btxifchangecase {Dissertation}{Dissertation}, Universit{\"a}t
  Leipzig, 2014.

\bibitem {MP11}
\btxnamefont {\btxlastnamefont {M{\"u}hling},~C.} \btxandshort {.}\
  \btxnamefont {C.~\btxlastnamefont {Pfeffing}}\btxauthorcolon\ \btxtitlefont
  {\btxifchangecase {{\"U}ber ein matrizielles {P}otenzmomentenproblem vom
  {S}tieltjes-{T}yp}{{\"U}ber ein matrizielles {P}otenzmomentenproblem vom
  {S}tieltjes-{T}yp}}.
\newblock \btxifchangecase {Diplomarbeit}{Diplomarbeit}, Universit{\"a}t
  Leipzig, 2011.

\bibitem {MR0163346}
\btxnamefont {\btxlastnamefont {Rosenberg},~M.}\btxauthorcolon\ \btxjtitlefont
  {\btxifchangecase {The square-integrability of matrix-valued functions with
  respect to a non-negative {H}ermitian measure}{The square-integrability of
  matrix-valued functions with respect to a non-negative {H}ermitian measure}}.
\newblock \btxjournalfont {Duke Math. J.}, 31:291--298, 1964\ifbtxprintISSN {,
  \mbox{\btxISSN~\btxISSNfont {0012-7094}}}.

\bibitem {MR1631843}
\btxnamefont {\btxlastnamefont {Sakhnovich},~L.\btxfnamespaceshort
  A.}\btxauthorcolon\ \btxtitlefont {Interpolation theory and its
  applications}, \btxvolumeshort {.}\ \btxvolumefont {428} \btxofseriesshort
  {.}\ \btxtitlefont {Mathematics and its Applications}.
\newblock \btxpublisherfont {Kluwer Academic Publishers, Dordrecht},
  1997\ifbtxprintISBN {, \mbox{\btxISBN~\btxISBNfont {0-7923-4830-3}}}.

\bibitem {Sch11}
\btxnamefont {\btxlastnamefont {Scheithauer},~P.}\btxauthorcolon\ \btxtitlefont
  {\btxifchangecase {{\"U}ber ein matrizielles {P}otenzmomentenproblem vom
  {S}tieltjes-{T}yp}{{\"U}ber ein matrizielles {P}otenzmomentenproblem vom
  {S}tieltjes-{T}yp}}.
\newblock \btxifchangecase {Diplomarbeit}{Diplomarbeit}, Universit{\"a}t
  Leipzig, 2011.

\bibitem {MR1627806}
\btxnamefont {\btxlastnamefont {Simon},~B.}\btxauthorcolon\ \btxjtitlefont
  {\btxifchangecase {The classical moment problem as a self-adjoint finite
  difference operator}{The classical moment problem as a self-adjoint finite
  difference operator}}.
\newblock \btxjournalfont {Adv. Math.}, 137(1):82--203, 1998\ifbtxprintISSN {,
  \mbox{\btxISSN~\btxISSNfont {0001-8708}}}.

\bibitem {MR1508747}
\btxnamefont {\btxlastnamefont {Stieltjes},~T.\btxfnamespaceshort
  J.}\btxauthorcolon\ \btxjtitlefont {\btxifchangecase {Quelques recherches sur
  la th\'eorie des quadratures dites m\'ecaniques}{Quelques recherches sur la
  th\'eorie des quadratures dites m\'ecaniques}}.
\newblock \btxjournalfont {Ann. Sci. \'Ecole Norm. Sup. (3)}, 1:409--426,
  1884\ifbtxprintISSN {, \mbox{\btxISSN~\btxISSNfont {0012-9593}}}.
\newblock {\latintext
  \btxurlfont{http://www.numdam.org/item?id=ASENS_1884_3_1__409_0}}.

\bibitem {MR1508159}
\btxnamefont {\btxlastnamefont {Stieltjes},~T.\btxfnamespaceshort
  J.}\btxauthorcolon\ \btxjtitlefont {\btxifchangecase {Recherches sur les
  fractions continues}{Recherches sur les fractions continues}}.
\newblock \btxjournalfont {Ann. Fac. Sci. Toulouse Sci. Math. Sci. Phys.},
  8(4):J1--J122, 1894\ifbtxprintISSN {, \mbox{\btxISSN~\btxISSNfont
  {0996-0481}}}.
\newblock {\latintext
  \btxurlfont{http://www.numdam.org/item?id=AFST_1894_1_8_4_J1_0}}.

\bibitem {Thi06}
\btxnamefont {\btxlastnamefont {Thiele},~H.\btxfnamespaceshort
  C.}\btxauthorcolon\ \btxtitlefont {Beitr{\"a}ge zu matriziellen
  {P}otenzmomentenproblemen}.
\newblock \btxifchangecase {Dissertation}{Dissertation}, Universit{\"a}t
  Leipzig, \btxprintmonthyear{.}{5}{2006}{short}.

\end{thebibliography}
\bibliographystyle{bababbrv}

\vfill
\begin{minipage}{0.48\textwidth}
 Universit\"at Leipzig\\
 Fakult\"at f\"ur Mathematik und Informatik\\
 PF~10~09~20\\
 D-04009~Leipzig
\end{minipage}
\begin{minipage}{0.48\textwidth}
 \begin{flushright}
  \texttt{
   fritzsche@math.uni-leipzig.de\\
   kirstein@math.uni-leipzig.de\\
   maedler@math.uni-leipzig.de\\
   makarevichtanya@yahoo.com
  } 
 \end{flushright}
\end{minipage}

\end{document}